\documentclass[a4paper,11pt]{amsart}
\usepackage[utf8]{inputenc}
\usepackage[T1]{fontenc}
\usepackage{lmodern}
\usepackage[english]{babel}
\usepackage{microtype}
\usepackage{amsmath,amssymb,amsfonts,amsthm}
\usepackage{cases}
\usepackage{mathtools,accents}
\usepackage{mathrsfs}
\usepackage{aliascnt}
\usepackage{bm}
\usepackage{bbm}
\usepackage{tikz}
\usepackage[citecolor=blue,colorlinks]{hyperref}
\usepackage{verbatim}
\usepackage{enumerate}
\usepackage{tikz-cd}
\usepackage{fullpage}

\makeatletter
\g@addto@macro\@floatboxreset\centering
\makeatother

\makeatletter
\def\newaliasedtheorem#1[#2]#3{
  \newaliascnt{#1@alt}{#2}
  \newtheorem{#1}[#1@alt]{#3}
  \expandafter\newcommand\csname #1@altname\endcsname{#3}
}
\makeatother

\newtheorem{theorem}{Theorem}[section]
\newtheorem{lemma}[theorem]{Lemma}
\newtheorem{corollary}[theorem]{Corollary}
\newtheorem{proposition}[theorem]{Proposition}
\newtheorem{conjecture}[theorem]{Conjecture}
\newtheorem{remark}[theorem]{Remark}

\theoremstyle{definition}
\newtheorem{definition}[theorem]{Definition}
\newtheorem{notation}[theorem]{Notation}
\newtheorem{question}[theorem]{Question}
\newtheorem{example}[theorem]{Example}

\numberwithin{equation}{section}

\newcommand{\p}{\partial}

\newcommand{\bl}{\bm{l}}
\newcommand{\btheta}{\bm{\vartheta}}

\newcommand{\C}{\mathbb{C}}

\newcommand{\cA}{\mathcal A}
\newcommand{\dH}{\mathbb{H}}
\newcommand{\R}{\mathbb{R}}
\newcommand{\dP}{\mathbb{P}}
\newcommand{\CP}{\mathbb{CP}}

\newcommand{\dT}{\mathbb{T}}
\newcommand{\dZ}{\mathbb{Z}}
\newcommand{\dS}{\mathbb{S}}

\newcommand{\mv} {\mathfrak v}
\newcommand{\mw} {\mathfrak w}
\newcommand{\ms} {\mathfrak s}

\newcommand{\mI} {\mathfrak I}
\newcommand{\mJ} {\mathfrak J}
\newcommand{\mR} {\mathfrak R}
\newcommand{\mt} {\mathfrak t}
\newcommand{\mL} {\mathfrak L}
\newcommand{\mC} {\mathfrak C}

\newcommand{\bP}{\mathbb P}
\newcommand{\bH}{\mathbb H}

\DeclareMathOperator{\dvol}{dvol}

\DeclareMathOperator{\Id}{Id}

 \DeclareMathOperator{\Lie}{Lie}

\DeclareMathOperator{\Ric}{Ric}
\DeclareMathOperator{\Rm}{Rm}

\DeclareMathOperator{\bbv}{\mathbbm{v}}
\DeclareMathOperator{\dist}{dist}

\DeclareMathOperator{\I}{I}
\DeclareMathOperator{\II}{II}
\DeclareMathOperator{\III}{III}

\DeclareMathOperator{\Tr}{Tr}
\DeclareMathOperator{\Vol}{Vol}

\DeclareMathOperator{\bv}{{\bf v}}
\DeclareMathOperator{\Aea}{AE_\alpha}
\DeclareMathOperator{\Ae}{AE}
\DeclareMathOperator{\ALe}{ALE}
\DeclareMathOperator{\ALF}{ALF}

\DeclareMathOperator{\TN}{TN}
\DeclareMathOperator{\dTN}{dTN}
\DeclareMathOperator{\TB}{TB}
\DeclareMathOperator{\EH}{EH}

\DeclareMathOperator{\TNap}{TN^+_\alpha}
\DeclareMathOperator{\TNam}{TN^{--}_\alpha}
\DeclareMathOperator{\TNp}{TN^+}
\DeclareMathOperator{\TNm}{TN^{--}}
\DeclareMathOperator{\ALE}{ALE}

\DeclareMathOperator{\ALFp}{ALF^+}
\DeclareMathOperator{\ALFm}{ALF^{--}}

\DeclareMathOperator{\Kerr}{Kerr}
\DeclareMathOperator{\AF}{AF}
\DeclareMathOperator{\AFa}{AF_\beta}

\usetikzlibrary{decorations.pathreplacing}

\title{Gravitational instantons and harmonic maps}
\author{
Mingyang Li\textsuperscript{*}}
\thanks{\textsuperscript{*} Department of Mathematics, University of California, Berkeley; \url{mingyang_li@berkeley.edu}}
\author{Song Sun\textsuperscript{\dag}}
\thanks{\dag\  Institute for Advanced Study in Mathematics, Zhejiang University; \url{songsun@zju.edu.cn}}

\begin{document}

\begin{abstract}
We study the interaction between toric Ricci-flat metrics in dimension 4 and axisymmetric harmonic maps from the 3-dimensional Euclidean space into the hyperbolic plane. Applications include 
\begin{itemize}
    \item The construction of complete Ricci-flat 4-manifolds that are non-spin, simply-connected, and with arbitrary $b_2$. Our method is non-perturbative and is based on ruling out conical singularities arising from axisymmetric harmonic maps. When $b_2$ is at most $2$, these recover the Euclidean Kerr metrics and the Chen-Teo metrics; when $b_2$ is bigger than $2$, these spaces are not locally Hermitian in any orientation and they provide infinitely many new diffeomorphism types of complete Ricci-flat 4-manifolds. They also give systematic counterexamples to various versions of the Riemannian black hole uniqueness conjecture. 
    \item  A PDE classification result for axisymmetric harmonic maps of degree at most $1$, via the Gibbons-Hawking ansatz and the LeBrun-Tod ansatz, in terms of axisymmetric harmonic functions. This is motivated by the study of hyperk\"ahler and conformally K\"ahler gravitational instantons.
\end{itemize}
  
\end{abstract}

\maketitle

\tableofcontents

\section{Introduction}
 
 In this paper, a \emph{gravitational instanton} is defined as a complete oriented Ricci-flat 4-dimensional  manifold $(M, g)$ that satisfies the \emph{finite energy} condition \begin{equation}\int_{M}|\Rm_g|^2\dvol_g<\infty.
 \end{equation}
 This notion was first introduced by Hawking \cite{Hawking} to study Euclidean quantum gravity. Mathematically gravitational instantons provide the most basic singularity models for Einstein metrics in Riemannian geometry. It is a fundamental question to construct and classify gravitational instantons.

 In general, it is difficult to directly tackle the fully non-linear PDE system governing gravitational instantons.  Almost all known gravitational instantons,  up to orientation reversal, exhibit \emph{special} geometry, i.e.,  they are locally \emph{Hermitian} (indeed, hyperk\"ahler or conformally K\"ahler) with respect to some complex structure. These special gravitational instantons are closely related to complex geometry and mathematical physics. They have been widely studied, and by now we have a satisfactory classification theory. However, little is known previously about gravitational instantons without special geometry.

The main goal of this paper is to construct gravitational instantons on infinitely many new diffeomorphism types of 4-manifolds. These new metrics are not locally Hermitian. The underlying oriented smooth 4-manifolds are 
\begin{equation}\label{eqn:definition of Xn}X_n:=\begin{cases}(S^2\times \R^2)\# k {\CP^2}\# k \overline{\CP^2},  & \text{if } n=2k+1;
\\ (S^2\times \R^2)\# k{\CP^2} \# (k+1)\overline{\CP^2}, & \text{if } n=2k+2.
\end{cases}
\end{equation}
 The metrics we construct are asymptotic to a flat quotient $(\R^3\times \R^1)/\langle q_\beta\rangle$ (for any $\beta \in (0, 1)$),  where $q_\beta$ acts as a rotation by angle $2\pi\beta$ on $\R^3$ and a translation by  $1$ on $\R^1$. These asymptotic models are referred to as the $\AFa$  (asymptotically flat with rotation angle $2\pi\beta$; for a precise definition, see Definition \ref{def:Type of gravitational instantons}) models. Such models have cubic volume growth but the asymptotic cones at infinity are  more complicated: when $\beta$ is rational, the asymptotic cone is a 3-dimensional flat cone; when $\beta$ is irrational, the asymptotic cone is the 2-dimensional half-plane.

\begin{theorem}\label{t:main}
	For each $n\geq 1$ and $\beta\in (0, 1)$, there is an $\AFa$ gravitational instanton on $X_n$.
\end{theorem}

 Some remarks are in order.

 \

 (1). Since $X_n$ is non-spin, there is no hyperk\"ahler gravitational instanton on $X_n$, with respect to either orientation. The metrics constructed above are Hermitian if and only if $n\in \{1, 2\}$. In these cases, the metrics coincide with the known ones: a Euclidean Kerr metric (when $n=1$) or a Chen-Teo metric \cite{Chen-Teo} (when $n=2$). Our method for proving Theorem \ref{t:main} recovers them in these special cases and provides a geometric explanation for the existence of these metrics. When \( n \geq 3 \), previous classification results by the first author \cite{Li2} show that $X_n$ cannot support a  Hermitian gravitational instanton (for both orientations). Therefore, the metrics constructed in Theorem~\ref{t:main} provide new diffeomorphism types of smooth 4-manifolds that support gravitational instantons.

\

(2). Theorem \ref{t:main} implies that with a fixed asymptotics at infinity, there is no uniform bound on the topological type of gravitational instantons. This is new and different from the hyperk\"ahler and the Hermitian setting. 

\

 (3). Theorem \ref{t:main} is related to  questions of Gibbons \cite{Gibbons}, Lapedes \cite{Lapedes}  and Yau (\cite{Yau1982problem}, Problem 116) regarding the Riemannian version of the black hole uniqueness conjecture. In general relativity, the  black hole uniqueness conjecture states that the only asymptotically flat stationary vacuum Einstein space-time is given by the Kerr space-time.  The naive Riemannian version of the conjecture  was disproved by Chen-Teo \cite{Chen-Teo},  who wrote down explicit families of asymptotically flat toric gravitational instantons on $X_1$. More recently, it was discovered by Aksteiner-Andersson \cite{AA} that the Chen-Teo spaces are Hermitian (or one-sided Petrov Type D, in the language of relativity); subsequently Biquard-Gauduchon \cite{BG} fits the Chen-Teo spaces in the framework of  toric complex geometry via the LeBrun-Tod ansatz. In \cite{AADNS} a modified uniqueness conjecture was proposed, which says that asymptotically flat gravitational instantons must be Hermitian. Theorem \ref{t:main} showed this is still too optimistic. In Section \ref{sec:discussion} we will discuss versions of the uniqueness conjecture which are more likely to hold. 

\

(4). Our construction is non-perturbative and non-explicit. It remains an interesting question to understand the geometry of these metrics.  When the parameters degenerate (i.e. as $\beta\rightarrow0$ or $\beta\rightarrow 1$), one can show that the metrics split in the pointed Gromov-Hausdorff sense to a finite union of Taub-NUT, Schwarzschild and Taub-Bolt metrics (under suitable orientations). See the discussion in Section \ref{sec:discussion}. For the Euclidean Kerr spaces, with some computation one can explicitly see the splitting into the union of a Taub-NUT space and an anti-Taub-NUT space (i.e., a negatively oriented Taub-NUT space), See Example \ref{ex:Kerr}.

\

(5). In this paper we focus on constructing the special family of examples in Theorem \ref{t:main}. It is possible that similar techniques can be exploited to construct $\AFa$ gravitational instantons on non-compact 4-manifolds with other  topologies. One may also attempt to apply the same strategy  to construct gravitational instantons of other Asymptotic Types, for example, the $\ALF$ or $\ALE$ asymptotics. See Section \ref{sec:discussion} for more discussion.

\

To prove Theorem \ref{t:main}, we study gravitational instantons with toric symmetry, using harmonic maps. In the Lorentzian setting there is a well-known reduction of the axisymmetric stationary vacuum Einstein equation to an axisymmetric harmonic map from $\mathbb R^3$ into the hyperbolic plane $\mathcal H$ \cite{MM, Ernst, Carter, BZ, WS}. Similar reductions also exist in the Riemannian setting (see, for example, \cite{Lott, Kunduri}). In both settings, the topology of the 4-dimensional space is encoded in the data called a \emph{rod structure} (see Definition \ref{def:rod structures} for the precise definition in our specific setting), and we are interested in axisymmetric harmonic maps which are \emph{strongly tamed} by a rod structure (see Definition \ref{def:strong tameness}). Using  general theory of harmonic maps (see, for example, \cite{Weinstein,  Kunduri}), we know the existence and uniqueness of axisymmetric harmonic maps which are tamed by a given rod structure and with desired asymptotics at infinity (see  Theorem \ref{t:existence result}). Together with a choice of an enhancement (see Definition \ref{def:enhancement}) we obtain a 4-dimensional toric Ricci-flat metric, which however may have \emph{conical singularities} (see Definition \ref{def:cone angle}) along a union of  codimension 2 submanifolds. In the Lorentzian setting, the axisymmetric stationary case of the black hole uniqueness conjecture is equivalent to the statement that conical singularities must be present unless the spacetime is a Kerr space-time. The difficulty in proving this conjecture lies in the lack of explicit knowledge relating the cone angles to the rod structure.  

\

In our Riemannian setting, by contrast, we will show that solutions without conical singularities exist in abundance, which give rise to toric gravitational instantons. More precisely, for each $n\geq 1$ we will design a specific rod structure which gives rise to the smooth 4-manifold $X_n$ and such that, by adjusting the rod lengths, we can achieve that the cone angles along the rods are all equal to $2\pi$, so conical singularities are ruled out.  Our argument is based on a topological fixed point theorem.  The key observation is that, although there is in general no explicit knowledge on the cone angles, it is still possible in our setting to read out the behavior of cone angles when the rods degenerate to zero or infinity. To achieve the latter, we need to perform a bubbling analysis of axisymmetric harmonic maps. We show that the bubble limits are, up to bounded hyperbolic distance, modeled on four types of explicit axisymmetric harmonic maps (see Section \ref{subsec:local regularity} for the definitions): $\Ae$ (Asymptotically Euclidean), $\AF_0$ (Asymptotically Flat), $\ALF^+$ (positively oriented Asymptotically Locally Flat), $\ALF^-$ (negatively oriented Asymptotically Locally Flat). These are the harmonic maps generalizing those associated to the standard gravitational instantons, given respectively by the flat $\R^4$, the flat product $S^1\times \R^3$, the Taub-NUT space and the anti-Taub-NUT space. Heuristically, these model spaces serve as the fundamental building blocks for constructing the gravitational instantons in Theorem \ref{t:main}. To realize the above strategy, there are a few subtleties

\begin{itemize}
    \item In general, there are topological constraints on the rod structure of a toric gravitational instanton (see for example \cite{Nilsson}). This  means that we need to design the correct rod structure to start with (see Section \ref{ss:general case}). 
    \item  In general, it is not true that for an arbitrarily given rod structure one can always tune the cone angle by varying the rod lengths: for example, in the  hyperk\"ahler case (i.e. the case of degree $0$, see Remark \ref{rem:angle invariant in Type I}) the cone angles stay \emph{constant} when we vary the rod lengths. Our observation is that in suitable settings the cone angles can be shown to diverge to infinity or go to a number smaller than $2\pi$, due to the degeneration of the enhancement in the process of the bubbling convergence, see Proposition \ref{p:cone angle AF rod length go to zero}, Proposition \ref{p:deep scale angle goes to infinity}, Proposition \ref{p:cone angle AF rod length go to infinity}, Proposition \ref{p:cone angle TN rod length go to infinity}. Clearly, this can only occur in the non-hyperk\"ahler case. 
    \item In the scale when the enhancement does not degenerate, there is again no precise information that can be said on the limit cone angles. Here the new ingredient is that there is an \emph{angle comparison theorem} (see Proposition \ref{p:angle comparison}), which is motivated by the Bishop-Gromov volume comparison for Riemannian metrics. It provides some key inequalities in our argument.
\end{itemize}

\

In the reverse direction, we will prove a PDE classification result for axisymmetric harmonic maps, motivated by the theory of gravitational instantons.     
 Locally, an axisymmetric harmonic map $\Phi$ into the hyperbolic plane defines a 4-dimensional toric Ricci-flat metric $g$. If $g$ has special geometry (i.e., locally Hermitian), then $\Phi$ can be  described explicitly in terms of an axisymmetric harmonic function: this is given by the \emph{Gibbons-Hawking ansatz} if $g$ is hyperk\"ahler and by the \emph{LeBrun-Tod ansatz} if $g$ is Hermitian non-K\"ahler. In this paper, we give a characterization of these special axisymmetric harmonic maps in terms of a topological invariant of the rod structure. Namely, we define a notion of \emph{degree} for a rod structure (see Definition \ref{def:degrees}), which takes value in the set of non-negative integers and distinguishes different connected components of the set of all rod structures of a given  Asymptotic Type. We prove the following
\begin{theorem}\label{t:main2}
    Given an axisymmetric harmonic map $\Phi$  which is strongly tamed by a rod structure $\mR$ (see  Definition \ref{def:strong tameness}), then 
    \begin{itemize}
        \item $\Phi$ defines a locally hyperk\"ahler metric if and only if $\mR$ has degree 0;
        \item $\Phi$ defines a locally Hermitian non-K\"ahler metric if and only if $\mR$ has degree 1.
    \end{itemize}
\end{theorem}

As a consequence, when $\mR$ has degree at most $1$, $\Phi$ can be written in terms of an axisymmetric harmonic function, and the cone angles can be explicitly determined (see Remark \ref{rem:angle invariant in Type I} and Remark \ref{rem: angle in Type II}). In the degree 0 case, the proof is proceeded by showing that the Gibbons-Hawking construction yields all the possible degree 0 rod structures and then by appealing to the uniqueness of axisymmetric harmonic maps. In the degree 1 case our proof makes use of the geometry of Hermitian gravitational instantons.  The heart of the argument is to show that the class of axisymmetric harmonic maps locally generating Hermitian non-K\"ahler Ricci-flat metrics is open under deformation of rod structures. The latter then relies on the result of Biquard-Gauduchon-LeBrun \cite{BGL} for gravitational instantons.

\

This paper is organized as follows. In Section \ref{sec:dimension reduction} we recall the dimension reduction of 4-dimensional toric Ricci-flat metrics to augmented, axisymmetric harmonic maps  into the hyperbolic plane. In Section \ref{sec:rod structures} we make precise definitions of rod structures and enhancements and describe the rod structure and the axisymmetric harmonic maps for a few key examples of toric gravitational instantons. In Section \ref{sec:harmonic maps} we discuss tame harmonic maps (i.e. axisymmetric harmonic maps tamed by a rod structure) and study their local and global regularity. In particular, we define the \emph{cone angle} along each rod. Then in the case of  Asymptotic Type  $\Ae$ we prove a rigidity result (i.e., globally tame to strongly tame) and an angle comparison theorem (as a result of the Bishop volume comparison). In Section \ref{sec:degeneration of rod structures} we analyze the bubbling behavior of tame harmonic maps when the rod structure degenerates. This is done by constructing good background maps which are almost harmonic. This analysis allows us to gain information on the asymptotic behavior of the cone angles when the rod structure degenerates. In Section \ref{sec:applications} we apply the results in Section \ref{sec:degeneration of rod structures} to prove Theorem \ref{t:main}. In Section \ref{sec:degree 0 and 1} we prove Theorem \ref{t:main2}. Section \ref{sec:discussion} is devoted to further discussion and open questions. In particular, we determine the splitting of the gravitational instantons constructed in Theorem \ref{t:main} when $\beta$ approaches $0$.

\

 {\bf Acknowledgements}: We thank Lars Andersson for helpful discussions on gravitational instantons with symmetry and Marcus Khuri for helpful comments on the first version of this paper. We thank Simon Donaldson, Claude LeBrun and John Lott for their interest in this work.  The first author is grateful to the IASM at Zhejiang University for hospitality during his visit in the year 2024 and Spring 2025. The second author thanks Shaosai Huang for conversations on related topics in Fall 2021. The first author was partially supported by NSF Grant DMS-2304692.

\

 {\bf Notations and Conventions}
 
 \begin{itemize}
 	\item We denote by $\mathbb H=\{(\rho, z)|\rho\geq 0\}$ the right half plane with boundary. We also use polar coordinates $(r, \theta)$ ($\theta\in [0, \pi]$) on $\mathbb H$, where $\theta$ is the angle measured from the positive $z$-axis.  We have 
 \begin{align*}
    \rho = r\sin\theta, \ \ \ 
    z = r\cos\theta.
\end{align*}

\item We have the following actions on $\mathbb H$: 

Rescaling: 
\begin{equation}\label{eqn:dilation}\varphi_\lambda: \bH\rightarrow\bH; (\rho, z)\mapsto (\lambda^{-1}\rho, \lambda^{-1}z), \ \ \ \ \lambda>0.
\end{equation}

Translation: 
\begin{equation}\label{eqn:translation} T_{\underline z}:\bH\rightarrow\bH; (\rho, z)\mapsto (\rho, z-\underline z), \ \ \ \underline z\in \R.
\end{equation}
 	\item We denote by $\mathbb H^\circ$ the interior of $\dH$ and we identify $\Gamma=\p \mathbb H$ with the real line $\R$ parameterized by $z$. We denote 
    $$\Gamma_+=\Gamma\cap \{z>0\},$$
    $$\Gamma_-=\Gamma\cap \{z<0\}.$$
 	\item For an open subset $\Omega\subset \dH$, we denote  $\Omega^\circ=\Omega\cap \dH^\circ$. 
 	\item On $\R^3$, we use the standard coordinates $(x_1, x_2, x_3)$ and the cylindrical coordinates $(\rho, \phi, z)$. They are related via
\begin{align*}
    x_1 &= \rho\cos\phi, \\
    x_2 &= \rho\sin\phi, \\
    x_3 &= z.
\end{align*}We also identify the boundary $\Gamma=\p\mathbb H$ with the $x_3$-axis in $\R^3$.
\item We denote by $\mathcal H$ the hyperbolic upper half plane $\{(X,Y)|X>0\}$ endowed with the Riemannian hyperbolic metric $$g_{\mathcal H}=X^{-2}(dX^2+dY^2).$$
\item We also identify $\mathcal H$ with $SL(2;\R)/SO(2)$, the space of $2\times 2$ positive symmetric matrices of determinant 1. The identification is given by sending $(X, Y$) to the matrix
\begin{equation}\label{e:1.1}\left(\begin{matrix}X+X^{-1}Y^2&X^{-1}Y\\X^{-1}Y&X^{-1} \end{matrix}\right).
\end{equation}
Under this identification, for $B\in \mathcal H$,  for $B_1,  B_2\in T_B\mathcal{H}$, we have   $$g_{\mathcal H}( B_1, B_2)=\frac12\Tr(B^{-1} B_1B^{-1} B_2).$$ 
\item For two   matrices $A,B\in \mathcal H$, we denote by $d(A, B)$ the hyperbolic distance between them. It is easy to see that there is a monotonically increasing homeomorphism $F: [0, \infty)\rightarrow [0, \infty)$ such that for any $C\geq 0$,  
\begin{equation}\label{e:hyperbolic distance equivalent control}d(A, B)\leq C \Longleftrightarrow \Tr(A^{-1}B-Id)\leq F(C).	
\end{equation}
Moreover, we can see that $F(C)$ is uniformly comparable to $C^2$ for $C\in (0, 1]$.

\item We denote by $SL_\pm(2; \R)$ the subgroup of $GL(2; \R)$ with determinant $\pm1$ and by $SO_-(2)$ the subset of $O(2)$ with determinant $-1$. Denote 
$$PSO(2)=SO(2)/\{\pm1\}$$
$$PSO_-(2)=SO_-(2)/\{\pm 1\}.$$
 There is a natural (left) action of $SL_\pm(2;\R)$ on $\mathcal H$, given by 
 \begin{equation}\label{eqn:group actionon Phi}D. \Phi=(D^{-1})^{\top}\Phi D^{-1}.
 \end{equation}
\item For $\lambda\neq 0$ and $\mu\in \R$, we denote the following matrices in $GL(2; \R)$
\begin{equation}\label{e:definition of matrices}D_\lambda=\left(\begin{matrix} \lambda &0\\ 0 & \lambda^{-1}\end{matrix} \right);
 \ \ \ \ \ 
	L_{\mu}=\left(\begin{matrix}
1 &0\\ \mu & 1	
\end{matrix}
\right);
\ \ \ \ \ 
	U_{\mu}=\left(\begin{matrix}
1 & \mu \\ 0 & 1	
\end{matrix}
\right);
\ \ \ \ \ 
	R_{\lambda}=\left(\begin{matrix}
0 & \lambda \\ \lambda & 0	
\end{matrix}
\right).
\end{equation}
\item A vector in $\R^2$ will be viewed as column vectors and denoted by $(a, b)^\top$ for $a, b\in \R$.
\item For $\alpha\in \R$ we denote the vector  
\begin{equation}\label{eqn: defintion of bvalpha}\bv_{\alpha}=(1,  \alpha)^{\top}.
\end{equation}
\item An element in $\R\dP^1$ is denoted  by fraktur symbols such as $\mv$, $\mw$, and so on. It is represented as $[(a, b)^\top]$ for a vector $(a, b)^{\top}\in \R^2$.
\item Given two elements $\mv\neq \mv'\in \R\bP^1$, there is a unique $P_{\mv, \mv'}\in PSO(2)$ and $\alpha>0$  such that 
\begin{equation}\label{e:definition of P}
    P_{\mv, \mv'}.\mv=\bv_{\alpha/2}, \quad P_{\mv, \mv'}.\mv'=\bv_{-{\alpha}/2}.
\end{equation}
We denote 
\begin{equation}\label{e:definition of alpha}
    \mv\wedge\mv':=\alpha.
\end{equation}
We also make a convention that $\mv\wedge \mv'=0$ if $\mv=\mv'$.
There is also a unique $\overline{P}_{\mv,\mv'}\in PSO_-(2)$ 
such that 
\begin{equation}
    \overline{P}_{\mv,\mv'}.\mv=\bv_{2/\alpha}, \quad \overline{P}_{\mv,\mv'}.\mv'=\bv_{-2/\alpha}.
\end{equation}
We have 
\begin{equation} \label{e:definition of bar P}
    \overline{P}_{\mv,\mv'}=R_1 P_{\mv,\mv'}.
\end{equation}
\item Given an element $\mv\in \R\bP^1$, we denote by $\mv^{\perp}\in \R\dP^1$ the  line orthogonal to $\mv$.  There is a unique element $P_{\mv}=P_{\mv, \mv}\in PSO(2)$ such that 
\begin{equation}\label{e:definition of P1}P_{\mv}.\mv=[(1, 0)^{\top}], \ \ P_{\mv}.\mv^{\perp}=[(0, 1)^{\top}].
\end{equation}

 \item We denote by $\dT$ a real 2-dimensional torus, viewed as a compact abelian Lie group and by  $\mL$ the co-character lattice in $\Lie(\dT)$, i.e, $\mL$ is the lattice in $\mt$ consisting of all elements $\xi$ such that $\exp(\xi)=\Id$. A $\mathbb Z$-basis of $\mL$ is called an integral basis of $\dT$.
 \end{itemize}

\section{Local theory}
\label{sec:dimension reduction}
\subsection{Dimension reduction of the Ricci-flat equation}
\label{subsec:reduction of Ricci-flat}
Let $(M,g)$ be an oriented  4-dimensional Riemannian manifold. Let $X$ be a Killing vector field on $M$ and denote by $\eta=X^{\flat}$ its dual 1-form. We have $\mathcal L_Xg=0$, which implies that
\begin{itemize}
    \item The 1-form $\eta$ is divergence free: \begin{equation}\label{divfree}d^*\eta=0\end{equation}
    and satisfies the equation 
     \begin{equation}\label{Bochner2} \nabla^*\nabla\eta=Ric.\eta.\end{equation}
    \item Acting on $p$-forms, $\mathcal L_X$ commute with $d^*$, so the formal adjoint $\mathcal L_X^*\cdot=\eta \wedge d^*\cdot+d^*(\eta\wedge\cdot)$  commutes with $d$: \begin{equation}\label{Liedaulcommuted}\mathcal L_X^*d=d\mathcal L_X^*.\end{equation}   
\end{itemize}
Combined with the usual Bochner formula we get 
\begin{equation}\label{Laplacian eta}d^*d\eta=\Delta_H\eta-d d^*\eta=2Ric.\eta.
\end{equation}
We also have $$\mathcal L_X^*(d\eta)=d\mathcal L_X^*\eta=d(\eta\wedge d^*\eta+d^*(\eta\wedge\eta))=0.$$
 and 
 $$\mathcal L_X^*(d\eta)=\eta\wedge d^*d\eta+d^*(\eta\wedge d\eta). $$
 In particular, if $g$ is Einstein,  then 
 $$d^*(\eta\wedge d\eta)=0. $$
Locally we can find a function  $\mathcal E$, called the \emph{Ernst potential}, which is unique up to constant addition,  such that  $$d\mathcal{E}=*(\eta\wedge d\eta).$$ 
The \emph{twist 1-form} $d\mathcal{E}$ measures the non-integrability of the orthogonal distribution to $X$.

 \begin{definition}
 A 4-dimensional \emph{toric} Ricci-flat manifold  consists of an oriented 4-dimensional Ricci-flat manifold $(M,g)$  and a 2-torus $\mathbb T$ action with the property that every point on $M$ has a connected stabilizer group.  We denote by $N$ the quotient space $M/\dT$.
 \end{definition}
 
 \begin{remark}
 The last condition is purely technical: it ensures that the quotient space $N$ is a smooth manifold with boundary. Without this condition, one would have to deal with orbifolds. 	
 \end{remark}

Let $\{X_1, X_2\}$ be an integral basis of $Lie(\dT)$. They can be viewed as Killing fields on $M$. Denote by $\eta_i$ the dual 1-form of $X_i$.  Then from the above discussion we have the Ernst potentials $\mathcal{E}_1, \mathcal{E}_2$ respectively.  We have 
$$\mathcal L_{X_i}\mathcal{E}_i=\langle \eta_i, d\mathcal{E}_i\rangle=*(\eta_i\wedge *d\mathcal{E}_i)=0.$$ 
Denote $$\Theta_i:=*(\eta_1\wedge\eta_2\wedge d\eta_i),\ i=1, 2.$$  They measure the non-integrability of the horizontal distribution orthogonal to the $\mathbb T$ orbits. Then 
$$\Theta_1=\langle\eta_2, d\mathcal{E}_1\rangle=\mathcal L_{X_2}\mathcal{E}_1$$ 
and 
$$\Theta_2=-\mathcal L_{X_1}\mathcal{E}_2.$$
Since $\mathcal L_{X_1}\eta_2=\mathcal L_{X_2}\eta_1=0$, we see that $d\Theta_1=d\Theta_2=0$, so $\Theta_1$ and $\Theta_2$ are both constants. 
 
\begin{definition}
	A 4-dimensional toric Ricci-flat manifold $(M, g)$ is called \emph{integrable} if the horizontal distribution is integrable, or equivalently, if $\Theta_1=\Theta_2=0$.
\end{definition}
The integrability property is also known in the literature as \emph{surface-orthogonality} \cite{Calderbank2002,Joyce1995}. In this paper, our main interest lies in the case when the $\mathbb T$ action has a fixed point,  so we will restrict our discussion to the integrable setting.

In the rest of this subsection we will derive the local dimensional reduction of the Ricci-flat equation. It is well-known that this results into an axisymmetric harmonic map into the hyperbolic plane, see for example  \cite{Lott, Kunduri}. For the convenience of readers we include a brief derivation here.  Denote by $M^\circ$ the set of points in $M$ with trivial stabilizer and denote $N^\circ=M^\circ/\dT$.

Locally in $M^\circ$ we can write the integrable metric $g$ in the form 
\begin{equation} \label{eqn-local form for T2 invariant metric}
	g=\sum_{i,j}G_{ij}(x_1, x_2)d\phi_id\phi_j+\sum_{\alpha, \beta}h_{\alpha\beta}(x_1, x_2)dx_\alpha dx_\beta,
\end{equation}
where $\{x_1, x_2\}$ are local coordinates on $N^\circ$, $\{\phi_1, \phi_2\}$ are fiber coordinates such that $\p_{\phi_i}=X_i$. We orient the $\dT$ fibers using the form $d\phi_1\wedge d\phi_2$, which then induces an orientation on the base $N^\circ$. If we change the choice of the integral basis via an element $A\in GL(2; \dZ)$ then these orientations are preserved if and only if $\det(A)=1$.

Notice that we have $\eta_i=\sum_jG_{ij}d\phi_j$. We write  the \emph{Gram matrix} $G=(G_{ij})>0$ as $$G=\rho\Phi,$$
 where $\rho=\sqrt{\det G}>0,$ and $\Phi$ is a map into the hyperbolic plane $\mathcal H$. 
Denote \begin{equation}\label{defining equation for J-2}J:=G^{-1}dG=\rho^{-1}d\rho\cdot \Id+\Phi^{-1}d\Phi.
\end{equation}

\noindent Below will use $\ \widehat{}\ $ to denote operators on the base and $\ \widetilde{} \ $ to denote operators on the fibers. Notice that
\begin{gather}
	d\phi_j\wedge \widetilde{*}\eta_i=\delta_{ij} \rho d\phi_1\wedge d\phi_2,\notag\\
	d(\rho^{-1}\widetilde{*}\eta_i)=0.\notag
\end{gather} 
\noindent We investigate the Ricci-flat equation for a metric $g$ of the form  \eqref{eqn-local form for T2 invariant metric}.  By  \eqref{Laplacian eta},  
\begin{equation}
	d^*d\eta_i=2Ric.\eta_i.
\end{equation}
We compute \begin{equation}
 	d\eta_i=\sum_jd G_{ij}\wedge d\phi_j=\rho\sum_j J_{ji}\wedge \rho^{-1}\eta_j,
\end{equation}
\begin{equation}
	*d\eta_i=-\rho\sum_j\widehat{*} J_{ji}\wedge \widetilde{*}\rho^{-1}\eta_j,
\end{equation}
\begin{equation}
d*d\eta_i=-\sum_j\widehat d(\rho\widehat{*}J_{ji})\wedge \widetilde{*}\rho^{-1}\eta_j.
\end{equation}
It follows that the horizontal components of $Ric.\eta_i$ always vanishes,  and its vertical component vanishes if and only if 
\begin{equation}
	\widehat d(\rho \widehat{*}J)=0.
\end{equation}
Next we deal with the horizontal part of the Ricci curvature. Here we use the Bochner-Weitzenb\"ock formula. For $f=f(x_1, x_2)$, we have
\begin{equation}\label{Weitzenbock1}
	\frac{1}{2}\Delta|\nabla f|^2=|\text{Hess}f|^2+\langle \nabla \Delta f, \nabla f\rangle+Ric(\nabla f, \nabla f),
\end{equation}
\begin{equation}\label{Weitzenbock2}
	\frac{1}{2}\widehat\Delta|\widehat{\nabla} f|^2=|\widehat{\text{Hess}}f|^2+\langle \widehat{\nabla} \widehat\Delta f, \widehat{\nabla} f\rangle+\widehat{Ric}(\widehat{\nabla} f, \widehat{\nabla} f).
\end{equation}
Notice that $\nabla f$ is the horizontal lift of $\widehat{\nabla}f$. 
\begin{lemma}
The horizontal leaves are totally geodesic.	
\end{lemma}
\begin{proof}
For a general Riemannian submersion, given a geodesic on the base and a lift of the initial point to a fiber, there is always a unique horizontal lift to a geodesic on the total space. So each horizontal tangent direction is tangential to a horizontal geodesic.  Since the horizontal distribution is integrable in our case, this shows that the horizontal leaves are totally geodesic. 
\end{proof}

From the lemma it follows that
\begin{equation}\label{eq:hess}
	\text{Hess} f=\widehat{\text{Hess}}f-df(\Pi), 
\end{equation}
where $\Pi$ is the second fundamental form of the fibers and is given by
\begin{equation}
	\Pi_{ij}=-\frac{1}{2}\widehat{\nabla} G_{ij}.
\end{equation}
So 
\begin{eqnarray*}
	|\text{Hess}f|^2&=&|\widehat{\nabla} ^2f|^2+\frac{1}{4}\sum_{i,j,k,l}G^{ik}G^{jl}\langle\widehat{\nabla} G_{ij}, \widehat{\nabla} f\rangle\langle\widehat{\nabla} G_{kl}, \widehat{\nabla} f\rangle\\&=&|\widehat{\nabla} ^2f|^2+\frac{1}{4}\Tr\langle J, \widehat{\nabla} f\rangle\langle J, \widehat{\nabla} f\rangle.
\end{eqnarray*}
Taking the trace of \eqref{eq:hess} we obtain
\begin{equation}
	\Delta f=\widehat\Delta f-\langle\widehat{\nabla} f, H\rangle,
\end{equation}
where $$H=-\widehat{\nabla} \log\rho$$ is the mean curvature vector field.
By comparing \eqref{Weitzenbock1} and \eqref{Weitzenbock2}, we obtain that the vanishing of the horizontal part of the Ricci curvature is equivalent to that for all $f$,
\begin{align*}
\widehat{Ric}(\widehat{\nabla} f, \widehat{\nabla} f)&=-\frac{1}{2}\langle \widehat{\nabla}\log\rho, \widehat{\nabla}|\widehat{\nabla} f|^2\rangle+\langle\widehat{\nabla}\langle \widehat{\nabla}\log\rho, \widehat{\nabla} f\rangle, \widehat{\nabla} f\rangle+\frac{1}{4}\Tr\langle J, \widehat{\nabla} f\rangle\langle J, \widehat{\nabla} f\rangle
	\\
&=\widehat{\text{Hess}}\log\rho (\widehat{\nabla} f, \widehat{\nabla} f)+\frac{1}{4}\Tr\langle J, \widehat{\nabla} f\rangle\langle J, \widehat{\nabla} f\rangle
	+\langle\widehat{\nabla} \log\rho, \widehat{\nabla}_{\widehat{\nabla} f}\widehat{\nabla} f-\frac{1}{2}\widehat{\nabla}|\widehat{\nabla} f|^2\rangle.
\end{align*}
It is easy to check that $\widehat{\nabla}_{\widehat{\nabla} f}\widehat{\nabla} f-\frac{1}{2}\widehat{\nabla}|\widehat{\nabla} f|^2\equiv0$. So we arrive at 
\begin{equation}
\widehat{Ric}=\widehat{\text{Hess}}\log \rho+\frac{1}{4}\Tr(J\otimes J).	
\end{equation}
Combining the above we get the dimension reduction of the Ricci-flat equation in terms of a system of two equations
\begin{numcases}{}
		\widehat d(\rho \widehat{*}J)=0, \label{dimension reduction equation 1}
\\
		\widehat{Ric}=\widehat{\text{Hess}}\log \rho+\frac{1}{4}\Tr(J\otimes J).	
\label{dimension reduction equation 2}
		\end{numcases}
Taking the trace of \eqref{dimension reduction equation 1} we see that $$\widehat{d}^*\widehat{d}\rho=0,$$ so $\rho$ is a harmonic function on $N^\circ$. Let $z$ be a local harmonic conjugate of $\rho$, that is, a function such that $\widehat dz=\widehat{*}\widehat d\rho$; it is well-defined up to an additive constant. 
\begin{definition}
	We say that $(M, g)$ is regular if $d\rho\neq 0$ on $M^\circ$. 
\end{definition}
Suppose $(M, g)$ is regular, then $\{\rho, z\}$ form local conformal coordinates on $N^\circ$ in which we may write 
$$\widehat g=e^{2\nu} (d\rho^2+dz^2)$$
for a smooth local function $\nu$.  Equation \eqref{dimension reduction equation 1} then becomes
\begin{equation}\label{eqn-EL-2}
	(\rho\Phi^{-1}\Phi_\rho)_\rho+(\rho\Phi^{-1}\Phi_z)_z=0.
\end{equation}
If we use the  cylindrical coordinates $(\rho, \phi, z)$ on $\R^3$ with  $x_1=\rho \cos\phi$, $x_2=\rho\sin\phi$ and $x_3=z$, then \eqref{eqn-EL-2} is exactly the equation for $\Phi$ which induces an axisymmetric harmonic map from a domain in $\R^3$ into $\mathcal H$. 

Denote 
$$\mathcal U=\Id+\rho\Phi^{-1}\Phi_\rho,$$ 
and 
$$\mathcal V=\rho\Phi^{-1}\Phi_z.$$ 
 By looking at different components of \eqref{dimension reduction equation 2}, it is equivalent to a combination of three equations below
\begin{gather} 
	\widehat{R}_{\rho z}=0 \Leftrightarrow \p_z\nu=\frac{1}{4\rho}\Tr(\mathcal U \mathcal V), \label{eqn-Ricci-1}\\
	\widehat{R}_{\rho\rho}=\widehat{R}_{zz}\Leftrightarrow \p_\rho\nu=-\frac{1}{2\rho}+\frac{1}{8\rho}\Tr(\mathcal U^2-\mathcal V^2), \label{eqn-Ricci-2}\\
	\widehat{R}_{zz}=\frac{1}{\rho}\partial_{\rho}\nu+\frac{1}{4\rho^2}\Tr(\mathcal V^2). \label{eqn-Ricci-3}
\end{gather}
One can  check that \eqref{eqn-Ricci-3}  follows from \eqref{eqn-EL-2}, \eqref{eqn-Ricci-1} and \eqref{eqn-Ricci-2} by taking further derivatives, hence is redundant. The compatibility condition for \eqref{eqn-Ricci-1} and \eqref{eqn-Ricci-2} also  follows from \eqref{eqn-EL-2}.

The upshot is that the Ricci-flat equation for $g$ is locally reduced to the following system of equations on the pair $(\Phi, \nu)$
\begin{align} \label{e:eqn2.33}
	\left\{\begin{aligned}
		&(\rho\Phi^{-1}\Phi_\rho)_\rho+(\rho\Phi^{-1}\Phi_z)_z=0.\\
		&\p_z\nu=\frac{1}{4\rho}\Tr(\mathcal U\mathcal V),\\
		&\p_\rho\nu=-\frac{1}{2\rho}+\frac{1}{8\rho}\Tr(\mathcal U^2-\mathcal V^2).
	\end{aligned}\right.
\end{align}

\subsection{Augmented harmonic maps}
 
Let $U$ be an open subset of $\dH^\circ$.  
\begin{definition}
	A smooth map $\Phi: U\rightarrow \mathcal H$ is called \emph{harmonic} if it induces an axisymmetric harmonic map from a domain $U\subset\R^3$ into $\mathcal{H}$, that is, $\Phi$ satisfies the equation
	\begin{equation}
		(\rho\Phi^{-1}\Phi_\rho)_\rho+(\rho\Phi^{-1}\Phi_z)_z=0.
	\end{equation}
\end{definition}
\begin{remark}
    The terminology of being harmonic used here is simply for notational convenience. It should not be confused with the usual harmonic map from the 2-dimensional space $U$.
\end{remark}

\begin{definition}
	An \emph{augmentation} for a harmonic map $\Phi:U\rightarrow\mathcal H$ is  a function $\nu$ on $U$ satisfying \eqref{e:eqn2.33}. 

A pair $(\Phi, \nu)$ is called an \emph{augmented} harmonic map.
\end{definition}

The above discussion says that a 4-dimensional regular, integrable, toric Ricci-flat metric gives rise to an augmented harmonic map. Conversely, given an augmented harmonic map $(\Phi, \nu)$ on an open subset $U\subset \dH^\circ$, on $U\times \R^2$ we can write down a Ricci-flat metric 
\begin{equation}\label{e:from aug harmonic map to Ricci flat}g=e^{2\nu}(dz^2+d\rho^2)+\rho\Phi,
\end{equation}
where we view $\rho\Phi$ as a family of flat metrics on $\R^2$ (with coordinates $\phi_1, \phi_2$ as above). Choosing any lattice $\Lambda\subset\R^2$ then yields a toric Ricci-flat metric on $U\times \R^2/\Lambda$.

Notice that if  $U$ is simply-connected, given a harmonic map $\Phi$, one can use \eqref{e:eqn2.33} to uniquely solve the function $\nu$ up to an additive constant. Giving an augmentation  amounts to fixing a choice of $\nu$.

\begin{remark}
  On the original 4-manifold $M$, the local coordinates $(\rho, z, \phi_1, \phi_2)$ are usually referred to as the {Weyl-Papapetrou} coordinates.
\end{remark}

Given an augmented harmonic map $(\Phi, \nu)$,  it is easy to check that 

\begin{itemize}
	\item  For any $\lambda>0$, $\underline z\in \mathbb R$, $(\varphi_\lambda^*T_{\underline z}^*\Phi, \varphi_\lambda^*T_{\underline z}^*\nu)$  defines an augmented harmonic map. Here $\varphi_\lambda$ and $T_{\underline z}$ are defined in \eqref{eqn:dilation} and \eqref{eqn:translation}.
	\item For $P \in SL_{\pm}(2,\mathbb{R})$, $(P. \Phi, \nu)$  defines an augmented harmonic map, where $P.\Phi$ is defined in \eqref{eqn:group actionon Phi}.
\end{itemize}

\subsection{Points of non-trivial stabilizer}\label{ss:nuts and bolts}

We introduce a few notations first. 

\begin{notation}
	Given a toric Ricci-flat metric ($M,g)$, we denote 
	\begin{itemize}
		\item by $M^\sharp$ the set of points fixed by the $\dT$ action;
		\item by $M^{*}$ the set of points with stabilizer isomorphic to $\dS^1$; 
        	\item $N^\sharp=M^\sharp/\dT\simeq M^\sharp$;
		\item $N^*=M^*/\dT$;
		\item $N=M/\dT$.
	\end{itemize} 
\end{notation}
It follows from definition that 
$$M=M^\circ\cup M^\sharp\cup M^*.$$

	Locally around a point $p\in M^*$, the $\dT$ action is modeled on the standard action of $\dT=\dS^1\times \dS^1$ on $\R^2\times \dS^1\times I$, where the first $\dS^1$ acts by rotation on the $\R^2$ factor and the second $\dS^1$ action by rotation on the $\dS^1$ factor. It is then easy to see that $N\setminus N^\sharp$ is naturally a smooth manifold with boundary and the quotient metric is  smooth up to the boundary.
	
	Locally around a point $p\in M^\sharp$, the $\dT$ action is modeled on the standard action of $\dT=\dS^1\times \dS^1$ on $\R^2\times \R^2$, where the first $\dS^1$ acts by rotation on the first $\R^2$ and the second $\dS^1$ acts by rotation on the second $\R^2$. By studying the local structure one can see that $M^\sharp$ (and also $N^\sharp$) consists of only isolated points and  $N$ is naturally a smooth manifold with corner at points in $N^\sharp$; the local model around $N^\sharp$ is the first quadrant in $\R^2$. The metric is  smooth up to the corner.

 \section{Rod structures and toric gravitational instantons}\label{sec:rod structures}
 \subsection{Rod structures} \label{ss:defintion of rod structures} To understand the global theory  we introduce the notion of an \emph{enhanced rod structure}. The terminology ``rod structure'' has been widely used in the physics literature, see for example \cite{OrlikRaymond, Harmark, Hollands, ChenTeo2, Kunduri}. Mathematically, a rod structure serves as a substitute  of the Delzant polytope in toric K\"ahler geometry when the complex structure is absent. Our precise definition used in this paper is given as follows, which suits our purpose  and differs slightly from what is used in most literature.
\begin{definition}[Rod structure]\label{def:rod structures}

A \emph{rod structure} \( \mathfrak{R} = (\{z_j\}_{j=1}^{n}, \{\mathfrak{v}_j\}_{j=0}^{n}, \sharp) \) consists of:  
\begin{itemize}
    \item A collection of \emph{turning points} \( z_1 = 0 < \dots < z_n \) on \( \Gamma := \mathbb{R} \).
    
    We also set \( z_0 = -\infty \) and \( z_{n+1} = +\infty \). 
    
    For \( 0 \leq j \leq n \), the open interval \( \mathfrak{I}_j := (z_j, z_{j+1}) \subset \Gamma \) is called a \emph{rod}.
    \item A collection of \emph{rod vectors} \( \mathfrak{v}_0, \dots, \mathfrak{v}_n \in \mathbb{RP}^1 \) satisfying the \emph{non-degeneracy condition} that \( \mathfrak{v}_j \neq \mathfrak{v}_{j+1} \) for \( 0 \leq j \leq n-1 \).
    
    Each \( \mathfrak{v}_j \) can be represented as \( \mathfrak{v}_j = [(a_j, b_j)^{\top}] \), where \( (a_j, b_j)^{\top} \) is a  vector in \( \mathbb{R}^2 \).
    \item An Asymptotic \emph{Type} $\sharp\in \{\Ae, \ALFp, \ALFm\}$ if $\mv_0\neq\mv_n$, and $\sharp \in \{\AFa|\beta\in \R)\}$ if $\mv_0=\mv_n$. 
\end{itemize}

The quantity \( \mathfrak{s}(\mathfrak{R}) = |z_n| \) is referred to as the \emph{size} of $\mR$. 

\

In the above we assume $n\geq 1 $. We also define a rod structure $\mR$ with $n=0$ to be one where the set of turning points is empty and there is exactly one rod vector $\mv_0$ (in particular, $\mv_0=\mv_n$).

When we do not specify the Type, we call $\mR$ an \emph{untyped rod structure}; when we relax the non-degeneracy condition we call $\mR$ a \emph{weak rod structure}. These notions will make convenient our later discussion. For a weak rod structure $\mR$, there is a natural rod structure $\mR'$ associated to $\mR$, obtained by combining all the adjacent rods with the same rod vectors into a single rod. 

 \end{definition} 
It is often helpful to represent an untyped rod structure by an unbounded polygon in $\R^2$ such that the length of each side represents the rod length and the normal direction of each side represents the rod vector. See Figure \ref{Figure:Example of polygon} for an illustration of a rod structure with 2 turning points, $\mv_0=\mv_2=[(1,0)^\top]$, $\mv_1=[(0, 1)^\top]$.

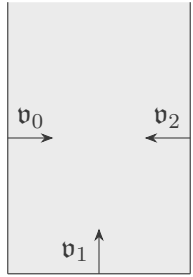
\begin{figure}[ht]
    \begin{tikzpicture}[scale=0.6]
        \draw (-2,6)--(-2,0);
        \draw (-2,0)--(2,0);
        \draw (2,0)--(2,6);

        \draw[-Stealth] (-2,3)--node [above]{$\mathfrak{v}_0$} (-1,3);
        \draw[-Stealth] (2,3)--node [above]{$\mathfrak{v}_2$} (1,3);
        \draw[-Stealth] (0,0)--node [left]{$\mv_1$} (0,1);

        \filldraw[color=gray!50,fill=gray!50,opacity=0.3](-2,6)--(-2,0) --(2,0)--(2,6);
    \end{tikzpicture}
	\caption{A polygon that represents an untyped rod structure.}
	 \label{Figure:Example of polygon}   
\end{figure}

 \begin{remark}\label{rem:term}
      Here $\Ae$ stands for {Asymptotically Euclidean}, $\ALF^{+}$ stands for  positively oriented {Asymptotically Locally Flat}, $\ALF^{-}$ stands for  negatively oriented {Asymptotically Locally Flat}, $\AFa$ stands for {Asymptotically Flat} with rotation angle $2\pi\beta$. These are possible Asymptotic Types of toric gravitational instantons that we will be interested in this paper, see Section \ref{ss:model examples} for model examples. When we discuss harmonic maps which are strongly tamed by a rod structure, the Asymptotic Type indeed uniquely determines the asymptotic model; see the discussion in Section \ref{ss:model maps and tameness}.
     \end{remark}

 \begin{definition}[Enhancement]\label{def:enhancement}
 	Given a rod structure \( \mathfrak{R} \), an \emph{enhancement} for \( \mathfrak{R} \) is a lattice \( \Lambda \subset \mathbb{R}^2 \) such that all the rod vectors can be represented by vectors in \( \Lambda \).  

A rod structure is \emph{enhanceable} if admits an enhancement.

	An \emph{enhanced rod structure} $(\mathfrak R, \Lambda)$ consists of a rod structure $\mathfrak R$ and an enhancement $\Lambda$ for $\mR$.

\end{definition}

\begin{definition}[Primitive rod vector]
Given an enhanced rod structure $(\mR, \Lambda)$, for each rod vector $\mv_j$, a \emph{primitive rod vector} $\bbv_j$ is defined to be a primitive vector in $\Lambda$ along the line defined by $\mv_j$.
\end{definition}

Notice that the primitive rod vector is unique up to multiplication by $-1$.  In particular, the norm $|\bbv_j|$ is a well-defined positive number. 
 
 \begin{definition}[Smoothness]
	An enhanced rod structure \( (\mathfrak{R}, \Lambda) \) is \emph{smooth} if, for any \( 0 \leq j \leq n-1 \), 
    $\Lambda$ is generated by primitive rod vectors \( \bbv_j \) along $\mv_j$ and \( \bbv_{j+1} \) along $\mv_{j+1}$.
\end{definition}

\begin{remark}
    Comparing with the notion of a rod structure used in most literature, the main difference here is that we have distinguished the lattice data as an ``enhancement'' for the rod structure. The main purpose is that in this way we have separated the analytic information and the topological information. As can be seen in Section \ref{sec:harmonic maps}, a rod structure defined as above  is sufficient for the analytic study of tame harmonic maps, whereas the enhancement is only necessary when we pass from harmonic maps to toric gravitational instantons. 
\end{remark}
\begin{definition}
We define the following operations on an untyped rod structure \( \mathfrak{R} = (\{z_j\}_{j=1}^n, \{\mathfrak{v}_j\}_{j=0}^n)\)
\begin{itemize}
    \item For \( \lambda > 0 \) and \( \underline{z} \in \mathbb{R} \), we define  
    \[
    \varphi_\lambda ^* \mathfrak{R} = \left(\{\lambda z_j\}_{j=1}^{n}, \{\mathfrak{v}_j\}_{j=0}^{n} \right),
    \]
    and  
    \[
    T_{\underline{z}}^* \mathfrak{R} = \left(\{ z_j + \underline{z} \}_{j=1}^{n}, \{\mathfrak{v}_j\}_{j=0}^{n} \right).
    \]
   
        \item For \( P \in GL_{+}(2; \mathbb{R}) \), we define  
    \[
    P . \mathfrak{R} = \left(\{z_j\}_{j=1}^n, \{P. \mathfrak{v}_j\}_{j=0}^n \right).
    \]
    \item For \( P \in GL_{-}(2; \mathbb{R}) \), we define  
    \[
    P . \mathfrak{R} = \left(\{z_n-z_{n+1-j}\}_{j=1}^n, \{P. \mathfrak{v}_{n-j}\}_{j=0}^n \right).
    \]
     \item { We define the \emph{opposite} rod structure
       \[
   \overline{\mathfrak{R}} = \left(\{z_j\}_{j=1}^n, \{R_1.\mathfrak{v}_{j}\}_{j=0}^n \right), 
    \]
    where $R_1$ is defined in \eqref{e:definition of matrices}.}
\end{itemize}

These also induce operations on an enhanced untyped rod structure: if $\mR$ is further enhanced by $\Lambda$, then we define
    $\varphi_\lambda^*\Lambda:=\Lambda$, 
 $T_{\underline z}^*\Lambda:=\Lambda$, and define $P.\Lambda$  by the natural action.
\end{definition}

\begin{definition}
    For a rod structure $\mathfrak R=(\{z_j\}_{j=1}^n, \{P. \mv_j\}_{j=0}^n, \sharp)$, 
  we define the \emph{opposite} rod structure
       \[
   \overline{\mathfrak{R}} = \left(\{z_j\}_{j=1}^n, \{R_1.\mathfrak{v}_{j}\}_{j=0}^n, \overline\sharp \right), 
    \]
    where $R_1$ is defined in \eqref{e:definition of matrices}.
    Here we make the convention that 
    \begin{equation}
        \overline{\ALFp}=\ALFm, \ \ \overline{\ALFm}=\ALFp, \ \ \overline{\Ae}=\Ae,  \ \ \overline{\AFa}=\AF_{-\beta}.
    \end{equation}

\end{definition}
\begin{remark}Notice that we do not define the other operations on Asymptotic Types. In particular, by our convention, when we perform the operation of rescaling, translation, or $GL(2;\R)$ action, to a rod structure, we only get an untyped rod structure. 
\end{remark}

\begin{definition}[Equivalence]

\

\begin{itemize}
	\item Two untyped rod structures $\mathfrak R_1$ and $\mathfrak R_2$ are called \emph{equivalent} if 
    there is an element $P\in GL(2;\mathbb{R})$ such that $\mathfrak{R}_1=P.\mathfrak{R}_2$.
    \item Two enhanced untyped rod structures $(\mathfrak R_1, \Lambda_1)$ and $(\mathfrak R_2, \Lambda_2)$ are called  \emph{equivalent} if there is an element $P\in GL(2;\mathbb{R})$ such that $\mathfrak{R}_2=P.\mathfrak{R}_1$ and $\Lambda_2=P.\Lambda_1$. 
\end{itemize}
	\end{definition}
In general $\mR$ and $\overline{\mR}$ are not equivalent. Instead, $\overline{\mR}$ is equivalent to the untyped rod structure $(\{z_n-z_{n+1-j}\}_{j=1}^n, \{\mv_j\}_{j=0}^n)$.

\

Given an enhanced untyped rod structure $(\mathfrak{R}, \Lambda)$ we consider the  4-manifold $M^\circ=\mathbb H^\circ\times\mathbb{R}^2/\Lambda$, with the canonical orientation induced by the 4-form $d\rho\wedge dz\wedge d\phi_1\wedge d\phi_2$. It admits the action of a 2-torus $\mathbb T$ which can be identified with $\mathbb R^2/\Lambda$. The open manifold $M^\circ$ can be compactified into a toric orbifold $M$ by adding points with stabilizers over $\Gamma=\p\mathbb H^\circ$. Over each rod $\mathfrak I_j$ the stabilizer is an $\dS^1$ whose Lie algebra is the line in $\R^2=\Lambda\otimes _{\mathbb Z}\mathbb R$ corresponding to the rod vector $\mathfrak v_j$. At each turning point $z_j$ we obtain an orbifold point in $M$, a neighborhood of which can be topologically identified with an open subset in $\mathbb R^4/H$. The orbifold group $H$ can be determined as follows. We may identify $\Lambda$ with $\mathbb Z^2$ by an element $D\in GL(2,\mathbb{R})$, such that $\mathfrak v_{j-1}=[(0,1)^{\top}]$ and $\mathfrak v_j=[(p,-q)^{\top}]$ with $0\leq q\leq p-1$ and $(p, q)=1$. Then $H$ is isomorphic to the cyclic group $\mathbb{Z}_p$ acting on $\mathbb{R}^4=\mathbb{C}^2$ in the way that $(z_1,z_2)\mapsto (\zeta_{p}z_1,\zeta_p^q z_2)$. In particular if $(\mR, \Lambda)$ is smooth, then $M$ is a smooth manifold. It is also easy to see that equivalent enhanced untyped rod structures yield isomorphic toric orbifolds. If we change $\mR$ to the opposite untyped rod structure $\overline {\mR}$, then we simply obtain an isomorphic toric manifold with the opposite orientation.  Obviously for this topological discussion the  Asymptotic Type  does not play a role. 

 \

We next define the notion of \emph{degree} for rod structures. Given a rod structure $\mathfrak{R}$, we can naturally associate a loop $l_\mathfrak{R}$ in $\mathbb{RP}^1$. Start with the point $\mathfrak{v}_0\in \R\mathbb P^1$ and choose a unit representative vector $v_0\in\mathbb{R}^2$ for $\mv_0$. Then we inductively choose a unit representative vector $v_j$ for $\mathfrak{v}_j$ (for all $j\leq n$) in a way such that $v_j$ is given by rotating $v_{j-1}$ counter-clockwise by an angle smaller than $\pi$. Then for $j\leq n-1$ we connect each $v_j$ and $v_{j+1}$ via the natural counter-clockwise rotation in $\mathbb{R}^2$. If $\mathfrak{v}_0\neq\mathfrak{v}_n$, then we further rotate $v_n$ counter-clockwise by an angle smaller than $\pi$ to connect it to $\pm v_0$ (there is exactly one choice). Passing to the projectivization $\mathbb{RP}^1$, this defines a loop $l_{\mathfrak R}$. It is straightforward to verify that the homotopy class of $l_{\mathfrak{R}}$ is well-defined and independent of the initial choice of $v_0$.
Denote
    \begin{align}
        e(\mathfrak{R})=\frac{1}{\pi}\int_{l_{\mathfrak R}}d\theta\in\mathbb{Z}_{\geq0},
    \end{align}
where $d\theta$ is the standard angular 1-form on $\R\mathbb P^1$.
\begin{definition}\label{def:degrees}
 We define the degree $d(\mathfrak{R})$ of a rod structure $\mathfrak{R}$ to be 
\begin{equation}
    d(\mR)=\begin{cases}
        e(\mR)-1, &\ \text{if}\ \sharp=\Ae \ \text{or}  \ALFp;\\
        e(\mR), &\ \text{if}\ \sharp= \ALFm \ \text{or} \ \AFa.
    \end{cases}
\end{equation}
    
\end{definition}
 In the definition, the shift by $-1$ is chosen only for notational convenience, such that the statement of Theorem \ref{thm:equivalence of degree and type} is uniform for all  asymptotic types. Notice that we always have $d(\mR)\geq 0$. Moreover, when $\sharp=\ALFm$ we must have $d(\mR)\geq 1$ and when $\sharp=\AFa$ we must have $d(\mR)\geq 1$ unless there are no turning points. 

Notice that in general $d(\mathfrak{R})\neq d(\mathfrak{\overline{R}})$. The degree  precisely characterizes connected components of the space of rod structures of fixed Type $\sharp$. Denote by $\mathcal{S}({\sharp\mid z_1,\ldots,z_n})$ the space of rod structures of  Asymptotic Type  $\sharp$, with fixed  turning points $z_1,\ldots,z_n$. Note that by definition we have for all $j$, $\mathfrak{v}_j\neq\mathfrak{v}_{j+1}$. So $\mathcal{S}({\sharp\mid z_1,\ldots,z_n})$ is an open subset of $(\mathbb R\dP^1)^{n+1}$ and we equip it with the induced topology.
\begin{proposition}\label{p:degreee topological invariant}
    Two rod structures $\mathfrak{R}$ and $\mathfrak{R}'$ in $\mathcal{S}(\sharp\mid z_1,\ldots,z_n)$ are in the same connected component if and only if $d(\mathfrak R)=d(\mathfrak R')$.

\end{proposition}
\begin{proof}
    By passing to the universal cover $\R^1$ of $\mathbb{RP}^1$, where a fundamental domain is given by $[0,\pi)$, the representatives $v_0,\ldots,v_n$ that we selected above can be identified as real numbers $0=v_0<\ldots<v_n\in[0,\infty)$. The integer $e(\mathfrak{R})$ is exactly the integer $k$ such that $v_n\in((k-1)\pi,k\pi]$. Note that we moreover have $0<v_j-v_{j-1}<\pi$. Hence, the space $\mathcal{S}(\sharp\mid z_1,\ldots,z_n)$ can be identified with the set of $n$-tuple of numbers $0=v_0<v_1<\ldots<v_n$, with $0<v_j-v_{j-1}<\pi$ and 
    \begin{itemize}
        \item if $\sharp\in\{\Ae, \ALFp, \ALFm\}$, then $v_n\not\in\mathbb{Z}\pi$;
        \item if $\sharp=\AFa$, then $v_n\in\mathbb{Z}\pi$.
    \end{itemize}
From here one can conclude.
\end{proof}
It is easy to see that $\mathcal S(\sharp\mid z_1, \cdots, z_n)$ has exactly $n$ connected components except when it is of Asymptotic Type $\AFa$, in which case it has exactly $n-1$ connected components when $n\geq1$.

\subsection{Toric gravitational instantons}
\begin{definition}
	A toric gravitational instanton is a complete oriented toric Ricci-flat 4-manifold $(M,g)$ which is  integrable and regular, and satisfies that $$\int_{M}|\Rm_g|^2\dvol_g<\infty.$$ 
\end{definition}

Given a toric gravitational instanton, we choose an integral basis $\{\partial_{\phi_1},\partial_{\phi_2}\}$ of $\mathrm{Lie}(\dT)$  and use them to identify $\mathrm{Lie}(\dT)$ with $\mathbb{R}^2$. The \emph{Gram matrix} $G$ is globally defined on $N^\circ$.
As before we set $$\rho:=\sqrt{\det G}\geq0.$$ It is a priori defined on $N^\circ$, and it extends continuously to $N$ and vanishes on $N\setminus N^\circ$. There is also an induced orientation on the quotient $N$. If $N$ is simply-connected (in particular, if $M$ is simply-connected), then we can define a global harmonic conjugate $z$ of $\rho$ which is continuous up to $N$. Here we are taking the harmonic conjugate with the quotient metric and the induced orientation. That is, the orientation such that $d\rho\wedge dz\wedge d\phi_1\wedge d\phi_2$ is positively oriented on $M$.

\begin{definition}
	A toric gravitational instanton is \emph{proper} if there is a globally defined harmonic conjugate $z$ of $\rho$ on $N$ such that  the map $\mathcal F:=(\rho, z)$ induces a homeomorphism from $N$ to $\mathbb H$. 
\end{definition}

	This properness assumption is a natural one since it is satisfied by all the known examples of simply-connected toric gravitational instantons.
    In this paper we only consider toric gravitational instantons which are proper.

A proper  toric gravitational instanton naturally defines a smooth enhanced untyped rod structure  using the map $\mathcal F$. The turning points are given by the images under $\mathcal F$ of the points in $M^\sharp$ and the rods are given by the images of the components of $M^*$. The closure of each component of $M^*$ is topologically either $S^2$ or $\mathbb R^2$. The rod vector is given by the line in $\Lie(\dT)\simeq \R^2$ corresponding to the Lie algebra of the stabilizer $\dS^1$. The smoothness of $M$ implies that the rod structure is both enhanced and smooth.  

If we change the integral basis by an element $P$ in $GL(2;\dZ)$, namely, using $\{P.\partial_{\phi_1},P.\partial_{\phi_2}\}$ as  the new integral basis, then the untyped rod structure is changed to $P^{-1}.\mathfrak{R}$.
For a proper toric gravitational instanton $(M,g)$, if we keep the integral basis but consider the opposite orientation on $M$, then the  resulting  untyped rod structure is equivalent  to $\overline{\mR}$.

\subsection{Key Examples}
\label{ss:model examples}

In this subsection we discuss a few important classes of examples of toric gravitational instantons. We will write down the corresponding  rod structures and  augmented harmonic maps, which will form the local models and fundamental building blocks for our construction. 
\begin{example}[Flat $\R^4$]\label{ex1}
    Consider the flat space $\R^4$, which we identify with $\C^2$ using the complex coordinates $(z_1, z_2)$. Write $$z_1=se^{i\phi_1}\sin \theta, \ \ z_2=se^{i\phi_2}\cos\theta.$$ Then the flat metric takes the form
$$g=ds^2+s^2(d\theta^2+\sin^2\theta d\phi_1^2+\cos^2\theta d\phi_2^2).$$
Consider the $\dT$ action generated by $\p_{\phi_1}$ and $\p_{\phi_2}$. Then $\rho=\frac{1}{2}s^2\sin(2\theta)$ and one can take the harmonic conjugate $z=\frac{1}{2}s^2\cos(2\theta)$.  Denote $r=\sqrt{\rho^2+z^2}$. The corresponding augmented harmonic map is
$$\Phi=\rho^{-1}\begin{pmatrix}
r-z & 0 \\
0 & r+z
\end{pmatrix}, \ \ \nu=-\log (2r).
$$
and the quotient metric on $N$ is $$\widehat g=\frac{1}{2r}(d\rho^2+dz^2),$$
which is isometric to the first quadrant in $\mathbb R^4$. The corresponding rod structure  has exactly one turning point $z_1=0$ and the rod vectors are $\mv_0=[(0, 1)^{\top}]$, $\mv_1=[(1,0)^{\top}]$.

It turns out to be more natural to view $\R^4$ in terms of the Gibbons-Hawking ansatz (see Section \ref{subsubsec:Type I case}. Then it is convenient to  consider an equivalent untyped rod structure $\mR^{\Ae}$ where $\mv_0=[(1, \frac{1}{2})^{\top}]$, $\mv_1=[(1,-\frac{1}{2})^{\top}]$, and the augmented harmonic map is
\begin{equation}\label{e:standard harmonic map R4}\Phi^{\Ae}=\rho^{-1}\begin{pmatrix}
\frac r2 & z \\
z & 2r
\end{pmatrix}, \ \ \nu^{\Ae}=-\log (2r). 
\end{equation}
We make $\mR^{\Ae}$ a rod structure by assigning the Asymptotic Type to be $\Ae$, since the corresponding Ricci-flat metric is the model end  for \emph{Asymptotically Euclidean} spaces. See Figure \ref{Figure:rod structure for R^AE} for a picture of the rod structure.
 
Notice that $\Phi^{\Ae}$ is scale invariant, i.e., $\varphi_\lambda^*\Phi^{\Ae}=\Phi^{\Ae}$ for all $\lambda>0$. This is related to the conical nature of the flat metric on $\R^4$.
\begin{figure}[ht]
    \begin{tikzpicture}[scale=0.6]
        \draw (0,0)--(4,6);
        \draw (0,0)--(-4,6);
        \draw[-Stealth] (2,3)--node [above right]  {$\mathfrak{v}_1=[(1,-\frac{1}{2})^\top]$} (2.75,2.5);
        \draw[-Stealth] (-2,3)--node [above right]{$\mathfrak{v}_0=[(1,\frac{1}{2})^\top]$} (-1.25,3.5);

        \filldraw[color=gray!50,fill=gray!50,opacity=0.3](4,6)--(0,0) --(-4,6);
    \end{tikzpicture}
	\caption{The rod structure $\mR^{\Ae}$.}
	 \label{Figure:rod structure for R^AE}   
\end{figure}
\end{example}

\begin{example}[Taub-NUT]\label{example:positive Taub-NUT}
    Consider the \emph{Taub-NUT space}, which is by definition the (positively oriented) Taub-NUT metric.  In the {Weyl-Papapetrou} coordinates it can be explicitly written as
	\begin{equation*}
		g=H^{-1}(d\phi_1+{A}d\phi_2)^2+H(d\rho^2+dz^2+\rho^2d\phi_2^2),
	\end{equation*}
    where $$H=1+\frac{1}{2r},$$and $${A}=\frac{z}{2r}.$$ The $\dT$ action is generated by $\p_{\phi_1}$ and $\p_{\phi_2}$. 
    This is a hyperk\"ahler metric so can be obtained via the Gibbons-Hawking ansatz (see Section \ref{subsubsec:Type I case}). The orientation is determined by the hyperk\"ahler complex structures. The underlying manifold is diffeomorphic to $\R^4$.
    
    The augmented  harmonic map is given by 
    \begin{equation}\label{e:standard harmonic map ALF}\Phi^{\TNp}=\frac{1}{\rho}\left(\begin{matrix} H^{-1}  & H^{-1}A\\ H^{-1}A& H\rho^2+H^{-1}A^2\end{matrix}\right), \ \ \nu^{\TNp}=\frac{1}{2}\log(1+\frac{1}{2r}).
    \end{equation}
    The rod structure $\mathfrak{R}^{\TN^{+}}$ consists of two rods $\mathfrak{I}_0=(-\infty,0)$ and $\mathfrak{I}_1=(0,\infty)$, with rod vectors given by  ${\mathfrak{v}}_0=[(\frac{1}{2},1)^{\top}]$ and ${\mathfrak{v}}_1=[(-\frac{1}{2},1)^{\top}]$ respectively. We assign the Asymptotic Type to be $\ALFp$, since the corresponding Ricci-flat metric is the model end  for \emph{Asymptotically Locally Flat} spaces. See Figure \ref{Figure:rod structure for R^TN+}.

    \begin{figure}[ht]
    \begin{tikzpicture}[scale=0.6]
        \draw (0,0)--(6,3);
        \draw (0,0)--(-6,3);

        \draw[-Stealth] (3,1.5)--node [right]{$\mathfrak{v}_1=[(-\frac{1}{2},1)^\top]$} (3.5,0.5);
        \draw[-Stealth] (-3,1.5)--node [right]{$\mathfrak{v}_0=[(\frac{1}{2},1)^\top]$} (-2.5,2.5);

        \filldraw[color=gray!50,fill=gray!50,opacity=0.3](6,3)--(0,0) --(-6,3);
    \end{tikzpicture}
	\caption{The rod structure $\mR^{\TNp}$.}
	 \label{Figure:rod structure for R^TN+}   
    \end{figure}
    \end{example}
    
\begin{example}[Anti-Taub-NUT]\label{example:negative Taub-NUT}
    Consider the \emph{anti-Taub-NUT space}, which is by definition the negatively oriented Taub-NUT metric.   We can use the same formula as in the previous example, but with $\phi_1$ and $\phi_2$ interchanged:
	\begin{equation*}
		g=H^{-1}(d\phi_2+{A}d\phi_1)^2+H(d\rho^2+dz^2+\rho^2d\phi_1^2),
	\end{equation*}
This is \emph{negatively oriented}, because the orientation we use in our setting, which is defined by $d\rho dzd\phi_1d\phi_2$, is opposite to the orientation determined by the usual hyperk\"ahler complex structure. Note that the anti-Taub-NUT space is Hermitian and conformally K\"ahler with respect to a  different complex structure (see Section \ref{subsubsec:Type II case}).
    
    The augmented harmonic map is given by 
    \begin{equation}\Phi^{\TNm}=\frac{1}{\rho}\left(\begin{matrix}H\rho^2+H^{-1}A^2 & H^{-1}A\\ H^{-1}A& H^{-1}\end{matrix}\right), \ \ \nu^{\TNm}=\frac{1}{2}\log(1+\frac{1}{2r}).
    \end{equation}
    The rod structure $\mathfrak{R}^{\TN^{-}}$ consists of two rods $\mathfrak{I}_0=(-\infty,0)$ and $\mathfrak{I}_1=(0,\infty)$, with rod vectors given by  ${\mathfrak{v}}_0=[(1,\frac{1}{2})^{\top}]$ and ${\mathfrak{v}}_1=[(1,-\frac{1}{2})^{\top}]$ respectively. We assign the Asymptotic Type to be $\ALFm$. See Figure \ref{Figure:rod structure for R^TN-}.
    
    Notice that $\mR^{\TN^-}$   differs from $\mR^{\Ae}$ only by the Asymptotic Type. It is a manifestation of the fact that the underlying smooth toric 4-manifolds of the flat metric and the Taub-NUT metric are diffeomorphic. 
    \begin{figure}[ht]
    \begin{tikzpicture}[scale=0.6]
        \draw (0,0)--(4,6);
        \draw (0,0)--(-4,6);

        \draw[-Stealth] (2,3)--node [above right]{$\mathfrak{v}_1=[(1,-\frac{1}{2})^\top]$} (2.75,2.5);
        \draw[-Stealth] (-2,3)--node [above right]{$\mathfrak{v}_0=[(1,\frac{1}{2})^\top]$} (-1.25,3.5);

        \filldraw[color=gray!50,fill=gray!50,opacity=0.3](4,6)--(0,0) --(-4,6);
    \end{tikzpicture}
	\caption{The rod structure $\mR^{\TNm}$.}
	\label{Figure:rod structure for R^TN-}   
    \end{figure} 
\end{example}

\begin{example}[Flat $\R^3\times \dS^1$]\label{ex:AF}
Consider the product space $\R^3\times \dS^1$. In terms of the polar coordinates $(r,\theta, \phi_1)$ on $\R^3$ and the rotational coordinate $\phi_2$ on $\dS^1$, the standard flat metric can be written as
$$g=dr^2+r^2(d\theta^2+\sin^2\theta d\phi_1^2)+d\phi_2^2.$$
Consider the $\dT$ action generated by $\p_{\phi_1}$ and $\p_{\phi_2}$. Then $\rho=r\sin\theta$ and one can take $z=r\cos\theta$. The augmented  harmonic map has the form
\begin{equation}\label{e:standard harmonic map AF}\Phi^{\AF_0}=\begin{pmatrix}
\rho & 0 \\
0 & \rho^{-1}
\end{pmatrix}, \ \ \nu^{\AF_0}=0.
\end{equation}
The quotient metric on $N$ is given as $$\widehat g=d\rho^2+dz^2.$$  
The rod structure $\mR^{\AF_0}$ has no turning points and with exactly one rod vector $\mv_0=[\bv_0]$. We assign the Asymptotic Type to be $\AF_0$. 

More generally, for $\beta\in \mathbb R$, we  consider the quotient $M_\beta:=(\R^3\times \R)/\langle q_\beta\rangle$,  where $q_\beta$ acts as a rotation by angle $2\pi\beta$ on $\R^3$ and a translation by distance $1$ on $\R^1$. Under the basis $\partial_{\phi_1},\beta\partial_{\phi_1}+\partial_{\phi_2}$ the augmented harmonic map has the form
 \begin{equation}
	\Phi^{\AFa}=\frac{1}{\rho}\left(\begin{matrix}\rho^2&\beta\rho^2\\\beta\rho^2&\beta^2\rho^2+1\end{matrix}\right), \ \ \nu^{\AFa}=0.
\end{equation}
Note that for $\beta_1\equiv\beta_2 \mod 1$, the $\mathbb Z$ actions generated by $q_{\beta_1}$ and $q_{\beta_2}$ are the same. The difference only lies in our choices of basis to write down the harmonic map. Also notice that different $\Phi^{\AFa}$ are related by $\Phi^{\AFa}=U_{-\beta}.\Phi^{\AF_0}$. The rod structure agrees with that of $\mR^{\AF_0}$, except we now assign the Asymptotic Type to be $\AFa$. See Figure \ref{Figure:rod structure for R^AF}.

\begin{figure}[ht]
    \begin{tikzpicture}[scale=0.6]
        \draw (0,4)--(0,-4);

        \draw[-Stealth] (0,0)--node [above]{$\mathfrak{v}_0=[(1,0)^\top]$} (1,0);

        \filldraw[color=gray!50,fill=gray!50,opacity=0.3](0,4)--(0,-4) --(4,-4)--(4,4);
    \end{tikzpicture}
	\caption{The rod structure $\mR^{\AFa}$.}
	 \label{Figure:rod structure for R^AF}   
    \end{figure}
\end{example}

\begin{remark}\label{remark3.19}
    One may notice that in the above we distinguish only between the Taub-NUT and anti-Taub-NUT spaces, not in the $\Ae$ and $\AFa$ cases. The reason is as follows. In the $\AFa$ case, we know that the space $M_\beta$ with the opposite orientation is isomorphic to $M_{-\beta}$ as oriented Ricci-flat metrics.  In the $\Ae$ case, we will in Section \ref{ss:model maps and tameness} extend each model harmonic map above to a family of model harmonic maps. Using the notations in Definition \ref{def:Ae model family} and \ref{def:AF model family}, then we have 
    $$R_1.\Phi^{\Ae_\alpha}=\Phi^{\Ae_{4/\alpha}}.$$
    This means that the flat $\R^4$ with the opposite orientation can be described using the rod structure $\mR^{\Ae_4}$ and the harmonic map $\Phi^{\Ae_4}$.  This is not the case for the Taub-NUT space, since we only have (see Definition \ref{def:TNplus family})
    $$ R_1.\Phi^{\TNap}=\Phi^{\TNam}.$$
    So we need to use two different families of model harmonic maps. 
   
\end{remark}

In Section \ref{sec:harmonic maps}  we will write down more general augmented harmonic maps based on the above four examples. A central philosophy in this paper is that the above families form the essential building blocks for constructing new gravitational instantons.  As an illustration of this idea we investigate the example of  Euclidean Kerr spaces, and show that in a limit situation they decouple into the union of a  Taub-NUT space and an anti-Taub-NUT space.

\begin{example}[Kerr]\label{ex:Kerr}The Euclidean Kerr spaces, modulo rescalings, form a one parameter family of asymptotically flat gravitational instantons $g_a$ ($a\in (-\infty, \infty)$) on  $S^2\times \mathbb R^2$. Recall that there is an explicit formula (see for example \cite{ChenTeo2})
\begin{equation}
g_a = \frac{\Delta}{\Sigma} (d\phi_1 + a \sin^2 \vartheta \, d\phi_2)^2 +\frac{1}{\Sigma}{\sin^2 \vartheta} (a d\phi_1 - (s^2 - a^2) d\phi_2)^2  + \Sigma(\frac{1}{\Delta} ds^2 +  d\vartheta^2),
\label{eq:kerr_instanton}
\end{equation}
where \(\Sigma\) and \(\Delta\) are defined as
\begin{equation}
\Sigma = s^2 - a^2 \cos^2 \vartheta, \qquad \Delta = s^2 - 2ms - a^2.
\end{equation}
Here the parameters  \(a\), and the coordinates \(s\), \(\vartheta\) take the ranges
\[
\quad a\in (-\infty, \infty), \quad s \geq s_a, \quad  \vartheta \in [0,\pi],
\]
where \(s_a\) is defined as
\[
s_a = m + \sqrt{m^2+ a^2}.
\]
We always fix $m=\frac{1}{2}$ to eliminate the scaling freedom. 
The Weyl–Papapetrou coordinates \(( \rho, z, \phi_1, \phi_2)\) are related to the above coordinates $(r, \vartheta, \phi_1, \phi_2)$ via
\begin{equation}
\rho = \sqrt{s^2 - 2ms - a^2} \, \sin \vartheta, \qquad z = (s - m) \cos \vartheta+z_a,
\label{eq:weyl_coords}
\end{equation}where $$z_a:=\sqrt{m^2+a^2}.$$
The rod structure $\mR^{\Kerr_a}$ consists of 3 rods given by 
$$\mI_0=(-\infty, 0),\quad \mI_1=(0, 2z_a),\quad \mI_2=(2z_a, \infty). $$
Notice that here we have normalized the $z$ coordinate such that the first turning point occurs at $z=0$.
The corresponding rod vectors are given by 
$$\mv_0=\mv_2=[(0, 1)^\top], \qquad \mv_1=[(\frac{1}{\kappa_a}, \frac{\Omega_a}{\kappa_a})^\top], $$
where 
$$\kappa_a:=\frac{z_a}{2ms_a}, \qquad \Omega_a:=\frac{a}{2ms_a}.$$
The enhancement $\Lambda$ is generated by $(0, 1)^\top$ and $(\frac{1}{\kappa_a}, \frac{\Omega_a}{\kappa_a})^\top$. When $a=0$ we get the Euclidean Schwarzschild space; this is the only case when the metric $g_a$ is asymptotic to the flat product $\R^3\times S^1$ at infinity. When $\Omega_a/\kappa_a$ is rational the asymptotic cone at infinity of $g_a$ is a flat orbifold $ \mathbb R^2/\mathbb Z_p\times \R$; when $\Omega_a/\kappa_a$ is irrational the asymptotic cone at infinity of $g_a$ is isometric to the half plane $\mathbb H$. See the discussion in \cite{LS2024}. With the canonical orientation $d\rho\wedge dz\wedge d\phi_1\wedge d\phi_2$, we refer to these gravitational instantons as \emph{Kerr spaces}.

The augmented harmonic map for $g_a$ is given by 
\begin{equation}
    \Phi^{\Kerr_a}=\frac{1}{\Sigma\rho}\left(\begin{matrix} \Delta+a^2\sin^2\vartheta  & -2asm\sin^2\vartheta\\ -2asm\sin^2\vartheta & \sin^2\vartheta(a^2\rho^2+(s^2-a^2)^2)\end{matrix}\right),
    \end{equation}
    
    \begin{equation} \nu^{\Kerr_a}=\frac{1}{2}\log(\frac{\Sigma\sin^2\vartheta}{\rho^2\cos^2\vartheta+(s-m)^2\sin^4\vartheta}).
\end{equation}
Now we consider the limit of $(\Phi^{\Kerr_a}, \nu^{\Kerr_a})$ as $a\rightarrow\infty$.  Naturally we need to consider  the pointed limits based at the two turning points. First we consider the turning point $z=0$. We fix a constant $C>0$ large and consider the limit of $(\Phi^{\Kerr_a}, \nu^{\Kerr_a})$ as $a\rightarrow\infty$, on the region where $r^2=z^2+\rho^2\leq C$ and $|\rho|\geq C^{-1}$. Using the equation 
\begin{equation}\frac{\rho^2}{(s-m)^2-m^2-a^2}+\frac{(z-z_a)^2}{(s-m)^2}=1,
\end{equation}it is straightforward to obtain the following asymptotics 
\begin{equation}
    s=a+\frac{1}{2}(r-z+2m)+O(a^{-1}),
\end{equation}
\begin{equation}
    \sin^2\vartheta=(r+z)a^{-1}+O(a^{-2}),
\end{equation}
\begin{equation}
    \Sigma=2(m+r)a+O(1),
\end{equation}
\begin{equation}
    \Delta=(r-z)a+O(1).
\end{equation}
It then follows that 
\begin{equation}
\Phi^\infty:=\lim_{a\rightarrow\infty}\Phi^{\Kerr_a}=\frac{1}{2(m+r)\rho}\left(\begin{matrix} 2r & -2m(r+z)\\ -2m(r+z) & (r+z)(\rho^2+(r-z+2m)^2)\end{matrix}\right)
\end{equation}
and 
\begin{equation}
    \nu^\infty:=\lim_{a\rightarrow\infty}\nu^{\Kerr_a}=\frac{1}{2}\log \frac{m+r}{r}.
\end{equation}
Recall that we have fixed $m=\frac{1}{2}$. Then one can check that 
\begin{equation}
    U_{-m}.\Phi^\infty=\eta^*\Phi^{\TNp},\quad \nu^\infty=\nu^{\TNp}=\nu^{\TNm},
\end{equation}
where 
$$\eta:\mathbb H\rightarrow \mathbb H;\quad (\rho, z)\mapsto (\rho, -z).$$
Notice $\eta^*\Phi^{\TNp}$ is essentially equivalent to $\Phi^{\TNm}$: when we define the anti-Taub-NUT space by changing the orientation on the Taub-NUT space, instead of interchanging the order of $\phi_1$ and $\phi_2$, we may change the orientation on the quotient space $\dH$, which has the effect of replacing $z$ by $-z$.  Similarly, we can study the pointed limit based at the other turning point using the translation $T_{2z_a}$, where one sees instead $(U_m.\Phi^{\TNp}, \nu^{\TNp})$. The following Figure \ref{Figure:splitting of kerr} illustrates the situation.
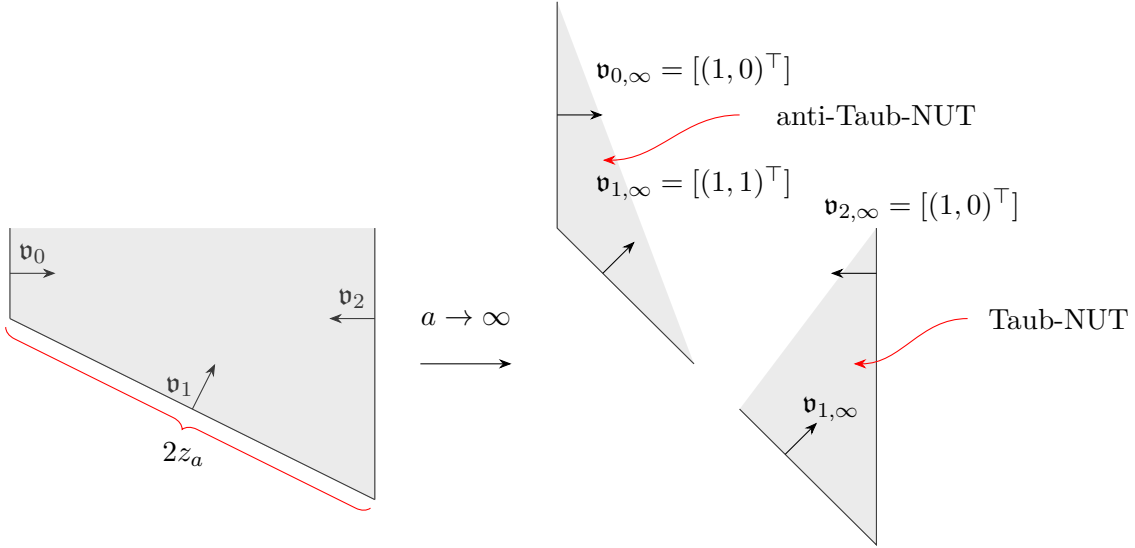
\begin{figure}[ht]
    \begin{tikzpicture}[scale=0.6]
        \draw (-4,2)--(-4,0);
        \draw (-4,0)--(4,-4);
        \draw (4,-4)--(4,2);

        \draw[-Stealth] (-4,1)--node [above]{$\mathfrak{v}_0$} (-3,1);
        \draw[-Stealth] (0,-2)--node [left]{$\mathfrak{v}_1$} (0.5,-1);
        \draw[-Stealth] (4,0)--node [above]{$\mv_2$} (3,0);

        \filldraw[color=gray!50,fill=gray!50,opacity=0.3](-4,2)--(-4,0) --(4,-4)--(4,2);
        \draw [draw = red,decorate,decoration={brace,amplitude=5pt,mirror,raise=4ex}]
        (-3.6,0.8) -- (4.4,-3.2) node[midway,yshift=-3em]{};
        \node at (-0.2,-3) {$2z_a$};

        \draw[-Stealth] (5,-1) -- (7,-1);
        \node at (6,0) {$a\to\infty$};

        \draw (8,7)--(8,2)--(11,-1);
        \filldraw[color=gray!50,fill=gray!50,opacity=0.3](8,7)--(8,2)--(11,-1);
        \draw[-Stealth] (8,4.5)-- (9,4.5);
        \node at (11,5.5) {$\mathfrak{v}_{0,\infty}=[(1,0)^\top]$};
        \draw[-Stealth] (9,1)-- (9.7,1.7);
        \node at (11,3) {$\mathfrak{v}_{1,\infty}=[(1,1)^\top]$};
        
        \draw (12,-2)--(15,-5)--(15,2);
        \filldraw[color=gray!50,fill=gray!50,opacity=0.3](12,-2)--(15,-5)--(15,2);
        \draw[-Stealth] (13,-3)-- (13.7,-2.3);
        \node at (14,-2) {$\mathfrak{v}_{1,\infty}$};
        \draw[-Stealth] (15,1)--(14,1);
        \node at (16,2.5) {$\mv_{2,\infty}=[(1,0)^\top]$};

        \draw [draw = red,-Stealth] (17,0) to [out=180,in=0] (14.5,-1);
        \node at (19,0) {Taub-NUT};

        \draw [draw = red,-Stealth] (12,4.5) to [out=180,in=0] (9,3.5);
        \node at (15,4.5) {anti-Taub-NUT};
        
    \end{tikzpicture}
	\caption{The splitting of Kerr spaces}
	 \label{Figure:splitting of kerr}   
\end{figure}

The above discussion shows that as $a\rightarrow 0$, we can see both $\Phi^{\TNp}$ and $\Phi^{\TNm}$ as pointed limits of $\Phi^{\Kerr_a}$. At the level of gravitational instantons, we may say that the Kerr spaces $g_a$ split into the union of a Taub-NUT space and an anti-Taub-NUT space. By Lemma \ref{lem:Gauss-Bonnet}, the $L^2$ energy of $g_a$ is given by
$$\int|\Rm(g_a)|^2\dvol_{g_a}=16\pi^2,$$ whereas the curvature energy of a Taub-NUT space is $8\pi^2$, so there is no energy loss in this splitting process. Geometrically this is an interesting phenomenon, which extends to  more general situation. See the discussion in Section \ref{sec:discussion}.

As a remark, notice that the Kerr spaces are Hermitian non-K\"ahler (or, Type $\II$ in terms of the terminology of Section \ref{sec:7.1}). According to \cite{BG} (see also Section \ref{sec:7.3}), they can alternatively be  described using toric K\"ahler geometry, in terms of a piecewise linear convex function in one variable.  However it does not seem straightforward to see the above splitting in this viewpoint. The reason is that among the two splitting limits, the anti-Taub-NUT is of Type $\II$ but the Taub-NUT is  of Type $\I$, so can not be put in the framework of toric K\"ahler geometry applied in \cite{BG}. 
\end{example}

\section{Tame harmonic  maps}
\label{sec:harmonic maps}

In this section we study the analytic aspect of axisymmetric harmonic maps tamed by a rod structure. 

\subsection{Model maps and tameness}\label{ss:model maps and tameness}First we write down local and infinity model harmonic maps. The local models involve \emph{rod models} and \emph{turning point models}. Then we define the notion of tameness for harmonic maps.
\begin{definition}[Rod model]
Consider the harmonic map defined on $\dH^\circ$ given by
  $$\Phi_{[\bv_0]}:=\left(\begin{matrix}
			     \rho&0\\0&\rho^{-1}
		    \end{matrix}\right).$$
		    This is the harmonic map associated to the flat product metric on $\dS^1\times \R^3$, see \eqref{e:standard harmonic map AF}. 
		    In general for $\mv\in \R\dP^1$  we define
\begin{equation} \label{e: definition of rod model}
	\Phi_{\mv}=P_\mv^{-1}. \Phi_{[\bv_0]}, 
\end{equation}
where $P_\mv$ is defined in \eqref{e:definition of P1}.
\end{definition}

\begin{definition}[Turning point model]
For $\alpha>0$ we consider the harmonic map defined on $\dH^\circ$ given by
	 \begin{equation}\label{e:turning point model}
	\Phi_{[\bv_{\frac{\alpha}{2}}], [\bv_{-\frac{\alpha}{2}}]}=\frac{1}{\rho}\left(\begin{matrix}
		\frac{\alpha r}{2} & z\\
		z & \frac{2r}{\alpha}
	\end{matrix}\right).
\end{equation}
When $\alpha=1$ this is the harmonic map associated to the flat metric on $\R^4$, see \eqref{e:standard harmonic map R4}. 
In general, for two elements $\mv\neq \mv'\in \R\mathbb P^1$, we define $$\Phi_{\mv, \mv'}=P_{\mv, \mv'}^{-1}. \Phi_{[\bv_{\frac{\alpha}{2}}], [\bv_{-\frac{\alpha}{2}}]},$$ where  $P_{\mv, \mv'}$ is  defined in \eqref{e:definition of P} and $\alpha=\mv\wedge\mv'$ is defined in \eqref{e:definition of alpha}.
\end{definition}

Next, we introduce 4 families of asymptotic (at infinity) models of augmented harmonic maps. These are motivated by the examples discussed in Section \ref{ss:model examples}.
Consider  
\begin{equation}\label{eq:Phi with alpha}
\Phi=\rho^{-1}\left(\begin{matrix}H_\alpha^{-1}A_\alpha^2+\rho^2H_\alpha& H_\alpha^{-1}A_\alpha\\ H_\alpha^{-1}A_\alpha & H_\alpha^{-1}\end{matrix}\right).
\end{equation}
and 
\begin{equation}\label{eq:nu with alpha}
    \nu=\frac{1}{2}\log H_\alpha,
\end{equation}
for functions $A_\alpha$ and $H_\alpha$ to be specified below.

\begin{definition}[$\Ae$ model family] \label{def:Ae model family}
    The $\Ae$ model family $(\Phi^{\Aea}, \nu^{\Aea})$ (for $\alpha>0$)  is formed by setting
$$H_\alpha=\frac{\alpha}{2r}, \ \ \ \  A_\alpha=\frac{\alpha z}{2r}.$$
One can check that indeed we have $$\Phi^{\Aea}=\Phi_{[\bv_{\frac{\alpha}{2}}], [\bv_{-\frac{\alpha}{2}}]}.$$
More generally, 
given  $\mv\neq \mv'\in \R\bP^1$, we define the $\Ae$ model augmented harmonic map to be 
\begin{equation}\label{e:definition of Phi Ae general}\Phi_{\mv, \mv'}^{\Ae}=P_{\mv, \mv'}^{-1}.\Phi^{\Aea}, \ \ \ \nu_{\mv, \mv'}^{\Ae}=\nu^{\Aea}=\frac{1}{2}\log(\frac{\alpha}{2r})
\end{equation}
for $\alpha=\mv\wedge\mv'$.
\end{definition}
\begin{definition}[$\TN^-$ model family]\label{def:TNplus family}
    The $\TN^{-}$ model family $(\Phi^{\TNam},  \nu^{\TNam})$ (for $\alpha\geq 0$) is formed by setting
\begin{equation}\label{e:defintion of Halpha and Aalpha}H_\alpha=1+\frac{\alpha}{2r}, \ \ \ \  A_\alpha=\frac{\alpha z}{2r}.\end{equation}
More generally, for $\mv\neq\mv'\in \R\dP^1$, 
 we define the $\ALF^-$ model augmented harmonic map to be
\begin{equation}\label{e:definition of Phi TN general}\Phi_{\mv, \mv'}^{\ALF^-}=P_{\mv, \mv'}^{-1}.\Phi^{\TNam}, \ \ \ \nu_{\mv, \mv'}^{\ALFm}=\nu^{\TNam}=\frac{1}{2}\log(1+\frac{\alpha}{2r}),
\end{equation}
for $\alpha=\mv\wedge\mv'$.
\end{definition}

\begin{definition}[$\TN^+$ model family]
The $\TN^{+}$ model family $(\Phi^{\TNap},  \nu^{\TNap})$ (for $\alpha>0$) is defined with 
    $$\Phi^{\TNap}=R_1.\Phi^{\TN^{-}_{\alpha}},$$
    $$\nu^{\TNap}=\nu^{\TN^-_{\alpha}}.$$
where $R_1$ is given in \eqref{e:definition of matrices}.
More generally, for $\mv\neq\mv'\in \R\dP^1$, 
 we define the $\ALF^+$ model augmented harmonic map to be
\begin{equation}\label{e:definition of Phi TN+ general}\Phi_{\mv, \mv'}^{\ALF^+}=P_{\mv, \mv'}^{-1}.\Phi^{\TN^+_{4/\alpha}},\ \ \ \nu_{\mv, \mv'}^{\ALFp}=\nu^{\TN_{4/\alpha}^+}=\frac{1}{2}\log(1+\frac{2}{\alpha r}), 
\end{equation}
for $\alpha=\mv\wedge\mv'$.
\end{definition}

Notice that $$\Phi_{\mv, \mv'}^{\ALF^+}=P_{\mv, \mv}^{-1}R_1P_{\mv, \mv'}. \Phi_{\mv, \mv'}^{\ALF^-}.$$

\begin{definition}[$\AF$ model family]\label{def:AF model family}
    The $\AF$ model family $(\Phi^{\AFa}, \nu^{\AFa})$ (for $\beta\in \R$)
is given by \begin{equation}
	\Phi^{\AFa}=\frac{1}{\rho}\left(\begin{matrix}\rho^2&\beta\rho^2\\\beta\rho^2&\beta^2\rho^2+1\end{matrix}\right)=U_{-\beta}. \Phi_{[\bv_0]}, \quad \nu^{\AFa}=0.
\end{equation}
Given $\mv\in \R\dP^1$ and $\beta\in \R$,  we define the $\AFa$ model augmented harmonic map to be
\begin{equation}\label{e:definition of Phi AFa in general}\Phi_{\mv, \mv}^{\AFa}=P_{\mv}^{-1}. \Phi^{\AFa}, \ \ \  \nu_{\mv, \mv}^{\AFa}=0.
\end{equation}
\end{definition} 

\begin{remark}
    Notice that by definition $$\Phi^{\TN^-_0}=\Phi^{\AF_0}.$$
    This fact will be important for our discussion in Section \ref{sec:degeneration of rod structures}.
\end{remark}

Let $U$ be an open subset of $\dH$. We denote $U^\circ=U\cap \dH^\circ$.
\begin{definition}[Local tameness] \label{def:local tameness}
A smooth map $\Phi: U^\circ\rightarrow \mathcal H$ is said to be  \emph{locally tamed} by an untyped rod structure $\mR$ if 
\begin{itemize}
	\item For any rod $\mI_j$ and $z\in \mJ_j\cap U$, there is a neighborhood $V$ of $z$ such that $$\sup_{V^\circ}d(\Phi,\Phi_{\mv_j})<\infty;$$
	\item For each turning point $z_j\in U$,  there is a neighborhood $V$ of $z_j$ such that   $$\sup_{V^\circ}d(\Phi,T_{z_j}^*\Phi_{\mv_{j}, \mv_{j+1}})<\infty.$$ 
\end{itemize}
\end{definition}

\begin{definition}[Global tameness]

A smooth map $\Phi: \dH^\circ\rightarrow \mathcal H$ is said to be \emph{globally tamed} by a rod structure $\mR$ of  Asymptotic Type  $\sharp$ if it is locally tamed by $\mR$ and $$\sup_{\dH^\circ}d(\Phi,\Phi^{\sharp}_{\mv_0, \mv_n})\leq C.$$
  \end{definition}

  \begin{definition}[Strong tameness]\label{def:strong tameness}
   We say a smooth map $\Phi:\dH^\circ\rightarrow \mathcal H$ is \emph{strongly tamed} by a rod structure $\mR$ of  Asymptotic Type  $\sharp$ if it is globally tamed by $\mR$ and 
    $$d(\Phi,\Phi^{\sharp}_{\mv_0, \mv_n})=O(r^{-1}), \ \  r\rightarrow\infty.$$ 
 \end{definition}

It is easy to check by definition that 
\begin{itemize}
    \item For $\alpha>0$, $\Phi^{\Aea}$ is strongly tamed by the rod structure  $\mR^{\Aea}$ which is of  Asymptotic Type  $\Ae$, with a single turning point $z_1=0$ and with rod vectors given by 
$$\mathfrak v_0=[\bv_{\frac{\alpha}{2}}], \ \ \mathfrak v_1=[\bv_{-\frac{\alpha}{2}}].$$
\item  For $\alpha>0$, $\Phi^{\TNam}$ is strongly tamed by the rod structure $\mR^{\TN^-_\alpha}$ which is of  Asymptotic Type  $\ALFm$,  with a single turning point $z_1=0$ and with rod vectors given by 
$$\mathfrak v_0=[\bv_{\frac{\alpha}{2}}], \ \ \mathfrak v_1=[\bv_{-\frac{\alpha}{2}}].$$
\item For $\alpha>0$, $\Phi^{\TNap}$ is strongly tamed by the rod structure $\mR^{\TN^+_{\alpha}}=\overline{\mR^{\TN^-_{\alpha}}}$ which is of  Asymptotic Type  $\ALFp$, with a single turning point $z_1=0$ and with rod vectors given by
$$\mathfrak v_0=[\bv_{\frac{2}\alpha}], \ \ \mathfrak v_1=[\bv_{-\frac{2}{\alpha}}].$$
\item For $\beta\in\R$, $\Phi^{\AFa}$ is strongly tamed by the rod structure $\mR^{\AFa}$ which is of  Asymptotic Type  $\AFa$, with no turning points and with $\mv_0=[\bv_0]$.  
\end{itemize}

In the next two subsections we will study the regularity of tame harmonic maps. From the PDE point of view, this question fits into the general study of harmonic maps with prescribed codimension 2 singularities, see \cite{Weinstein, LiTian, HKWX} for example. The focus of these previous works was on harmonic maps arising from the Lorentzian setting, which has similar but different asymptotics near the rods. In our Riemannian setting, we will combine these analytical results with the geometric point of view and take advantage of the regularity results for the Riemannian Einstein equation. The main result of this section is the following

\begin{theorem}\label{t:from tame harmonic to conical metric}
\ 
   \begin{itemize}
       \item (Local) An augmented harmonic map, which is  locally tamed by a smooth enhanced untyped rod structure, defines a toric Ricci-flat metric with possible codimension 2 conical singularities. 
       \item (Global) A smooth enhanced rod structure strongly tames a unique harmonic map on $\dH^\circ$. Moreover, there is a unique augmentation which defines a toric gravitational instanton with possible codimension 2 conical singularities. 
   \end{itemize} 
\end{theorem}

The first part will be proved in Section \ref{subsec:local regularity} and the second part will be proved in Section \ref{subsec:global regularity}.

\subsection{Local theory}\label{subsec:local regularity}

\subsubsection{Near a rod interior}
\label{subsubsec:local regularity near rods}
Let $(\Phi, \nu)$ be an augmented harmonic map defined on an open set $U\subset \dH$ which is locally tamed by  a rod structure $\mR$ and assume $U\cap\Gamma$ is contained in a rod $\mI$. Without loss of generality, we assume the rod vector is $\mathfrak{v}=[\bv_0]$, and the corresponding axisymmetric harmonic map   is defined on the ball $B_{1}(0)\subset \R^3$.

  Recall that $\mathcal H=\{(X, Y)|X>0\}$. We set
$$X=e^{u+\log\rho}, \quad Y=v, $$ then the harmonic map equation for $\Phi$ reduces to the following system on $(u, v)$
\begin{align}
	\Delta u=-\frac{1}{\rho^2}e^{-2u}|\nabla v|^2,\label{eq:u}\\
	\mathrm{div}\left(\frac{1}{\rho^2}e^{-2u}\nabla v\right)=0.\label{eq:v}
\end{align}
The tameness assumption implies that $\sup_Ud(\Phi, \Phi_{[\bv_0]})\leq C_0$, which gives 
\begin{align}\label{eq:bounded distance condition}
	|u|+\rho^{-1}|v|\leq C C_0.
\end{align}
Here, $C$ is a constant independent of $C_0$.
A straightforward application of \cite{LiTian,Nguyen} gives the following regularity result. 
\begin{proposition}[$C^{1,\alpha}$ regularity]\label{thm:C11 regularity}
	We have the following uniform bounds on $B_{\frac{1}{2}}(0)$:
	\begin{align*}
		&|\nabla^{k}v|\leq CC_0\rho^{2-k},\text{ for $k=0,1,2,3$},\\
		&\|u\|_{C^{1,\alpha}}\leq CC_0,\text{ for any $0<\alpha<1$},\\
        &|\nabla^ku|\leq CC_0\rho^{1-k}\text{ for any $k=1,2,3$}.
	\end{align*}
    Here the constant $C$ depends on $\alpha$ but is independent of $C_0$.
\end{proposition}
\begin{proof}
First by Theorem 4.1 and Lemma 5.1 in \cite{LiTian} we know that $u$ extends to a H\"older continuous function across $\Gamma$ and $$|\nabla u|\leq C \rho^{-1+\epsilon_0}$$ for a uniform constant $C$ and $\epsilon_0>0$. Notice that the arguments in \cite{LiTian} use a monotonicity formula for local minimizing harmonic maps. But as observed in \cite{Nguyen} (see also \cite{HKWX}), this can be replaced by a Caccioppoli inequality. 	Since we have $$C^{-1}\rho^{-2}\leq \rho^{-2}e^{-2u}\leq C\rho^{-2}$$ and $$|\nabla u|\leq C\rho^{-1+\epsilon_0},$$ with the bound $$\rho^{-1}|v|\leq CC_0,$$ by applying Theorem 1 in \cite{Nguyen} to the equation $$\mathrm{div}(\rho^{-2}e^{-2u}\nabla v)=0,$$ we get $|v|\leq CC_0\rho^2$. From here it is a simple rescaling argument, using the standard Schauder estimate, to bound the higher derivatives of $v$. The bound on $u$ then follows easily from \eqref{eq:u}.

\end{proof}
\begin{remark}It is worth pointing out that in the Lorentzian setting of \cite{LiTian, Nguyen, HKWX} the corresponding harmonic map has a different singular behavior and the regularity obtained therein is of $C^{3, \alpha}$. In our setting, we can only obtain $C^{1, \alpha}$ regularity using directly the PDE result.	
\end{remark}

Now we turn to the geometry of the  Ricci-flat metric associated to $\Phi$.  We write the Gram matrix as
$$G=\rho\Phi=\left(\begin{matrix}
	\rho^2e^u+e^{-u}v^2&e^{-u}v\\
	e^{-u}v&e^{-u}
\end{matrix}\right).$$
From this one can compute the 1-forms \begin{equation}\label{eq:U}
	\mathcal{U}=\rho G^{-1} G_\rho=\left(\begin{matrix}
		2+\rho u_\rho+\frac{1}{\rho}e^{-2u}v_\rho v & \frac{1}{\rho}e^{-2u}v_\rho\\ 
		-2v-2\rho u_\rho v-\frac{1}{\rho}e^{-2u}v_\rho v^2+\rho v_\rho &-\frac{1}{\rho}e^{-2u}v_\rho v-\rho u_\rho
	\end{matrix}\right),
\end{equation}

\begin{equation}\label{eq:V}
	\mathcal{V}=\rho G^{-1}G_z=\left(\begin{matrix}
		\rho u_z+\frac{1}{\rho}e^{-2u}vv_z & \frac{1}{\rho}e^{-2u}v_z\\
		-2\rho u_zv-\frac{1}{\rho}e^{-2u}v^2v_z+\rho v_z & -\frac{1}{\rho}e^{-2u}vv_z-\rho u_z
	\end{matrix}\right).
\end{equation}

We denote by $O'(\rho^k)$ a function which is bounded by $C\rho^k$ and with derivative bounded by $C\rho^{k-1}$. Then we see that
\begin{align*}
	\Tr(\mathcal{U}^2)&=2 \rho^2 u_\rho^2+4 \rho u_\rho+2 e^{-2 u} v_\rho^2+4=4+4\rho u_\rho+O'(\rho^2),\\
    \Tr(\mathcal{V}^2)&=2 \rho^2 u_z^2+2 e^{-2 u} v_z^2=O'(\rho^2),\\
	\Tr(\mathcal{U}\mathcal{V})&=2 \rho u_z \left(\rho u_\rho+1\right)+2 e^{-2 u} v_z v_\rho=2\rho u_z+O'(\rho^2).
\end{align*}
Using \eqref{eqn-Ricci-1} and \eqref{eqn-Ricci-2}, we see that
\begin{align}
    (2\nu-u)_\rho=O'(\rho),&& (2\nu-u)_z=O'(\rho). \label{eqn:nu}	
\end{align}
It follows that $\nu$ is a $C^{1,\alpha}$ function on $B_{\frac{1}{2}}(0)$. 
Moreover, there is a constant $c_0$ such that $$2\nu-u=c_0$$  on $\Gamma$. 
 
As in \eqref{e:from aug harmonic map to Ricci flat},  we consider $$g=e^{2\nu}(d\rho^2+dz^2)+G,$$ which is a well-defined Ricci-flat Riemannian metric on $U^\circ\times\mathbb{R}^2$. 
 Now we take the lattice $\Lambda$ generated by $[(2\pi e^{c_0/2}, 0)^{\top}]$ and $[(0, 2\pi)^{\top}]$, and consider the quotient metric $g_\Lambda$ on $U^\circ\times \dS^1\times \dS^1$. Consider the coordinate functions $x=\rho\cos(e^{-c_0/2}\phi_1),y=\rho\sin( e^{-c_0/2}\phi_1),z,\phi_2$, which diffeomorphically identify $U^\circ\times \dS^1\times \dS^1$ with an open subset $V^\circ\subset\mathbb{R}^2\times\mathbb{R}^1\times \dS^1$. The Killing fields $e^{-c_0/2}\partial_{\phi_1}$ and $\partial_{\phi_2}$ correspond to the rotation on $\mathbb{R}^2$ and the translation on  $\dS^1$. The $\mathbb{R}^1$ is parametrized by $z$.  We set $V\subset\mathbb{R}^2\times\mathbb{R}^1\times \dS^1$ to be the closure of $V^\circ$ by adding points with $\rho=0$.

\begin{lemma}\label{lem:Lipschitz metric}
	The Ricci-flat metric $g$ extends to a Lipschitz (i.e. $C^{0, 1}$) metric on $V$.
\end{lemma}
\begin{proof}
	This follows from a straightforward computation. Notice that the map $(x, y, z, \phi_2)\mapsto (\rho, z)$ is distance decreasing so it suffices to check that the components of $g$ in the $(dx, dy, dz, d\phi_2)$ frame are Lipschitz functions in $(\rho, z)$. To see this, one simply substitutes $d\rho=\rho^{-1}(xdx+ydy)$ and $d\phi_1=\rho^{-2}(xdy-ydx)$ and check the Lipschitz continuity of the components of $g$. We omit the details. 
\end{proof}

Now using standard theory of harmonic coordinates we arrive at 
\begin{proposition}\label{p:rod smooth compactification}
	One can find a new $C^{1, \alpha}$ coordinate system in a neighborhood of $0\in V$ in which the components of $g$ are smooth. In particular, $g$ defines a smooth toric  Ricci-flat metric.
\end{proposition}
\begin{proof}
Since $g$ is Lipschitz one can find a local $W^{2, p}$ (for $p$ arbitrarily large) harmonic coordinate system $(x_1, x_2, x_3, x_4)$. In these coordinates the components $g_{ij}$ are $W^{1, p}$. Since $g$ is Ricci-flat whenever it is smooth, we have the equation 
$$\Delta g_{ij}+\p g*\p g*g^{-1}*g^{-1}=0$$
holds on $V^\circ$. Since $dg_{ij}$ is in $L^p$ and $V\setminus V^0$ has Hausdorff codimension 2 (noticing that $g$ is Lipschitz on $V)$, it is easy to see that the equation holds weakly across the singularities.
Using elliptic regularity and bootstrapping we obtain that $g_{ij}\in C^\infty$. Since  a Killing field of a smooth Riemannian metric is always smooth, it follows that the metric $g$ is indeed a toric Ricci-flat metric.
\end{proof}

\begin{remark}
Since $\rho$ has intrinsic meaning given by  $\sqrt{\det G}$, it follows that a posteriori that $u$ and $v$ has improved regularity. For example we know that $u_\rho=0$ when $\rho=0$.
\end{remark}

Given a rod $\mI$ with rod vector $\mv$, for any nonzero vector $v$ in the line $\mv$, from the above discussion we see that 
$$\lim_{\rho\to0}\frac{1}{e^{\nu}\rho}\sqrt{\rho\Phi(v ,v)}=c,$$
for a positive constant $c$, where the limit is taken uniformly for points in $\dH^\circ$ approaching the interior of the rod $\mI$.
\begin{definition}
	We define the \emph{normalized rod vector $\overline{\mv}$} for the rod $\mI$ to be the vector  in the line corresponding to $\mv$, which is unique up to $\pm1$, such that along $\mI$ we have  
$$\lim_{\rho\to0}\frac{1}{e^{\nu}\rho}\sqrt{\rho\Phi(\overline{\mathfrak{v}},\overline{\mathfrak{v}})}=1.$$
\end{definition}
\begin{example}
    For the model augmented harmonic maps defined  at the beginning of this section, we can explicitly compute the normalized rod vectors:
    \begin{itemize}
        \item For the model $(\Phi^{\Aea}, \nu^{\Aea})$ and $(\Phi^{\TNam}, \nu^{\TNam})$, the normalized rod vectors are given by $\bv_{\frac{\alpha}{2}}$ and $\bv_{-\frac{\alpha}{2}}$ on $\mv_0$ and $\mv_1$;
        \item For the model $(\Phi^{\TNap}, \nu^{\TNap})$, the normalized rod vectors are given by $\bv_{\frac{2}{\alpha}}$ and $\bv_{-\frac{2}{\alpha}}$ on $\mv_0$ and $\mv_1$;
        \item For the model $(\Phi^{\AFa}, \nu^{\AFa})$, the normalized rod vector is given by $\bv_0$ on $\mv_0$.
    \end{itemize}
\end{example}

 Now suppose that we are given a smooth enhancement $\Lambda$ and let  $\bbv$ be a primitive rod vector along  $\mv$. Consider the manifold $M^\circ:= U^\circ\times \mathbb R^2/\Lambda$. Endowed with the quotient metric $g_\Lambda$, $M^\circ$ is naturally a toric Ricci-flat manifold. As in Section \ref{sec:rod structures}, there is a partial compactification $M$ of $M^\circ$ to a smooth toric manifold by adding points with stabilizers over $\Gamma\cap U$.
\begin{definition}[Cone angle] \label{def:cone angle}
	The cone angle $\vartheta(\mI)$ along a rod $\mI$, associated to the augmented harmonic map $(\Phi, \nu)$ locally tamed by  the enhanced rod structure $(\mR, \Lambda)$, is defined to be 
	$$\vartheta_{\Phi, \nu}(\mI)=2\pi\lim_{\rho\rightarrow 0}\frac{1}{e^\nu\rho}\sqrt{\rho \Phi(\bbv, \bbv)}.$$
    \end{definition}
In particular, we have $$2\pi\cdot\bbv=\pm\vartheta_{\Phi, \nu}(\mI)\cdot \overline{\mv},$$ and  for any constant $C$,
$$\vartheta_{\Phi, \nu+C}(\mI)=e^{-C}\vartheta_{\Phi, \nu}(\mI).$$
 By Proposition \ref{p:rod smooth compactification}, the metric $g_\Lambda$  has cone angle $\vartheta(\mI)$ along the 2-sphere in $M$ corresponding to the rod $\mI$. This is the precise meaning of being \emph{conical} in this paper. These types of conical singularities are very mild—they arise from an identification with an incorrect period, similar to orbifold singularities.  In particular,  the curvature of $g_\Lambda$ is uniformly bounded across the conical singularities.

\

Next we explain the convergence of normalized rod vectors for a sequence of augmented harmonic maps. First we introduce a notion of convergence of harmonic maps for the convenience of later discussion. Suppose we have a sequence of smooth maps $\Phi^k$ on $U^\circ\subset \dH^\circ$ where $U\subset\mathbb H$ is an open subset with $U^\circ:=U\cap\mathbb{H}^\circ$, which is locally tamed by an untyped rod structure $\mR^k$ such that for each $k$, $U\cap \Gamma$ is contained in a unique rod $\mI^k$ of $\mR^k$. Suppose that the corresponding rod vectors $\mv^k$ converge to a limit $\mv^\infty\in \R\dP^1$. Let $\mR^\infty$ be an untyped rod structure such that one of the rods contains $U\cap \Gamma$ and the corresponding rod vector is $\mv^\infty$. Let $\Phi^\infty$ be a smooth map on $U^\circ$ which is locally tamed by $\mR^\infty$.
\begin{definition}\label{defi: local uniform convergence}
    We say $\Phi^k$ converges to a limit map $\Phi^\infty$ in $C^{1, \alpha}$ locally over $U$ if the following hold
    \begin{itemize}
        \item For any compact $K\subset U\cap \dH^\circ$, $\Phi^k$ converges to $\Phi^\infty$ in $C^{1, \alpha}(K)$ for all $\alpha\in (0, 1)$.
        \item For any $z\in U\cap \Gamma$, there is a neighborhood $V$ of $z$ in $\dH$, such that if we represent $P_{\mv^k}.\Phi^k$ by functions $(u^k, v^k)$ and $P_{\mv^\infty}. \Phi^\infty$ by functions $(u^\infty, v^\infty)$, then $(u^k, v^k)$ converges to $(u^\infty, v^\infty)$ in $C^{1, \alpha}(V)$ for all $\alpha\in (0, 1)$.
    \end{itemize}
\end{definition}

Now suppose we have a sequence of augmented harmonic maps $(\Phi^k, \nu^k)$ on $U$ which are locally tamed by  rod structures $\mR^k$ and with a uniform bound $d(\Phi^k,\Phi_{\mv^k})\leq C$. From the uniform estimate in Proposition \ref{thm:C11 regularity}, by passing to a subsequence $\Phi^k$ converges $C^{1, \alpha}$ locally over $U$ to a limit harmonic map $\Phi^\infty$ locally tamed by $\mR^\infty$. Furthermore, we have 
$$d(\Phi^\infty, \Phi_{\mv^\infty})\leq C.$$Passing to a further subsequence we may find constants $C_k$ such that $\nu^k+C_k$ converges uniformly to a limit augmentation $\nu^\infty$. It follows that we have the convergence of normalized rod vectors
\begin{equation}\label{e:angle convergence}\lim_{k\rightarrow\infty}e^{C_k}\overline{\mv}^k=\overline{\mv}^\infty.
\end{equation}
Suppose furthermore that each $\mR^k$ is enhanced by $\Lambda_k$ and let $\bbv^k$ be the corresponding primitive vector along $\mv^k$. It follows that 
\begin{equation}
     \lim_{k\to\infty}\frac{1}{\vartheta_{\Phi^k, \nu^k}(\mI^k)}e^{C_k}\bbv^k=\frac{\overline{\mathfrak{v}}^\infty}{2\pi}.
\end{equation}
In particular, to understand the limiting behavior of the cone angles $\vartheta_{\Phi^k, \nu^k}(\mI^k)$, it suffices to study the limiting behavior of the enhancements $\Lambda_k$.

\subsubsection{Near a turning point} 
\label{subsubsec:local regularity near a turning point}

We again let $(\Phi, \nu)$ be an augmented harmonic map defined on  $U^\circ=U\cap \dH$ for some pen subset $U\subset \dH$,  which is locally tamed by an untyped rod structure $\mR$.  Without loss of generality we may assume that $U$ is a neighborhood of the turning point $(0,0)\in\Gamma$  and $U\cap\Gamma$ contains exactly two rods $\mathfrak I_0$ and $\mathfrak I_1$.
 
Again, there is a  PDE study of the singular behavior at the turning points  in the Lorentzian setting, see \cite{HKWX}. In our case we will instead use geometry to directly prove the regularity of the induced toric Ricci-flat metric, which fits our purpose in this paper.  As in \cite{HKWX} it is convenient to apply an element in $SL(2;\mathbb R)$ such that the rod vectors are given by $\tilde{\mathfrak{v}}_0=[v_0]$ and $\tilde{\mathfrak{v}}_1=[v_1]$ where  
\begin{align*}
    v_0=(1,1)^{\top}, && v_1=(1,-1)^{\top}.
\end{align*}
Then the model map $\Phi_{\mv_0, \mv_1}$ given in \eqref{e:turning point model} becomes 
\begin{equation}\label{eq:sin 1/sin}
\Phi_{\tilde{\mv}_0, \tilde{\mv}_1}=\left(\begin{matrix}
	\frac{1}{\sin\theta} &\frac{\cos\theta}{\sin\theta} \vspace{10pt}\\ \frac{\cos\theta}{\sin\theta}& \frac{1}{\sin\theta}
\end{matrix}\right), 
\end{equation} 
where we used the  polar coordinates $(r, \theta)$ on $\dH$, with $\rho=r\sin\theta$ and $z=r\cos\theta$ for $\theta\in[0,\pi]$.
Writing 
$$\Phi=\left(\begin{matrix}e^u+e^{-u}v^2&e^{-u}v\\e^{-u}v&e^{-u} \end{matrix}\right).$$
The harmonic map equation reduces to the following system of two equations:
\begin{align}
	\Delta u+e^{-2u}|\nabla v|^2=0,\label{eq:harmonic u}\\
	\Delta v-2\nabla u\cdot\nabla v=0,\label{eq:harmonic v}
\end{align}
 By assumption we have $d(\Phi, \Phi_{\tilde{\mv}_0, \tilde{\mv}_1})<\infty$,  which implies that 
 \begin{equation}
	|u-\log\sin\theta|\leq C, |v-\cos\theta|\leq C\sin\theta.
\end{equation} 
Consider the Ricci-flat metric on $U^\circ\times \R^2$ given by $$g=e^{2\nu}(d\rho^2+dz^2)+\rho\Phi$$
Using the study in Section \ref{subsubsec:local regularity near rods} we know that along each rod $\mathfrak {I}_\alpha (\alpha=0, 1)$ we have that 
$$\lim_{\rho\to0}\frac{1}{e^\nu\rho}\sqrt{{\rho\Phi(v_\alpha,v_\alpha)}}=\frac{1}{2\pi}c^2_\alpha>0.$$
Now we take the lattice $\Lambda$ in $\mathbb R^2$ generated by $v_\alpha/c_\alpha$. It follows that from Proposition \ref{p:rod smooth compactification} that $g$ descends to a toric Ricci-flat metric over $U^\circ\times \R^2/\Lambda$ and it  naturally extends to a smooth toric Ricci-flat metric on $ B_1(0)\setminus\{0\}\subset \mathbb R^2\times \mathbb R^2$, such that $v_0/c_0$ and $v_1/c_1$ generate the standard torus action.

\begin{lemma}\label{lem:rescaled estimate}
	Over $U\setminus\{0\}$ we have
	\begin{align}
		&|\partial_r^m\partial_{\theta}^n(v-\cos\theta)|\leq Cr^{-m}(\sin\theta)^{2-m-n}\text{ for $0\leq m+n\leq3$}, \label{eq:puncture harmonic estimate v}\\
		&|\partial_r^m\partial_{\theta}^n(u-\log\sin\theta)|\leq Cr^{-m}(\sin\theta)^{1-m-n}\text{ for $1\leq m+n\leq3$}. \label{eq:puncture harmonic estimate u}
	\end{align}
\end{lemma}
\begin{proof}

 Note that the harmonic map equations (\ref{eq:harmonic u}) and (\ref{eq:harmonic v}) are invariant under the rescaling $u_{R}(r,\theta)=u(R^{-1}r,\theta)$ and $v_{R}(r,\theta)=v(R^{-1}r,\theta)$. 
 By considering the rescaled maps $(u_R, v_R)$, it suffices to prove uniform estimates on a fixed annulus $A_{{\frac{1}{2}},1}$ in $\dH^\circ$. Away from the rods, the estimates follow from the standard interior estimates of Schoen-Uhlenbeck. It suffices to prove the estimates in a neighborhood of a rod. For simplicity we consider the rod $\theta=\pi$.
 
	Denote 
	$$\widehat\Phi=\frac12\left(\begin{matrix}
		1 &1\\ 1& -1\end{matrix}\right)\Phi\left(\begin{matrix}
			1 &1\\ 1& -1\end{matrix}\right)=\frac12\left(\begin{matrix}
				e^{u}+e^{-u}(v+1)^2&e^{u}+e^{-u}(v^2-1)\\
				e^{u}+e^{-u}(v^2-1)&e^{u}+e^{-u}(v-1)^2
				\end{matrix}\right).$$
				Then $\widehat \Phi$ has bounded distance to $\frac{1}{\rho}\mathrm{diag}(r+z,r-z)$, hence to the  model $\Phi_{[\bv_0]}$ on the annulus $ A_{{\frac{1}{2}},1}$ near the rod $\theta=\pi$. Set 
	$$\widehat\Phi=\left(\begin{matrix}
		e^{\widehat u}+e^{-\widehat u}{\widehat v}^2 & e^{-\widehat u} \widehat v\\e^{-\widehat u}\widehat v & e^{-\widehat u}
	\end{matrix}\right),$$
	then one can apply the estimates in Proposition \eqref{thm:C11 regularity} to the functions $(\widehat u-\log\rho,\widehat v)$ near $\theta=\pi$ and obtain that
	\begin{align*}
		&|\nabla^k\widehat v|\leq C\rho^{2-k}\text{ for $k=0,1,2,3$},\\
        &|\nabla^k(\widehat u-\log\rho)|\leq C\rho^{1-k}\text{ for $k=1,2,3$}.
	\end{align*}
Returning to the original $(u, v)$ we obtain the desired conclusion.\end{proof}

\begin{lemma}\label{lem:L infinity estimate}
	Over $B_1(0)\setminus\{0\}$, we have  $C^{-1}g_{\mathbb{R}^4}\leq g\leq C g_{\mathbb{R}^4}$. 
\end{lemma}
\begin{proof}
	Notice that the flat metric $g_{\mathbb{R}^4}$ can be written as 
	$$g_{\mathbb R^4}=\frac{1}{2r}(d\rho^2+dz^2)+\rho\Phi_{\tilde\mv_0,\tilde\mv_1}.$$
	It suffices to show that $e^{2\nu}$ is uniformly comparable to $r^{-1}$.
	
	Under the above rescaling  we have $\nu_R(r,\theta)=\nu(R^{-1}r,\theta)$. Applying the estimate \eqref{eqn:nu} to $\nu_R$ on the fixed annulus $A_{\frac{1}{2}, 1}$ we see that  $\nu_\theta$ is uniformly bounded. 

	Now consider $(u,v)$ on $A_{1/R,1}$. Let $\widehat \Phi$  be as given  in the proof of Lemma \ref{lem:rescaled estimate}. From the discussion in Section \ref{subsubsec:local regularity near rods}, we know that $2\nu-\widehat u+\log\rho$ is a constant on the rod $\theta=\pi$ in $A_{1/R,1}$.  This together with our assumption $|u-\log\sin\theta|<C$ imply that $$|(2\nu+\log r)|_{\{\theta=\pi\}}|<C$$ in $A_{1/R,1}$, for  $C$ independent of $R$. 
	Let $R\to\infty$, the uniform bound on $\nu_\theta$ then implies that $e^{2\nu}$ is uniformly comparable to $r^{-1}$ on  $B_1(0)\setminus\{0\}$, as desired.
\end{proof}
It follows that the metric completion of $g$ is given by adding the origin $0$ in $B_1(0)$ and the distance function to the origin $0$ with respect to the metric $g$ is comparable to $r^{{\frac{1}{2}}}$. From Lemma \ref{lem:rescaled estimate} with a rescaling argument, we obtain the curvature bound $$|\Rm_g|_g\leq C\dist_g(0,\cdot)^{-2}.$$

\begin{proposition}\label{p:smoothness near turning point}
	There is a diffeomorphism $F:B_1(0)\setminus\{0\}\to B_1(0)\setminus\{0\}$, such that $F^*g$ extends as a $C^{\infty}$ metric to $B_1(0)$.
\end{proposition}
\begin{proof}
Using the  chopping theorem of Cheeger-Gromov we may find a sequence of neighborhoods $V_j$ of $0$ such that $\partial V_j\subset A_{j^{-1}, 2j^{-1}}(0)$ and $\p V_j$ has second fundamental form bounded by $C_j$. Applying the Chern-Gauss-Bonnet theorem on $B_{{\frac{1}{2}}}(0)\setminus V_j$ and the Ricci-flat equation, and letting $j\rightarrow\infty$  we see that $$\int_{B_1(0)\setminus \{0\}}|\Rm_g|_g^2\dvol_g<\infty.$$
Then by the result of Bando-Kasue-Nakajima \cite{BKN} (see also \cite{DS12}, Section 5) the conclusion follows.
\end{proof}

\subsection{Global theory}
\label{subsec:global regularity}

The following existence  result should be standard (see for example \cite{Weinstein}, \cite{Kunduri}). For completeness and for our later purpose we include a short proof here.
\begin{theorem}\label{t:existence result}	
Given a rod structure $\mathfrak{R}$,  there is a unique harmonic map $\Phi: \dH^\circ\rightarrow \mathcal H$ which is strongly tamed by $\mR$. 
\end{theorem}
\begin{proof}
 We first prove the uniqueness statement.  Suppose we have two  harmonic maps $\Phi_0,\Phi_1$ which are strongly tamed by $\mR$. It is a standard computation that the function $f=d(\Phi_0, \Phi_1)$ can be viewed as an axisymmetric  sub-harmonic function on $\R^3\setminus \Gamma$. The strong tameness assumption implies that $f$ is uniformly bounded and decays to $0$ at infinity. In particular, $f$ extends to a nonnegative subharmonic function on the entire $\mathbb{R}^3$, so it must vanish identically by the maximum principle. 

Now we proceed to prove the existence. We only deal with the case of  Asymptotic Type  $\AFa$ (where $\mv_0=\mv_n$); the other cases are similar. 	Given $\mR$, choose $N$ bigger than the size $\ms(\mR)$, we want to construct a good background smooth map $\widehat{\Phi}: \dH^\circ\rightarrow \mathcal H$. Choose $\tau>0$ such that $100\tau\leq |z_j-z_{j-1}|$ for all $j$. We first define $\widehat{\Phi}$ to be equal to 
\begin{itemize}
    \item $\Phi_{\mv_0}^{\AFa}$ outside $B_{2N}(0)$;
    \item $T_{z_j}^*\Phi_{\mv_j, \mv_{j+1}}$ in $B_\tau(z_j)$ for a turning point $z_j$;
    \item $\Phi_{\mv_j}$ in $N_{\tau}(\mI_j)\setminus (B_{5\tau}(z_j)\cup B_{5\tau}(z_{j+1}))$ for a finite rod $\mI_j$;
    \item $\Phi_{\mv_0}$ in $(B_N(0)\cap N_{\tau}(\mI_0))\setminus B_{5\tau}(z_1)$ and $(B_N(0)\cap N_{\tau}(\mI_n))\setminus B_{5\tau}(z_n)$. 
\end{itemize}
It is easy to check that 
\begin{itemize}
    \item $d(\Phi_{\mv_j}, T_{z_j}^*\Phi_{\mv_j, \mv_{j+1}})<\infty$ in $(B_{5\tau}(z_j)\cap N_{\tau}(\mI_j))\setminus B_{\tau}(z_j)$ for a turning point $z_j$;
    \item $d(\Phi_{\mv_{j+1}}, T_{z_j}^*\Phi_{\mv_j, \mv_{j+1}})<\infty$ in $(B_{5\tau}(z_j)\cap N_{\tau}(\mI_{j+1}))\setminus B_{\tau}(z_j)$ for a turning point $z_j$;
    \item $d(\Phi^{\AFa}_{\mv_0},\Phi_{\mv_0})<\infty$ in $(B_{2N}(0)\cap N_{\tau}(\mI_0))\setminus B_{N}(0)$ and $(B_{2N}(0)\cap N_{\tau}(\mI_n))\setminus B_{N}(0)$. 
\end{itemize}
So by Lemma \ref{l:interpolation} below and a similar interpolation argument, one can extend $\widehat{\Phi}$ to $N_\tau(\Gamma)\setminus \Gamma$ such that $|\Delta\widehat{\Phi}|\leq C$.  Since $\mathcal H$ is simply-connected we can further extend $\widehat{\Phi}$ smoothly to $\dH^\circ$ such that on  $B_{2N}(0)\setminus B_{N}(0)$ we have $d(\widehat{\Phi}, \Phi^{\AFa}_{\mv_0})\leq C$. Notice that  $\Delta\widehat{\Phi}=0$ on $\mathbb H^\circ\setminus{B_{2N}(0)}$. The function 
	\begin{equation}\label{eq:potential v_0}
		F(x):=\int_{\R^3}|x-y|^{-1}|\Delta \widehat{\Phi}|(y)dy>0
	\end{equation}
	is well-defined, It satisfies that $\Delta F=-|\Delta\widehat{\Phi}|$ and $F=O(r^{-1})$ as $r\rightarrow\infty$. 
	Now for any $\delta>0$, on the domain $U_\delta=B_{1/\delta}(0)\cap\{\rho>\delta\}$ one can solve the Dirichlet problem  to get a smooth harmonic map $\Phi_\delta: U_\delta\to\mathcal{H}$ with boundary value $\widehat{\Phi}|_{\partial U_\delta}$.  As $\mathcal{H}$ is negatively curved,  we have $$\Delta{d(\Phi_\delta,\widehat{\Phi})}\geq-|\Delta\widehat{\Phi}|=\Delta F$$ in the weak sense, which gives $$\Delta(d(\Phi_\delta,\widehat{\Phi})-F)\geq0.$$
	Maximum principle then implies that  over $U_\delta$ we have
    $$d(\Phi_\delta,\widehat{\Phi})\leq F.$$
    Passing to a limit as $\delta\to0$, we get a smooth harmonic map $\Phi:\mathbb H^\circ\to\mathcal{H}$ with $d(\Phi, \widehat\Phi)\leq F$. 
\end{proof}

\begin{lemma}\label{l:interpolation}
	Suppose $\Psi$ is  harmonic map  on $B_\delta(0)\cap  \dH^\circ$ such that 
    \begin{itemize}
        \item $d(\Psi, \Phi_{[\bv_0]})\leq c$;
        \item writing $\Psi$ as $(u,v)$ as in Section \ref{subsubsec:local regularity near rods} then $|\nabla^{k}u|,|\nabla^{k}v|<C_k$. 
    \end{itemize} 
    Then there are $\epsilon\in (0, \frac{\delta}{100})$ and $C_k'>0$ depending only on $c$ and $C_k$, and  a smooth map  $\Phi: N_\epsilon(\Gamma)\cap \dH^\circ\rightarrow \mathcal H$, such that 
    \begin{itemize}
        \item $\Phi=\Psi$ in $B_\epsilon(0)$, $\Phi=\Phi_{[\bv_0]}$ in $N_\epsilon(\Gamma)\setminus B_{10\epsilon}(0)$;
        \item writing $\Phi$ in terms of $(u',v')$ as in Section \ref{subsubsec:local regularity near rods} again then $|\nabla^{k}u'|,|\nabla^kv'|<C_k'$.
    \end{itemize}
    In particular, $|\Delta\Phi|\leq C$.
\end{lemma}
\begin{proof}
We can express $\Psi$  in terms of a pair of two functions $(u, v)$ as in Section \ref{subsubsec:local regularity near rods}. On the other hand, the rod model map $\Phi_{[\bv_0]}$ corresponds to $(u_0, v_0)\equiv (0, 0)$. The assumption implies that $|u|+\rho^{-1}|v|\leq c$. Using the estimates in Proposition \ref{thm:C11 regularity} and a simple cut-off function $\chi=\chi(z)$ one can achieve the conclusion. 
\end{proof}

Let $\Phi$ be a harmonic map on $\dH^\circ$ which is strongly tamed by a rod structure $\mR$ of  Asymptotic Type  $\sharp$.  For each $\sharp$ we define the model augmentation 
$$\nu^\sharp=\nu_{\mv_0, \mv_n}^\sharp,$$
using the definitions in Section \ref{ss:model maps and tameness}.

\begin{proposition}\label{prop:normalizing augmentation}
	There is a unique choice of the augmentation $\nu$ such that $\nu=\nu^\sharp+O(r^{-1})$.
\end{proposition}
\begin{proof}
First consider the case where $\Phi$ is of  Asymptotic Type  $\Ae$. Without loss of generality we assume $\mv_0=[\bv_{\frac{\alpha}{2}}]$ and $\mv_1=[\bv_{-\frac{\alpha}{2}}]$. Similar to the local analysis in Section \ref{subsec:local regularity}, since $d(\Phi,\Phi^{\Ae_\alpha})<C/r$, by scaling down one can show that $\nu=\nu^{\Ae_\alpha}+\sigma+O(r^{-1})$ for a constant $\sigma$.  We only need to choose $\nu$ such that $\sigma=0$.

Next consider the $\AFa$ case. We may assume $\mv_0=[\bv_0]$. In this case for $R$ large we consider a modified rescaling 
$$\Phi_R=D_{\sqrt{R}}.U_\beta.\varphi_{R^{-1}}^*\Phi $$
$$\nu_R=\varphi_{R^{-1}}^*\nu.$$
It is easy to see that $$D_{\sqrt{R}}.U_\beta.\varphi_{R^{-1}}^*\Phi^{\AFa}=\Phi^{\AF_0}.$$
Then  on the annulus we have $A_{1,2}$  $$d(\Phi_R, \Phi^{\AF_0})\leq CR^{-1}.$$
Then we can apply the local uniform estimates as in Section \ref{subsec:local regularity} to get that $\nu=\sigma+O(r^{-1})$ for a constant $\sigma$. We choose $\nu$ such that $\sigma=0$. 

The $\ALFm$ case can be treated in a similar fashion. Again we assume $\mv_0=[\bv_{\frac{\alpha}{2}}]$ and $\mv_1=[\bv_{-\frac{\alpha}{2}}]$. As in the $\AFa$ case we need to perform a modified rescaling. 
First consider 
$$\Phi_R=D_{\sqrt{R}}. \varphi_{R^{-1}}^*\Phi.$$
One can check that for any $\delta>0$ small, on $A_{1, 2}\setminus N_\delta(\Gamma)$, we have a uniform control on  the corresponding operation on the model $\widetilde \Phi_R:=D_{\sqrt{R}}. \varphi_{R^{-1}}^*\Phi^{\TN^{-}_\alpha}$:
$$d(\widetilde\Phi_R, \Phi^{\AF_0})\leq \frac{C_\delta}{R}.$$
From this one obtains that $\nu=\sigma+O(r^{-1})$ as $r\rightarrow\infty$ on the set $\{\rho\geq \delta r\}$. We choose $\nu$ such that $\sigma=0$.

It remains to understand $\nu$ in the set $\{\rho\leq \delta r\}$. We only consider the case when $z<0$. Consider 
$$\Phi_R=D_{\sqrt{R}}.L_{-\alpha/2}. \varphi_{R^{-1}}^*\Phi. $$
Correspondingly for $\widetilde \Phi_R:=D_{\sqrt{R}}. L_{-\alpha/2}. \varphi_{R^{-1}}^*\Phi^{\TNam}$, we have that on $A_{1, 2}\cap N_\delta(\Gamma_{-})$, 
$$d(\widetilde\Phi_R, \Phi^{\AF_0})\leq \frac{C_\delta}{R}.$$ 
Then as before we again obtain that $\nu=O(r^{-1})$ as $r\rightarrow\infty$, on the set $\{\rho\leq \delta r\}$. 

The $\ALFp$ case is similar and we omit the details.
\end{proof}
\begin{definition}[Normalized augmentation]
    We call the unique choice $\nu$ in Proposition \ref{prop:normalizing augmentation} the  \emph{normalized} augmentation. 
\end{definition}

\begin{remark}\label{rem:continuity of cone angles}
Suppose we only consider the normalized augmentation, then by Theorem \ref{t:existence result} and the discussion at the end of Section \ref{subsec:local regularity} we know that for each rod $\mI$, the normalized rod vector associated to $\mI$ depends continuously on the rod vectors. If we further fix an enhancement, then the cone angle along $\mI$  depends continuously on the rod lengths. This is sufficient for our purpose in this paper. It may be interesting to explore better behavior of the cone angle function. We mention that in the Lorentzian setting, the analogous cone angle function was shown to be  differentiable in \cite{HKWX2}. 
    
\end{remark}

\subsection{From tame harmonic maps to gravitational instantons}

Consider an augmented harmonic  map $(\Phi, \nu)$ on $\dH^\circ$ which is globally tamed by a rod structure $\mathfrak R$. Given any smooth enhancement $\Lambda$, the quotient $M^\circ:=\mathbb H^\circ\times \mathbb R^2/\Lambda$ is then naturally a toric Ricci-flat manifold, with the metric given by 
\begin{equation}\label{eq:quotient toric metric}
	g_\Lambda=e^{2\nu}(d\rho^2+dz^2)+\rho\Phi.
\end{equation}
As before $M^\circ$ may be compactified into a smooth toric manifold $M$ by  adding points with stabilizers over $\Gamma=\p \mathbb H^\circ$. The regularity results in Section \ref{subsec:local regularity} imply that $g_\Lambda$ is conical,  with a cone angle $\vartheta_j=\vartheta_{\Phi, \nu}(\mI_j)$ along the spheres $S_j^2$ corresponding to the rod $\mathfrak I_j$.

Using the model augmented harmonic maps introduced at the beginning of this section, we obtain asymptotic models of gravitational instantons. 

\begin{itemize}
    \item $(M^{\Ae}, g^{\Ae})$: consider the rod structure $\mR^{\Ae_1}$ with the smooth enhancement $\Lambda$ generated by $\bv_{\frac{1}{2}}$ and $\bv_{-\frac{1}{2}}$, and use the augmented harmonic map $(\Phi^{\Ae_1}, \nu^{\Ae_1})$. 
    This is the flat $\R^4$ (as discussed in Example \ref{ex1}). 
 \item $(M^{\TNm}, g^{\TNm})$: consider the rod structure $\mR^{\TN^{-}_1}$ with the  smooth enhancement $\Lambda$ generated by $\bv_{\frac{1}{2}}$ and $\bv_{-\frac{1}{2}}$, and use the augmented harmonic map $(\Phi^{\TN^{-}_{1}}, \nu^{\TN^{-}_1})$. This is the anti-Taub-NUT space  (as discussed in Example \ref{example:negative Taub-NUT}).  
 \item $(M^{\TNp}, g^{\TNp})$: consider the rod structure $\mR^{\TN^{+}_1}$ with the smooth enhancement $\Lambda$ generated by $\bv_{2}$ and $\bv_{-2}$, and use the augmented harmonic map $(\Phi^{\TN^{+}_{1}}, \nu^{\TN^{+}_1})$. This is the  Taub-NUT space  (as discussed in Example \ref{example:positive Taub-NUT}). 

\item $(M^{\AFa}, g^{\AFa})(\beta\in \R)$: consider the rod structure $\mR^{\AFa}$ with the standard smooth enhancement $\Lambda=\dZ^2\subset \R^2$, and use $(\Phi^{\AFa}, \nu^{\AFa})$. This is the quotient $\R^3\times \R/\langle q_\beta\rangle$, as discussed in Example \ref{ex:AF}. Notice that $(M^{\AFa}, g^{\AFa})$ and $(M^{\AF_{\beta+1}}, g^{\AF_{\beta+1}})$ are isometric as toric flat manifolds. 
\end{itemize}

In each of the above model space, we denote by $\mathfrak r$ the distance function. 

\begin{definition}
     A \emph{$\mathbb T$-isometry} between two toric Ricci-flat 4-manifolds $(M_1, g_1)$ and $(M_2, g_2)$ is an orientation-preserving diffeomorphism $F:M_1\rightarrow M_2$ such that  $F^*g_2=g_1$ and  $F(t.x)=t. F(x)$ for all $t\in \dT$.
\end{definition}

\begin{definition}\label{def:Type of gravitational instantons}
 A 4-dimensional toric Ricci-flat end $(M,g)$ is said to be  of  Asymptotic Type
 \begin{itemize}
   \item   $\ALe$ if it has a finite cover which admits a $\dT$-isometry to the end of the model $(M^{\Ae}, g^{\Ae})$, such that $$|\nabla^k_{g}(g-g^{\Ae})|_g=O(\mathfrak r^{-k-\delta})$$ for some $\delta>0$ and all $k\in\mathbb{N}$.
   \item $\ALF^{\pm}$ if it has a finite cover which admits a $\dT$-isometry to the end of the model $(M^{\ALF^{\pm}}, g^{\ALF^{\pm}})$, such that $$|\nabla^k_{g}(g-g^{\ALF^{\pm}})|_g=O(\mathfrak r^{-k-\delta})$$ for some $\delta>0$ and all $k\in\mathbb{N}$.
  \item  $\AFa$ if it admits a $\dT$-isometry to the end of the model $(M^{\AFa}, g^{\AFa})$, such that $$|\nabla^k_{g}(g-g^{\AFa})|_g=O(\mathfrak r^{-k-\delta})$$ for some $\delta>0$ and all $k\in\mathbb{N}$.
 \end{itemize}
\begin{remark}
    One can check  that (see Remark \ref{remark3.19}) if we reverse the orientation,  then up to an automorphism of the torus $\dT$, an $\ALE$ end is still an $\ALE$ end, an $\ALF^{\pm}$ end becomes an $\ALF^{\mp}$ end,  and an $\AFa$ end becomes an $\AF_{-\beta}$ end. The automorphism of the torus $\dT$ changes the $\dT$ action, but it does not change the geometry of the metrics. 
    \end{remark}
\end{definition}

\begin{definition}
A toric gravitational instanton $(M, g)$ is said to be of  Asymptotic Type $\ALe/\ALFp/\ALFm/\AFa$ if its end is of Asymptotic Type $\ALe/\ALFp/\ALFm/\AFa$ .
\end{definition}
\begin{remark}
    In the literature sometimes the $\ALF^{\pm}$ and $\AFa$ Asymptotics Types are both referred to as the $\ALF$ Asymptotic Type. We choose to distinguish between these two cases, partly due to the fact that in the $\AFa$ case we have the extra continuous parameter $\beta$ which changes the asymptotic behavior of the metric, but in the $\ALF$ case such a continuous parameter does not exist, by the result proved in \cite{LS2024}.
\end{remark}
Now given a rod structure  $\mathfrak R$. Let  $(\Phi, \nu)$ be an augmented harmonic map on $\dH^\circ$ which is  globally tamed by $\mR$, with $\nu$ normalized. Given any lattice $\Lambda$ in $\R^2$ such that $\mv_0$ and $\mv_n$ are generated by vectors in $\Lambda$, as above by working on the complement of a large compact set $K\subset \dH$, we obtain a toric Ricci-flat \emph{end} $(E_\Lambda, g_{\Lambda})$ with possible conical singularities. We can make the end smooth by choosing a canonical lattice $\Lambda$.

\begin{definition}[Canonical lattice]\label{def:canonical enhancement}
    Let $\mR$ be a rod structure of Asymptotic Type $\sharp$. We define the canonical lattice $\Lambda_{\mR}$
 associated to $\mR$ as follows. 
 \begin{itemize}
     \item If $\sharp=\Ae, \ALF^{-}$, then $\Lambda_{\mR}$ is generated by $P_{\mv_0, \mv_n}^{-1}.\bv_{\frac{\alpha}{2}}$ and $P_{\mv_0, \mv_n}^{-1}.\bv_{-\frac{\alpha}{2}}$, where $\alpha=\mv\wedge\mv'$;
     \item If $\sharp=\ALF^+$, then $\Lambda_{\mR}$ is generated by $\overline P_{\mv_0, \mv_n}^{-1}.\bv_{\frac{2}\alpha}$ and $\overline P_{\mv_0, \mv_n}^{-1}.\bv_{-\frac{2}{\alpha}}$, where $\alpha=\mv\wedge\mv'$;
     \item If $\sharp=\AFa$, then $\Lambda_{\mR}$ is generated by $P_{\mv}^{-1}. \bv_0$ and $P_{\mv}^{-1}. [(0, 1)^T]$. \end{itemize}
 \end{definition}
 \begin{remark}
     If $\mR$ is enhanceable, then $\Lambda_{\mR}$ is an enhancement for $\mR$. \end{remark}

\begin{proposition}\label{prop:smooth end lattice}
    The toric  Ricci-flat end $(E_{\Lambda_\mR},g_{\Lambda_{\mR}})$ is smooth without conical singularities and of  Asymptotic Type  $\sharp$.
\end{proposition}
\begin{proof}
   We only discuss the case of Asymptotic Type $\ALF^-$; the other cases are similar. Without loss of generality we assume $\mv_0=[\bv_{\frac{\alpha}{2}}]$ and $\mv_n=[\bv_{-\frac{\alpha}{2}}]$ as before. In this case,  $\Lambda_{\mR}$ is generated by $\bv_{\frac{\alpha}{2}}$ and $\bv_{-\frac{\alpha}{2}}$.
    Then the cone angle 
        $$\vartheta_{\Phi, \nu}(\mI_0)=\vartheta_{\Phi, \nu}(\mI_n)=2\pi.$$
    In particular we know $(M, g)$ is smooth outside a compact set. It follows from Proposition \ref{prop:appendix 1} proved in the Appendix, that $g$ has an end of  Asymptotic Type  $\sharp$.
\end{proof}

Now let $\Lambda$ be a smooth enhancement for $\mR$. If $\Lambda$ coincides with $\Lambda_{\mR}$, then $(M, g_{\Lambda_{\mR}})$ is a gravitational instanton of  Asymptotic Type  $\sharp$, but with possible  codimension 2 conical singularities in the compact region. In general, the space is asymptotic to a quotient of the model ends.

\begin{remark}Conversely, one can show that any toric gravitational instanton of Asymptotic Type  $\ALE/\ALF^{\pm}/\AF_\beta$ (up to scaling) may arise from a harmonic map which is strongly tamed by an enhanced rod structure of Asymptotic Type $\ALE/\ALF^{\pm}/\AFa$, for an appropriate choice of a smooth enhancement. Here in the $\ALF^{\pm}/\AFa$ case a scaling of the metric is necessary, since the model end is not invariant under rescaling. 
\end{remark}

\subsection{The case of Asymptotic Type  \texorpdfstring{$\Ae$}{}} \label{ss: Type AE case}
We first prove a rigidity result.
\begin{proposition} \label{p:AE rigidity}
    Suppose $\Phi$ is a harmonic map globally tamed by a rod structure $\mR$ of  Asymptotic Type  $\Ae$. Then there is a $Q\in SL(2; \R)$ with $Q.\mv_0=\mv_0$, $Q.\mv_n=\mv_n$, such that  $Q.\Phi$ is the unique harmonic map which is strongly tamed by $Q.\mR$ (as produced in Theorem \ref{t:existence result}). Here we make $Q.\mR$ into a rod structure by assigning its Asymptotic Type to be $\Ae$.
\end{proposition}
\begin{proof}
  Without loss of generality we may assume $\bv_0=[\bv_{\frac{\alpha}{2}}]$ and $\mv_n=[\bv_{-\frac{\alpha}{2}}]$ for some $\alpha>0$.
  We first choose an arbitrary augmentation $\nu$. Let $\overline{\mv}_0=\mu_0\bv_{\frac{\alpha}{2}}$ and $\overline{\mv}_n=\mu_n \bv_{-\frac{\alpha}{2}}$ be the normalized rod vectors on the rod $\mI_0$ and $\mI_n$ respectively. Then we adjust $\nu$ by a constant such that $\mu_0\cdot \mu_n=1$. There is the matrix $P\in SL(2;\mathbb{R})$ that maps $\mu_0\bv_{\frac{\alpha}{2}},\mu_n\bv_{-\frac{\alpha}{2}}$ to $\bv_{\frac{\alpha}{2}},\bv_{-\frac{\alpha}{2}}$. 
  Replacing $\Phi$ by $P. \Phi$ we can assume $\mu_0=\mu_n=1$.
  
  Define $\Lambda$ to be the lattice generated by $\overline{\mv}_0$ and $\overline{\mv}_n$. 
     From Section \ref{subsec:local regularity}, we know that $(\mR, \Lambda, \Phi, \nu)$ defines a smooth toric Ricci-flat end $(M, g)$. Denote by $\p_{\phi_1}$ and $\p_{\phi_2}$ the Killing fields on $M$ corresponding to $\overline{\mv}_0$ and $\overline{\mv}_n$ respectively.

    As in the proof of the local regularity near a turning point in Section \ref{subsubsec:local regularity near a turning point}, by scaling down 
    we know  $g$ is uniformly equivalent to the flat metric $g_{\R^4}$ on $\R^4$ and with quadratic curvature decay.  Applying the Chern-Gauss-Bonnet theorem in Lemma \ref{lem:Gauss-Bonnet} we see that $\int_{M}|\Rm_g|^2\dvol_g<\infty$. 

   Now we may apply the arguments of Bando-Kasue-Nakajima \cite{BKN} to conclude that \( g \) is of Asymptotic Type \( \Ae \). Note that only the information of a Ricci-flat \emph{end} is needed for the arguments in \cite{BKN} to go through. One can diffeomorphically identify \( M \) with \( \mathbb{R}^4 \setminus K \) for some compact \( K \), such that  
\[
|\nabla^k(g - g_{\mathbb{R}^4})| = O(\mathfrak{r}^{-4-k}) \quad \text{for all } k \geq 0.
\]
Since a Killing field \( X \) for \( g \) satisfies \( \nabla_g^* \nabla_g X = 0 \), using the asymptotics of \( g \) and standard elliptic theory for \( g_{\mathbb{R}^4} \), we see that there is a rotational Killing field \( \widetilde{X} \) for \( \mathbb{R}^4 \) such that  
\[
|\nabla^k (X - \widetilde{X})|_g = O(\mathfrak{r}^{1-\epsilon-k})
\]
for some \( \epsilon > 0 \). It follows that the torus \( \mathbb{T} \) action is asymptotic to a standard \( \mathbb{T} \) action on \( \mathbb{R}^4 \), in the sense that the distance between the corresponding \( \mathbb{T} \) orbits is of order \( O(\mathfrak{r}^{1-\epsilon}) \).  

By the standard center-of-mass techniques of Grove-Karcher (see, for example, \cite{Rong}), we may assume the diffeomorphism is equivariant with respect to the \( \mathbb{T} \) action. Thus, one can identify the Killing fields \( \p_{\phi_1} \) and \( \p_{\phi_2} \) as the Killing fields on \( (\mathbb{R}^4, g_{\R^4}) \) which give rise to the harmonic map \( \Phi^{\Aea} \).  

Denote by \( G = \rho \Phi \) and \( \widetilde{G} = \widetilde{\rho} \widetilde{\Phi} \) the Gram matrices associated to \( g \) and \( g_{\mathbb{R}^4} \), respectively. Notice that  
\[
\widetilde{\Phi} = \frac{1}{\widetilde{\rho}}  \left(
\begin{matrix}  
\frac{\alpha \widetilde{r}}{2} & \widetilde{z} \\  
\widetilde{z} & \frac{2\widetilde{r}}{\alpha}  
\end{matrix}\right).  
\]
The asymptotics of \( g \) imply that  
\[
|\nabla^k(\widetilde{G} G^{-1})|_{g} = O(\mathfrak{r}^{-\epsilon-k}) \quad \text{for all } k \geq 0.
\]
In particular, \( d(\widetilde{\Phi}, \Phi) = O(\mathfrak{r}^{-\epsilon}) \). On the \( \mathbb{T} \)-quotient, we have \( \widetilde{\rho} \rho^{-1} = 1 + O(r^{-\epsilon}) \), which then gives \( \widetilde{z} = z + O(r^{1-\epsilon}) \). It then follows that  
\[
d(\widetilde{\Phi}, \Phi^{\Aea}) = O(r^{-\epsilon}).
\]
Altogether, we obtain  
\[
d(\Phi, \Phi^{\Aea}) = O(r^{-\epsilon}).
\]
Arguing as in the proof of the uniqueness part in Theorem \ref{t:existence result} we conclude that $\Phi=\Phi^{\Aea}$.

\end{proof}

\begin{remark}\label{rem:TN AF rigidity}
    It may be possible to prove the above result by a pure PDE method. Although not needed in the proof of Theorem \ref{t:main}, similar rigidity results also hold in the $\ALF^\pm$ and $\AFa$ case, see Proposition \ref{p:ALF/AF rigidity}.  
\end{remark}

\begin{lemma}\label{lem:Gauss-Bonnet}
    Let $(M, g)$ be a 4-dimensional Ricci-flat end satisfying $|\Rm_g|=O(r^{-2})$ as $r\rightarrow\infty$, then we have 
    $$\int_M |\Rm_g|^2\dvol_g<\infty. $$In particular, if $(M, g)$ is complete, then it is a gravitational instanton, and moreover,
        \begin{equation}\chi(M)=
        \int_M |\Rm_g|^2\dvol_g+\delta(M),
        \end{equation}
        where $\delta(M)=\frac{1}{|\Gamma|}$ if $M$ is of Asymptotic Type $\ALe$ with asymptotic cone $\R^4/\Gamma$, and $\delta(M)=0$ otherwise.
    \begin{proof}
      The first fact  follows from the  theorem of Cheeger-Gromov and Gauss-Bonnet-Chern theorem on 4-manifolds with boundary. For the second statement,  if $(M, g)$ has Euclidean volume growth then we know it is of Asymptotic Type $\ALe$ by \cite{BKN} and  $|\Rm|=O(r^{-3})$ as $r\rightarrow\infty$, then the desired formula follows from Gauss-Bonnet-Chern theorem. If $(M,g)$ has non-Euclidean volume growth then we know the end of $M$ has an $\mathcal F$-structure so we can exhaust $M$ by bounded domains $M_j$ with $\chi(M_j)=\chi(M)$. Then we apply the  chopping theorem of Cheeger-Gromov and the Gauss-Bonnet-Chern theorem to conclude, noticing that the boundary integral goes to zero as $j\rightarrow\infty$.
    \end{proof}
\end{lemma}

In the case of $2$ turning points, we can now classify all harmonic maps which are strongly tamed by a rod structure of  Asymptotic Type  $\Ae$. 
  Given $n_1, n_2, a>0$, we consider the \emph{Eguchi-Hanson harmonic map}
\begin{equation}
	\Phi_{\EH;n_1, n_2, a}=\frac{1}{\rho}\left(\begin{matrix}
		H^{-1}\cA^2+\rho^2H & H^{-1} \cA \\
		H^{-1}\cA & H^{-1}, 
	\end{matrix}\right),
\end{equation}
where $$H=\frac{n_1}{2r}+\frac{n_2}{2r_{2a}},$$ $$\cA=\frac{n_1z}{2r}+\frac{n_2(z-2a)}{2r_{2a}},$$  and $$r_{2a}=\sqrt{\rho^2+(z-2a)^2}.$$  
The associated rod structure $\mR_{\EH; n_1, n_2, a}$ consist of two turning points $z_1=0$ and $z_2=2a$, and three rod vectors 
 $$\mv_0=[\bv_{\frac{n_1+n_2}{2}}], \ \ \mv_1=[\bv_{\frac{n_2-n_1}{2}}], \ \ \mv_2=[\bv_{-\frac{n_1+n_2}{2}}].$$
 We choose the augmentation  $\nu=\frac{1}{2}\log H$, then one can compute the normalized rod vectors are given by $$\overline{\mv}_0=\bv_{\frac{n_1+n_2}{2}}, \overline{\mv}_1=\bv_{\frac{n_2-n_1}{2}}, \ \ \overline{\mv}_2=\bv_{-\frac{n_1+n_2}{2}}.$$ The associated gravitational instantons are the well-known Eguchi-Hanson gravitational instantons. These are hyperk\"ahler metrics, which can also be obtained via the Gibbons-Hawking ansatz. See the discussion in Section \ref{subsubsec:Type I case}.

Suppose we are given an enhancement $\Lambda$. Then we can find primitive rod vectors $\bbv_0, \bbv_1, \bbv_2$ such that 
$$\bbv_1=\gamma_0\bbv_0+\gamma_2\bbv_2$$
for some $\gamma_0, \gamma_2>0$. 
Denote $\vartheta_j=\vartheta_{\Phi, \nu}(\mI_j)$. Then we have 
$$\vartheta_1\bv_{\frac{n_2-n_1}{2}}=\gamma_0 \vartheta_0\bv_{\frac{n_1+n_2}{2}}+\gamma_2\vartheta_2\bv_{-\frac{n_1+n_2}{2}}.$$
 In particular, we have 
 \begin{equation}\label{e:angle equality}\vartheta_1=\gamma_0\vartheta_0+\gamma_2\vartheta_2.
 \end{equation}
 The upshot is that the cone angles along any two rods determine the cone angle along the third rod. 

\begin{lemma}\label{lem:uniquness of EH}
    Any harmonic map $\Phi$ strongly tamed by $\mR$ of  Asymptotic Type  $\Ae$ with $2$ turning points is given by the Eguchi-Hanson harmonic map $P_{\mv_0,\mv_2}^{-1}.\Phi_{\EH; n_1, n_2}$ for some $n_1, n_2>0$, where $\mv_0\wedge\mv_2=n_1+n_2$.
\end{lemma}
\begin{proof} 
It follows from the definition that $P_{\mv_0,\mv_2}.\Phi$ is strongly tamed by the rod structure $P. \mR$ (with Asymptotic Type again $\Ae$).  We can find unique $n_1, n_2>0$ such that the rod vectors coincide with that of $\mR_{\EH; n_1, n_2}$. 
The conclusion then follows from the uniqueness statement in Theorem \ref{t:existence result}. 
\end{proof}

We now prove  an \emph{angle comparison} result, which is a consequence of the Bishop-Gromov volume monotonicity. Let $(\Phi, \nu)$ be an augmented harmonic map which is strongly tamed by a rod structure $\mR$ of  Asymptotic Type  $\Ae$ and suppose that  $\Lambda$ is a smooth enhancement generated by vectors $\bbv_0\in \mv_0$ and $\bbv_n\in \mv_n$. For each $j=0, \cdots, n$, we denote $\vartheta_j=\vartheta_{\Phi, \nu}(\mI_j)$.

\begin{proposition} \label{p:angle comparison}
   Suppose $n\geq 2$, then we have
    \begin{align}\label{e:angle comparison}
        \vartheta_1\geq\vartheta_{n}, \ \ \ \  and\ \ \ \ 
        \vartheta_{n-1}\geq\vartheta_0.
    \end{align}
\end{proposition}
\begin{proof}
Since the inequalities are preserved if we adjust $\nu$ by a constant, we may assume all the cone angles $\vartheta_j\in (0, 2\pi)$. We only prove the first inequality; the second one can be proved in the same way. 

As before we let $(M, g)$ be the toric Ricci-flat metric with cone angle $\vartheta_j$ along the 2-sphere $S_j$ corresponding to $\mI_j$. Since the enhancement is smooth, $M$ is a smooth manifold. For $q\in M$ denote by $B_r(q)\subset M$ the geodesic ball of radius $r$ around $q$.   Denote by $p\in M^\sharp$ the point corresponding to the turning point $z_1$. 

We claim that the function $$F(r)=\omega_4^{-1}r^{-4}\Vol_g(B_r(p))$$ is decreasing in $r$, where $\omega_4$ is the volume of the unit ball in $\R^4$. This is a version of the Bishop-Gromov inequality and it is not difficult to adapt the usual proof to our conical setting. The key observation is the weak convexity lemma below, which says that $M^\circ$ is geodesically convex. Choose a sequence of points $p_j\in M^\circ$ that converges to $p$ and denote 
$$F_j(r)=\omega_4^{-1}r^{-4}\Vol_g(B_r(p_j)).$$
From the geodesic convexity we know that the inequality $$\Delta d(p_j, \cdot)\leq \frac{3}{d(p_j, \cdot)}$$ holds weakly on $M^\circ\setminus \{p_j\}$. Integrating this and noticing that the complement $M\setminus M^\circ$ has zero Hausdorff measure, we obtain the monotonicity of $F_j$. Taking limit as $j\rightarrow\infty$ the claim then follows. 

Now it is a simple computation that 
$$\lim_{r\rightarrow0}F(r)=\vartheta_0\cdot\vartheta_1.$$
On the other hand, since $\Phi$ is of  Asymptotic Type  $\Ae$, we see that
$$\lim_{r\rightarrow\infty}F(r)=\vartheta_0\cdot\vartheta_n.$$
So we have $\vartheta_1\geq\vartheta_n$. 
 
\end{proof}

\begin{lemma}[Weak convexity]
    $M^\circ$ is geodesically convex when all $\vartheta_j<2\pi$, i.e., any minimizing geodesic connecting $q_1$ and $q_2$ in $M^\circ$ is contained in $M^\circ$.
\end{lemma}
\begin{proof}
    Let $\gamma:[0, L]\rightarrow M^\circ$ be a unit-speed minimizing geodesic from $q_1$ to $q_2\in M^\circ$. Suppose the claim does not hold and  let $t_0>0$ be the first time such that $w=\gamma(t_0)\in M\setminus M^\circ$. 
    
    Consider the tangent cone $\mathcal C$ at $w$, which is given by $\mathbb R^2_{\vartheta_j}\times \mathbb R^2$ if $w$ is contained in $M^*$, and is $\mathbb{R}^2_{\vartheta_j}\times\mathbb R^2_{\vartheta_{j+1}}$ if $w$ is a point in $M^\sharp$. Here $\mathbb R^2_\vartheta$ denotes the 2-dimensional flat cone with angle $\vartheta$.  It is easy to see that after passing to a subsequence $\gamma$ scales to a limit  minimizing geodesic line $\gamma_\infty$ in $\mathcal C$ passing through the limit $w_\infty$ of $w$. In particular, $w$ must be contained in $M^*$, since there is no geodesic line in $\mathbb{R}^2_{\vartheta_j}\times\mathbb R^2_{\vartheta_{j+1}}$ when $\vartheta_j,\vartheta_{j+1}<2\pi$ passing through the origin $w_\infty$. The line $\gamma_\infty$ must be contained in the singular set of $\mathcal C$. So $\gamma$ must be tangential to the closure of some component of $M^*$, say $S_j$, at the point $w$. The latter is a submanifold diffeomorphic to $S^2$ or $\R^2$.

    Next we claim that for some $\epsilon>0$ small, $\gamma|_{[t_0-\epsilon, t_0]}$ must be contained in $M\setminus M^\circ$ entirely, which leads to a contradiction. To see this, notice that by the local analysis in Section \ref{subsec:local regularity},
    locally near $w$ we may rescale the primitive rod vector $\bbv_j$ to unwrap the conical singularity. More precisely, in a neighborhood \( U \) of \( w \), consider the \( \mathbb{Z} \)-cover \( \widehat{U} \) of \( U \setminus S_j \), obtained by unwrapping along the primitive rod vector \( \bbv_j \). Next, take the \( \mathbb{Z} \)-quotient of \( \widehat{U} \), generated by the normalized rod vector \( \overline{\mv}_j \), and complete the space by adding back the part \( U \cap S_j \), denoted by \( S_j' \). This construction yields a smooth toric Ricci-flat metric \( (U', g') \), containing a totally geodesic 2-dimensional submanifold \( S_j' \).

Now, for sufficiently small \( \epsilon > 0 \), the restriction \( \gamma|_{[t_0 - \epsilon, t_0]} \) lifts to a minimizing geodesic \( \gamma' \) in \( U' \), which is tangent to some point \( w' \in S_j' \) but not contained entirely in \( S_j' \). However, since \( S_j' \) is totally geodesic and \( g' \) is smooth, it follows that any geodesic tangent to \( S_j' \) at a point must lie entirely within \( S_j' \). This leads to a contradiction.
\end{proof}

\section{Degeneration of rod structures and bubbling of harmonic maps}\label{sec:degeneration of rod structures}
In this section, we examine the limiting behavior of a sequence of  harmonic maps which are strongly tamed by a sequence of enhanced rod structures. For our purpose in Section \ref{sec:applications}, it is crucial to analyze the bubbling phenomenon when turning points come together or diverge to infinity. This is done in Section \ref{ss:background map in a fixed scale} and \ref{ss:uniform control of tame harmonic maps}. We were informed by Marcus Khuri that similar bubbling analysis in the Lorentzian setting was performed in \cite{HKWX2} for a different purpose. In Section \ref{ss:limiting behavior of the cone angles} we apply the bubbling analysis to obtain information on the limiting behavior of cone angles when we work with a fixed enhancement and Asymptotic Type but vary the rod lengths.

\subsection{Rescaling of model harmonic maps}
\label{ss:rescaling of model harmonic maps}
We first study the effect of the rescaling map $\varphi_\lambda:\dH\rightarrow \dH$ on  the model families of harmonic maps we wrote down in Section \ref{ss:model maps and tameness}. For the $\Ae$ family we have rescaling invariance, that is, for any $\lambda>0$,
\begin{equation}
	\varphi_\lambda^*\Phi^{\Aea}=\Phi^{\Aea}.
\end{equation}For the $\AF$ family we have
$$\varphi_\lambda^*\Phi^{\AFa}=\varphi_\lambda^*(U_{-\beta}.\Phi_{[\bv_0]})=U_{-\beta}.D_{{\sqrt{\lambda}}}. \Phi_{[\bv_0]}=(U_{-\beta} D_{{\sqrt{\lambda}}}U_{\beta}). \Phi^{\AFa}.$$
For the $\TN^-$ family we have 
\begin{equation}
	\varphi_\lambda^*\Phi^{\TNam}=\frac{\lambda}{\rho}\left(\begin{matrix}
		(1+\frac{{\alpha}\lambda}{2r})^{-1}(\frac{{\alpha}z}{2r})^2+\frac{\rho^2}{\lambda^2}(1+\frac{{\alpha}\lambda}{2r}) & \ \  \ \ (1+\frac{{\alpha}\lambda}{2r})^{-1}\frac{{\alpha}z}{2r}  \\
		
		\\
		
		(1+\frac{{\alpha}\lambda}{2r})^{-1}\frac{{\alpha}z}{2r}   & (1+\frac{{\alpha}\lambda}{2r})^{-1}
	\end{matrix}\right)
\end{equation}
We now examine the behavior of $\varphi_\lambda^*\Phi^{\TNam}$ as $\alpha\rightarrow 0$ when $\lambda$ also varies. 
Notice that  in any bounded subset of $\mathbb H^\circ$,  the family $\Phi^{\TNam}$ is uniformly bounded for $\alpha\in [0, 2]$. The following lemma says that as $\alpha\rightarrow 0$, we also maintain a uniform control of $\Phi^{\TNam}$ outside a fixed neighborhood of  the turning point. Recall that on $\mathbb H$ we have polar coordinates $(r, \theta)$ satisfying $\rho=r\sin\theta$ and $z=r\cos\theta$.

\begin{lemma}\label{l:ALF bound away turning points}
For all $r_0>0$ there is a $C(r_0)>0$  such that for all $\alpha\in [0, 2]$ we have
\begin{itemize}
\item 
    $d(\Phi^{\TNam}, \Phi_{[\bv_{\alpha/2}]})\leq  C(r_0)$, if $r>r_0$ and $\theta>\frac{\pi}{4}$;
\item 
    $d(\Phi^{\TNam}, \Phi_{[\bv_{-\alpha/2}]})\leq  C(r_0)$,  if $r>r_0$ and $\theta<\frac{3\pi}{4}$.
\end{itemize}
\end{lemma}
\begin{proof}
	It suffices to prove the first item; the other one is similar. By a straightforward computation, we have 
    \begin{align*}
        \Tr((L_{\alpha/2}.\Phi_{[\bv_0]})^{-1}\Phi^{\TNam})&=\Tr(\Phi^{-1}_{[\bv_0]}(L_{-\alpha/2}.\Phi^{\TNam}))=\rho^{-2}H_\alpha^{-1}(A_{\alpha}+\frac{\alpha}{2})^2+H_\alpha+H_{\alpha}^{-1},\\
        &=H_{\alpha}^{-1}\frac{\alpha^2}{4r^2}\frac{(1+\cos\theta)^2}{\sin^2\theta}+H_{\alpha}+H_{\alpha}^{-1},
    \end{align*}
    where $H_\alpha$ and $A_\alpha$ are given in \eqref{e:defintion of Halpha and Aalpha}.
    The right hand side is clearly bounded by $C(r_0)$ when $\theta\in(\frac{\pi}{4},\pi]$. Next we have        
    $$\Tr((L_{\alpha/2}.\Phi_{[\bv_0]})^{-1}\Phi_{[\bv_{\alpha/2}]})=\Tr(\Phi_{[\bv_0]}^{-1}(L_{-\alpha/2}P_{[\bv_{\alpha/2}]}^{-1}.\Phi_{[\bv_0]}))<C.$$
    The conclusion then follows from \eqref{e:hyperbolic distance equivalent control}. The last inequality is also a direct computation.
\end{proof}

For the $\TN^+$ family we have a similar statement, with essentially the same proof. We omit the details.
\begin{lemma}\label{l:TN+ bound away turning points}
  For all $r_0>0$ there is a $C(r_0)>0$ such that for all $\alpha\in [1,\infty)$ we have
\begin{itemize}
    \item $d(\Phi^{\TN^+_{4/\alpha}}, \Phi_{[\bv_{\alpha/2}]})\leq  C(r_0)$, if $r>r_0$ and $\theta>\frac{\pi}{4}$;
    \item $d(\Phi^{\TN^+_{4/\alpha}}, \Phi_{[\bv_{-\alpha/2}]})\leq  C(r_0)$, if $r>r_0$ and $\theta<\frac{3\pi}{4}$.
\end{itemize}
\end{lemma}

Now, we take into account the parameter \( \lambda>0 \).  If \( \alpha \lambda \leq 2 \), then we have  
\begin{equation}\label{e:scaling ALF model}
    \varphi_\lambda^*\Phi^{\TNam} = D_{\sqrt{\lambda}}. \Phi^{\TN^-_{\alpha\lambda}},
\end{equation}
and we may apply Lemma \ref{l:ALF bound away turning points}. In particular, the behavior of \( \varphi_\lambda^*\Phi^{\TNam} \) is controlled by the action of the twist \( D_{\sqrt{\lambda}} \) away from the turning point.  

If \( \alpha\lambda \geq \frac{1}{4} \), then we instead write  
\begin{equation}
    \varphi_\lambda^*\Phi^{\TNam} = D_{{1}/{\sqrt{\alpha}}}. \Upsilon_{\alpha\lambda}.
\end{equation}
Here, we define  
\begin{equation}
    \Upsilon_{\alpha} = \frac{1}{\rho} \begin{pmatrix}
        \alpha H_\alpha^{-1} A_1^2 + \rho^2 \alpha^{-1} H_\alpha & \alpha H_\alpha^{-1} A_1 \\
        \alpha H_\alpha^{-1} A_1 & \alpha H_\alpha^{-1}
    \end{pmatrix}.
\end{equation}
Notice that if \( \alpha = \infty \), we have  
\begin{equation}
    \Upsilon_{\infty} = \frac{1}{\rho} \begin{pmatrix}
        \frac{r}{2} & z \\
        z & 2r
    \end{pmatrix} = \Phi^{\Ae_1},
\end{equation}
which is the  harmonic map associated with the standard flat metric on \( \mathbb{R}^4 \).  

\begin{lemma}\label{l:TN rescale upsilon}
For all $r_0\in(0,1)$ there is a $C(r_0)$ such that for all ${\alpha}\geq \frac{1}{4}$ we have
\begin{itemize}
    \item $d(\Upsilon_{{\alpha}}, \Phi_{[\bv_{{1}/{2}}]})\leq  C(r_0)$, if $r_0<r<1/r_0$ and $\theta>\frac{\pi}{4}$;
    \item $d(\Upsilon_{{\alpha}}, \Phi_{[\bv_{-{1}/{2}}]})\leq  C(r_0)$, if if $r_0<r<1/r_0$ and $\theta<\frac{3\pi}{4}$.
\end{itemize}

\end{lemma}
\begin{proof}
	Again it suffices to consider the first inequality. We have 
    \begin{align*}
        \Tr((L_{1/2}.\Phi_{[\bv_0]})^{-1}\Upsilon_{ {\alpha}})&=\Tr(\Phi_{[\bv_0]}^{-1} (L_{{-{1}/{2}}}.\Upsilon_{ {\alpha}}))= \rho^{-2}{\alpha}H_\alpha^{-1}(A_1+\frac{1}{2})^2+({\alpha}^{-1}H_\alpha+{\alpha}H_\alpha^{-1})\\
        &=\alpha H_{\alpha}^{-1}\frac{1}{4}\frac{(1+\cos\theta)^2}{\sin^2\theta}+({\alpha}^{-1}H_\alpha+{\alpha}H_\alpha^{-1}).
    \end{align*}
    This is uniformly bounded as long as $r$ is bounded. Now notice as before we have
    $$\Tr((L_{1/2}.\Phi_{[\bv_0]})^{-1}\Phi_{[\bv_{1/2}]})=\Tr(\Phi_{[\bv_0]}^{-1}(L_{-1/2}P_{[\bv_{1/2}]}^{-1}.\Phi_{[\bv_0]}))<C.$$
    The last inequality can be proved in a similar way.
\end{proof}

In particular, if $\alpha\lambda\geq\frac14$, the behavior of $\varphi_\lambda^*\Phi^{\TNam}$ is controlled by the action of the twist $D_{{1}/{\sqrt{\alpha}}}$ away from the turning point. 

\

Fix a smooth function  
\[
\sigma: \mathbb{R}_{\geq 0} \times \mathbb{R}_{>0} \to \mathbb{R}_{>0}
\]  
such that  
\begin{equation}\label{e:definition of sigma}
\sigma(\alpha, \lambda) =
\begin{cases}
    \sqrt{\lambda}, & \text{if } \alpha\lambda < \frac{1}{4}, \\ 
 \sqrt{\lambda}/2 \leq \sigma(\alpha, \lambda) \leq 2\sqrt{\lambda}, & \text{if } \alpha\lambda \in [\frac{1}{4},1],\\
    1/\sqrt{\alpha}, & \text{if } \alpha\lambda >1.
\end{cases}
\end{equation}
Notice that for all $\alpha>0$ and $\lambda>0$ we have
\begin{equation}\label{e:bound on sigma}\alpha\sigma(\alpha, \lambda)^2\leq 4,
\end{equation}
and 
\begin{equation}\label{e:bound on sigma2}
\alpha \sigma(4/\alpha, \lambda)^{-2}\geq 1. 
\end{equation}

For $\lambda>0$  and $\alpha\geq0$
we set 
\begin{equation}\label{eq:scaling relation for TN}
	\Phi^{\lambda; \alpha}=	
     D_{\sigma(\alpha,\lambda)}^{-1}.\varphi_{\lambda}^*\Phi^{\TNam}.
	\end{equation}
Notice that here we allow $\alpha=0$. One can see that the harmonic map $\Phi^{\lambda;\alpha}$ is locally tamed by the untyped weak rod structure with only one turning point $z_1=0$ and two rod vectors $\mv_0,\mv_1$ given by
\begin{align}\label{eq:lambda;alpha rod vector 1}
    [\bv_{\alpha\sigma(\alpha,\lambda)^2/2}]\text{ and }[\bv_{-\alpha\sigma(\alpha,\lambda)^2/2}].
\end{align}
The upshot is that, by Lemma \ref{l:ALF bound away turning points}-\ref{l:TN rescale upsilon}, for any value of $\alpha\geq0,\lambda>0$, up to some explicit twist that lies in a bounded subset of $SL(2;\R)$, the map $\Phi^{\lambda;\alpha}$ always has a controlled distance to one of the model maps on the corresponding regions away from the turning point.
For the $\TN^+$ family the rescaled harmonic map can be controlled using the fact that 
\begin{equation}\label{eq:rescaling relation for TN+}
    \varphi_{\lambda}^*\Phi^{\TNap}=R_1.D_{\sigma(\alpha,\lambda)}.\Phi^{\lambda;\alpha}=D_{\sigma(\alpha,\lambda)}^{-1}.R_1.\Phi^{\lambda;\alpha}.
\end{equation}

\subsection{Background map in a fixed scale}
\label{ss:background map in a fixed scale}

\begin{definition}[$(\epsilon, \delta)$-separate rod structure]
Given \( \epsilon \in (0, \frac{1}{2}) \) and \( \delta \in (0, \frac{\epsilon}{100}) \), a rod structure is called \emph{\((\epsilon, \delta)\)-separate} if any two turning points \( z_j \) and \( z_k \) satisfy either \( |z_j - z_k| \leq \delta \) or \( |z_j - z_k| \geq \epsilon \).  

\end{definition}
Now suppose we are given a rod structure \( \mathfrak{R} \) with $n$ turning points. If $\mR$ is \((\epsilon, \delta)\)-separate, then we can partition the set of turning points of \( \mathfrak{R} \) into the union of \( s\leq n \) clusters, \( \mathfrak{C}_1\), \( \cdots\), \(  \mathfrak{C}_s \). Each cluster is given by  
\[
\mathfrak{C}_p = \{ z_{p, 1}, \dots, z_{p, n_p} \}
\]  
with \( z_{p,1} < \dots < z_{p, n_p} \) and $p=1,\ldots,s$, satisfying the conditions  
\[
z_{p, n_p} - z_{p, 1} \leq \delta, \quad \text{and} \quad z_{p+1, 1} - z_{p, n_p} \geq \epsilon,
\]  
where we drop the second condition when $p=s$.
\begin{figure}[ht]
    \begin{tikzpicture}[scale=0.9]
        \draw (-5,0) -- (0,0);
        \node at (0,0) {$\bullet$};
        \draw (0,0) -- (0.4,0);
        \node at (0.4,0) {$\bullet$};
        \draw[dotted] (0.4,0)--(1.1,0);
        \node at (1.1,0) {$\bullet$};
        \draw (1.1,0)--(5,0);
        \node at (5,0){$\bullet$};
        \draw (5,0) -- (5.3,0);
        \node at (5.3,0) {$\bullet$};
        \draw[dotted] (5.3,0)--(6.1,0);
        \node at (6.1,0) {$\bullet$};
        \draw (6.1,0)--(10,0);
        
        \draw [decorate,decoration={brace,amplitude=5pt,mirror,raise=4ex}]
        (0,0.4) -- (1.1,0.4) node[midway,yshift=-3em]{$\leq\delta$};
        \draw [decorate,decoration={brace,amplitude=5pt,mirror,raise=4ex}]
        (5,0.4) -- (6.1,0.4) node[midway,yshift=-3em]{$\leq\delta$};
        \draw [decorate,decoration={brace,amplitude=5pt,mirror,raise=4ex}]
        (1.3,0.4) -- (4.8,0.4) node[midway,yshift=-3em]{$\geq\epsilon$};

        \node at (0.5,0.6) {$\mC_p$};
        \node at (5.5,0.6) {$\mC_{p+1}$};
        \node at (4,0.6) {$\mathfrak{J}_p$};
        \node at (2.5,0.6) {$\mathfrak{w}_{p}$};
        \draw[-Stealth] (2.1,0)--(2.1,0.4);

        \draw[-Stealth] (-3,0)--(-3,0.4);
        \node at (-2.4,0.6) {$\mathfrak{w}_{p-1}$};
        \draw[-Stealth] (7.5,0)--(7.5,0.4);
        \node at (8.1,0.6) {$\mathfrak{w}_{p+1}$};
    \end{tikzpicture}
	\caption{Clusters of $(\epsilon,\delta)$-separate rods.}
	 \label{Figure:rod clusters}   
\end{figure}
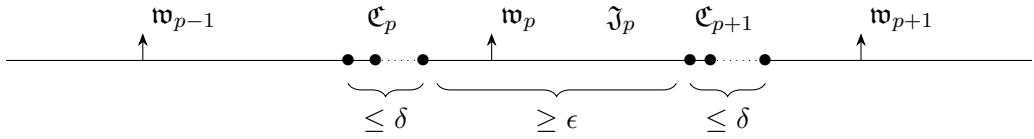
For \( p \geq 1 \), denote by \( \mathfrak{J}_p \) the rod of \( \mathfrak{R} \) connecting \( z_{p, n_p} \) and \( z_{p+1,1} \). Denote by \( \mathfrak{w}_p \) the rod vector associated with the rod \( \mathfrak{J}_p \). We also set \( \mathfrak{w}_0 = \mathfrak{v}_0 \) and \( \mathfrak{w}_s = \mathfrak{v}_n \).  
For \( \tau > 0 \), we define  
\[
\mathfrak{J}_p^\tau = N_{\tau}(\mathfrak{J}_p) \setminus N_{\tau}(\mathfrak{C}_p \cup \mathfrak{C}_{p+1}).
\]  
For each \( p \) with \( n_p > 1 \), we define 
\[
\underline{z}(\mC_p)=z_{p, 1},
\] 
\[
\lambda(\mathfrak{C}_p) = (z_{p, n_p} - z_{p, 1})^{-1},
\]  
and (see \eqref{e:definition of alpha}) \[
\alpha(\mathfrak{C}_p) = \mathfrak{w}_{p-1}\wedge \mathfrak{w}_{p}.
\]
We then define an untyped rod structure \( \widetilde{\mathfrak{R}}_{\mathfrak{C}_p} \) associated with \( \mathfrak{C}_p \), 
whose turning points are given by  
\[
\lambda(\mathfrak{C}_p) (z_{p, \alpha} - z_{p, 1}), \quad
\alpha=1, \cdots, n_p,
\]  
and whose rod vectors coincide with those of \( \mathfrak{R} \) inside the cluster $\mC_p$ on the finite rods and are given by \( \mathfrak{w}_{p-1} \) and \( \mathfrak{w}_{p} \) on the infinite rods. Finally, we define the untyped rod structure
\begin{equation}\label{eq:def R mC_p}
    \mR_{\mC_p}=\begin{cases}
        D_{\sigma(\alpha(\mC_{p}),\lambda(\mC_{p}))}^{-1}. P_{\mathfrak{w}_{p-1},\mathfrak{w}_{p}}. \widetilde{\mR}_{\mC_p}, & \text{if }\alpha(\mC_p)<1,\\
        D_{\sigma(4/\alpha(\mC_p),\lambda(\mC_p))}.P_{\mw_{p-1},\mw_{p}}.\widetilde{\mR}_{\mC_{p}}, & \text{if }\alpha(\mC_p)\geq1.
    \end{cases} 
\end{equation}
By construction we have $\ms(\mR_{\mC_p})=1$. Note that the infinite rod vectors of $\mR_{\mC_p}$ are given by 
\begin{equation} \label{e: rod vectors two cases}
    \begin{cases}
        [\bv_{\frac{1}{2}\alpha(\mC_p)\sigma(\alpha(\mC_{p}),\lambda(\mC_{p}))^2}]\text{ and }[\bv_{-\frac{1}{2}\alpha(\mC_p)\sigma(\alpha(\mC_{p}),\lambda(\mC_{p}))^2}],&\text{if }\alpha(\mC_p)<1,\\
        [\bv_{\frac{1}{2}\alpha(\mC_p)\sigma(4/\alpha(\mC_{p}),\lambda(\mC_{p}))^{-2}}]\text{ and }[\bv_{-\frac{1}{2}\alpha(\mC_p)\sigma(4/\alpha(\mC_{p}),\lambda(\mC_{p}))^{-2}}], &\text{if }\alpha(\mC_p)\geq1.
    \end{cases}     
\end{equation}
 Recall that for $\mv\in \R\dP^1$, we have defined
 \begin{equation} 
 	\Phi_{\mv}=P_\mv^{-1}. \Phi_{[\bv_0]}. 
 \end{equation}
For $\mv, \mv'\in \R\dP^1$ and $\underline z\in \R$, we now  define
\begin{equation}\label{e:definition of Phi mv mv' z}
    \Phi_{\mv, \mv'; \underline z}:=T_{\underline z}^*P_{\mv, \mv'}^{-1}.\Phi^{\TNam},
\end{equation}
where $\alpha=\mv\wedge\mv'$ (see \eqref{e:definition of alpha}), and  
\begin{equation}\label{e:definition of Phi mv mv' z TN+}
    \overline{\Phi}_{\mv, \mv'; \underline z}:=T_{\underline z}^*\overline{P}_{\mv, \mv'}^{-1}.\Phi^{\TN^-_{4/\alpha}}=T_{\underline z}^*{P}_{\mv, \mv'}^{-1}.\Phi^{\TN^+_{4/\alpha}}.
\end{equation}
For the definition of  $P_{\mv, \mv'}$ and $\overline{P}_{\mv, \mv'}$ see \eqref{e:definition of P} through \eqref{e:definition of bar P}.

The proofs of the following two propositions are very similar to the construction in the proof of Theorem \ref{t:existence result}, using Lemma \ref{l:interpolation}	and the fact that $\mathcal H$ is simply-connected. See Figure \ref{Figure:rod clusters model map}. We omit the details. 

\begin{proposition}[Unrescaled model maps] \label{p:background map}
	For all $\epsilon \in(0, \frac{1}{2})$, $\delta\in (0, \frac{\epsilon}{100})$ and $\alpha_0>1$,  there exits a constant $C=C(\epsilon, \alpha_0)$ such that for any rod structure $\mR$  satisfying 
    \begin{itemize}
        \item $\mR$ is $(\epsilon, \delta)$-separate with size $\ms(\mR)\leq 1$;
        \item if the Asymptotic Type $\sharp\in \{\Ae, \ALF^{-}, \ALF^+\}$, then $\mv_0=\bv_{\alpha/2}$ and $\mv_n=\bv_{-\alpha/2}$ with
        \begin{itemize}
            \item  $\alpha\in (0, \alpha_0]$ if $\sharp=\Ae$ or $\sharp=\ALFm$;
            \item $\alpha\in [\alpha_0^{-1}, \infty)$ if $\sharp=\Ae$ or $\sharp=\ALFp$;        \end{itemize}

        \item if the Asymptotic Type $\sharp=\AFa$, then $\mv_0=\mv_n=\bv_0$ and $|\beta|\leq \alpha_0$,
    \end{itemize}
    there is a smooth map
	$$\Psi_{\mR}^\sharp: \mathbb H^\circ\setminus \bigcup_{p=1}^sN_{\frac{\epsilon}{20}}(\mC_p)\rightarrow \mathcal H$$
	such that 
	\begin{itemize}
		\item $\Psi_{\mR}^\sharp=\Phi^\sharp_{\mv_0, \mv_n}$ $($see \eqref{e:definition of Phi Ae general}, \eqref{e:definition of Phi TN general}, \eqref{e:definition of Phi TN+ general}$)$ on $\mathbb H^\circ\setminus B_{\frac{1}{\epsilon}}(0)$;
		\item $\Psi_{\mR}^\sharp=\begin{cases}
		    \Phi_{\mw_{p-1}, \mw_{p}; z_{p,1}},&\text{if }\alpha(\mC_{p})<1,\\
            \overline{\Phi}_{\mw_{p-1}, \mw_{p}; z_{p,1}},&\text{if }\alpha(\mC_{p})\geq1,
		\end{cases}$ on $N_{\frac{\epsilon}{10}}(\mC_p)\setminus N_{\frac{\epsilon}{20}}(\mC_p)$ for $p=1, \cdots, s$;
		\item  $\Psi_{\mR}^\sharp=\Phi_{\mw_p}$ on $\mathbb H^\circ\cap  \mJ_p^{\frac{\epsilon}{5}}$;
		\item  $|\Delta\Psi_{\mR}^\sharp|\leq C$ on $\mathbb H^\circ\setminus \bigcup_{p=1}^sN_{\frac{\epsilon}{20}}(\mC_p)$;
        \item For any $k\geq 1$, $|\nabla^k\Psi_{\mR}^\sharp|\leq C(k, \epsilon)$ on $B_{\frac{1}{\epsilon}}(0)\setminus \bigcup_{p=1}^sN_{\frac{\epsilon}{10}}(\mC_p)$.
	\end{itemize}
\end{proposition}
\begin{remark}
The constant $\alpha_0$ is imposed to prevent the degeneration of the infinite rods which may destroy the uniform estimates. However we allow $\alpha\rightarrow 0$ in the $\ALFm$ case and $\alpha\rightarrow\infty$ in the $\ALFp$ case, since we have uniform control of the model maps in these settings, as discussed in Section \ref{ss:rescaling of model harmonic maps}.
\end{remark}

\begin{proposition}[Rescaled model maps] \label{p:background map2}
	For all $\epsilon \in(0, \frac{1}{2})$ and $\delta\in (0, \frac{\epsilon}{100})$,  there exits a $C=C(\epsilon)$ such that for all $\lambda>0$ and $\alpha\geq 0$, and  any rod structure $\mR$  with rod vectors $\mv_0$ and $\mv_n$ given by
    $$\begin{cases}
        [\bv_{\alpha\sigma(\alpha,\lambda)^2/2}]\text{ and }[\bv_{-\alpha\sigma(\alpha,\lambda)^2/2}], \quad \text{if }\alpha<1;\\
        [\bv_{\alpha\sigma(4/\alpha,\lambda)^{-2}/2}]\text{ and }[\bv_{-\alpha\sigma(4/\alpha,\lambda)^{-2}/2}], \quad \text{if }\alpha\geq1,
    \end{cases}$$
    there is a smooth map 
	$$\Psi_{\mR}^{\lambda; \alpha}: \mathbb H^\circ\setminus \bigcup_{p=1}^sN_{\frac{\epsilon}{20}}(\mC_p)\rightarrow \mathcal H$$
	such that 
	\begin{itemize}
		\item $\Psi_{\mR}^{\lambda;\alpha}=\begin{cases}\Phi^{\lambda;\alpha}, &\text{if } \alpha<1, \\
        R_1. \Phi^{\lambda; 4/\alpha}, &\text{if } \alpha\geq 1,\end{cases}$ on $\mathbb H^\circ\setminus B_{\frac{1}{\epsilon}}(0)$;
		\item $\Psi_{\mR}^{\lambda;\alpha}=\begin{cases}
		    \Phi_{\mw_{p-1}, \mw_{p}; z_{p,1}},&\text{if }\alpha(\mC_p)<1,\\
            \overline{\Phi}_{\mw_{p-1}, \mw_{p}; z_{p,1}},&\text{if }\alpha(\mC_p)\geq1,
		\end{cases}$ on $N_{\frac{\epsilon}{10}}(\mC_p)\setminus N_{\frac{\epsilon}{20}}(\mC_p)$;
		\item  $\Psi_{\mR}^{\lambda;\alpha}=\Phi_{\mw_p}$ on $\mathbb H^\circ\cap  \mJ_p^{\frac{\epsilon}{5}}$;
		\item  $|\Delta\Psi_{\mR}^{\lambda;\alpha}|\leq C$ on $\mathbb H^\circ\setminus \bigcup_{p=1}^sN_{\frac{\epsilon}{20}}(\mC_p)$;
        \item For any $k\geq 1$, $|\nabla^k\Psi_{\mR}^{\lambda;\alpha}|\leq C(k, \epsilon)$ on $B_{\frac{1}{\epsilon}}(0)\setminus \bigcup_{p=1}^sN_{\frac{\epsilon}{10}}(\mC_p)$.
	\end{itemize}
\end{proposition}

\begin{figure}[ht]
    \begin{tikzpicture}[scale=0.7]
        \draw (-7,0) -- (0,0);
        \node at (0,0) {$\bullet$};
        \draw (0,0) -- (0.4,0);
        \node at (0.4,0) {$\bullet$};
        \draw[dotted] (0.4,0)--(1.1,0);
        \node at (1.1,0) {$\bullet$};
        \draw (1.1,0)--(3,0);

        \draw[dashed] (3.2,0)--(4.8,0);

        \draw (5,0)--(7,0);
        \node at (7,0){$\bullet$};
        \draw (7,0) -- (7.3,0);
        \node at (7.3,0) {$\bullet$};
        \draw[dotted] (7.3,0)--(8.1,0);
        \node at (8.1,0) {$\bullet$};
        \draw (8.1,0)--(15,0);

        \filldraw[color=red!50,fill=red!50,opacity=0.5](15,0)--(11,0) arc (0:180:7) --(-3,0)--(-7,0)--(-7,9)--(15,9)--(15,0);

        \filldraw[color=brown!50,fill=brown!50,opacity=0.5](1.7,0)--(1.2,0) arc (0:180:0.7) --(-0.2,0)--(-0.7,0) arc (180:0:1.2) (1.7,0);

        \filldraw[color=brown!50,fill=brown!50,opacity=0.5](8.7,0)--(8.2,0) arc (0:180:0.7) --(6.8,0)--(6.3,0) arc (180:0:1.2) (8.7,0);

        \filldraw[color=blue!50,fill=blue!50,opacity=0.5](10.7,0)--(9,0) --(9,0.3)to[out=90,in=180](9.2,0.5)--(10.5,0.5)to[out=0,in=90](10.7,0.3)--(10.7,0);

        \filldraw[color=blue!50,fill=blue!50,opacity=0.5](-1,0)--(-2.7,0) --(-2.7,0.3)to[out=90,in=180](-2.5,0.5)--(-1.2,0.5)to[out=0,in=90](-1,0.3)--(-1,0);
        
        \filldraw[color=blue!50,fill=blue!50,opacity=0.5](2,0)--(2.8,0)--(2.8,0.5)--(2.2,0.5)to[out=180,in=90](2,0.3)--(2,0);

        \filldraw[color=blue!50,fill=blue!50,opacity=0.5](6,0)--(5.2,0)--(5.2,0.5)--(5.8,0.5)to[in=90,out=0](6,0.3)--(6,0);

        \draw[dotted] (4.9,0.25)--(3.1,0.25);

        \node at (13,6) {$\mathbb H^\circ\setminus B_{1/\epsilon}(0)$};
        
        \node at (9.8,1.4) {\color{blue} $\mathbb{H}^\circ\cap\mJ^{\frac{\epsilon}{5}}$};
        \node at (-1.8,1.4) {\color{blue} $\mathbb{H}^\circ\cap\mJ^{\frac{\epsilon}{5}}$};
        
        \node at (1.4,2) {\color{brown} $N_{\frac{\epsilon}{10}}(\mC)\setminus N_{\frac{\epsilon}{20}}(\mC)$};

        \node at (6.6,2) {\color{brown} $N_{\frac{\epsilon}{10}}(\mC)\setminus N_{\frac{\epsilon}{20}}(\mC)$};

    \end{tikzpicture}
	\caption{Construction of $\Psi_{\mR}^\sharp$ in Proposition \ref{p:background map} and \ref{p:background map2}.}
	 \label{Figure:rod clusters model map}   
\end{figure}
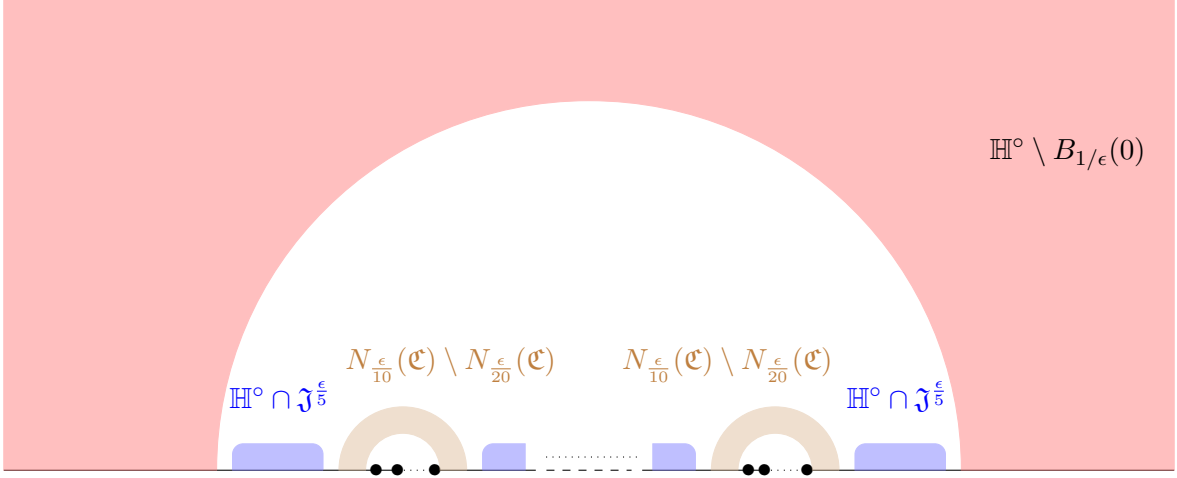

\subsection{Uniform control of tame harmonic maps} 
\label{ss:uniform control of tame harmonic maps}

Given a sequence of rod structures \( \mathfrak{R}_i \) with r \( \mathfrak{n}(\mathfrak{R}_i) \leq N \). We assume that for all $i$,  the Asymptotic Type $\sharp_i$ of $\mR_i$ is either $\Ae$, or $\ALFp$, or $\ALFm$, or $\AF_{\beta_i}$, and  \( \mathfrak{s}(\mathfrak{R}_i) \) uniformly bounded. Without loss of generality we may assume that for each $i$, $\ms(\mR_i)\leq 1$ and $\mR_i$ has $n$ turning points. We denote  \[\alpha(\mR_i):=\mv_{i,0}\wedge \mv_{i,n}\in [0, \infty).\] In addition, we assume there exists $\alpha_0>1$ such that
\begin{itemize}
    \item  $\alpha(\mR_i)\leq \alpha_0$ if $\sharp_i=\Ae$ or $\sharp_i=\ALF^-$;
    \item  $ \alpha(\mR_i)\geq \alpha_0^{-1}$ if $\sharp_i=\Ae$ or $\sharp_i=\ALF^+$;
     \item  $|\beta_i|\leq \alpha_0$ if $\sharp_i=\AF_{\beta_i}$; 
\end{itemize}
Again these assumptions are imposed in order to gain a uniform control for all $i$. Passing to a subsequence we may assume $\beta_i$ has a finite limit $\beta_\infty$ if $\sharp_i=\AF_{\beta_i}$, and  $\alpha(\mR_i)$ has a limit in $[0, \infty]$ otherwise.

By Theorem \ref{t:existence result}, there is a unique harmonic map $\Phi_i:\dH^\circ\rightarrow\mathcal H$ which is strongly tamed by $\mR_i$. Our goal is to understand the limit and bubbling behavior of $\Phi_i$ by comparing it with an explicitly constructed background map $\Psi_i$.

We can pass to a subsequence and assume that \( \mathfrak{R}_i \) is \( (\epsilon, \delta_i) \)-separate for some fixed \( \epsilon > 0 \) and a sequence \( \delta_i \to 0 \).  
Moreover, there exist exactly \( s \) clusters \( \mathfrak{C}_{i, 1}, \dots, \mathfrak{C}_{i, s} \), with $s$ independent of \( i \), such that as \( i \to \infty \), each \( \mathfrak{C}_{i,p} \) converges to a unique limit point \( z_{\infty, p} \), and each \( \mathfrak{w}_{i,p} \) converges to a limit rod vector \( \mathfrak{v}_{\infty, p} \). We also set \( z_{\infty, 0} = -\infty \) and \( z_{\infty, s+1} = +\infty \). We denote 
\begin{equation}\label{eq:definition of lambda for clusters}
    \lambda_{i, p}:=\lambda(\mC_{i, p})
\end{equation}
and 
\begin{equation}\label{eq:definition of alpha for clusters}
    \alpha_{i, p}:=\alpha(\mC_{i, p}).
\end{equation}  
In this way, we obtain a limit weak rod structure \( \mathfrak{R}_\infty \), with turning points given by \( \{z_{\infty, p}\}_{p=1}^{s} \) and rod vectors given by \( \{\mathfrak{v}_{\infty, p}\}_{p=0}^{s} \).  We assign the Asymptotic Type of $\mR_\infty$ to be
\begin{itemize}
    \item $\Ae$ if $\mR_i$ are of Asymptotic Type $\Ae$;
    \item $\AF_{\beta_\infty}$ if $\mR_i$ are of Asymptotic Type $\AF_{\beta_i}$ with $\beta_i\rightarrow\beta_\infty$;
    \item $\ALF^-$ if $\mR_i$ are of Asymptotic Type $\ALF^-$ and $\lim_{i\rightarrow\infty}\alpha(\mR_i)>0$;
    \item $\ALFp$ if $\mR_i$ are of Asymptotic Type $\ALFp$ and $\lim_{i\rightarrow\infty}\alpha(\mR_i)<\infty$;
    \item $\AF_0$ if  $\mR_i$ are of Asymptotic Type $\ALF^-$ and $\lim_{i\rightarrow\infty}\alpha(\mR_i)=0$;
    \item $\AF_0$ if $\mR_i$ are of Asymptotic Type $\ALFp$ and $\lim_{i\rightarrow\infty}\alpha(\mR_i)=\infty$.
\end{itemize}

The rod structure is \emph{weak} in general, since it is possible that some adjacent rod vectors coincide.  Each limit turning point \( z_{\infty, p} \) is associated with a multiplicity \( m_p \in \mathbb{Z}_{>0} \), namely, the number of turning points of $\mR_i$ converging to $z_{\infty,p}$.  

\begin{definition}
    The weak rod structure $\mR_\infty$ is referred to as a \emph{weak limit} of the rod structures $\mR_i$.
\end{definition}

\

To define the background maps $\Psi_i$ we first set $$\Psi_{i}: \mathbb H^\circ\setminus \bigcup_{p}N_{\frac{\epsilon}{20}}(\mC_{i,p})\rightarrow \mathcal H$$ 
to be  the map 
\begin{equation}\label{eq:Psi_i in trivial scale}
    P_{\mw_{i,0},\mw_{i,s}}^{-1}.\Psi_{P_{\mw_{i,0},\mw_{i,s}}.\mR_i}^\sharp
\end{equation} 
given by Proposition \ref{p:background map}. By definition and Proposition \ref{p:background map}, in each cluster region $N_{\frac{\epsilon}{10}}(\mC_{i,p})\setminus N_{\frac{\epsilon}{20}}(\mC_{i,p})$, the expression \eqref{eq:Psi_i in trivial scale} coincides with
\begin{equation}\label{eq:Psi_i near the cluster region}
    \begin{aligned}
        \Phi_{\mw_{i,p-1},\mw_{i,p};z_{i,p,1}}=P_{\mw_{i,p-1},\mw_{i,p}}^{-1}. T_{z_{i,p,1}}^*\Phi^{\TN^-_{\alpha_{i, p}}},\quad \text{if }\alpha_{i, p}<1,\\
        \overline{\Phi}_{\mw_{i,p-1},\mw_{i,p};z_{i,p,1}}=P_{\mw_{i,p-1},\mw_{i,p}}^{-1}.T_{z_{i,p,1}}^*\Phi^{\TN^+_{4/\alpha_{i, p}}},\quad \text{if }\alpha_{i, p}\geq1.
    \end{aligned}
\end{equation}
In particular, we see that for each $p$ with $m_p=1$, $\Psi_{i}$ obviously extends to  $N_{\frac{\epsilon}{20}}(\mC_{i,p})$ on which it is harmonic and locally tamed by the rod structure $\mR_i$. 

For each $p$ with $m_p>1$,  we consider the sequence of rescaled rod structures (as defined in \eqref{eq:def R mC_p}) $$\mR_{i, p}:=\mR_{\mC_{i, p}}.$$   Again by passing to a further subsequence we may assume this has a limit weak, untyped rod structure $\mR_{\infty, p}$ as above. Furthermore,  \( \mathfrak{R}_{i,p} \) is \( (\epsilon', \delta_i') \)-separate for some fixed \( \epsilon' > 0 \) and a sequence \( \delta_i' \to 0 \).    

Next we extend the domain of definition of $\Psi_i$ to a smaller scale.  Denote by $N'_{i, p}$ the subset of  $N_{\frac{\epsilon}{10}}(\mC_{i, p})$ obtained by removing the union of the $\frac{\epsilon'}{20\lambda_{i, p}}$ neighborhood of the clusters of $\mR_{i, p}$. We now divide into two cases.

If $\alpha_{i,p}<1$, then we define $\Psi_i$ on $N_{i, p}'$ to be
\begin{equation}\label{eq:Psi_i in scale of mC}
    P_{\mw_{i,p-1},\mw_{i,p}}^{-1}D_{\sigma(\alpha_{i, p}, \lambda_{i, p})}. T_{{z}_{i, p,1}}^*\varphi_{1/\lambda_{i, p}}^*\Psi_{\mR_{i, p}}^{\lambda_{i, p}; \alpha_{i, p}},
\end{equation} 
where $\sigma$ is defined in \eqref{e:definition of sigma} and  $\Psi_{\mR_{i, p}}^{\lambda_{i, p}; \alpha_{i, p}}$ is given in  Proposition \ref{p:background map2}.  Note that for $\Psi_{\mR_{i, p}}^{\lambda_{i, p}; \alpha_{i, p}}$, on $\mathbb{H}^\circ\setminus B_{\frac{1}{\epsilon'}}(0)$, by  definition  it satisfies 
$$D_{\sigma(\alpha_{i, p}, \lambda_{i, p})}.\Psi_{\mR_{i, p}}^{\lambda_{i, p}; \alpha_{i, p}}= D_{\sigma(\alpha_{i, p}, \lambda_{i, p})}.\Phi^{\lambda_{i, p}; \alpha_{i, p}}=\varphi_{\lambda_{i, p}}^*\Phi^{\TN^-_{\alpha_{i, p}}}.$$ Hence, \eqref{eq:Psi_i in scale of mC} clearly matches with \eqref{eq:Psi_i near the cluster region} on the overlapping region $N_{\epsilon/10}(\mC_{i,p})\setminus N_{\epsilon/20}(\mC_{i, p})$ (see Figure \ref{Figure:patching clusters}). In other words, we have extended the domain of definition of $\Psi_i$ into a smaller scale around the cluster $\mC_{i, p}$.

If $\alpha_{i,p}\geq1$, then we define $\Psi_i$ on $N_{i, p}'$ to be 
\begin{equation}\label{eq:Psi_i in scale of mC 2}
    P_{\mw_{i,p-1},\mw_{i,p}}^{-1}D_{\sigma(4/\alpha_{i, p}, \lambda_{i, p})}^{-1}. T_{{z}_{i, p,1}}^*\varphi_{1/\lambda_{i, p}}^*\Psi_{\mR_{i, p}}^{\lambda_{i, p}; \alpha_{i, p}}.
\end{equation}
For $\Psi^{\lambda_{i,p};\alpha_{i,p}}_{\mR_{i,p}}$, near infinity it satisfies
$$D_{\sigma(4/\alpha_{i, p}, \lambda_{i, p})}^{-1}.\Psi_{\mR_{i, p}}^{\lambda_{i, p}; \alpha_{i, p}}= D_{\sigma(4/\alpha_{i, p}, \lambda_{i, p})}^{-1}.R_1.\Phi^{\lambda_{i, p}; 4/\alpha_{i, p}}=\varphi_{\lambda_{i, p}}^*\Phi^{\TN^+_{4/\alpha_{i, p}}}.$$
Again, \eqref{eq:Psi_i in scale of mC 2} also matches with \eqref{eq:Psi_i near the cluster region} on the overlapping region $N_{\epsilon/10}(\mC_{i,p})\setminus N_{\epsilon/20}(\mC_{i, p})$.

Note that from the matching above, we moreover have 
\begin{equation}\label{eq:Psii formula}
    \Psi_i=\begin{cases}
		    \Phi_{\mw_{i,p-1}, \mw_{i,p}; z_{i,p,1}},&\text{if }\alpha(\mC_{i,p})<1,\\
            \overline{\Phi}_{\mw_{i,p-1}, \mw_{i,p}; z_{i,p,1}},&\text{if }\alpha(\mC_{i,p})\geq1,
		\end{cases}
\end{equation} 
on $N_{\frac{\epsilon}{10}}(\mC_{i,p})\setminus N_{\frac{1}{\epsilon'\lambda_{i,p}}}(\mC_{i,p})$.

\begin{figure}[ht]
    \begin{tikzpicture}[scale=0.6]
        \draw (0,-4.5) -- (0,-1.5);
        \node at (0,-1.5) {$\bullet$};
        \draw (0,-1.5) -- (0,-1.1);
        \node at (0,-1.1) {$\bullet$};
        \draw[dotted] (0,-1.1)--(0,-0.4);
        \node at (0,-0.4) {$\bullet$};
        \draw (0,-0.4)--(0,2.5);

        \filldraw[color=violet!50,fill=violet!50,opacity=1](0,0.5)--(0,-0.3) arc (90:-90:0.7) --(0,-1.7)--(0,-2.5) arc (-90:90:1.5) (0,0.5);

        \draw[-stealth] (3,-1)--(8,-1);

        \draw (10,-9) -- (10,-3.3);
        \node at (10,-3.3) {$\bullet$};
        \draw (10,-3.3) -- (10,-3);
        \node at (10,-3) {$\bullet$};
        \draw[dotted] (10,-3) -- (10,-2.4);
        \node at (10,-2.4) {$\bullet$};
        
        \draw (10,-2.4)--(10,-2);
        \draw[dashed] (10,-2)--(10,0);

        \draw (10,0)--(10,0.3);
        \node at (10,0.3) {$\bullet$};
        \draw (10,0.3)--(10,0.6);
        \node at (10,0.6) {$\bullet$};
        \draw[dotted] (10,0.6)--(10,1.2);
        \node at (10,1.2) {$\bullet$};

        \draw (10,1.2) -- (10,7);

        \filldraw[color=violet!50,fill=violet!50,opacity=1](10,6.5)--(10,4.5) arc (90:-90:5.5) --(10,-6.5)--(10,-8.5) arc (-90:90:7.5) (10,6.5);

        \filldraw[color=brown!50,fill=brown!50,opacity=0.7](10,1.8)--(10,1.3) arc (90:-90:0.55) --(10,0.2)--(10,-0.3) arc (-90:90:1.05) (10,1.8);

        \filldraw[color=brown!50,fill=brown!50,opacity=0.7](10,-1.8)--(10,-2.3) arc (90:-90:0.55) --(10,-3.4)--(10,-3.9) arc (-90:90:1.05) (10,-1.8);

        \node at (-2,-1.1) {\small cluster $\mC_{i,p}$};
        \node at (0,-9.7) {$\mR_{i}$};
        \node at (2.4,1) {\small \color{violet} $\Phi_{\mw_{i,p},\mw_{i,p+1};z_{i,p,1}}$};
        \node at (2.8,-3) {\small \color{violet} $N_{\frac{\epsilon}{10}}(\mC_{i,p})\setminus N_{\frac{\epsilon}{20}}(\mC_{i,p})$};

        \node at (10.3,-9.7) {$\mR_{i, p}$};

        \node at (7.7,1) {\small deeper cluster};

        \filldraw[color=pink!50,fill=pink!50,opacity=0.5](10,7.3)--(10,2) arc (90:-90:3) --(10,-4)--(10,-9.3) arc (-90:90:8.3) (10,7.3);

        \node at (8.5,3) {{\small \color{pink}$\Phi^{\lambda_{i, p};\alpha_{i, p}}$}};

        \node at (8.4,-5) {{\small \color{pink}$\mathbb{H}^\circ\setminus B_{\frac{1}{\epsilon'}}(0)$}};

        \node at (5.4,-7.6) {{\small \color{violet} $\varphi_{\lambda_{i,p}^{-1}} T_{z_{i,p,1}}(N_{\frac{\epsilon}{10}}(\mC_{i,p})\setminus N_{\frac{\epsilon}{20}}(\mC_{i,p}))$}};

    \end{tikzpicture}
	\caption{\small Patching between scales when $\alpha_{i,p}<1$: on the cluster region $N_{\frac{\epsilon}{10}}(\mC_{i,p})\setminus N_{\frac{\epsilon}{20}}(\mC_{i,p})$ in $\mR_i$, the map ${\color{violet} \Phi_{\mw_{i,p},\mw_{i,p+1};z_{i,p,1}}}$ is rescaled to 
    $${\color{violet} \varphi_{\lambda_{i, p}}^*T^*_{-z_{i,p,1}}D_{\sigma(\alpha_{i, p},\lambda_{i, p})}^{-1}. P_{\mathfrak{w}_{i,p},\mathfrak{w}_{i,p+1}}.\Phi_{\mw_{i,p},\mw_{i,p+1};z_{i,p,1}}}$$
    over the purple region in $\mR_{i, p}$, which is $\color{violet} \varphi_{\lambda_{i,p}^{-1}}T_{z_{i,p,1}}(N_{\frac{\epsilon}{10}}(\mC_{i,p})\setminus N_{\frac{\epsilon}{20}}(\mC_{i,p}))$. The purple region is contained in $\mathbb{H}^\circ\setminus B_{\frac{\epsilon}{30}\lambda_{i,p}}(0)$, which is particularly contained in ${\color{pink}\mathbb{H}^\circ\setminus B_{\frac{1}{\epsilon'}}(0)}$.
    By \eqref{e:definition of Phi mv mv' z} and \eqref{eq:scaling relation for TN} this is 
    $${\color{violet}D_{\sigma(\alpha_{i, p},\lambda_{i, p})}^{-1}.\varphi_{\lambda_{i, p}}^*\Phi^{\TN^-_{\alpha_{i, p}}}},$$
    which exactly matches $\color{pink} \Phi^{\lambda_{i, p};\alpha_{i, p}}$ over the overlap.}
	\label{Figure:patching clusters}   
\end{figure}
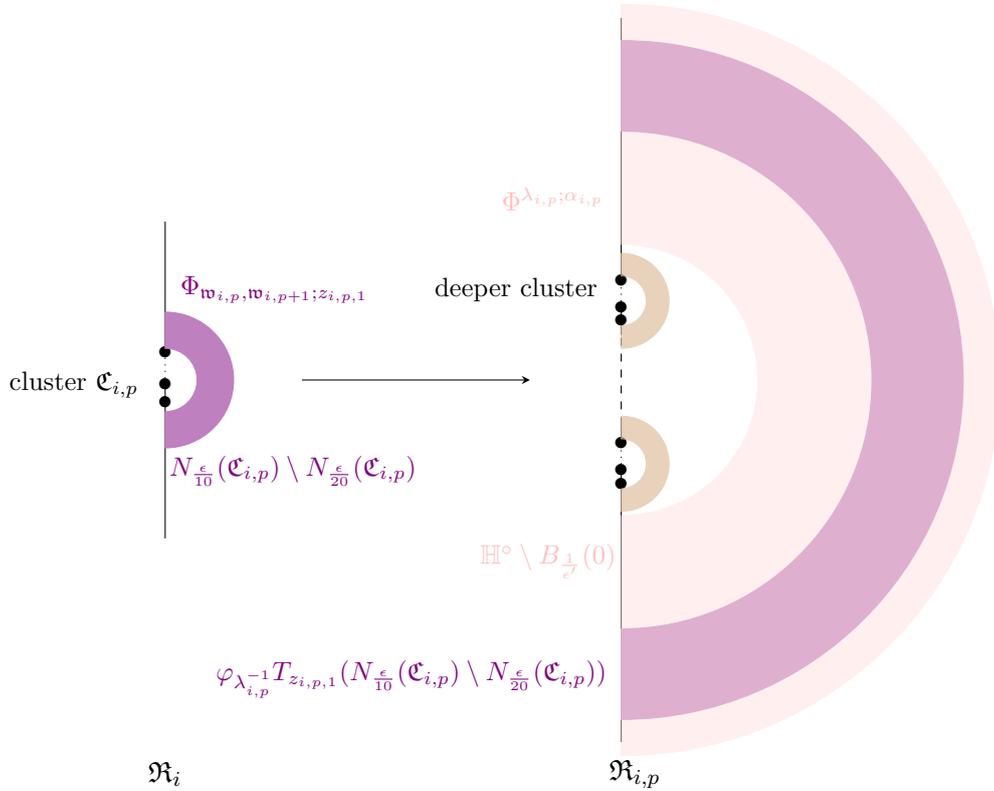

We may perform the above construction for each cluster $\mC_{i, p}$. Then we repeat this procedure for deeper scales, after passing to further subsequences.   Notice that this process must terminate after finitely many steps, since the number of turning points of $\mR_i$ is uniformly bounded. In particular, in each scale we have $(\epsilon, \delta_i)$ separation for the rescaled rod structure, for a possibly smaller $\epsilon>0$.
 
The upshot is that after passing to a subsequence we have obtained a good model map $\Psi_i$ for each $i$. By construction we know  $\Delta\Psi_i$ is zero except on the union of controlled number of annuli of the form $A_{r, Nr}$ (centered at different points in $\Gamma$) in $\dH^\circ$ for some uniform constant $N>0$. On each annuli we have a bound $|\Delta \Psi_i|\leq cr^{-2}$ for a uniform constant $c>0$ by scaling. We have a simple lemma.

 \begin{lemma}
 	Suppose $f$ is a smooth function on $\R^3$ which is supported in an annulus of the form $A_{r,  Nr}$  on which we have a bound $|f|\leq cr^{-2}$. Then there is a constant $C>0$ depending only on $c$ and $N$, such that on $\R^3$
 	$$(1+r)\cdot |\mathcal N_f|\leq C, $$
    where  $$
 	\mathcal N_f(x)=\int_{\R^3}\frac{f(y)}{|x-y|}dy.$$
 	 \end{lemma}
 \begin{proof}
 	Denote $\bar f(x)=r^2f(rx)$. Then $\bar f$ is supported on $A_{1, N}$ and $|\bar f|\leq c$. Also $\mathcal N_{\bar f}(x)=\mathcal N_f(rx)$.  It is then easy to estimate that $(1+r)\mathcal N_{\bar f}$ is uniformly bounded on $\R^3$ by a constant $C>0$ depending only on $c$ and $N$. \end{proof}

\begin{proposition}\label{t:bounded distance to model map}
   We have the following uniform pointwise bound for all $i$
\begin{equation}\label{e:uniform bound}\sup_{\dH^\circ} d(\Phi_i, \Psi_i)\leq C(1+r)^{-1}.	
\end{equation}
\end{proposition}

\begin{proof}
    This follows from the same arguments as in the proof of Theorem \ref{t:existence result}, using the above Lemma. 
\end{proof}

By construction we have a good understanding on the sequence $\Psi_i$ in all scales. The above  bound then leads to a uniform control of $\Phi_i$ in all scales, which allows us to get a \emph{bubble tree} convergence for the family. 

Recall that in the construction we have obtained a limit weak rod structure $\mR_\infty$. Denote by $\mathcal T$ the set of turning points of $\mR_\infty$. Let $\mR'_\infty$ be the natural rod structure associated to $\mR_\infty$ obtained by merging possible adjacent rods sharing the same rod vector, and let $\Phi'_\infty$ be the unique harmonic map which is strongly tamed by $\mR_\infty'$, which is obtained by Theorem \ref{t:existence result}. 
Away from $\mathcal T$, by construction we know $\Psi_i$ converges in $C^{1, \alpha}$ locally over $\dH\setminus \mathcal T$ to a limit map $\Psi_\infty$, which is locally tamed by $\mR_\infty'$. One can check that even though the convergence is only locally over $\dH\setminus \mathcal T$, due to the the explicit model we use in the construction of $\Psi_i$, \eqref{eq:Psii formula}, and Lemma \ref{l:ALF bound away turning points} and \ref{l:TN+ bound away turning points}, the limit map $\Psi_\infty$ is indeed smooth on $\dH^\circ$ and strongly tamed by $\mR_\infty'$, and it is harmonic outside a compact subset of $\dH^\circ$.

By the bound in Proposition \ref{t:bounded distance to model map}, passing to a subsequence we see that $\Phi_i$ also converges, in $C^{1, \alpha}$ locally over $\dH\setminus \mathcal T$, to a limit harmonic map $\Phi_\infty$, which is locally tamed by $\mR_\infty'$ and by Proposition \ref{t:bounded distance to model map} we have 
$$d(\Phi_\infty, \Phi_\infty')\leq C(1+r)^{-1}.$$
It follows that $\Phi_\infty$ is also strongly tamed by $\mR_\infty'$, so the uniqueness statement in Theorem \ref{t:existence result} implies that $\Phi_\infty=\Phi_\infty'$. 

Near a limit turning point $z_{\infty, p}$ of $\mR_{\infty}'$, as before after passing to a subsequence the rescaled sequence $\mR_{i, p}$ converges to a limit  weak, untyped rod structure $\mR_{\infty, p}$. Denote by $\mathcal T_{p}$ the set of turning points of $\mR_{\infty, p}$. Again we get an untyped rod structure $\mR_{\infty, p}'$ naturally associated to $\mR_{\infty, p}$. Furthermore passing to a further subsequence we may assume the following to hold
\begin{itemize}
\item Either $\alpha_{i, p}\geq 1$ for all $i$, or $\alpha_{i, p}<1$ for all $i$; 
\item There is a limit $\alpha_{i, p}\rightarrow\alpha_{\infty, p}\in [0, \infty]$ as $i\rightarrow\infty$;
\item If $\alpha_{i, p}<1$ for all $i$, then as $i\rightarrow\infty$, there are limits  $$\alpha_{i,p}\lambda_{i, p}\rightarrow \kappa_{\infty, p}\in [0, \infty];$$ 
\begin{equation}\label{eqn:definition of sigma ip}\sigma_{i,p}:=\alpha_{i, p}\sigma(\alpha_{i, p}, \lambda_{i, p})^2\rightarrow\sigma_{\infty, p}\in [0, \infty].
\end{equation}
\item If $\alpha_{i, p}\geq 1$ for all $i$,   then as $i\rightarrow\infty$, there are limits  
$$\alpha_{i, p}^{-1}\lambda_{i, p}\rightarrow \kappa_{\infty, p}\in [0, \infty];
$$
$$\sigma_{i,p}:=\alpha_{i, p}\sigma(4/\alpha_{i, p}, \lambda_{i, p})^{-2}\rightarrow\sigma_{\infty, p}\in [0, \infty].$$ 
\end{itemize}
Denote
\begin{equation}
  \Psi_{i, p}:=\begin{cases}
      \varphi_{\lambda_{i, p}}^*T_{-z_{i,p,1}}^* D_{\sigma(\alpha_{i, p}, \lambda_{i, p})}^{-1}P_{\mw_{i,p-1},\mw_{i,p}}.\Psi_{i}, & \text{if }\alpha_{i,p}<1;\\
      \varphi_{\lambda_{i, p}}^*T_{-z_{i,p,1}}^*D_{\sigma(4/\alpha_{i,p},\lambda_{i,p})} P_{\mw_{i,p-1},\mw_{i,p}}.\Psi_i, & \text{if }\alpha_{i,p}\geq1.
  \end{cases}  
\end{equation}
Then $\Psi_{i, p}$ is locally tamed by $\mR_{i, p}$.
There are five cases:
\begin{enumerate}
    \item $\sigma_{\infty, p}=0$. By \eqref{e:bound on sigma2} in this case we must have $\alpha_{i, p}<1$. By \eqref{e: rod vectors two cases}, the two infinite rod vectors coincide. Then we endow $\mR'_{\infty, p}$ with Asymptotic Type $\AF_0$. In this case $\kappa_{\infty,p}=0$.
    \item $\sigma_{\infty, p}=\infty$. By \eqref{e:bound on sigma} in this case we must have $\alpha_{i, p}\geq 1$. Again by \eqref{e: rod vectors two cases}, the two infinite rod vectors coincide. Then we endow $\mR'_{\infty, p}$ with Asymptotic Type $\AF_0$. In this case $\kappa_{\infty,p}=0$.
    \item $\sigma_{\infty, p}\in (0, \infty)$ and $\kappa_{\infty, p}<\infty$ and $\alpha_{i, p}<1$. By \eqref{e:definition of sigma} in this case we must have $\kappa_{\infty,p}>0$. Notice by construction $\lambda_{i, p}\rightarrow\infty$ so $\alpha_{\infty, p}=0$.  In this case  the two infinite rod vectors are different. Then we endow $\mR'_{\infty, p}$ with Asymptotic Type $\ALF^-$. 
     \item $\sigma_{\infty, p}\in (0, \infty)$ and $\kappa_{\infty, p}<\infty$ and $\alpha_{i, p}\geq 1$. By \eqref{e:definition of sigma} in this case we must have $\kappa_{\infty,p}>0$. Notice by construction $\lambda_{i, p}\rightarrow\infty$ so $\alpha_{\infty, p}=\infty$.  In this case  the two infinite rod vectors are different. Then we endow $\mR'_{\infty, p}$ with Asymptotic Type $\ALF^+$.  
    \item $\sigma_{\infty, p}\in (0, \infty)$,  $\kappa_{\infty, p}=\infty$.  In this case the two infinite rod vectors are different. Then we endow $\mR'_{\infty, p}$ with  Asymptotic Type  $\Ae$.  
       
\end{enumerate}
  In each case, Using Lemma \ref{l:ALF bound away turning points}, Lemma \ref{l:TN rescale upsilon}, and the construction in Proposition \ref{p:background map2}, one can check that $\Psi_{i, p}$ converges in $C^{1, \alpha}$, locally over $\dH\setminus \mathcal T_{p}$, to a smooth limit map $\Psi_{\infty, p}$. As before, an explicit computation show that $\Psi_{\infty, p}$  is globally tamed by $\mR'_{\infty, p}$.
  
By Proposition \ref{t:bounded distance to model map}, we see that after passing to a subsequence 
\begin{equation}\label{eq:harmonic map in the scale of mC}
   \Phi_{i,p}:=\begin{cases}
       \varphi_{\lambda_{i, p}}^*T_{-z_{i,p,1}}^* D_{\sigma(\alpha_{i, p}, \lambda_{i, p})}^{-1}P_{\mw_{i,p-1},\mw_{i,p}}.\Phi_{i}, &\text{if }\alpha_{i,p}<1;\\
       \varphi_{\lambda_{i, p}}^*T_{-z_{i,p,1}}^* D_{\sigma(4/\alpha_{i, p}, \lambda_{i, p})}P_{\mw_{i,p-1},\mw_{i,p}}.\Phi_{i}, &\text{if }\alpha_{i,p}\geq1,
   \end{cases}  
\end{equation}
converges in $C^{1, \alpha}$ locally over $\dH\setminus\mathcal T_p$, to a limit harmonic map $\Phi_{\infty, p}$  with
$$\sup_{\dH^\circ} d(\Phi_{\infty, p}, \Psi_{\infty, p})\leq C.$$
In particular,  $\Phi_{\infty, p}$ is globally tamed by $\mR'_{\infty, p}$. If $\mR'_{\infty, p}$ is of  Asymptotic Type  $\Ae$, then by Proposition \ref{p:AE rigidity}, we know that for some $Q\in SL(2; \R)$, $Q.\Phi_{\infty, p}$ is strongly tamed by $Q.\mR'_{\infty, p}$ (where we assign the Asymptotic Type to be $\Ae$). In the other cases we have similar statements (see Proposition \ref{p:ALF/AF rigidity}), but we do not need them until Section \ref{sec:discussion}.  See also Remark \ref{rem:TN AF rigidity}.

One can continue this process and obtain a bound on all the rescaled limits of $\Phi_i$ in terms of the explicit limits of the model map $\Psi_i$. We call any one of such rescaled limits of $\Phi_i$ a \emph{bubble limit}.
By the above discussion we know that a bubble limit is globally tamed by some rod structure. Passing to a subsequence we then obtain a bubble tree structure for the sequence $\Phi_i$.

\subsection{Limiting behavior of the cone angles}
\label{ss:limiting behavior of the cone angles}

Fix  $n\geq 2$ and an Asymptotic Type $\sharp\in \{\Ae, \ALFm,\ALFp, \AFa(\beta\in \R)\}$. 
Consider the space $\mathcal S(n, \sharp)$ of all rod structures $\mR$ which are of  Asymptotic Type  $\sharp$ and with $n$ turning points. It can be naturally identified with a subset of $\R_{+}^{n}\times (\R\dP^1)^{n+1}$ so has the induced topology.

For each $\mR\in \mathcal S(n, \sharp)$, by Theorem \ref{t:existence result} and Proposition \ref{prop:normalizing augmentation} there is a unique augmented harmonic map $(\Phi_{\mR}, \nu_{\mR})$ which  is strongly tamed by $\mR$ and is  such that $\nu_{\mR}$ is normalized. From the discussion in Section \ref{ss:uniform control of tame harmonic maps} (see also Remark \ref{rem:continuity of cone angles})  we know that for each fixed $j\in\{0, \cdots, n\}$, the normalized rod vector $\overline{\mv}_j$ (as an element in $\R^2/\{\pm 1\}$) depends continuously on $\mR$. Recall that there is a canonical lattice $\Lambda_{\mR}$ given in Definition \ref{def:canonical enhancement}  that ensures the toric Ricci-flat end is smooth without conical singularities.

From now on we further fix $n+1$ elements  $\mv_0, \cdots, \mv_n\in \R\dP^1$ with $\mv_{j}\neq \mv_{j+1}$, $j=0, \cdots, n-1$. Consider the space $\mathcal S(\sharp| \mv_0, \cdots, \mv_n)$  of all enhanced rod structures $(\mR, \Lambda_{\mR})$ such that $\mR\in\mathcal S(n, \sharp)$, with rod vectors given by $\mv_0, \cdots, \mv_n$ and with enhancement  given by the canonical lattice $\Lambda_{\mR}$. Since $\mv_0,\ldots,\mv_n$ are fixed, it follows that  $\Lambda_{\mR}$ is isomorphic to a fixed lattice $\Lambda$. Then $\mathcal S(\sharp| \mv_0, \cdots, \mv_n)$ is parametrized by the positions of the turning points. For $j\in \{0, \cdots, n\}$ we then obtain the cone angle function $$\vartheta_j:=\vartheta_{\Phi_{\mR}, \nu_{\mR}}(\mI_j).$$ It can be viewed as a continuous function of  $z_2, \cdots, z_n$ (recall that $z_1=0$ by definition), or equivalently of the rod lengths $l_j=z_{j+1}-z_{j}$, due to our considerations in Section \ref{ss:uniform control of tame harmonic maps}. By assumption we have $$\vartheta_0=\vartheta_n\equiv 2\pi.$$ For our applications in this paper we are interested in the behavior of $\vartheta_j$ when some of  the turning points come together or diverge to infinity. The main complication lies in the fact when we rescale the rod structure and the harmonic map to obtain the bubble limits, we need to apply various twists (see the $D_{\sigma(\alpha_{i, p}, \lambda_{i, p})}^{-1}P_{\mw_{i,p-1},\mw_{i,p}}$ introduced in \eqref{eq:harmonic map in the scale of mC}). 

\subsubsection{Turning points coming together}\label{sss:limit of cone angles when turning points come together}
Suppose now we have a sequence of enhanced rod structures $(\mR_i, \Lambda)$ in $\mathcal S(\sharp| \mv_0, \cdots, \mv_n)$ with $\ms(\mR_i)$ uniformly bounded and a fixed lattice $\Lambda$. Suppose we have a limit weak untyped rod structure $\mR_\infty$ as $i\rightarrow\infty$. Denote by $\mathcal T$ the set of turning points of $\mR_\infty$, and let $\mR_\infty'$ be the natural untyped rod structure associated to $\mR_\infty$. We also endow $\mR_\infty'$ with the enhancement $\Lambda$ and the Asymptotic Type $\sharp$. Notice that the infinite rod vectors of $\mR_{\infty}'$ are again given by $\mv_0$ and $\mv_n.$

Let $(\Phi_i, \nu_i)$ be the unique augmented harmonic map which is strongly tamed by $\mR_i$ such that $\nu_i$ is normalized, and let $(\Phi_\infty, \nu_\infty)$ be the unique augmented harmonic map which is strongly tamed by $\mR_\infty'$ such that $\nu_\infty$ is normalized. Any rod $\mI$ of $\mR_\infty$ is the limit of rods $\mI^i_j$ of $\mR_i$ for some $j\in\{0, \cdots, n\}$. 
The discussion in Section \ref{ss:uniform control of tame harmonic maps} implies that $\Phi_i$ converges in $C^{1, \alpha}$ locally over $\dH^\circ\setminus \mathcal T$ to $\Phi_\infty$. Since both $\nu_i$ and $\nu_\infty$ are normalized, we also have the local convergence of $\nu_i$ to $\nu_\infty$ on $\dH^\circ$. Moreover, the normalized rod vector $\overline {\mv}_j^i$ for the rod $\mI_j^i$ determined by $(\Phi_i, \nu_i)$  converges to the normalized rod vector for the rod $\mI_\infty$ determined by $(\Phi_\infty, \nu_\infty)$, hence we have 
\begin{equation}\label{e:angle continuity}
    \lim_{i\rightarrow\infty}\vartheta_{\Phi_i, \nu_i}(\mI_j^i)=\vartheta_{\Phi_\infty, \nu_\infty}(\mI).
\end{equation}
Since $\Phi_\infty$ is strongly tamed by $\mR_\infty'$, it follows that if the limits of $\mI_{j}^i$ and $\mI_{k}^i$ get merged into the same rod of $\mR_\infty'$, then we have
\begin{equation}\lim_{i\rightarrow\infty}\vartheta_{\Phi_i, \nu_i}(\mI_j^i)=\lim_{i\rightarrow\infty}\vartheta_{\Phi_i, \nu_i}(\mI_k^i).
\end{equation}

Now consider a turning point $z_{\infty, p}$ of $\mR_{\infty}$ with $m_p\geq 2$. For $i$ large there is a corresponding cluster $\mC_{i,p}=\{z_{i, p, 1}, \cdots,z_{i,p,n_p}\}$. We rescale  $\mR_{i}$ around $z_{i, p,1}$ to obtain rod structures $\mR_{{i, p}}$ (see \eqref{eq:def R mC_p}). The latter converges to a limit weak rod structure $\mR_{\infty, p}$, which then induces an associated rod structure $\mR_{\infty, p}'$. The maps $\Phi_{i, p}$ as given in \eqref{eq:harmonic map in the scale of mC} converge in $C^{1, \alpha}$ locally over $\dH^\circ\setminus \mathcal T_{p}$ to the first bubble  limit $\Phi_{\infty, p}$ at $z_{\infty, p}$. 
We know the latter is globally tamed by the rod structure $\mR_{\infty, p}'$. Notice that for the bubble limits, in general we  do not have naturally associated enhancements, due to the fact that the fixed lattice $\Lambda$ may get distorted in the rescaling process. 

Denote the rescaled augmentation 
\begin{equation}\label{eq:define tilde nu}
    \widetilde\nu_i:=\varphi_{\lambda_i}^*T_{-{z}_{i, p,1}}^*\nu_{i}.
\end{equation} 
There are constants $c_i$ such that $\nu_{i,p}:=\widetilde\nu_i+c_i$ converges smoothly to a limit augmentation $\nu_{\infty, p}$ for $\Phi_{\infty, p}$  locally on $\dH^\circ$, by the convergence of $\Phi_{i,p}$. Suppose we have a sequence of rods $\mI^i_j$ of $\mR_i$ which converges to a limit rod $\mI$ in $\mR_{\infty, p}$. Denote by $\mI^{i'}_j$ the corresponding rod in $\mR_{i, p}$ and $\mv_j'$ the corresponding rod vector. Let $\overline{\mv}_j^i\in \mv_j$ the normalized rod vector with respect to $(\Phi_i, \nu_i)$, and denote by $\overline{\mv}_{j}^{i'}\in \mv_{j}'$ the normalized rod vector with respect to $(\Phi_{i,p},\widetilde{\nu}_i+c_i)$. The above discussion implies that $\overline{\mv}_j^{i'}$ converges to the normalized rod vector $\overline{\mv}_\infty$ along the rod $\mI$ with respect to $(\Phi_{\infty, p}, \nu_{\infty, p})$.

\begin{lemma}
   In the above setting $\mR_{\infty, p}'$ is either of  Asymptotic Type  $\Ae$ or of Asymptotic Type $\AF_0$.
\end{lemma}
\begin{proof}
    By definition by passing to a subsequence we know that in our setting $\alpha_{i, p}$ is independent of $i$, since $\mv_0,\ldots,\mv_n$ are fixed. If $\alpha_{i, p}\equiv0$, then it follows that $\sigma_{\infty, p}=0$ and $\mR_{\infty, p}'$ is of Asymptotic Type $\AF_0$. If $\alpha_{i, p}\equiv \alpha>0$, then since $\lambda_{i, p}\rightarrow\infty$, we must have $\kappa_{\infty ,p}=\infty$ and $\mR_{\infty, p}'$ is then of  Asymptotic Type  $\Ae$.
\end{proof}

Now  we first assume $\mR'_{\infty,p}$ is of  Asymptotic Type  $\Ae$, then by Proposition \ref{p:AE rigidity} we know  that $Q.\Phi_{\infty, p}$ is strongly tamed by $Q.\mR_{\infty, p}'$ for a matrix $Q\in SL(2; \R)$. Moreover, by our analysis in Section \ref{ss:uniform control of tame harmonic maps} in this case we have $\sigma_{\infty, p}\in (0, \infty)$. Indeed, since we have fixed the rod vectors $\mv_0,\ldots,\mv_n$, $\sigma_{i, p}$ is independent of $i$ for $i$ large. 

Let $\mI$  be an infinite rod of $\mR_{\infty, p}'$. We know by construction that before rescaling the normalized rod vector $\overline{\mv}_j^i$ converges to a normalized rod vector in $\mR_\infty$. So we have 
$$C^{-1}\leq |\overline{\mv}_j^i|\leq C.$$
It is straightforward to see that 
\begin{equation}\label{eq:AE normalized rod vector after rescaling}
    \overline{\mv}_j^{i'}=\begin{cases}
    e^{c_i}\sqrt{\lambda_{i,p}}D_{\sigma(\alpha_{i, p}, \lambda_{i, p})}^{-1}P_{\mw_{i,p-1},\mw_{i,p}}\overline{\mv}_j^i, &\text{if } \ \alpha_{i,p}<1;\\
    e^{c_i}\sqrt{\lambda_{i,p}}D_{\sigma(4/\alpha_{i, p}, \lambda_{i, p})}P_{\mw_{i,p-1},\mw_{i,p}}\overline{\mv}_j^i, &\text{if} \ \alpha_{i,p}\geq1.
\end{cases}
\end{equation}
It follows from the uniform bound on $\sigma_{i, p}$ and the convergence of $\overline{\mv}_j^{i'}$ that $e^{c_i}$ is uniformly comparable to $1/\sqrt{\lambda_{i,p}}$, hence can be chosen such that $$e^{c_i}\equiv1/\sqrt{\lambda_{i,p}}.$$ Moreover, we let $\Lambda_{i, p}$ be the enhancement for $\mR_{i, p}$ defined by 
\begin{equation}\label{eq:AE choice of lattice after rescaling}
    \Lambda_{i, p}=\begin{cases}
       D_{\sigma(\alpha_{i, p}, \lambda_{i, p})}^{-1}P_{\mw_{i,p-1},\mw_{i,p}}.\Lambda_{\mR}, &\text{if }\alpha_{i,p}<1;\\
        D_{\sigma(4/\alpha_{i, p}, \lambda_{i, p})}P_{\mw_{i,p-1},\mw_{i,p}}.\Lambda_{\mR}, &\text{if }\alpha_{i,p}\geq1.
       \end{cases}
\end{equation}
Now for any $\mI$ in $\mR_{\infty,p}$, because of \eqref{eq:AE normalized rod vector after rescaling}, $e^{c_i}\equiv1/\sqrt{\lambda_{i,p}}$, and our particular choice of lattice \eqref{eq:AE choice of lattice after rescaling}, we have
\begin{equation}
    \vartheta_{\Phi_i, \nu_i}(\mI^i_j)=\vartheta_{\Phi_{i, p}, \nu_{i, p}}(\mI^{i'}_j).
\end{equation}
Note that since the rod vectors $\mv_0,\ldots,\mv_n$ are fixed, the above $\Lambda_{i,p}$ does not depend on $i$ if $i$ is large. Passing to a subsequence there is a limit enhancement $\Lambda_{\infty, p}$  for $\mR_{\infty, p}'$.  Let $\bbv_j\in \mv_j$ be the primitive rod vector in $\Lambda$, then it naturally induces a primitive rod vector $\bbv_j^{i'}$ along the rod $\mI^{i'}_j$ in $\Lambda_{i,p}$, which then converges to a primitive rod vector $\bbv_\infty$ along the rod $\mI$ in $\Lambda_{\infty,p}$. It follows that
\begin{equation}\label{eqn:AE convergence of cone angles}
    \lim_{i\to\infty}\vartheta_{\Phi_i, \nu_i}(\mI^i_j)=\vartheta_{\Phi_{\infty, p}, \nu_{\infty, p}}(\mI).
\end{equation}

If  $\mR'_{\infty, p}$ has 2 turning points, then  by Proposition \ref{p:AE rigidity} and Lemma \ref{lem:uniquness of EH}, without loss of generality, up to modification by some $Q\in SL(2;\R)$, we may assume $\mR_{\infty, p}'=\mR_{\EH; n_1, n_2}$, and $\Phi_{\infty,p}'=\Phi_{\EH;n_1, n_2}$ for some $n_1, n_2$.    Denote the 3 rods of $\mR'_{\infty, p}$ by $\mI^\infty_0, \mI^\infty_1, \mI^\infty_2$. They come as the limits of 3 rods $\mI_{j_0}^i, \mI_{j_1}^i, \mI_{j_2}^i$ of $\mR_i$. We may find primitive rod vectors $\bbv_{j_0}, \bbv_{j_1}, \bbv_{j_2}$ in $\Lambda$ along these rods respectively, such that 
\begin{equation}\label{e:primitive vector relations}
    \bbv_{j_1}=\gamma_0\bbv_{j_0}+\gamma_2\bbv_{j_2}
\end{equation}
for some $\gamma_0, \gamma_2>0$. Then by \eqref{e:angle equality} we have 
\begin{equation}
\vartheta_{\Phi_{\infty,p}, \nu_{\infty, p}}(\mI^\infty_1)=\gamma_0\vartheta_{\Phi_{\infty,p}, \nu_{\infty, p}}(\mI^\infty_0)+\gamma_2\vartheta_{\Phi_{\infty,p}, \nu_{\infty, p}}(\mI^\infty_2)
\end{equation}
Hence,
\begin{equation}\label{e:angle equality before limit}
    \lim_{i\rightarrow\infty}\frac{\vartheta_{\Phi_i, \nu_i}(\mI_{j_1}^i)}{\gamma_0\vartheta_{\Phi_i, \nu_i}(\mI_{j_0}^i)+\gamma_2\vartheta_{\Phi_i, \nu_i}(\mI_{j_2}^i)}=1.
\end{equation}

If $\mR'_{\infty, p}$ has more than 2 turning points, then we do not have as explicit information on the angle function as above, but we always  have a \emph{comparison property}. Namely, suppose $\mv_{\infty, 0}, \cdots, \mv_{\infty, m}$ are the rod vectors of $\mR_{\infty, p}$ and suppose furthermore that the 
limit lattice $\Lambda_{\infty,p}$ 
\begin{itemize}
    \item is generated by primitive vectors $\bbv_{0}\in\mv_{\infty,0},\bbv_{m}\in\mv_{\infty,m}$;
    \item gives a smooth enhancement for $\mR_{\infty, p}'$.
\end{itemize}
Notice that the first assumption can always be achieved by passing to a coarser lattice before taking limit. 
Under the above assumptions, by the the above discussion and Proposition \ref{p:angle comparison} we get
\begin{proposition}\label{p:angle comparision before limit}
    For $i$ sufficiently large, suppose $\mI_{0}^i$, $\mI_{1}^i, \mI_{m-1}^i, \mI_{m}^i$ are the rods of $\mR_i$ that converge to the limit rods in $\mR_{\infty,p}$ with rod vector $\mv_{\infty,0},\mv_{\infty, 1}, \mv_{\infty, m-1}, \mv_{\infty,m}$ respectively.
   Then 
    \begin{align}\label{e:angle comparision before limit AE}
         \lim_{i\rightarrow\infty}\vartheta_{\Phi_i, \nu_i}(\mI_1^i)\geq \lim_{i\to\infty}\vartheta_{\Phi_i, \nu_i}(\mI_m^i), &&
         \lim_{i\rightarrow\infty}\vartheta_{\Phi_i, \nu_i}(\mI_{m-1}^i)\geq 
         \lim_{i\to\infty}\vartheta_{\Phi_i, \nu_i}(\mI_0^i).
   \end{align}
\end{proposition}

\

Now we assume $\mR_{\infty, p}'$ is of Asymptotic Type $\AF_0$, then we must have 
$$\alpha_{i, p}\equiv0 \quad \text{and}\quad \sigma(\alpha_{i, p}, \lambda_{i, p})=\sqrt{\lambda_{i, p}}.$$   Let $\mI$ be a finite rod of $\mR_{\infty, p}$, then it arises as the rescaled limit of the  corresponding rod $\mI^i_j$ in $\mR_i$ for some $j\in \{0, \cdots, n\}$. The corresponding rod vector $\mv_j$ is independent of $i$.
Let ${\mI}^i_{l}$ be the rod  of $\mR_i$ that converges to an infinite rod $\widetilde{\mI}$ in $\mR_{\infty, p}'$. Denote by $\mv$ the rod vector along  $\widetilde{\mI}$.
\begin{proposition} \label{p:cone angle AF rod length go to zero}
	If $\mv_j\neq\mv$, then $$\lim_{i\rightarrow\infty}\vartheta_{\Phi_i, \nu_i}(\mI_j^i)=\infty.$$
\end{proposition}
\begin{proof}
Without loss of generality we can assume $\mv=[\bv_0]$, hence for $\mC_{i,p}$, we know $\mw_{i,p-1}=\mw_{i,p}=[\bv_0]$ and $P_{\mw_{i,p-1},\mw_{i,p}}$ is the identity map.
By construction we have the convergence of the harmonic maps
$$\Phi_{i, p}=\varphi_{\lambda_{i, p}}^*T_{-z_{i,p,1}}^* D_{\sqrt{\lambda_{i, p}}}^{-1}.\Phi_{i}$$   
to $\Phi_{\infty,p}$ in $C^{1, \alpha}$ locally on $\dH^\circ\setminus \mathcal T_{\infty, p}$, which is globally tamed by $\mR'_{\infty,p}$.  Let $\overline{\mv}_j^i$ and $\overline{\mv}_{l}^i$ be the normalized rod vectors for the rod $\mI_j^i$ and $\mI_{l}^i$ respectively, which are determined by $(\Phi_i, \nu_i)$. 
Denote by $\overline{\mv}_j^{i'}$ and $\overline{\mv}_{l}^{i'}$ the corresponding normalized rod vectors determined by $(\Phi_{i,p}, \widetilde{\nu}_i+c_i)$, where recall $\widetilde{\nu}_i$ is defined in \eqref{eq:define tilde nu} and $c_i$ is a suitably chosen constant to ensure the convergence of the augmentation $\tilde{\nu}_i+c_i$.
Since $\Phi_{\infty, p}$ is the first bubble limit, in the original scale ${\mI}_{l}^i$ converges to a limit rod in $\mR_\infty$, so we have a uniform bound 
$$C^{-1}\leq |\overline{\mv}_l^i|\leq C. $$
It is easy to check that $$\overline{\mv}_l^{i'}=e^{c_i}D^{-1}_{\sqrt{\lambda_{i,p}}}\sqrt{\lambda_{i,p}}\overline{\mv}_l^i.$$
Since $[{\mv}_l^i]=[\bv_0]$, we see that $$\overline{\mv}_l^{i'}=e^{c_i}\overline{\mv}_l^i$$ and
$$C^{-1} e^{c_i}\leq |\overline{\mv}_{l}^{i'}|\leq C e^{c_i}. $$
Since we have the convergence of normalized rod vectors of $(\Phi_{i,p},\widetilde{\nu}_{i}+c_i)$ in the rescaled limit $\mR_{\infty,p}$, we must have $|c_i|<C$.  Suppose now that $\mv_j=[(a, b)^\top]$, where $b\neq0$ since $\mv_j\neq\mv$. Then we can write $$\overline{\mv}_j^i=\mu_i(a, b)^\top$$ and
$$\overline{\mv}_j^{i'}=\mu_ie^{c_i} (a, \lambda_{i, p}b)^\top.$$
If $b\neq0$, since $\overline{\mv}_j^{i'}$ converges, it is easy to see that $\mu_i$ is uniformly comparable to $\lambda_{i,p}^{-1}$, so it converges to $0$. Hence we conclude that $\vartheta_{\Phi_i, \nu_i}(\mI)\rightarrow\infty$.
\end{proof}

The above discussion provides information on the behavior of  the cone angle function along the rods of $\mR_i$ which survive in the first bubble limit (namely, $(\Phi_{\infty, p}, \mR_{\infty, p}')$). Note that by our assumption $(\mR_i,\Lambda)\in\mathcal S(\sharp|\mv_0,\ldots,\mv_n)$, so all the rod vectors are fixed and do not depend on $i$. However, it does not yield a useful statement if the first bubble limit has a \emph{trivial} rod structure, namely, if $\mR_{\infty, p}'$ has no finite rods. This means that either $\mR_{\infty, p}'$ has only one rod if it is of Asymptotic Type $\AF_0$, or $\mR_{\infty, p}'$ has exactly two rods (which are the two infinite rods) if it is of Asymptotic Type $\Ae$. If these occur, we need to pass to even smaller scales. 
Notice that from our initial hypothesis, it follows that the untyped rod structure $\mR_{\infty, p}$ always has a finite rod. So in this situation we consider a turning point $q$ of $\mR_{\infty, p}$ which is not a turning point of $\mR_{\infty, p}'$. In other words, $q$ lies in the interior of a rod $\mI$ of $\mR_{\infty, p}'$ and adjacent rods of $\mR_{\infty,p}$ (sharing same rod vectors however) merge together at $q$. Denote by $\mv$ the rod vector along $\mI$.  In this case we see in the next scale the bubble limit at $q$, denoted by $(\Phi_{\infty,2}, \mR_{\infty,2}')$, must be of Asymptotic Type $\AF_0$, since both of its infinite rod vectors coincide with $\mv$. It is possible that $\mR_{\infty,2}'$ is again trivial, then  there is another merging of rods in this bubble limit. If this occurs then we pass to a further smaller scale and obtain another bubble limit of Asymptotic Type $\AF_0$ whose infinite rod vectors again coincide with $\mv$. We can continue this process until at some scale we  see a bubble limit $(\Phi_{\infty,k}, \mR_{\infty,k}')$ whose rod structure contains more than one rod, then it has at least one finite rod. In all the intermediate scales we obtain bubble limits which are trivial and with the rod vector given by $\mv$. By construction, $\mR_{\infty,k}'$ is non-trivial so it has a finite rod whose rod vector differs from $\mv$. Suppose we have a rod $\mI_j^i$ of $\mR_i$ which survives in this bubble limit and the limit rod vector differs from $\mv$. We may view such a bubble limit as a first \emph{non-trivial} bubble limit. 

 Figure \ref{Figure:Merging and deeper scale} gives an illustration of this phenomenon. We consider  rod structures $\mR_i$ of Asymptotic Type $\Ae$ with $4$ turning points,  whose rod length $l_4\ll l_2\ll l_3\ll l_1\to0$. The limit rod structure is a trivial rod structure of Asymptotic Type $\Ae$. In the first bubble limit only  $\mI_1$ survives and the rod structure is  of Asymptotic Type $\Ae$. It is again trivial since $\mv_1$ coincides with the infinite rod vector $\mv_5$. In the next scale  $\mI_3$ survives in the bubble limit, and the rod structure is of Asymptotic Type $\AF_0$ but is again trivial, since $\mv_3=\mv_1=\mv_5$. Then in the next scale $\mI_2$ survives in the bubble limit. The corresponding  rod structure that is of Asymptotic Type $\AF_0$ is non-trivial and this bubble limit is the first non-trivial bubble limit. The rod $\mI_4$ only survives after passing to the next scale. 
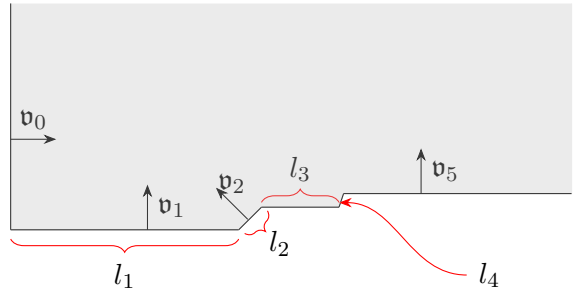
\begin{figure}[ht]
    \begin{tikzpicture}[scale=0.6]
        \draw (-2,0)--(-2,5);
        \draw (-2,0)--(3,0);
        \draw (3,0)--(3.5,0.5);
        \draw (3.5,0.5)--(5.2,0.5);
        \draw (5.2,0.5)--(5.3,0.8);
        \draw (5.3,0.8)--(10.3,0.8);

        \draw[-Stealth] (-2,2)--node [above]{$\mathfrak{v}_0$} (-1,2);
        \draw[-Stealth] (1,0)--node [right]{$\mathfrak{v}_1$} (1,1);
        \draw[-Stealth] (3.2,0.23)--node [above]{$\mathfrak{v}_2$} (2.5,0.93);

        \draw[-Stealth] (7,0.8)--node [right]{$\mathfrak{v}_5$} (7,1.8);

        \draw [draw = red,decorate,decoration={brace,amplitude=5pt,mirror,raise=4ex}]
  (-2,0.9) -- (3,0.9) node[midway,yshift=-3em]{$l_1$};

        \draw [draw = red,decorate,decoration={brace,amplitude=5pt,mirror,raise=4ex}]
  (2.35,0.7) -- (2.85,1.2) node[midway,yshift=-3em]{};
        \node at (3.9,-0.3) {$l_2$};

        \draw [draw = red,decorate,decoration={brace,amplitude=5pt,mirror,raise=4ex}]
  (5.2,-0.5) -- (3.5,-0.5) node[midway,yshift=-3em]{};
        \node at (4.35,1.35) {$l_3$};

        \draw [draw = red,-Stealth] (8,-1) to [out=180,in=0] (5.2,0.6);
        \node at (8.5,-1) {$l_4$};
        
        \filldraw[color=gray!50,fill=gray!50,opacity=0.3](-2,5)--(-2,0) --(3,0)--(3.5,0.5)--(5.2,0.5)--(5.3,0.8)--(10.3,0.8)--(10.3,5);
        
    \end{tikzpicture}
	\caption{A sequence of rod structures with 4 finite rods $\mI_1,\mI_2,\mI_3,\mI_4$ whose lengths are given by $l_4\ll l_2\ll l_3\ll l_1\to0$. The infinite rods are $\mI_0$ and $\mI_5$.}
	 \label{Figure:Merging and deeper scale}   
\end{figure}

For the above situation, we have a similar conclusion as in Proposition \ref{p:cone angle AF rod length go to zero}. 

\begin{proposition}\label{p:deep scale angle goes to infinity}
    We have  $$\lim_{i\rightarrow\infty}\vartheta_{\Phi_i, \nu_i}(\mI_j^i)=\infty.$$
\end{proposition}
\begin{proof}
     The key point to notice is that all the intermediate scales share the same infinite rod vector so when we pass to a smaller scale the accumulated twists are simply diagonal matrices of the form $D_{\sqrt{\lambda_i}}$ for $\lambda_i>0$. The case when the first bubble limit $\mR_{\infty, p}'$ is of Asymptotic Type $\AF_0$ can be proved by the same arguments as in the proof of Proposition \ref{p:cone angle AF rod length go to zero}. In the case when $\mR_{\infty, p}'$ is of Asymptotic Type $\Ae$, one may reduce to the $\AF_0$ case due to \eqref{eqn:AE convergence of cone angles}. Essentially, in this case the twists we need to use in order to obtain the first bubble limit $(\Phi_{\infty, p}, \mR_{\infty, p}')$ are uniformly bounded. 
\end{proof}

By passing to  smaller scales one may gain more information on the cone angles along other even shorter rods of $\mR_i$.  The analysis in general would involve understanding the iterative effect of the twists from all the intermediate scales. For our purpose  we will not need this so we do not investigate it in detail.

\subsubsection{Turning points diverging to infinity}\label{sss:limit of cone angles with turning points go to infinity}
We first  consider a sequence $(\mR_i, \Lambda)$  in $\mathcal S(\AFa| \mv_0, \cdots, \mv_n)$ for some $\beta\in \R$ such that $\ms_i:=\ms(\mR_i)\rightarrow\infty$. Again notice that the enhancement is a fixed lattice $\Lambda$. Without loss of generality we assume $\mv_0=\mv_n=[\bv_0]$, and it follows that $\Lambda_{\mR}=\mathbb{Z}^2$. Let $(\Phi_i, \nu_i)$ ($i\in \dZ_{\geq 0}$) be the unique augmented harmonic map which is strongly tamed by $\mR_i$ such that $\nu_i$ is normalized.

We define the twisted scaling  down of the form $$\widehat{\mR}_i:=(U_{-\beta}D_{\sqrt{\ms_i}}U_{\beta}).\varphi_{1/\ms_i}^*\mR_i,$$ which satisfies that $\ms(\widehat{\mR}_i)=1$.  The twist by the action of $U_{-\beta} D_{\sqrt{\ms_i}}U_\beta$ is necessary in order to apply the discussion in Section \ref{ss:uniform control of tame harmonic maps}. The point is that while our original harmonic map $\Phi_i$ is asymptotic to a fixed $\AFa$ model, the scaling down $\varphi_{1/\ms_i}^*\Phi_i$ is instead asymptotic to $$\varphi_{1/\ms_i}^*\Phi^{\AFa}=(U_{-\beta}D_{1/\sqrt{\ms_i}}U_{\beta}). \Phi^{\AFa}.$$ However, the twisted scaling down  $$\widehat{\Phi}_i:=(U_{-\beta}D_{\sqrt{\ms_i}}U_{\beta}).\varphi_{1/\ms_i}^*\Phi_i$$ is again asymptotic to $\Phi^{\AFa}$, and it is indeed strongly tamed by $\widehat{\mR}_i$. So we can apply the discussion in Section \ref{ss:uniform control of tame harmonic maps} to the new sequence of rod structures $\widehat{\mR}_i$. We also define $$\widehat{\nu}_i=\varphi_{1/\ms_i}^*\nu_i,$$ then one can check that it gives the normalized augmentation for $\widehat{\Phi}_i$, as now we have 
$$\overline{\widehat{\mv}}_0^i=U_{-\beta}D_{\sqrt{\ms_i}}U_\beta\sqrt{\ms_i^{-1}}\overline{\mv}_0\equiv\bv_0.$$

As before passing to a subsequence there is a weak limit rod structure $\mR_\infty$ of  Asymptotic Type  $\AFa$. Again let $\mR_\infty'$ be the rod structure associated to $\mR_\infty$ and denote by $\mathcal T$ the set of turning points of $\mR_\infty$. We may assume that $\widehat\Phi_i$ converges, in $C^{1, \alpha}$ locally over $\dH^\circ\setminus \mathcal T$, to a limit harmonic map which is strongly tamed by $\mR_\infty'$.

Suppose a rod $\mI$ of $\mR_\infty$ is the limit of the rod $\widehat{\mI}_j^i$ for some fixed $j\in\{0, \cdots, n\}$.  It is easy to see that the limit rod vector $\mv$ along $\mI$ is given by
\begin{equation*}
\mv=\begin{cases}
[\bv_0], &  \text{if} \   \mv_j\neq [(-\beta, 1)^\top];\\
[(-\beta, 1)^\top], & \text{if}\   \mv_j=[(-\beta, 1)^\top].
\end{cases}
\end{equation*}

 The infinite rods of $\mR_\infty'$ clearly have rod vector $[\bv_0]$. It is possible that some finite rods of $\mR_\infty$ are merged into the infinite rods of $\mR_\infty'$; such rods must arise from limits of rods $\mI_j^i$ with $\mv_j\neq [(-\beta, 1)^\top]$.

\begin{proposition} \label{p:cone angle AF rod length go to infinity}
Let $(a,b)^T$ be a primitive rod vector in $\mv_j$ for $(\mR_i, \Lambda)$. If the corresponding limit rod $\mI$ belongs to the infinite rod or $\mv_j= [(-\beta, 1)^\top]$, then 
    $$\lim_{i\rightarrow\infty}\vartheta_{(\Phi_i, \nu_i)}(\mI_j^i)= 2\pi |a+\beta b|.$$
\end{proposition}
\begin{proof}
Since $\widehat{\nu}_i$ is normalized,  the normalized rod vector along the infinite rods $\widehat{\mI}_0^i$ and $\widehat{\mI}_n^i$ of $\widehat{\mR}_i$, with respect to $(\widehat{\Phi}_i, \widehat{\nu}_i)$, is given by $\bv_0$. It follows that the normalized rod vector along other $\widehat{\mI}_j^i$ satisfies that 
$$ C^{-1}\leq |\overline{\widehat{\mv}}_{j}^i|\leq C$$
due to the convergence of $(\widehat{\Phi}_i,\widehat{\nu}_i)$. Note that there is an explicit relation between $\overline{\mv}_j^i$ and $\overline{\widehat{\mv}}_j^i$
$$\overline{\widehat{\mv}}_j^i=\frac{1}{\sqrt{\ms_i}}U_{-\beta}D_{\sqrt{\ms_i}}U_{\beta}\overline{\mv}_j^i.$$
Now let $(a, b)^\top$ be a primitive rod vector in $\mv_j$, where $\overline{\mv}_j^i=\mu_i(a,b)^\top$. It follows that
$$\overline{\widehat{\mv}}_j^i=\mu_i(a+b\beta-\frac{1}{\ms_i}b\beta,\frac{1}{\ms_i}b)^\top.$$
If $\mI$ belongs to the infinite rod in $\mR_\infty$ then we have $\overline{\widehat\mv}_j^i\to\bv_0$, which implies that 
$$\lim_{i\rightarrow\infty}\mu_i=\frac{1}{a+b\beta}.$$ 
So we have 
$$\lim_{i\to\infty}\vartheta_{(\Phi_i,\nu_i)}(\mI_j^i)=2\pi|a+\beta b|.$$
If $\mv_j=[(-\beta,1)^\top]$, then as $|\overline{\widehat{\mv}}_j^i|$ remains bounded, we have $\mu_i$ is uniformly comparable to $\ms_i$.  In this case we then have $$\lim_{i\to\infty}\vartheta_{\Phi_i,\nu_i}(\mI_j^i)=0.$$
\end{proof}
\begin{remark}
    If $\mI$ does not belong to the infinite rod and $\mv_j\neq [(-\beta, 1)^\top]$, then from the above argument we can only say that the limit of the angle function is finite. 
\end{remark}

Finally we consider a sequence $(\mR_i, \Lambda)$  in $\mathcal S(\ALFm| \mv_0, \cdots, \mv_n)$ such that $\ms_i:=\ms(\mR_i)\rightarrow\infty$. Without loss of generality we assume in addition that $\mv_0=[\bv_{\alpha/2}]$ and $\mv_n=[\bv_{-\alpha/2}]$ for some fixed $\alpha>0$, where then it follows $\Lambda_{\mR}$ is generated by $\bv_{\alpha/2}$ and $\bv_{-\alpha/2}$. Let $(\Phi_i, \nu_i)$ ($i\in \dZ_{\geq 0}$) be the unique augmented harmonic map which is strongly tamed by $\mR_i$ such that $\nu_i$ is normalized. 

Again we need to consider the scaling down. For the model map, we have by \eqref{e:scaling ALF model} that 
$$\varphi_{1/\ms_i}^*\Phi^{\TNam}=D_{1/\sqrt{\ms_i}}. \Phi^{\TN^-_{\alpha/\sqrt{\ms_i}}}.$$
So we should consider the twisted scaling down $$\widehat{\mR}_{i}:=D_{\sqrt{\ms_i}}. \varphi_{1/\ms_i}^*\mR_i$$ and $$\widehat{\Phi}_i:=D_{\sqrt{\ms_i}}. \varphi_{1/\ms_i}^*\Phi_i. $$
One can see that $\widehat{\Phi}_i$ is strongly tamed by $\widehat{\mR}_i$ which is of Asymptotic Type $\ALF^-$.  Then we can apply the discussion in Section \ref{ss:uniform control of tame harmonic maps}.

As before passing to a subsequence there is a weak limit rod structure $\mR_\infty$. Again let $\mR_\infty'$ be the rod structure associated to $\mR_\infty$ and denote by $\mathcal T$ the set of turning points of $\mR_\infty$. Since $\alpha/\sqrt{\ms_i}\rightarrow 0$ one sees that $\mR_\infty'$ should be equipped with Asymptotic Type $\AF_0$. We may assume that $\widehat\Phi_i$ converges, in $C^{1, \alpha}$ locally over $\dH^\circ\setminus \mathcal T$, to a limit harmonic map which is strongly tamed by $\mR_\infty'$ by passing to a subsequence. 

Suppose a rod $\mI$ of $\mR_\infty$ is the limit of the rod $\widehat{\mI}_j^i$ for some fixed $j\in \{0, \cdots, n\}$. One can see that the limit rod vector $\mv$ along $\mI$ is given by $[\bv_0]$ if $\mv\neq [(0, 1)^\top]$, and agrees with $\mv$ if $\mv=[(0, 1)^\top]$. Similar to the proof in the $\AFa$ case, we have 

\begin{proposition} \label{p:cone angle TN rod length go to infinity}
Let $(a, b)^\top$ be the primitive rod vector in $\mv_j$ for $(\mR_j, \Lambda)$.  If $\mI$ belongs to the infinite rod or $a=0$, then
    $$\lim_{i\rightarrow\infty}\vartheta_{\Phi_i, \nu_i}(\mI_j^i)= 2\pi |a|.$$
\end{proposition}

\section{Applications}\label{sec:applications}
In this section we prove Theorem \ref{t:main}. Our strategy is to design a smooth enhanced rod structure \((\mR, \Lambda)\) and demonstrate that the rod lengths can be adjusted such that the augmented harmonic map which is strongly tamed by \(\mR\) and of a given Type, has  cone angle  \(2\pi\) along every rod. This will then generate a toric gravitational instanton.

\subsection{The case of two turning points}
As a warm-up, we consider rod structures with  exactly two turning points. It is easy to see that in this case the degree of the rod structure is at most one.  According to Theorem \ref{thm:equivalence of degree and type} later the gravitational instantons we obtain below must be of Type $\I$ or Type $\II$ (for their definitions, see Section \ref{sec:7.1}), thus belong to the known class of examples, but the arguments below illustrates the key idea in proving Theorem \ref{t:main}. It also explains the fact that the Taub-NUT metric serves as building blocks for obtaining the other Type $\II$ gravitational instantons, including the Schwarzschild, Kerr, Taub-Bolt spaces, albeit in a non-explicit way.

\subsubsection{The \texorpdfstring{$\AF$}{} case}
Let $(\mR_l, \Lambda)$ be the smooth enhanced rod structure of  Asymptotic Type  $\AFa$ with exactly two turning points $z_0=0$ and $z_1=l>0$ and with the enhancement $\Lambda$ generated by 
$$\mv_0=\mv_2=[\bv_{0}],  \quad \mv_1=[(0, 1)^{\top}].$$
It is easy to see that $d(\mR_l)=1$.
Fix $\beta\in (-1, 1)$. By Theorem \ref{t:existence result} and Proposition \ref{prop:normalizing augmentation}, we have a unique augmented harmonic map $(\Phi_l, \nu_l)$ which is strongly tamed by $\mR$, such that $\nu_l$ is normalized. By definition, we have  
$$\vartheta_{\Phi_l, \nu_l}(\mI_0)=\vartheta_{\Phi_l, \nu_l}(\mI_2)=2\pi.$$

Denote 
$$\Theta(l)=\vartheta_{\Phi_l, \nu_l}(\mI_1).$$ 
By \eqref{e:angle continuity}, $\Theta$ is a continuous function of $l\in (0,\infty)$. 

\begin{proposition}\label{p:Kerr family}
    We have $$\lim_{l\rightarrow0}\Theta(l)= \infty,$$ and $$\lim_{l\rightarrow\infty}\Theta(l)=2\pi|\beta|.$$
\end{proposition}
\begin{proof}

The first claim follows from Proposition \ref{p:cone angle AF rod length go to zero} and the second claim      follows from Proposition \ref{p:cone angle AF rod length go to infinity}.
\end{proof}

It follows that for each $\beta\in (-1,1)$ there is an $l\in (0,\infty)$ such that $\Theta(l)=2\pi$. Then the corresponding  $(\Phi_l, \nu_l)$ defines a smooth AF gravitational instanton. These are exactly the Kerr spaces by Corollary \ref{c:classification of toric GI}.

Using the fact that Kerr spaces are of Type $\II$ and the result of Biquard-Gauduchon \cite{BG}, we have an explicit formula for $\Theta(l)$. In particular, we know that for each $\beta\in (-1, 1)$, there is a unique $l=l(\beta)>0$ such that $\Theta(l)=2\pi$. Moreover, when $\beta=0$ we get the Schwarzschild space and when $\beta\rightarrow 1$ we have $l(\beta)\rightarrow\infty$. As discussed in Example \ref{ex:Kerr}, as $\beta\rightarrow 1$, in the pointed Gromov-Hausdorff sense the Kerr space splits into the union of a Taub-NUT space and an anti-Taub-NUT space.  The case of $\beta\to-1$ is similar.

\subsubsection{The \texorpdfstring{$\ALFm$}{} case}\label{sss:negative taub bolt}

Denote by  $\mR_{\dTN^{-}, l}$ the smooth enhanced rod structure of Asymptotic Type $\ALF^-$, with exactly two turning points $z_1=0$ and $z_2=l>0$ and with the enhancement $\Lambda$ generated by primitive rod vectors 
$$\bbv_0=\bv_{\frac12}, \quad \bbv_1=(2, 0)^{\top}, \quad \bbv_2=\bv_{-\frac12}.$$ 

Denote by  $\mR_{\TB^-, l}$ the smooth enhanced rod structure of Asymptotic Type $\ALF^-$, with exactly two turning points $z_1=0$ and $z_2=l>0$, with  the enhancement $\Lambda$ generated by primitive rod vectors 
$$\bbv_0=\bv_{\frac12}, \quad \bbv_1= (0, 1)^{\top}, \quad  \bbv_2=\bv_{-\frac12}.$$ 
See Figure \ref{Figure:rod structure for dTN and TB} for an illustration of  the rod structures $\mR_{\dTN^-,l}$ and $\mR_{\TB^-,l}$. It is easy to see that $d(\mR_{\dTN^-, l})=2$ and $d(\mR_{\TB^{-}, l})=1$.

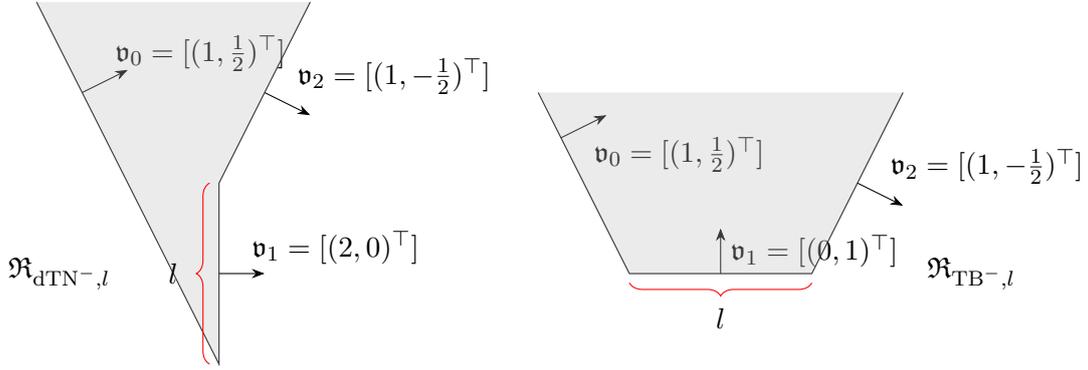
\begin{figure}[ht]
    \begin{tikzpicture}[scale=0.6]
        \draw (-4,4)--(-6,0)--(-6,-4)--(-10,4);

        \draw[-Stealth] (-5,2)--node [above right]{$\mathfrak{v}_2=[(1,-\frac{1}{2})^\top]$} (-4,1.5);
        \draw[-Stealth] (-6,-2)--node [above right] {$\mv_1=[(2,0)^\top]$} (-5,-2);
        \draw[-Stealth] (-9,2)--node [above right]{$\mathfrak{v}_0=[(1,\frac{1}{2})^\top]$} (-8,2.5);

        \filldraw[color=gray!50,fill=gray!50,opacity=0.3](-4,4)--(-6,0)--(-6,-4)--(-10,4);

        \node at (-9.5,-2) {$\mR_{\dTN^-,l}$};

        \node at (10.5,-2) {$\mR_{\TB^-,l}$};
        
        \draw [draw = red,decorate,decoration={brace,amplitude=5pt,mirror,raise=4ex}]
        (-5.1,0) -- (-5.1,-4) node[midway,yshift=-3em]{};
        \node at (-7,-2) {$l$};

        \draw [draw = red,decorate,decoration={brace,amplitude=5pt,mirror,raise=4ex}]
        (3,-1.1) -- (7,-1.1) node[midway,yshift=-3em]{};
        \node at (5,-3) {$l$};

        \draw (1,2)--(3,-2) --(7,-2)--(9,2);

        \draw[-Stealth] (8,0)--node [above right]{$\mathfrak{v}_2=[(1,-\frac{1}{2})^\top]$} (9,-0.5);
        \draw[-Stealth] (5,-2)--node [right] {$\mv_1=[(0,1)^\top]$} (5,-1);
        \draw[-Stealth] (1.5,1)--node [below right]{$\mathfrak{v}_0=[(1,\frac{1}{2})^\top]$} (2.5,1.5);

        \filldraw[color=gray!50,fill=gray!50,opacity=0.3](1,2)--(3,-2) --(7,-2)--(9,2);
    \end{tikzpicture}
	\caption{The rod structures $\mR_{\dTN^-,l}$ and $\mR_{\TB^-,l}$.}
	 \label{Figure:rod structure for dTN and TB}   
    \end{figure}

Denote by $(\Phi_{\dTN^{-}, l}, \nu_{\dTN^{-}})$ and $(\Phi_{\TB^-, l}, \nu_{\TB^-})$ the corresponding augmented harmonic maps tamed by $\mR_{\dTN^{-}, l}$ and $\mR_{\TB^-, l}$ respectively. Then in each case the cone angles along the rods $\mv_0$ and $\mv_2$ are fixed to be $2\pi$. We denote  
$$\Theta_{\TB^-}(l)=\vartheta_{\Phi_{\TB^-, l}, \nu_{\TB^-, l}}(\mI_1)$$
and 
$$\Theta_{\dTN^{-}}(l)=\vartheta_{\Phi_{\dTN^-, l}, \nu_{\dTN^-, l}}(\mI_1).$$
By \eqref{e:angle continuity}, $\Theta_{\TB^-}$ and $\Theta_{\dTN^{-}}$ are  continuous functions of  $l\in (0, \infty)$. 

\begin{proposition}
	We have $$\lim_{l\rightarrow 0}\Theta_{\TB^-}(l)=\lim_{l\rightarrow 0}\Theta_{\dTN^{-}}(l)=4\pi.$$
\end{proposition}
\begin{proof}
	As $l\rightarrow 0$, both family of rod structures converges to the same limit rod structure $\mR_\infty=\mR^{\TNm}$.  By Proposition \ref{t:bounded distance to model map} in both cases we know there is a unique  bubble limit $\Phi_{\infty, 1}$ which is tamed by a rod structure of  Asymptotic Type  $\Ae$. By Proposition \ref{p:AE rigidity} and \eqref{e:angle equality before limit} we get that 
    $$\lim_{l\rightarrow 0}\Theta_{\TB^-}(l)=\lim_{l\rightarrow 0}\Theta_{\dTN^{-}}(l)=4\pi.$$
\end{proof}

\begin{proposition}
  We have $$\lim_{l\rightarrow \infty}\Theta_{\TB^-}(l)=0,$$
  and $$\lim_{l\rightarrow \infty}\Theta_{\dTN^{-}}(l)= 4\pi.$$
\end{proposition}
\begin{proof}
  This is an immediate consequence of Proposition \ref{p:cone angle TN rod length go to infinity}.
\end{proof}

As a consequence there must exist an $l>0$ such that $(\Phi_{\TB^-, l}, \nu_{\TB^-, l})$ defines a gravitational instanton of Asymptotic Type $\ALF^-$. Indeed, it is of Type $\II$ and by Proposition \ref{c:classification of toric GI}, we know there is exactly one $l$ such that $\Theta_{\TB^-}(l)=2\pi$, which isometrically is the \emph{Taub-Bolt metric}. Indeed, the Taub-Bolt metric (up to scaling) can be explicitly written as 
\begin{equation}\label{eq:explicit Taub-bolt}
    f(r)(d\phi_2+\frac{1}{2}\cos\theta d\phi_1)^2+\frac{dr^2}{f(r)}+(r^2-\frac{1}{16})(d\theta^2+\sin^2\theta d\phi_1^2)
\end{equation}
with 
$$f(r)=\frac{r^2+\frac{1}{16}-\frac{5}{8}r}{r^2-\frac{1}{16}},\quad r\geq \frac{1}{2}\text{ and }\theta\in[0,\pi].$$
The enhancement is generated by $(1,\frac{1}{2})^\top,(1,-\frac{1}{2})^\top$, with
$$\rho=\sqrt{r^2-\frac{5}{8}r+\frac{1}{16}}\sin\theta,\quad z=(r-\frac{5}{16})\cos\theta+\frac{3}{16}.$$
One can check that the rod structure is exactly given by the $\mR_{\TB^-,l}$ in Figure \ref{Figure:rod structure for dTN and TB} with here $l=\frac{3}{8}$. 
We shall refer to the Taub-Bolt metric with the orientation $d\rho\wedge dz \wedge d\phi_1\wedge d\phi_2$ as the \emph{anti-Taub-Bolt space} (so with the opposite orientation there is the \emph{Taub-Bolt space}). Note that the natural $S^2$  has self-intersection $-1$.

Note that similar discussion in the $\dTN^-$ case is not sufficient to yield a gravitational instanton. This is not surprising, since if we consider the opposite rod structure which we denote by $\mR_{\dTN^+_l}$, then one can see that $d(\mR_{\dTN^+_l})=0$ and the associated harmonic map must be of Type $\I$ in the sense of Definition \ref{def:Type IIIIII} by Theorem \ref{thm:equivalence of degree and type}. 
As discussed in Remark \ref{rem:angle invariant in Type I}, in the Type $\I$ case the cone angle functions do not depend on the lengths of the rods. In particular, going back to $\mR_{\dTN^-}$, this implies that for all $l>0$, $\Theta_{\dTN^{-}}=4\pi$. One can take an obvious $\dZ_2$ quotient to obtain a gravitational instanton of Asymptotic Type $\ALF^-$. The latter is exactly a negatively oriented double Taub-NUT space (by Proposition  \ref{c:classification of toric GI}).

\subsubsection{The \texorpdfstring{$\ALFp$}{} case}
We can also recover the \emph{Taub-Bolt space}. Recall the metric is explicitly given by \eqref{eq:explicit Taub-bolt}. This time we consider the rod structure $\mR_{\TB^+,l}$ with turning points $z_1=0,z_2=l>0$ and with the enhancement $\Lambda$ generated by primitive rod vectors
$$\bbv_0=(\frac{1}{2},1)^\top,\quad \bbv_1=(1,0)^\top,\quad \bbv_2=(\frac{1}{2},-1)^\top.$$
The Asymptotic Type is taken to be $\ALFp$. See Figure \ref{Figure:positive taub-bolt}.
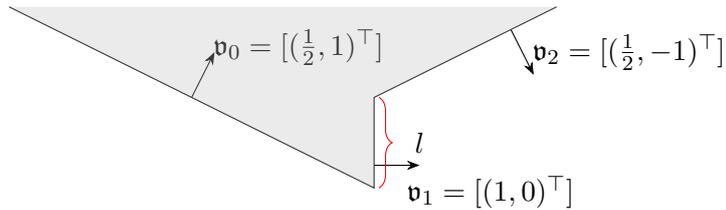
\begin{figure}[ht]
    \begin{tikzpicture}[scale=0.6]
        \draw (-8,4)--(0,0)--(0,2)--(4,4);

        \draw[-Stealth] (-4,2)--node [above right]{$\mathfrak{v}_0=[(\frac{1}{2},1)^\top]$} (-3.5,3);
        \draw[-Stealth] (3,3.5)--node [right] {$\mv_2=[(\frac{1}{2},-1)^\top]$} (3.5,2.5);
        \draw[-Stealth] (0,0.5)--node [below right]{$\mathfrak{v}_1=[(1,0)^\top]$} (1,0.5);

        \filldraw[color=gray!50,fill=gray!50,opacity=0.3](-8,4)--(0,0)--(0,2)--(4,4);

        \draw [draw = red,decorate,decoration={brace,amplitude=5pt,mirror,raise=4ex}]
        (-1,0) -- (-1,2) node[midway,yshift=-3em]{};
        \node at (1,1) {$l$};
    \end{tikzpicture}
    \caption{The rod structure $\mR_{\TB^+,l}$.}
    \label{Figure:positive taub-bolt}
\end{figure}
As in Section \ref{sss:negative taub bolt}, one can argue that the angle along the rod $\mI_1$ satisfies that $\Theta(l)\to0$ as $l\to\infty$ and $\Theta(l)\to4\pi$ as $l\to0$, hence there is an $l$ such that the harmonic map strongly tamed by $\mR_{\TB^+,l}$ with normalized augmentation and the enhancement $\Lambda$ gives rise to a  toric gravitational instanton of Asymptotic Type $\ALFp$. Notice that $d(\mR_{\TB^+, l})=1$ so by Theorem \ref{thm:equivalence of degree and type} the gravitational instanton is of Type $\II$ in the sense of Definition \ref{def:Type IIIIII}. Again  by Proposition \ref{c:classification of toric GI} this is isometric to the Taub-Bolt metric, explicitly given by \eqref{eq:explicit Taub-bolt}. We shall refer to the Taub-Bolt metric with this orientation as the \emph{Taub-Bolt space}.

\subsection{The general case}\label{ss:general case}

Fix $\beta\in (0, 1)$. Given an integer $n\geq 1$, and $\bl=(l_1, \cdots, l_n)\in \R^n_+$,  we consider the rod structure  $\mathfrak{R}_{\bl}=(\{z_j\}_{j=1}^{n+1},\{\mathfrak{v}_j\}_{j=0}^{n+1}, \AFa)$, where
\begin{itemize}
\item $z_1=0$, $z_{j+1}-z_{j}=l_j$ for $j=1, \cdots, n$;
    \item $\mathfrak{v}_0=\mathfrak{v}_{n+1}=[(1,0)^{\top}]$;
    \item for $j\neq0,n+1$, $\mathfrak{v}_{j}=[(0,1)^{\top}]$ when $j$ is odd, and $\mathfrak{v}_{j}=[(-1,1)^{\top}]$ when $j$ is even.
\end{itemize}
See Figure \ref{Figure:AF rod structure} for an illustration of $\mathfrak{R}_{\bl}$ with  $5$ turning points.
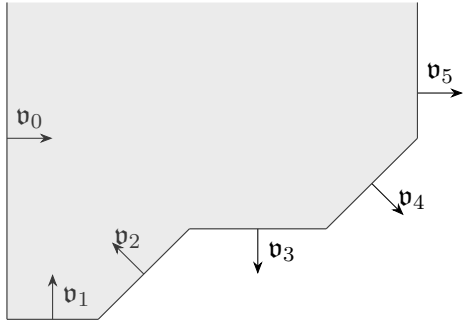
\begin{figure}[ht]
    \begin{tikzpicture}[scale=0.6]
        \draw (0,7)--(0,0);
        \draw (0,0)--(2,0);
        \draw (2,0)--(4,2);
        \draw (4,2)--(7,2);
        \draw (7,2)--(9,4);
        \draw (9,4)--(9,7);

        \draw[-Stealth] (0,4)--node [above]{$\mathfrak{v}_0$} (1,4);
        \draw[-Stealth] (1,0)--node [right]{$\mathfrak{v}_1$} (1,1);
                \draw[-Stealth] (3,1)--node [above]{$\mathfrak{v}_2$} (2.3,1.7);
        \draw[-Stealth] (5.5,2)--node [right]{$\mathfrak{v}_3$} (5.5,1);
                \draw[-Stealth] (8,3)--node [right]{$\mathfrak{v}_4$} (8.7, 2.3);
        \draw[-Stealth] (9,5)--node [above]{$\mathfrak{v}_5$} (10, 5);

        \filldraw[color=gray!50,fill=gray!50,opacity=0.3](0,7)--(0,0) --(2,0)--(4,2)--(7,2)--(9,4)--(9,7);
        
    \end{tikzpicture}
	\caption{A rod structure with 5 turning points.}
	 \label{Figure:AF rod structure}   
\end{figure}

One can compute that 
\begin{equation}\label{eqn:degree formula}
    d(\mathfrak{R}_{\bl})=\lfloor\frac{n+1}{2}\rfloor.
\end{equation}
We will always take the smooth enhancement $\Lambda$ to be the standard lattice $\dZ^2$. One can check that it gives rise to the  toric 4-manifold diffeomorphic to $X_{n}$ (see \eqref{eqn:definition of Xn}).  We let $(\Phi_{\bl, \beta}, \nu_{\bl, \beta})$ be the augmented harmonic map which is strongly tamed by the enhanced rod structure $(\mR_{\bl}, \Lambda)$ and such that $\nu_{\bl, \beta}$ is normalized.

Theorem \ref{t:main} is a consequence of the following 
\begin{theorem}\label{t:adjusting angle}
There exists an $\bl\in \R^n_+$ such that $(\mR_{\bl}, \Lambda, \Phi_{\bl, \beta},\nu_{\bl, \beta})$ defines a toric $\AF_\beta$ gravitational instanton.
\end{theorem}
\begin{remark}
    From our proof it is not clear whether $\bl$ depends uniquely or continuously on $\beta$. See Section \ref{sec:discussion} for more discussion.
\end{remark}
\begin{remark}
When $n=1$ or $n=2$, we have $d(\mR_l)=1$, and the gravitational instantons constructed above coincide with the Kerr spaces and the Chen-Teo spaces; when $n\geq 3$ we have $d(\mR_l)>1$, and the gravitational instantons constructed above  is not locally Hermitian. These follow from the general result in Theorem \ref{thm:degree equivalence Type II}, using \eqref{eqn:degree formula} and the classification of Hermitian toric AF gravitational instantons by Biquard-Gauduchon \cite{BG}. 
\end{remark}
Denote by $\vartheta_j(\bl)=\vartheta_{\Phi_{\bl, \beta}, \nu_{\bl, \beta}}(\mI_j)$ the cone angle along the rod $\mI_j$. Since $\nu_{\bl, \beta}$ is normalized we know $\vartheta_0=\vartheta_{n+1}=2\pi$.  We define a map 
\begin{equation}
    \mathcal{A}:\mathbb{R}^n_+\to\mathbb{R}^n_+;(l_1,\ldots,l_n)\mapsto (\vartheta_1(\bl),\ldots,\vartheta_n(\bl)).
    \end{equation}
It suffices to show that the point ${\btheta}_0:=(2\pi,\ldots,2\pi)$ is contained in the image of $\mathcal{A}$. We introduce an auxiliary map 
    $$\mathcal{A}_{0}:\mathbb{R}^n_+\to\mathbb{R}^n_+; (l_1,\ldots,l_n)\mapsto(e^{-l_1}(2\pi+1),\ldots,e^{-l_n}(2\pi+1)).$$
    Define a subset of $\R^n_+$ by
    $$\mathcal S:=\{\bl=(l_1, \cdots, l_n)\in \R_n^+\mid t\mathcal{A}(\bl)+(1-t)\mathcal{A}_0(\bl)\neq \btheta_0, \ \ \forall \ \  t\in [0, 1]\}.$$
For $N>\epsilon>0$, we denote the cube $$\mathcal{Q}_{\epsilon,N} =[\epsilon,N]^n\subset \R_+^n.$$

\begin{proposition} \label{p:boundary avoid}There exist $\epsilon$ and $N$ such that  $$\p\mathcal Q_{\epsilon, N}\subset \mathcal S.$$
\end{proposition}

\begin{proof}[Proof of Theorem \ref{t:adjusting angle}]
   We prove by contradiction. Suppose $\btheta_0$ is not in the image of $\mathcal A$. Denote by $\pi$ the radial projection map from $\R^n_+\setminus \{\btheta_0\}$ to the sphere $\mathbb S^{n-1}$.  Let $\epsilon$ and $N$ be given in Proposition \ref{p:boundary avoid}. Then the map $$\pi\circ \mathcal A|_{\p \mathcal Q_{\epsilon, N}}: {\p \mathcal Q_{\epsilon, N}}\rightarrow \mathbb S^{n-1}$$ is homotopic to $\pi\circ \mathcal A_0|_{\p\mathcal Q_{\epsilon, N}}$. It is clear that $\pi\circ \mathcal A_0|_{\p\mathcal Q_{\epsilon, N}}$ has a nonzero mod-2 degree but $\pi\circ \mathcal A|_{\p\mathcal Q_{\epsilon, N}}$ is contractible. Contradiction. 
\end{proof}   

\

Now we prove Proposition \ref{p:boundary avoid}.

\begin{lemma}\label{lem:AF example interpolatable 1}
    There exists an $N_0>0$, such that 
    $$\mathcal{C}_{N_0}:=\{(l_1,\ldots,l_n)\mid l_j>N_0 \ \text{for some $j$}\}\subset \mathcal S.$$ 

\end{lemma}
\begin{proof}
    Otherwise, there is a sequence $\bl_k=(l_{k, 1}, \cdots, l_{k, n})\notin \mathcal S$ and $j_0\in \{1, \cdots, n\}$ with $$l_{k, j_0}=\max_{j=1,\ldots,n}l_{k,j}\to\infty$$ as $k\rightarrow\infty$. Since $$\lim_{k\rightarrow\infty}e^{-l_{k, j_0}}(2\pi+1)= 0,$$ it follows that
    $\vartheta_{k, j_0}(\bl_k)\geq2\pi$ for $k$ large. However, according to Proposition \ref{p:cone angle AF rod length go to infinity},  we must have $$\lim_{k\rightarrow\infty}\vartheta_{j_0}(\bl_k)\leq 2\pi\max(\beta, 1-\beta)<2\pi,$$which is a contradiction.
\end{proof}

\begin{lemma}
   There exits an $\epsilon_0>0$ small such that 
    $$\mathcal{C}_{\epsilon_0, N_0}:=\{(l_1,\ldots,l_n)\mid\text{$l_j<2N_0$ for all $j$, and  $l_j<\epsilon_0$ for some $j$}\}\subset\mathcal S.$$
\end{lemma}
\begin{proof}
    Otherwise, there exists a sequence $\bl_k\notin \mathcal S$ and $j_0$ with $\sup_{j}l_{k, j}<2N_0$ and $l_{k,j_0}\rightarrow 0$ as $k\rightarrow\infty$. Passing to a suitable subsequence we may obtain a bubble tree convergence of the rod structures. We pick a first  bubble limit $(\Phi_{\infty, p},\mR_{\infty, p}')$ as shown in Section \ref{ss:uniform control of tame harmonic maps}. There are two cases.

    \
    
    {\bf Case (1)}: $\mR_{\infty, p}'$ is of  Asymptotic Type  $\AF_0$. If $\mR_{\infty, p}'$ has more than one rod, then we must see a finite rod  $\mI$ of $\mR_{\infty, p}'$ whose rod vector is different from the infinite rod vector. Suppose $\mI$ contains the rescaled  limit of the rod $\mI_{j_{\flat}}$. 
     By Proposition \ref{p:cone angle AF rod length go to zero} we know the cone angle $\vartheta_{j_\flat}(\bl_k)\rightarrow\infty$. Since $l_{k,j_\flat}\to0$,  for $k$ large we have  $$e^{-l_{k, j_\flat}}(2\pi+1)>2\pi.$$ So $\bl_k\in \mathcal S$. This is a contradiction. 

If  $\mR_{\infty, p}'$ has only one rod, i.e., $\mR_{\infty, p}'$ is trivial, then we can pass to smaller scales and go to a first non-trivial bubble limit, which must be of Asymptotic Type $\AF_0$. It follows from Proposition \ref{p:deep scale angle goes to infinity} that  there is a $j_\flat$ such that $\vartheta_{j_\flat}(\bl_k)\rightarrow\infty$, so as  before we obtain a contradiction.  

\

{\bf Case (2)}: $\mR_{\infty, p}'$ is of  Asymptotic Type  $\Ae$.  If $\mR_{\infty, p}'$ has no finite rods, i.e., $\mR_{\infty, p}'$ is trivial, then we can pass to smaller scales and go to a first non-trivial bubble limit that is of Asymptotic Type $\AF_0$. Then as in {\bf Case (1)}  we obtain a contradiction.  

Now we assume $\mR_{\infty, p}'$ has a finite rod.  Suppose the $j_a$-th rod $\mathfrak{I}_{j_a}$ and the $j_b$-th rod $\mathfrak{I}_{j_b}$ converge to the infinite rods of $\mR$, where $j_a<j_b$. Then $\mathfrak{I}_{j_a}$ and $\mathfrak{I}_{j_b}$ do not vanish in the original scale. So passing to the $\Ae$ bubble limit and applying the angle  comparison (Proposition \ref{p:angle comparision before limit}) at $z_{j_a+1}$, it follows that we have
    $$\lim_{k\rightarrow\infty}\vartheta_{j_a+1}(\bl_k)\geq \lim_{k\rightarrow\infty}\vartheta_{j_b}(\bl_k).$$
However, since both $\mathfrak{I}_{j_a}$ and $\mathfrak{I}_{j_b}$ survive in the original scale, we have $$2N_0>l_{k, j_a},l_{k, j_b}>\delta$$ for a small constant $\delta$. As $l_{k, j_a+1}\to0$,  for $k$ large we have 
    $$e^{-l_{k,j_a+1}}(2\pi+1)>e^{-l_{k, j_b}}(2\pi+1),$$ 
Hence for $t\in [0, 1]$ and $k$ large enough,  we have
    $$(1-t)\vartheta_{j_a+1}(\bl_k)+te^{-l_{k,j_a+1}}(2\pi+1)>(1-t)\vartheta_{j_b}(\bl_k)+te^{-l_{k,j_b}}(2\pi+1),$$
  
    These imply that $\bl_k\in \mathcal S$. Contradiction.
\end{proof}

Now choose $\epsilon<\epsilon_0$ and $N\in (N_0, 2N_0)$, then we have  $$\partial\mathcal{Q}_{\epsilon,N}\subset \mathcal{C}_{N_0}\cup \mathcal{C}_{\epsilon_0, N_0}\subset \mathcal S.$$ This completes the proof of Proposition \ref{p:boundary avoid}.

\section{Rod structures of low degrees}\label{sec:degree 0 and 1}
\subsection{Ricci-flat metrics with special geometry}\label{sec:7.1}

As in \cite{Li1,Li2}, for convenience of discussion we divide 4-dimensional oriented Ricci-flat metrics into three classes.

\begin{enumerate}
	\item Type $\I$: $W^+\equiv0$. This is equivalent to saying that $g$ is locally K\"ahler (hence locally hyperk\"ahler). 
	\item Type $\II$: $W^+$ has exactly one eigenvalue of multiplicity one everywhere. This is equivalent to saying that $g$ is locally conformally K\"ahler but non-K\"ahler.  
	\item Type $\III$: $W^+$ has three distinct eigenvalues generically.
\end{enumerate}
In the first two cases the Ricci-flat metrics have special geometry. In particular, they are closely related to complex geometry. The latter fact leads to many deep results for Type $\I$ and $\II$ gravitational instantons.

Given a harmonic map $\Phi$ in our setting, we can similarly define its Type. For this purpose, we need to pick an augmentation $\nu$ and consider  the Ricci-flat metric $$g=e^{2\nu}(d\rho^2+dz^2)+\rho\Phi,$$ endowed  with the canonical orientation  given by $d\rho\wedge dz\wedge d\phi_1\wedge d\phi_2$.

\begin{definition}\label{def:Type IIIIII}
    We say a harmonic map $\Phi$ is Type $\I/\II/\III$ if the metric $g$ is Type $\I/\II/\III$ with respect to the canonical orientation. 
\end{definition}
It is easy to see that the Type is independent of the choice of the augmentation $\nu$.
Notice that for $P\in SL_+(2;\mathbb R)$, we know $P. \Phi$ is of the same Type as $\Phi$. But if $P\in SL_-(2;\mathbb R)$, then in general $P. \Phi$ may not be of the same Type as $\Phi$. Given a sequence of harmonic maps $\Phi^k$ that converge to a limit harmonic map $\Phi^\infty$, it is easy to see that if $\Phi^k$ are all Type $\I$ then $\Phi^\infty$ is also Type $\I$, and if $\Phi^k$ are all Type $\II$ then $\Phi^\infty$ can only be Type $\II$ or Type $\I$.

In the case of Type $\I$ and Type $\II$, for integrable toric Ricci-flat metrics, the axisymmetric harmonic maps further simplify to axisymmetric harmonic functions and  one can carry out explicit studies. We will explain this below.

\subsubsection{Type $\I$ case} 
\label{subsubsec:Type I case}

In this case, we will explain that the axisymmetric harmonic map reduces to an axisymmetric harmonic function, using the \emph{Gibbons-Hawking ansatz}.

Consider the local situation where we have a hyperk\"ahler metric $g$. Denote by $\mathbb S$ the 2-sphere of compatible parallel orthogonal complex structures. Notice that $\mathbb S$ is endowed with a natural round metric (it naturally sits inside $\R^3$ as the unit sphere). An isometry $\psi$ of $g$ induces an action on $\mathbb S$ via $\psi. I=\psi^*I$. One can see the latter is an isometric action.  In this way a Killing field $X$ for $g$ generates a Killing vector field $V_X$ on $\mathbb S$.  So, if nonzero, then $V_X$ is an infinitesimal rotation. We may assume it vanishes at one complex structure (say $I$) and rotates the plane generated by $J$ and $K$. 

Next we assume further that $g$ is toric and integrable.  Then we have two commuting Killing fields $X_1$ and $X_2$, so $V_{X_1}$ and $V_{X_2}$ also commute, thus they must be proportional. Up to a linear combination we may assume that $V_{X_1}=0$. This implies that $X_1$ is tri-holomorphic, i.e., its flow preserves all the hyperk\"ahler complex structures. Hence $g$ is locally given by the Gibbons-Hawking construction applied to a positive harmonic function $V$ on some domain in $\R^3$. It is easy to see that, up to an additive constant, $V^{-1}$ is exactly the Ernst potential $\mathcal{E}_1$. By the integrability assumption we have $\Theta_1=\Theta_2=0$, it follows that $V$ is invariant under $X_2$. As we are interested in gravitational instantons where the toric action has a fixed point, the Killing field $X_2$ globally has zero, and it follows that $X_2$ must be a rotation in $\mathbb R^3$ in the Gibbons-Hawking ansatz. Hence $V$ is locally an axisymmetric harmonic function.

To connect to  axisymmetric harmonic maps, we write the flat metric on $\R^3$ as $d\rho^2+dz^2+\rho^2d\phi_2^2$. For an axisymmetric harmonic function $V$ such that $\frac{1}{2\pi}[*dV]$ defines an integral cohomology class,  we can write  down the Gibbons-Hawking ansatz as 
\begin{equation}\label{eq:Type I metric}
    g=V(d\rho^2+dz^2+\rho^2d\phi_2^2)+V^{-1}(d\phi_1+Fd\phi_2)^2
\end{equation}
with $$dF=\rho V_\rho dz-\rho V_zd\rho.$$ The latter admits a solution locally since $V$ is axisymmetric harmonic. Here \eqref{eq:Type I metric} is already in the Weyl-Papapetrou coordinates. The corresponding augmented harmonic map is  explicitly given by
\begin{equation}\label{eqn:GH harmonic map}
    \Phi=\frac{1}{\rho}\left(\begin{matrix}
        V^{-1} & V^{-1}F \\ V^{-1}F & \rho^2 V+V^{-1}F^2
    \end{matrix}\right),
\end{equation}
\begin{equation}\label{eqn:GH augmentation}
\nu=\frac{1}{2}\log V.
\end{equation}
Particularly, when we take $V=1+\frac{1}{2r}$ and $F=\frac{z}{2r}$ we get the Taub-NUT space, discussed in Example \ref{example:positive Taub-NUT}.

More generally,  given $k+1$ distinct points $p_1, \cdots, p_{k+1}\in \Gamma\subset \mathbb R^3$, if we take 
$$V=\sum_{\alpha=1}^{k+1} \frac{1}{2|x-p_\alpha|},$$
then we obtain a  \emph{multi-Eguchi-Hanson space}, which is an $\ALE$-$A_k$ hyperk\"ahler gravitational instanton. If we instead take for $T>0$
$$V=T+\sum_{\alpha=1}^{k+1} \frac{1}{2|x-p_\alpha|},$$
then we obtain a \emph{multi-Taub-NUT space}, which is (up to scaling according to our Definition \ref{def:Type of gravitational instantons}) an $\ALF$-$A_k$ hyperk\"ahler gravitational instanton. In both cases, the underlying 4-manifold is diffeomorphic to the minimal resolution of $\C^2/\mathbb Z_k$.

\subsubsection{Type $\II$ case}
\label{subsubsec:Type II case}
In this case we again reduce the axisymmetric harmonic map  to an axisymmetric harmonic function, using instead the \emph{LeBrun-Tod ansatz}. 

Given a Type $\II$ Ricci-flat metric $g$,  we define a conformal metric 
$$g_K=(2\sqrt{6}|W^+|_g)^{2/3}g.$$
Then by \cite{Der} $g_K$ is a locally K\"ahler metric, which is moreover \emph{extremal} in the sense of Calabi. That is, the scalar curvature function $|s_{g_K}|=(2\sqrt{6}|W^+|_g)^{1/3}$ is not a constant, and  
    $$X_1:=J\nabla_{g_K}s_{g_K}$$ 
is a Hamiltonian Killing field. The LeBrun-Tod ansatz for K\"ahler metrics with a Hamiltonian Killing field now gives that locally,
\begin{equation}\label{eq:Type II metric}
    g=V(d\xi^2+e^u(dx_2^2+dx_3^2))+V^{-1}\eta^2,
\end{equation}
where $(u, V, \eta)$ satisfy the following system of equations
\begin{equation}\label{eq:Toda}
    e^u_{\xi\xi}+u_{x_2x_2}+u_{x_3x_3}=0,
\end{equation}
\begin{equation}\label{eq:Type II V}
    V=-12\xi+6\xi^2u_{\xi},
\end{equation}
\begin{equation}\label{eq:Type II deta}
    d\eta=-\xi^2(\xi^{-2}e^uV)_\xi dx_2dx_3-V_{x_2}dx_3d\xi-V_{x_3}d\xi dx_2.
\end{equation}
Here $x_2+ix_3$ is a local holomorphic coordinate over the K\"ahler quotient and $\xi$ is the moment map for $X_1$. See Equations (28)-(31) in \cite{BG}. To be consistent with \cite{BG}, one can instead work with the scaled Killing field $\sqrt{6k}X_1$, where the only difference for \eqref{eq:Type II metric}-\eqref{eq:Type II deta} is \eqref{eq:Type II V} becomes 
\begin{equation}\label{eq:Type II scaled V}
    V=\frac{1}{k}(-2\xi+6\xi^2u_\xi).
\end{equation}
Later we will see that $k$ has the significance of ensuring $V$ tends to $1$ at infinity in the global setting.

When we have an additional Killing field $X_2$, it preserves $|W^+|_g$ hence the conformal factor $s_{g_K}$, so $X_2$ is also a Killing field for $g_K$. The complex structure $J$ is determined by the non-repeated eigen-two-form of $W^+$ and $g_K$, hence is also preserved by $X_2$, which implies that $X_2$ is holomorphic. Since $X_1=J\nabla_{g_K}s_{g_{K}}$, $X_1$ and $X_2$ commute automatically. Therefore, by rotating  the complex coordinate $x_2+ix_3$ by some $e^{i\theta}$ one can arrange that 
$$X_1=\partial_t, X_2=\partial_{x_3},  \eta=dt-Fdx_3$$ 
for functions $u, V, F$ independent of $x_3$. The Equation \eqref{eq:Toda} becomes
\begin{equation}\label{eqn:reduced Toda}
e^u_{\xi\xi}+u_{x_2x_2}=0.
\end{equation}Performing Ward's B\"acklund transformation \cite{Ward} one can show that the reduced equation is equivalent to the axisymmetric Laplace's equation. More precisely, given an axisymmetric harmonic function $U$, we define $$\xi=\frac12\rho U_{\rho}$$ and $$x_2=-\frac12U_z. $$Restricting to the domain where $\frac{\rho}{4}(U_{zz}^2+U_{\rho z}^2)>0$,  $(\xi,x_2)$ gives rise to local coordinates and $u=\log\rho^2$ is a solution to \eqref{eqn:reduced Toda}. Conversely, any solution to \eqref{eqn:reduced Toda} locally arises in this way. See Proposition 3.5 in \cite{BG} for details. 

Combining the previous Ward ansatz, using the second Killing field $X_2$, we can further reduce \eqref{eq:Type II metric}, \eqref{eq:Toda}, \eqref{eq:Type II scaled V}, and \eqref{eq:Type II deta}  to the following
\begin{equation}\label{eq:Type II metric harmonic map}
    g=e^{2\nu}(d\rho^2+dz^2)+V\rho^2dx_3^2+V^{-1}(dt-Fdx_3)^2,
\end{equation}
\begin{equation}\label{eq:Type II nu harmonic}
    e^{2\nu}=\frac14V\rho^2(U_{\rho z}^2+U_{zz}^2),
\end{equation}
\begin{equation}\label{eq:Type II V harmonic}
    V=-\frac{1}{k}\left(\rho U_\rho+\frac{U_\rho^2U_{zz}}{U_{\rho z}^2+U_{zz}^2}\right),
\end{equation}
\begin{equation}\label{eq:Type II F harmonic}
    F=-\frac{1}{k}\left(-\frac{\rho U_\rho^2 U_{\rho z}}{U_{\rho z}^2+U_{zz}^2}+\rho^2U_z+2H\right).
\end{equation}
Here the function $H$ is the harmonic conjugate of $U$ and the formulae make sense over the domain where $V>0$ and $U_{\rho z}^2+U_{zz}^2>0$. This is Corollary 3.6 in \cite{BG}.  The metric $g$ in \eqref{eq:Type II metric harmonic map} is already in the Weyl-Papapetrou coordinates and one can easily write down the corresponding harmonic map. With $\phi_1=x_3$ and $\phi_2=-t$, the harmonic map is explicitly given by
\begin{equation}\label{eq:Type II harmonic map}
    \Phi=\frac{1}{\rho}\left(\begin{matrix}
        V\rho^2+V^{-1}F^2 & V^{-1}F\\ V^{-1}F & V^{-1}
    \end{matrix}\right).
\end{equation}
Taking $U=\frac12\log\rho^2+2r-z\log\frac{r+z}{r-z}$ with $k=1$, one gets $V=1+\frac{1}{2r}$ and $F=\frac{z}{2r}$, and we obtain the anti-Taub-NUT space, discussed in Example \ref{example:negative Taub-NUT}. As explained in \cite{BG}, such construction can give rise to the Kerr, Taub-Bolt, anti-Taub-Bolt, and Chen-Teo spaces.

In the next subsections  we will deal with the global situation. 

\begin{definition}
    We define the Type of a rod structure $\mR$ to be the Type of the unique harmonic map which is strongly tamed by $\mR$.
\end{definition}
\begin{remark}This notion of \emph{Type} should not be confused with the notion of \emph{Asymptotic Type} introduced in Definition \ref{def:rod structures}.
\end{remark}Clearly from the definition the Type of $\mR$
 is only determined after solving the global harmonic map equation, which is in general difficult to determine. Nevertheless, we will prove the following  theorem in the next two subsections. It gives a complete characterization of the Type in terms of the degree (see Definition \ref{def:degrees}).
\begin{theorem}\label{thm:equivalence of degree and type}
Given a rod structure $\mR$, we have
    \begin{itemize}
        \item $\mR$ is Type $\I$ if and only if it has degree $d(\mathfrak{R})=0$.
        \item $\mR$ is Type $\II$ if and only if it has degree $d(\mathfrak{R})=1$. 
         \item $\mR$ is Type $\III$ if and only if it has degree $d(\mathfrak{R})\geq 2$.
    \end{itemize}
\end{theorem}
Together with the discussion at the end of Section \ref{ss:defintion of rod structures} we see that when the number of turning of points is at most 2, either $\mR$ or $\overline {\mR}$ is of Type $\I$ or Type $\II$, i.e., the associated harmonic map has special geometry.

Since Type $\I$ and Type $\II$ harmonic maps are explicitly understood using axisymmetric harmonic functions by the discussion above, Theorem \ref{thm:equivalence of degree and type} gives an interesting PDE result, i.e., an explicit classification of axisymmetric harmonic maps when the rod structure has degree at most $1$, in terms of axisymmetric harmonic functions.

Together with the classification of hyperk\"ahler and Hermitian gravitational instantons, we have 
\begin{proposition}
    \label{c:classification of toric GI}
    Let $(M, g)$ be a simply-connected toric gravitational instanton of Asymptotic Type $\ALF/\AFa$, such that the associated rod structure has degree at most $1$.  Then it is isometric to one in the following list
    \begin{enumerate}
    \item (Hyperk\"ahler case):
   \begin{itemize}

       \item A multi-Taub-NUT space diffeomorphic to  the crepant resolution of $\C^2/\mathbb Z_k$.
   \end{itemize}
    \item (Hermitian, non-K\"ahler case):
    \begin{itemize}
             \item An anti-Taub-NUT space diffeomorphic to $\R^4$;
        \item A Kerr space diffeomorphic to $S^2\times \R^2$ (including the Schwarzschild space);
   
        \item A Taub-Bolt space diffeomorphic to $\mathbb R^4\sharp \C\mathbb P^2$;
   \item An anti-Taub-Bolt space diffeomorphic to $\mathbb R^4\sharp \overline{\C\mathbb P^2}$
        \item A Chen-Teo space diffeomorphic to $S^2\times \R^2\sharp \C\mathbb P^2$.
    \end{itemize}
    \end{enumerate}
\end{proposition} 

 In particular, 
\begin{corollary}
    Suppose $(M, g)$ is a simply-connected toric gravitational instanton of Type $\ALE/\ALF/\AFa$ and let $\mR$ be the associated rod structure. If $\mR$ has more than 2 turning points, then $d(\mR)>1$.
\end{corollary}

\subsection{Degree 0 = Type \texorpdfstring{$\I$}{I}}

In this subsection we prove the first part of Theorem \ref{thm:equivalence of degree and type}. 
\begin{theorem} \label{thm:Type I equivalent to degree 0}
A rod structure $\mR$ is Type $\I$ if and only if $d(\mathfrak R)=0$. In this case $\mR$ must have Asymptotic Type  $\Ae$ or $\ALFp$, unless $\mR$ has no turning points and is of  Asymptotic Type  $\AFa$ for some  $\beta\in \R$.
\end{theorem}

\begin{proof}
 Let $\Phi$ be a harmonic map  which is  strongly tamed by a rod structure $\mR$. 
If $\mR$ has no finite rods, then by definition it has to be of Asymptotic Type  $\AFa$ for some $\beta\geq 0$. By the uniqueness result in Theorem \ref{t:existence result} we know $\Phi$ must be equal to $\Phi^{\AFa}_{\mv, \mv}$ for some $\beta$. The associated toric Ricci flat metric is flat, so by definition it is of Type $\I$. It is also obvious that $d(\mR)=0$.  In the following we assume that there is at least 1 finite rod.
    
    $(\Rightarrow)$: Suppose $\Phi$ is Type $\I$.   Recall that from our discussion in Section \ref{sec:7.1} we know $\Phi$ is locally  given by \eqref{eq:Type I metric}  for an axisymmetric harmonic function $V$ on $\dH^\circ$. 
    Due to the strong tameness we know that $V$ needs to stay bounded in a neighborhood of a point in the interior of a rod and blow up at the rate $r^{-1}$ near each turning point, and $V$ is uniformly bounded at infinity. It is then easy to see that $V$ must be of the form
        $$V=T+\sum_j\frac{\lambda_j}{(\rho^2+(z-z_j)^2)^{1/2}}$$
 for some $T\geq 0$, where $\lambda_j>0$ and $z_1<\ldots<z_n$ are the turning points of the rod structure. Since there is at least 1 finite rod, $V$ can not be a constant. If $T=0$, then $\mR$ is of Asymptotic Type  $\Ae$; if $T>0$, then we must take $T=1$ (since it has decaying distance to $\Phi^{\TNap}$ by definition) and $\mR$ is of Asymptotic Type  $\ALFp$. It is easy to see that in both cases $d(\mR)=0$.

    $(\Leftarrow)$: Suppose $d(\mR)=0$. By definition, $\mR$ has to be of Asymptotic Type  $\Ae$ or $\ALFp$ (since in all the other cases we have $d(\mR)\geq 1$).

   We focus on the case $\sharp=\Ae$; the other case can be argued similarly. Since $d(\mathfrak R)=0$, by applying a suitable element in $SO(2;\mathbb R)$, one can always assume the rod vectors are of the form
   \begin{equation}\label{eqn:normalized rod vectors Gibbons-Hawking}\mathfrak{v}_j=[(-\sum_{l\leq j}\lambda_j+\sum_{l>j}\lambda_j,1)^{\top}]
   \end{equation}
    for some $\lambda_j>0$.
    Then there is a harmonic map $\Phi'$  strongly tamed by this rod structure, which is   explicitly given by the $(\Rightarrow)$ part of the proof, which is Type $\I$. From the uniqueness result in Theorem \ref{t:existence result}, it follows that $\Phi=\Phi'$, hence $\Phi$ is Type $\I$.
\end{proof}

\begin{remark}\label{rem:angle invariant in Type I}
It is easy to see, using the explicit formulae \eqref{eqn:GH harmonic map}, \eqref{eqn:GH augmentation}, that the  rod vectors $\mv_j$ in \eqref{eqn:normalized rod vectors Gibbons-Hawking} are indeed normalized. Notice that they do not depend on the rod lengths. In particular, if we are given an enhancement and an augmentation, then the cone angle along each rod can be seen explicitly, and it does not depend on the rod length. This gives a complication when one applies the strategy of tuning the cone angles by varying the rod lengths -- the proof of Theorem \ref{t:main2} would not work if one instead tries to construct Type I gravitational instantons (which we already know do not exist with the $\AFa$ Asymptotic Types anyway).
\end{remark}

\subsection{Degree 1 = Type \texorpdfstring{$\II$}{II}}
\label{sec:7.3}

The main result of this subsection is
\begin{theorem}\label{thm:degree equivalence Type II}
A rod structure $\mR$ is Type $\II$ if and only if $d(\mathfrak R)=1$.
\end{theorem}

In the following we will first prove Theorem \ref{thm:degree equivalence Type II} for the $\ALF^\pm/\AFa$ case and then deal with the $\Ae$ case in Section \ref{sss:case Ae}. Harmonic maps strongly tamed by a rod structure of  Asymptotic Type  $\ALF^\pm/\AFa$ have already been classified by Biquard-Gauduchon \cite{BG}, which we now recall.

For each positive convex piece-wise affine function
\begin{equation}
    f(z)=A+\sum_{j=1}^na_j|z-z_j|
\end{equation}
with  $A>0$, $z_1<\cdots <z_n$ and $-1=a_1<\cdots< a_n=1$, one can assign a canonical axisymmetric harmonic function
\begin{equation}
    U(\rho,z)=A\log\rho^2+\sum_{j=1}^na_j U_0(\rho,z-z_j),
\end{equation}
where
$$U_0(\rho,z)=2r-z\log\frac{r+z}{r-z}.$$
We write $f_j=f(z_j)$ and denote by $f_j'$  the slope of $f$ over $(z_j,z_{j+1})$. In particular, we have $$a_j=\frac12(f_j'-f_{j-1}').$$
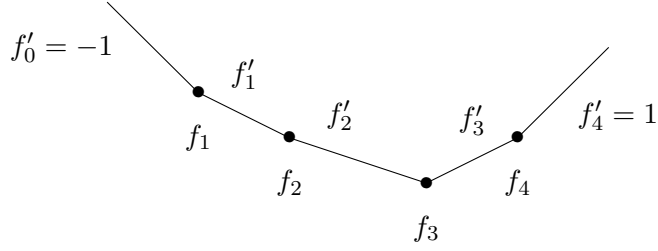
\begin{figure}[ht]
    \begin{tikzpicture}[scale=0.6]
        \draw (0,4) -- (2,2);
        \draw (2,2) -- (4,1);
        \draw (4,1) -- (7,0);
        \draw (7,0) -- (9,1);
        \draw (9,1) -- (11,3);
        \node at (2,2) {$\bullet$};
        \node at (2,1) {$f_1$};
        \node at (4,1) {$\bullet$};
        \node at (4,0) {$f_2$};
        \node at (7,0) {$\bullet$};
        \node at (7,-1) {$f_3$};
        \node at (9,1) {$\bullet$};
        \node at (9,0) {$f_4$};

        \node at (-1,3) {$f_0'=-1$};
        \node at (3,2.4) {$f_1'$};
        \node at  (5.1,1.5) {$f_2'$};
        \node at (8,1.4) {$f_3'$};
        \node at (11.2,1.5) {$f_4'=1$};
    \end{tikzpicture}
	\caption{A piece-wise affine function $f$ with four turning points.}
	 \label{Figure:piece-wise affine function}   
\end{figure}
Then with the ansatz in \eqref{eq:Type II metric harmonic map}-\eqref{eq:Type II F harmonic} and choosing $k=2A$ (this ensures $V\to1$ at infinity), we get an augmented harmonic map $(\Phi,\nu)$ which is strongly tamed by a rod structure $\mR$ with $n$ turning points.  The rods are given by $\mathfrak{I}_j=(z_j,z_{j+1})$ and the normalized rod vector along $\mI_j$ is
\begin{equation}\label{eq:BG normailzed rod vector}
    \overline{\mv}_j=
    \begin{cases}
         \overline{\mathfrak{v}}_j=f_j'(\partial_{x_3}+F_j\partial_t), & \text{if} \ \ \ \  f_j'\neq 0; \\
         \overline{\mathfrak v}_j=\frac{1}{A}f_j^2\partial_t,  & \text{if} \ \ \ \  f_j'=0.
         \end{cases}
\end{equation}
Here the constants $F_j$ are determined by the following conditions 
\begin{itemize}
\item We have $$F_0=-\frac{1}{A}(A+\sum_{j=1}^na_jz_j)^2+\frac{1}{A}\sum_{j=1}^na_jz_j^2;$$
    \item if $f_j'\neq0$ then \begin{equation}
        F_j-F_{j-1}=\frac{1}{A}f_j^2(\frac{1}{f_j'}-\frac{1}{f_{j-1}'});
    \end{equation}
    \item if $f_j'=0$ then \begin{equation}
        F_{j+1}-F_{j-1}=\frac{1}{A}(f_j^2(\frac{1}{f_{j+1}'}-\frac{1}{f_{j-1}'})-2(z_{j+1}-z_j)f_j).
    \end{equation}
\end{itemize}
Notice that these imply  $$F_n=\frac{1}{A}(A-\sum_{j=1}^na_jz_j)^2-\frac{1}{A}\sum_{j=1}^na_jz_j^2.$$
For the convenience of discussion, in the ansatz \eqref{eq:Type II metric harmonic map}-\eqref{eq:Type II F harmonic}, we take $\phi_1$ and $\phi_2$ such that $$\partial_{\phi_1}=\partial_{x_3}+\frac{F_0+F_n}{2}\partial_t, \quad \partial_{\phi_2}=-\partial_t.$$ Then  the harmonic map $\Phi$ is given by
\begin{equation}\label{eq:Type II harmonic map under correct basis}
    \Phi=\frac{1}{\rho}\left(\begin{matrix}
        V\rho^2+V^{-1}(F-\frac{F_0+F_n}{2})^2 & V^{-1}(F-\frac{F_0+F_n}{2}) \\ V^{-1}(F-\frac{F_0+F_n}{2}) & V^{-1}
    \end{matrix}\right).
\end{equation}
The  Asymptotic Type  of $\mR$ is $\AF_0$ when $F_n-F_0=0$, is $\ALFm$ if $F_n-F_0>0$ and is $\ALFp$ when  $F_n-F_0<0$.

By construction we know that the harmonic map given by \eqref{eq:Type II harmonic map under correct basis} is of Type $\II$.  Conversely, for any Type $\II$ harmonic map strongly tamed by an $\ALF^\pm/\AFa$ rod structure $\mR$, by \cite{BG}, up to conjugation by a suitable element in $SL(2;\mathbb{R})$ (in particular by $U_\beta$ we can make the Asymptotic Type change from $\AFa$ to $\AF_0$), it always takes the form \eqref{eq:Type II harmonic map under correct basis}. 

\begin{remark}\label{rem: angle in Type II}
    In this setting, the normalized rod vectors are explicitly given by \eqref{eq:BG normailzed rod vector}, so as in the Type $\I$ case if we are given an enhancement and an augmentation, then the cone angle along each rod can be seen explicitly.  But unlike the Type $\I$ case they do depend on the rod lengths. This is the philosophy which makes it possible to apply the strategy of tuning the cone angles by varying rod lengths in our proof of Theorem \ref{t:main2}.
\end{remark}

\begin{proposition}
    The above rod structure $\mR$ has degree $1$.
    \end{proposition}
\begin{proof}
 First we observe that 
\begin{itemize}
    \item if $f_j'<0$, then $F_j-F_{j-1}<0$;
    \item if $f_j'>0$, then $F_{j+1}-F_j<0$.
\end{itemize}
If $f_j'=0$, we can formally view $F_j$ as $-\infty$. 

We only consider the $\ALFm$ case, as the other cases are similar. Thus, we assume $F_n - F_0 > 0$.  
We begin with the rod $\mathfrak{I}_0$ and choose  
\[
\partial_{\phi_1} + \frac{F_n - F_0}{2} \partial_{\phi_2} = \partial_{x_3} + F_0 \partial_t
\]  
as the representative vector $v_0$.  

From the above observation, if we assign the following representative vectors $v_j$ for $\mathfrak{I}_j$,  

\begin{equation}
    v_j=\begin{cases}
        \partial_{\phi_1} + \left( \frac{F_0 + F_n}{2} - F_j \right) \partial_{\phi_2} = \partial_{x_3} + F_j \partial_t, & \text{if} \quad f_j' < 0, \\
        \partial_{\phi_2} = -\partial_t, & \text{if} \quad f_j' = 0, \\
        -\left( \partial_{\phi_1} + \left( \frac{F_0 + F_n}{2} - F_j \right) \partial_{\phi_2} \right) = -(\partial_{x_3} + F_j \partial_t), & \text{if} \quad f_j' > 0,
    \end{cases}
\end{equation}  
then $v_j$ rotates in a counterclockwise direction. It is then straightforward to see that $d(\mathfrak{R}) = 1$.  
\end{proof}

To finish the proof of Theorem \ref{thm:degree equivalence Type II} in the $\ALF^\pm/\AFa$ case, we fix the  positions of the turning points $z_1,\ldots,z_n$ and $\sharp\in \{\ALFm, \ALFp, \AFa (\beta\in \R)\}$. Denote by $\mathcal{S}_{1,\sharp}\subset\mathcal S(\sharp|z_1, \cdots, z_n)$ the subset consisting of rod structures  of degree $1$. By Theorem \ref{p:degreee topological invariant} we know that $\mathcal S_{1,\sharp}$ is connected.  Denote by $\mathcal S_{1,\alpha,\sharp}$ the subset of $\mathcal S_{1,\sharp}$ which consists of rod structures $\mR$ with $\mv_{0}\wedge\mv_{n}=\alpha$ and $\mathcal{S}'_{1,\alpha,\sharp}$ those with $\mv_{0}\wedge\mv_{n}=\alpha$ that are of Type $\II$.  Note that if $\sharp=\AFa$ then $\alpha=0$ and if $\sharp=\ALF^\pm$ then $\alpha>0$. Similar to the proof of Theorem \ref{p:degreee topological invariant} we also have that each $\mathcal S_{1,\alpha,\sharp}$ is connected. From the Biquard-Gauduchon construction we know each $\mathcal S_{1,\alpha,\sharp}'$ is nonempty. All we need to show is that for each $\alpha$, $\mathcal{S}_{1,\alpha,\sharp}'=\mathcal{S}_{1,\alpha,\sharp}$. This implies that $\mathcal S_1'=\mathcal{S}_1$, which finishes the proof.

As in the discussion in Section \ref{ss:uniform control of tame harmonic maps} if we have the convergence of rod structures $\mR_i\rightarrow \mR_\infty$, then the associated harmonic maps $\Phi_i$ converges to $\Phi_\infty$ smoothly away from the rods. In particular, this implies that $\mathcal S_{1,\alpha,\sharp}'$ is closed. It  remains to prove that $\mathcal S_{1,\alpha,\sharp}'$  is open. One possible strategy to show this is by using the implicit function theorem and the explicit knowledge above on the rod structures of elements in $\mathcal S_1'$. The naive count of parameters matches, but it does not seem obvious to see the infinitesimal openness (i.e., non-vanishing of the associated derivative map). We will instead use the approach motivated to prove related result on gravitational instantons by Biquard-Gauduchon-LeBrun \cite{BGL}. The latter is a generalization of the surprising openness of Type $\II$ Einstein metrics on compact 4-manifolds, discovered by Wu \cite{Wu} and LeBrun \cite{LeBrun}. The advantage of this approach is that one obtains the openness without proving infinitesimal openness.

\subsubsection{Openness of $\mathcal S_{1,\alpha,\sharp}'$} 
\label{subsubsec:openness of Type II}

In this section we prove
\begin{proposition}\label{prop:openness Type II}
    The set $\mathcal{S}_{1,\alpha,\sharp}'$ is open.
\end{proposition}

Given a rod structure $\mR$ in $\mathcal{S}_{1,\alpha,\sharp}'$ we let $\Phi$ be the associated harmonic map which is strongly tamed by $\mR$. Suppose the rod vectors of $\mR$ are given by $\mv_0,\ldots,\mv_n$. Without loss of generality we may assume $\mv_0=[\bv_{\alpha/2}]$ and $\mv_n=[\bv_{-\alpha/2}]$. Fix $\epsilon>0$ small. Then an open neighborhood of $\mR$ in $\mathcal{S}_{1,\alpha,\sharp}$ consists of rod structures with the same positions of the turning points but with rod vectors given by $\mv_j^\epsilon$ where for $j=0,\ldots,n$,
$$\text{the angle between }\mv_j^{\epsilon} \text{ and }\mv_j<\epsilon.$$
It suffices to show that for $\epsilon>0$ small any harmonic map $\Phi_\epsilon$ strongly tamed by such a rod structure $\mR_\epsilon$ is also of Type $\II$. Without loss of generality, by applying a suitable $P\in PSO(2)$ on $\mR_\epsilon$ we can again assume $\mv_0^\epsilon=[\bv_{\alpha/2}]$ and $\mv_n^\epsilon=[\bv_{-\alpha/2}]$.
We first prove

\begin{lemma}\label{lem:finer control on Phi epsilon}
    There is a constant $C(\epsilon)>0$ with $C(\epsilon)\to0$ as $\epsilon\to0$, such that over $\mathbb{H}^\circ\setminus B_{N}(0)$ for a fixed large $N$, $d(\Phi_\epsilon,\Phi)<C(\epsilon)/r$.
\end{lemma}
\begin{proof}
    We modify $\Phi$ to a model map $\Psi_{\epsilon}$ which is strongly tamed by $\mR_\epsilon$ and has controlled hyperbolic distance to $\Phi_\epsilon$. For $\delta>0$, we divide $\mathbb{H}^\circ$ into a union of regions
    \begin{align*}
        &U_{\infty,\delta}=\mathbb{H}^\circ\setminus N_{\delta}((z_1,z_n)),\\
        &U_{j,1,\delta}=N_{2\delta}((z_j,\frac{z_j+z_{j+1}}{2})),\\
        &U_{j,2,\delta}=N_{2\delta}((\frac{z_j+z_{j+1}}{2},z_{j+1})).
    \end{align*}
    Here $j\in\{1,\ldots,n-1\}$. Fix $\delta_0$ to be a small positive constant. Define 
    $$\alpha_j:=\mv_{j-1}\wedge \mv_j,\quad \alpha_j^\epsilon:=\mv_{j-1}^\epsilon\wedge \mv_j^\epsilon.$$
    Moreover, set $P_{j,\epsilon}\in PSL(2;\R)$ to be such that  
    $$P_{j,\epsilon}.[\bv_{\pm\frac{\alpha_j}{2}}]=[\bv_{\pm\frac{\alpha_j^\epsilon}{2}}].$$
    Define the map $\Psi_{j,1}$ over $U_{j,1,3\delta_0}$ as 
    $$\Psi_{j,1}:=P_{\mv_{j-1}^\epsilon,\mv_j^\epsilon}^{-1}P_{j,\epsilon} P_{\mv_{j-1},\mv_j}.\Phi$$ 
    and $\Psi_{j,2}$ over $U_{j,2,3\delta_0}$ as $$\Psi_{j,2}:=P_{\mv_j^\epsilon,\mv_{j+1}^\epsilon}^{-1}P_{j+1,\epsilon}P_{\mv_j,\mv_{j+1}}.\Phi.$$
    
    Next, we glue the above maps together to obtain the map $\Psi_\epsilon$. We first glue the maps $\{\Psi_{j,1},\Psi_{j,2}\}_{j=1,\ldots,n-1}$. On each overlap $U_{j,1,3\delta_0}\cap U_{j,2,3\delta_0}$, we can write $P_{\mv_{j}^\epsilon}.\Psi_{j,1}$ and $P_{\mv_{j}^\epsilon}.\Psi_{j,2}$ as 
    $$P_{\mv_j^\epsilon}.\Psi_{j,1}=\frac{1}{\rho}\left(\begin{matrix}
        \rho^2e^{u_1}+e^{-u_1}v_1^2 & e^{-u_1}v_1 \\ e^{-u_1}v_1 & e^{-u_1}
    \end{matrix}\right),\quad P_{\mv_j^\epsilon}.\Psi_{j,2}=\frac{1}{\rho}\left(\begin{matrix}
        \rho^2e^{u_2}+e^{-u_2}v_2^2 & e^{-u_2}v_2 \\ e^{-u_2}v_2 & e^{-u_2}
    \end{matrix}\right).$$
    Pick a cut-off function $0\leq\chi_j\leq1$ which equals one over $U_{j,1,2\delta_0}$ and has support in $U_{j,1,3\delta_0}$. Set $$u:=\chi_j u_1+(1-\chi_j)u_2, \quad v:=\chi_j v_1+(1-\chi_j)v_2$$ and define
    $$\Psi'_j:=\begin{cases}
        \Psi_1,\text{ over }U_{j,1,\delta_0},\\
        \frac{1}{\rho} P_{\mv_j^\epsilon}^{-1}. \left(\begin{matrix}
        \rho^2e^{u}+e^{-u}v^2 & e^{-u}v \\ e^{-u}v & e^{-u}
    \end{matrix}\right), \text{ over } U_{j,1,\delta_0}\cap U_{j,2,\delta_0},\\
    \Psi_2, \text{ over }U_{j,2,\delta_0}.
    \end{cases}$$
    This way we have obtained the glued map $\Psi_j'$ over $U_{j,1,2\delta_0}\cup U_{j,2,2\delta_0}$. Similarly we can glue all the $\Psi_j'$ on the overlaps of $U_{j,1,2\delta_0}\cup U_{j,2,2\delta_0}$ to obtain a glued map $\Psi'$ over $N_{2\delta_0}((z_1,z_n))$.

    Finally, on the overlap $N_{2\delta_0}((z_1,z_n))\cap U_{\infty,\delta_0/2}$, we can write $\Psi'$ and $\Phi$ as 
    $$\Psi'=\left(\begin{matrix}
        e^{u'}+e^{-u'}v'^2 & e^{-u'}v' \\ e^{-u'}v' & e^{-u'}
    \end{matrix}\right),\quad \Phi=\left(\begin{matrix}
        e^{u}+e^{-u}v^2 & e^{-u}v \\ e^{-u}v & e^{-u}
    \end{matrix}\right).$$
    Pick a cut-off function $0\leq\chi\leq1$ which equals one over $U_{\infty,\delta_0}$ and supports in $U_{\infty,\delta_0/2}$. Set $\bar u:=\chi u+(1-\chi)u',\bar v:=\chi v+(1-\chi)v'$ and define
    $$\Psi_\epsilon:=\begin{cases}
        \Phi,\text{ over }U_{\infty,\delta_0},\\
        \left(\begin{matrix}
        e^{\bar u}+e^{-\bar u}\bar v^2 & e^{-\bar u}\bar v \\ e^{-\bar u}\bar v & e^{-\bar u}
    \end{matrix}\right), \text{ over } U_{\infty,\delta_0/2}\cap N_{\delta_0}((z_1,z_n)),\\
    \Psi', \text{ over }N_{\delta_0}((z_1,z_n))\setminus U_{\infty,\delta_0/2}.
    \end{cases}$$
    From the construction and our local regularity result in Section \ref{subsec:local regularity}, there is a constant $C(\epsilon)$ which depends on $\epsilon$ and goes to zero as $\epsilon$ tends to zero, such that $|\Delta\Psi_\epsilon|<C(\epsilon)$. The map $\Psi_\epsilon$ is globally tamed by the rod structure $\mR_\epsilon$, and moreover, an important feature is that the map $\Psi_\epsilon$ coincides with $\Phi$ over $\mathbb{H}^\circ\setminus B_N(0)$ for large $N$.

    Now as in the proof of Theorem \ref{t:existence result}, setting $$F_\epsilon(x)=\int_{\mathbb{R}^3}|x-y|^{-1}|\Delta\Psi_{\epsilon}|(y)dy,$$ we have $$d(\Phi_\epsilon,\Psi_\epsilon)\leq F_\epsilon,$$
    which particularly gives $$d(\Phi_\epsilon,\Psi_\epsilon)<C(\epsilon)/r$$ when $r$ is large. This implies that $$d(\Phi_\epsilon,\Phi)<C(\epsilon)/r$$ when $r$ is large.
\end{proof}

When $\epsilon$ is sufficiently small, for the normalized augmentation $\nu_\epsilon$, we have  the metric $$g_\epsilon:=e^{2\nu_\epsilon}(d\rho^2+dz^2)+\rho \Phi_\epsilon.$$  
Denote by $g$ the corresponding metric associated to $\Phi$. From the lemma above, by the local regularity estimate in Proposition \ref{thm:C11 regularity} we know that on the end the difference between the Weyl curvatures of $g_{\epsilon}$ and $g$ is controlled as $$|W^+_{g_\epsilon}-W^+_g|_g<C(\epsilon) r^{-3}.$$ Moreover, since $g$ is Type $\II$, we know that $|W^+_g|$ is exactly comparable to $r^{-3}$ (see \cite{BG} for example). Hence when $\epsilon$ small, it satisfies that $\det(W^+_{g_\epsilon})>0$. Therefore, by \cite{BGL}, setting $\alpha_{\epsilon}$ to be the unique positive eigenvalue of $W^+_{g_\epsilon}$, which is a positive function over $\mathbb{H}^\circ$, we define $$f_{\epsilon}:=\alpha_{\epsilon}^{-1/3}$$ and $$g_{K,\epsilon}:=f_{\epsilon}^2g_{\epsilon}.$$ Moreover, let $\omega_\epsilon$ be an eigen-two-form of $W^+_{g_\epsilon}$ such that $|\omega_\epsilon|^2_{g_{K,\epsilon}}=2$. Over the manifold $\mathbb{H}^\circ\times\mathbb{R}^2$ there is a pointwise inequality
\begin{equation}\label{eq:pointwise 1}
    3d(\omega_{\epsilon}\wedge *d\omega_{\epsilon})\geq *(\frac{1}{2}|\nabla\omega_\epsilon|^2+3|d\omega_{\epsilon}^2|).
\end{equation}
We also have
\begin{equation}\label{eq:pointwise 2}
    2(2\sqrt{6}|W^+|+|s|)^{1/2}\geq|\omega_{\epsilon}\wedge *d\omega_{\epsilon}|.
\end{equation}
All the norms, Hodge stars, etc. in the above two inequalities are taken under $g_{K,\epsilon}$ and $s$ denotes the scalar curvature of $g_{K,\epsilon}$. For the derivation of these two inequalities, see \cite{BGL}.

\begin{lemma}
    We have
    \begin{equation}
        \iota_{\partial_{\phi_2}}\iota_{\partial_{\phi_1}}d(\omega_{\epsilon}\wedge *d\omega_{\epsilon})=d(\iota_{\partial_{\phi_2}}\iota_{\partial_{\phi_1}}(\omega_{\epsilon}\wedge *d\omega_{\epsilon})).
    \end{equation}
\end{lemma}
\begin{proof}
    Since the form $\omega_\epsilon\wedge*d\omega_\epsilon$ is invariant under $\partial_{\phi_1},\partial_{\phi_2}$ and the two Killing fields commute, we have
    \begin{align*}
        d(\iota_{\partial_{\phi_2}}\iota_{\partial_{\phi_1}}(\omega_{\epsilon}\wedge *d\omega_{\epsilon}))&=\mathcal{L}_{\partial_{\phi_2}}(\iota_{\partial_{\phi_1}}(\omega_{\epsilon}\wedge *d\omega_{\epsilon}))-\iota_{\partial_{\phi_2}}(d(\iota_{\partial_{\phi_1}}\omega_{\epsilon}\wedge *d\omega_{\epsilon}))\\
        &=-\iota_{\partial_{\phi_2}}\left(\mathcal{L}_{\partial_{\phi_1}}(\omega_{\epsilon}\wedge *d\omega_{\epsilon})-\iota_{\partial_{\phi_1}}(d(\omega_{\epsilon}\wedge *d\omega_{\epsilon}))\right)\\
        &=\iota_{\partial_{\phi_2}}\iota_{\partial_{\phi_1}}d(\omega_{\epsilon}\wedge *d\omega_{\epsilon}).
    \end{align*}
\end{proof}

Denote $$\gamma:=\iota_{\partial_{\phi_2}}\iota_{\partial_{\phi_1}}(\omega_{\epsilon}\wedge *d\omega_{\epsilon}),$$ which descends to a 1-form on $\mathbb{H}^\circ$. Denote $$\overline{\mathrm{dvol}}:=\iota_{\partial_{\phi_2}}\iota_{\partial_{\phi_1}}\dvol_{g_{K,\epsilon}},$$ which descends to a 2-form on $\mathbb{H}^\circ$. Then  \eqref{eq:pointwise 1} implies that over $\mathbb{H}^\circ$
\begin{equation}\label{eq:pointwise 3}
    3d\gamma\geq(\frac{1}{2}|\nabla\omega_\epsilon|^2+3|d\omega_{\epsilon}^2|)\overline{\dvol}.
\end{equation}
For any small $\tau>0$ and large $N>0$, define $V_{\tau,N}:=B_{N}(0)\cap\{\rho>\tau\}$, whose boundary consists of the union of $\partial V_{\tau,N}^1:=\partial V_{\tau,N}\cap\{\rho=\tau\}$ and $\partial V_{\tau,N}^2:=\partial B_N(0)\cap\{\rho>\tau\}$. By integrating by parts,
$$3\int_{\partial V_{\tau,N}^1}\gamma+3\int_{\partial V_{\tau,N}^2}\gamma\geq \int_{V_{\tau,N}}(\frac{1}{2}|\nabla\omega_\epsilon|^2+3|d\omega_{\epsilon}^2|)\overline{\dvol}.$$

\begin{lemma} We have 
$$\lim_{\tau\rightarrow0}\int_{\partial V_{\tau,N}^1}\gamma\to0.$$    
\end{lemma}
\begin{proof}
    One only needs to compute $\int_{\partial V_{\tau,N}^1}\gamma$ locally. Near the interior of each rod or near each turning point, we can always pick the correct lattice to make the quotient of $g_{\epsilon}$ smooth. It is then easy to see the behavior of $\gamma$ and conclude the lemma.
\end{proof}
By setting $\tau\to0$, we have 
$$3\int_{\partial V_{0,N}^2}\gamma\geq \int_{V_{0,N}}(\frac{1}{2}|\nabla\omega_\epsilon|^2+3|d\omega_{\epsilon}^2|)\overline{\dvol}.$$
Next we take the canonical lattice $\Lambda_{\mR_\epsilon}$ as given in Definition \ref{def:canonical enhancement}, which is the same as $\Lambda_{\mR}$. The lattice $\Lambda_{\mR}$ gives rise to smooth Ricci-flat ends $(E_\epsilon,g_{\epsilon,\Lambda_{\mR}})$ and $(E,g_{\Lambda_{\mR}})$, which are obtained by taking the quotient of $g_{\epsilon}$ and $g$ respectively. By Proposition \ref{prop:appendix 2} we know that there is a diffeomorphism $\varpi:E\to E_\epsilon$ such that
     $$|\nabla_{g_{\Lambda_\mR}}^k(\varpi^*g_{\epsilon,\Lambda_\mR}-g_{\Lambda_\mR})|_{g_{\Lambda_\mR}}<C_k(\epsilon)/\mathfrak{r}^{k+1}.$$ 
     Here $C_k(\epsilon)$ depends on $k$ and  $C_k(\epsilon)\to0$ as $\epsilon\to0$, and $\mathfrak{r}$ denotes the distance function under $g_{\Lambda_\mR}$.

Over the smooth end we have
\begin{align}
    \int_{\partial V_{0,N}^2}\gamma&=C\int_{\pi^{-1}(\partial V_{0,N}^2)}\omega_{\epsilon}\wedge* d\omega_{\epsilon}\notag\\
    &\leq 2C\left(\int_{\pi^{-1}(\partial V_{0,N}^2)}\dvol_{g_{K,\epsilon}}\right)^{1/2}\left(\int_{\pi^{-1}(\partial V_{0,N}^2)}(2\sqrt{6}|W^+|+|s|)\dvol_{g_{K,\epsilon}}\right)^{1/2},\label{eq:lebrun estimate for boundary}
\end{align}
where here $\pi:E\to\mathbb{H}^\circ$ is the projection map from the smooth end into the base $\mathbb{H}^\circ$ and the constant $C$ depends on the lattice $\Lambda_{\mR}$. Now the same arguments as in the proof of Theorem A in \cite{BGL}, but performed over the smooth end $E$, shows that both integrals in \eqref{eq:lebrun estimate for boundary} are going to 0 as $N\to\infty$. This shows $\nabla\omega_{\epsilon}\equiv0$ and $d\omega_\epsilon\equiv0$, which implies that the conformal metric $g_{K,\epsilon}$ is K\"ahler. Therefore $g_\epsilon$ is Type $\II$ and this finishes the proof of Proposition \ref{prop:openness Type II}.

\subsubsection{The  \texorpdfstring{$\Ae$}{} case} \label{sss:case Ae}

Here we prove Theorem \ref{thm:degree equivalence Type II} in the case of Asymptotic Type $\Ae$. It is possible to argue as in the other cases as above, but there is a simpler argument which we present below.

We first show that a rod structure $\mR$ with  Asymptotic Type  $\Ae$ and degree $d(\mR)=1$ is of Type $\II$. To see this, for $\delta>0$ we set $\mR'_{\delta}$ to be the rod structure with  Asymptotic Type  $\ALFp$ and turning points $\delta z_j$ and with the same rod vectors as $\mR$. Then $d(\mR'_{\delta})=1$. Hence the associated harmonic map $\Phi_\delta$ is of Type $\II$.  Letting $\delta\to0$, by our analysis in Section \ref{ss:uniform control of tame harmonic maps} and \ref{ss:limiting behavior of the cone angles}, the rescaled harmonic map $\varphi_{1/\delta}^*\Phi_\delta$ converges to a limit harmonic map $\Phi_\infty$, which is globally tamed by $\mR$. Since $\Phi_\delta$ is strongly tamed by $\mR_\delta'$, by Proposition \ref{p:AE rigidity} it follows that $\Phi_\infty$ is strongly tamed by $\mR$. Notice that being the limit of Type $\II$ harmonic maps,  $\Phi_\infty$ must be of Type $\I$ or Type $\II$, but we have $d(\mR)=1$ so by Theorem \ref{thm:Type I equivalent to degree 0}, $\Phi_\infty$ can not be of Type $\I$. Thus $\mR$ is of Type $\II$.

Next we show that a Type $\II$ rod structure $\mR$ of  Asymptotic Type  $\Ae$ must have degree $1$. Choosing the normalized augmentation $\nu$, we obtain a Type $\II$ Ricci-flat metric 
    $$g=e^{2\nu}(d\rho^2+dz^2)+\rho\Phi$$ 
over $\mathbb{H}^\circ\times\mathbb{R}^2$. As in Section \ref{subsubsec:openness of Type II}, $g$ is conformal to a Bach-flat K\"ahler metric $g_K=f^2g$. Using the canonical lattice $\Lambda_{\mR}$ given in Definition \ref{def:canonical enhancement}, we obtain a smooth toric Ricci-flat end $(E,g_{\Lambda_\mR})$ of Asymptotic Type $\ALE$. Indeed, it is asymptotic to the flat $\R^4$, and $g_{\Lambda_{\mR}}$ is conformal to a Bach-flat K\"ahler metric $g_{K,\Lambda_{\mR}}$. As is proved in \cite{Li1}, the K\"ahler metric $g_{K,\Lambda_\mR}$ is incomplete at the infinity of the end $E$, and  the metric completion is given by adding exactly one point $\{p_\infty\}$. Moreover, by  a singularity removal theorem for  $g_{K,\Lambda_\mR}$, it follows that the resulting metric is smooth, which we denote by $(\overline{E},\overline{g}_{K,\Lambda_\mR})$. The  torus action naturally extends to $\overline{E}$, fixing $p_\infty$. 

Notice that in general $\Lambda_{\mR}$ may not be an enhancement for $\mR$ and we do not have a global quotient by $\Lambda_{\mR}$. Nevertheless, we can work directly with the metric $g_K$ on $\dH^\circ\times \R^2$, with a natural $\R^2$ action. It is easy to see that this action admits a moment map, which descends to a map from $\mathbb{H}^\circ$ to $\mathbb{R}^2$, whose image $ \Delta$ is bounded. We claim that the closure $\overline{\Delta}$ (given by adding one point to $\Delta$, which corresponds to the point at infinity $p_\infty$) is a convex polygon, with the normal vector of each edge agreeing with the corresponding rod vector of $\mR$. From this, it is easy to see that $d(\mR)=1$.  
  
The claim follows from the Atiyah-Guillemin-Sternberg's convexity theorem. We are not in the exact situation to apply the convexity theorem, but its proof can be carried over without essential difficulty. The point is that, locally with a suitable change of the quotient lattice  we may obtain a smooth toric quotient. More precisely, near each turning point, the interior of each rod, and the infinity of $\mathbb{H}$, by picking the correct lattice, the quotient of $g_K$ locally extends as a smooth K\"ahler metric, with stabilizers described by the rod vectors. This implies that $\overline\Delta$ is a locally convex polygon, hence it is a convex polygon.

\section{Discussion}\label{sec:discussion}
\subsection{Riemannian black hole uniqueness}

In Theorem \ref{t:main} we require $\beta\in (0, 1)$. Notice that when $n=1$, by Proposition \ref{p:Kerr family} we may allow $\beta$ to be $0$, which then gives rise to the Schwarzschild space, but when $n\geq 2$ it is easy to see that the arguments in Section \ref{ss:general case} does not go through. Similarly the arguments fail to work for all $n\geq 1$ when $\beta=1$. There is a deeper reason behind these technical constraints.
In fact, in these cases it can be shown that one can not obtain gravitational instantons with the desired Asymptotic Type. In other words, conical singularities must appear.  This follows from the uniqueness result of Mars-Simon \cite{MarsSimon}, with a mild extension. See Appendix \ref{sec:appendix b}.

Motivated by this and Theorem \ref{t:main}, it seems that the following version of the Riemannian black hole conjecture is likely to hold

\begin{conjecture}[Uniqueness conjecture I]
         The Schwarzschild space is the unique $\AF_0$ gravitational instanton up to isometry and scaling. 
\end{conjecture}
    It is a natural question to investigate the geometric behavior of the $\AFa$ gravitational instanton $g_\beta$ constructed in Theorem \ref{t:adjusting angle} when $\beta\rightarrow 0$ or $\beta\rightarrow 1$. In the case $n=1$ from Example \ref{ex:Kerr} we know that as $\beta\rightarrow1$ the Kerr spaces split into the union of a  Taub-NUT space and an anti-Taub-NUT space. In general we have a similar picture. Below we assume $n\neq 2$ and only consider the case $\beta\rightarrow 0$. One can similarly deal with the case $\beta\rightarrow 1$, which we leave for the interested readers.

    We first prove a technical result, which generalizes Proposition \ref{p:AE rigidity} to the  $\ALF^{\pm}$ and $\AFa$ case. 
\begin{proposition} \label{p:ALF/AF rigidity}
    Suppose $\Phi$ is a harmonic map globally tamed by a rod structure $\mR$ of  Asymptotic Type  $\ALFm$, $\ALFp$, or $\AFa$. Then there is a $Q\in SL(2; \R)$ such that $Q.\mv_0=\mv_0$, $Q.\mv_n=\mv_n$, and $Q.\Phi$ is the unique harmonic map which is strongly tamed by $Q.\mR$ (as produced in Theorem \ref{t:existence result}). Here we make $Q.\mR$ into a rod structure by assigning its Asymptotic Type to be $\ALFm$, $\ALFp$, or $\AFa$ accordingly.
\end{proposition}
\begin{proof}
    Without loss of generality we may assume $\mv_0=[\bv_{\frac{\alpha}{2}}]$ and $\mv_n=[\bv_{-\frac{\alpha}{2}}]$ for some $\alpha>0$ when the Asymptotic Type is $\ALFm$ or $\ALFp$ and $\alpha=0$ when the Asymptotic Type is $\AFa$. Note that by Theorem \ref{t:existence result} there is the unique harmonic map $\Psi$ strongly tamed by $\mR$, and by the global tameness of $\Phi$ we have $d(\Phi,\Psi)<C$. 

    Next we first prove the proposition in the $\AFa$ case. For simplicity, by applying a $U_\beta$ action we may assume the Asymptotic Type is $\AF_0$. So on the end we have $d(\Phi,\Phi^{\AF_0})<C$. For each $R\gg1$, consider the scaling down 
    $$\Phi_R:=D_{\sqrt{R}}.\varphi_{R^{-1}}^*\Phi,$$
    and notice that 
    $$\Phi^{\AF_0}=D_{\sqrt{R}}.\varphi_{R^{-1}}^*\Phi^{\AF_0}.$$
    Hence for any $\epsilon>0$ and $R$ large enough depending on $\epsilon$, on the annulus $A_{\epsilon, \infty}\subset \dH^\circ$ we also have $d(\Phi_R,\Phi^{\AF_0})<C$. By the uniform estimate from Section \ref{subsec:local regularity}, up to passing to a subsequence we can take the limit of $\Phi_R$, which we denote by $\Phi_{\infty}$ defined on $\mathbb{H}^\circ$. The limit of the rod structure $\mR_R:=D_{\sqrt{R}}.\varphi_{R^{-1}}^*\mR$ is a weak rod structure $\mR_\infty$ with only two infinite rods, sharing the same rod vector $\bv_0$. After merging the two infinite rods we obtain the rod structure $\mR_\infty'=\mR^{\AF_0}$ assigned with Asymptotic Type $\AF_0$, and $\Phi_\infty$ is globally tamed by the rod structure $\mR^{\AF_0}$. As a consequence, we know that by choosing a suitable augmentation $\nu_\infty$ and a lattice $\Lambda_\infty$, $(\Phi_\infty,\nu_\infty)$ together with $(\mR'_\infty,\Lambda_\infty)$ defines a smooth complete toric Ricci-flat metric $(M_\infty,g_\infty)$ where $M_\infty$ is diffeomorphic to $\R^3\times S^1$. Moreover, we have  $d(\Phi_\infty,\Phi^{\AF_0})<C$, which together with our estimate Proposition \ref{thm:C11 regularity} imply that the curvature of $g_\infty$ decays quadratically, hence by Lemma \ref{lem:Gauss-Bonnet} we know $\int_{M_\infty}|\Rm_{g_\infty}|_{g_\infty}^2\dvol_{g_\infty}<\infty$. Because of Lemma \ref{lem:Gauss-Bonnet}, the Gauss-Bonnet formula holds and we then have $\int_{M_\infty}|\Rm_{g_\infty}|_{g_\infty}^2\dvol_{g_\infty}=\chi(M_\infty)=0$, which implies that $(M_\infty,g_\infty)$ is actually flat. Since $\Phi_\infty$ is associated to the flat metric $(M_\infty,g_\infty)$ and $d(\Phi_\infty,\Phi^{\AF_0})<C$, it follows that there is a positive constant $a>0$ such that $\Phi_\infty=\mathrm{diag}(a\rho,1/(a\rho))$. Hence, if we take $P=D_{\sqrt{a}}$ and consider the rod structure $P.\mR$ assigned with Asymptotic Type $\AF_0$, the harmonic map $P.\Phi$ satisfies $d(P.\Phi,\Phi^{\AF_0})\to0$ as we approach the infinity of $\mathbb{H}$. By the uniqueness part of proof of Theorem \ref{t:existence result}, this already implies that $P.\Phi$ is strongly tamed by $P.\mR$.

    Now we prove the proposition for the case of $\ALFm$, where the case of $\ALFp$ is similar. We can again perform the scaling down $$\Phi_R:=D_{\sqrt{R}}.\varphi_{R^{-1}}^*\Phi.$$ But this time notice that $$D_{\sqrt{R}}.\varphi_{R^{-1}}^*\Phi^{\TNam}=\Phi^{\TN^-_{\alpha/R}}.$$  We still have $d(\Phi_R,\Phi^{\TN^-_{\alpha/R}})<C$ and since $\Phi^{\TN^-_{\alpha/R}}$ converges to $\Phi^{\AF_0}$, the uniform estimate shows that $\Phi_R$ converges to a limit $\Phi_\infty$ globally tamed by $\mR^{\AF_0}$ as above. Taking a suitable augmentation $\nu_\infty$ and a lattice $\Lambda_\infty$ again we get a smooth complete toric Ricci-flat metric $g_\infty$ on $M_\infty=\mathbb{R}^3\times S^1$, which has $\int_{M_\infty}|\Rm_{g_\infty}|_{g_\infty}^2\dvol_{g_\infty}<\infty$. It follows from the Gauss-Bonnet formula that $g_\infty$ is flat, and $d(\Phi_\infty,\Phi^{\AF_0})<C$ implies that there is a positive constant $a>0$ such that $\Phi_\infty=\mathrm{diag}(a\rho,1/(a\rho))$. So take $P=D_{\sqrt{a}}$ we have $d(P.\Phi,\Phi^{\TNam})\to0$ as  we approach the infinity of $\mathbb{H}$. Assign $P.\mR$ with Asymptotic Type $\ALFm$, we then have $P.\Phi$ is strongly tamed by $P.\mR$.
\end{proof}

    Notice that in Theorem \ref{t:adjusting angle} for each $\beta$ the length vector $\bl(\beta)$ which gives rise to an $\AFa$ gravitational instanton is not a priori known to be unique. We make an arbitrary choice of such $\bl(\beta)$ for each $\beta$.

\begin{proposition}
   As $\beta\rightarrow 0$, in the pointed Gromov-Hausdorff sense, without rescalings, $(X_{n}, g_\beta)$ splits into the union of 
   \begin{itemize}
       \item A  Taub-NUT space, an anti-Taub-Bolt space and $k-1$ Schwarzschild spaces (when $n=2k\geq 2$);
       \item An anti-Taub-Bolt space, a  Taub-Bolt space and $k-1$ Schwarzschild spaces (when $n=2k+1\geq3$);
   \end{itemize}
\end{proposition}
\begin{remark}
    When $n=2$ it is also possible to see the above directly using the explicit expression of the Chen-Teo spaces \cite{Chen-Teo}, similar to the discussion on the Kerr spaces in Example \ref{ex:Kerr}.
\end{remark}
\begin{proof} 
    Clearly it suffices to prove that given an arbitrary sequence $\beta_i\rightarrow 0$ and $\bl^i=(l^{i}_1, \cdots, l^i_n)=\bl(\beta_i)$, after passing to a subsequence we have the desired splitting.

    Denote the corresponding enhanced rod structure by $(\mR_i, \Lambda)$  and the augmented harmonic map by $(\Phi_i, \nu_i)$.  By assumption we have  $\vartheta_{j}(\bl^i)=2\pi$ for all $i, j$. It follows that the normalized rod vector along each rod agrees with the primitive vector. 
    For simplicity we will call a rod \emph{vertical} if the normalized rod vector is given by $(1,0)^\top$,  \emph{horizontal} if the normalized rod vector is given by $(0, 1)^\top$ and  \emph{slant} if the normalized rod vector is given by $(1,-1)^\top$.  The key in our argument is to show that  the lengths of all horizontal rods are uniformly bounded above and bounded away from zero, whereas the lengths of all slant rods are going to infinity. This will be proved using the bubbling analysis in Section \ref{sec:degeneration of rod structures} and the uniqueness theorem of Mars-Simon (see Proposition \ref{p:MarsSimon}).

We denote by $S_i$ the maximal length of a slant rod and by $H_i$ the maximal length of a horizontal rod.  Passing to a subsequence we may assume that $H_i$ is achieved by some odd $j_0$ and $S_i$ is achieved by some even $j_1$, both independent of $i$. Denote by $\ms_i$  the size of $\mR_i$.
 
\

\noindent {\bf Step 1.} We must have $\ms_i\rightarrow\infty$. Suppose not, then passing to a subsequence we may assume $\ms_i\leq C$ uniformly. Passing to a further subsequence we may reduce to the following two cases. We will draw a contradiction in both cases.

\

{\bf Case (a).} $\ms_i\geq C'>0$ uniformly. Then using the analysis in Section \ref{sec:degeneration of rod structures} we obtain a limit harmonic map $\Phi_\infty$ tamed by an enhanced rod structure $(\mR_\infty,\Lambda)$ of Asymptotic Type $\AF_0$ and with all the cone angles $2\pi$. This leads to a toric gravitational instanton $(X, g)$ of Asymptotic Type $\AF_0$. As in the above discussion we may apply Proposition \ref{p:MarsSimon} to conclude that $(X, g)$ must be isometric to either a Schwarzschild metric or a Taub-Bolt metric. Since the Asymptotic Type of $(X,g)$ is $\AF_0$, we conclude that it is the Schwarzschild metric. Since $\mR_i$ has more than 2 turning points, there must be turning points coming together, then we arrive at a contradiction by Proposition \ref{p:angle comparision before limit} and Proposition \ref{p:deep scale angle goes to infinity}.

\

{\bf Case (b).} $\ms_i\rightarrow 0$ as $i\rightarrow\infty$. In this case we may apply the discussion in Section \ref{sss:limit of cone angles when turning points come together}. We rescale the sequence by $\ms_i^{-1}$ and consider the first bubble limit $(\mR_{\infty, p}', \Phi_{\infty, p}, \nu_{\infty, p})$. It is easy to see that $\mR_{\infty, p}'$ must be of Asymptotic Type $\AF_0$. Since the cone angles along all the rods are fixed to be $2\pi$, we arrive at a contradiction from Proposition \ref{p:cone angle AF rod length go to zero}.

\

\noindent {\bf Step 2.} We must have $H_i/S_i\rightarrow 0$. Suppose otherwise, then consider the scaling down by $\ms_i^{-1}\sim H_i^{-1}\to0$. In the limit the horizontal rod $\mI_{j_0}$ survives. We may argue as in the proof of Proposition \ref{p:cone angle AF rod length go to infinity} to conclude that the angle $\vartheta_{j_0}(\bl^i)\rightarrow 0$. Contradiction. 

\

\noindent {\bf Step 3.} From \textbf{Step 2} we know that $\ms_i\sim S_i$. We now consider the twisted scaling-down 
\begin{equation}\label{eq:first twist scale down}
    (\mR_{0,i},\Phi_{0,i},\nu_{0,i}):=(D_{\sqrt{S_i}}U_{\beta_i}.\varphi_{S_i^{-1}}^*\mR_i,D_{\sqrt{S_i}}U_{\beta_i}.\varphi_{S_i^{-1}}^*\Phi_i,\varphi_{S_i^{-1}}^*\nu_i),
\end{equation} 
which converges to the first bubble limit $(\mR, \Phi, \nu)$ where the slant rod $\mI_{j_1}$  survives.  As in Section \ref{sss:limit of cone angles with turning points go to infinity} we can analyze the twist action on the lattice $\Lambda$ due to the scaling effect. As a result we see that the normalized rod vector along a horizontal rod for $D_{\sqrt{S_i}}U_{\beta_i}.\varphi_{S_i^{-1}}^*\Phi_i$ becomes 
$$\overline{\mv}_{0,i, H}=(\beta_i, S_i^{-1})^\top,$$
 the normalized rod vector along a slant rod becomes 
$$\overline{\mv}_{0,i,S}=(1-\beta_i, -S_i^{-1})^\top,$$
and the normalized rod vector a long a vertical rod is still 
$$\overline{\mv}_{0}=\overline{\mv}_n=\bv_0=(1,0)^\top.$$ 
Notice that we have 
$$\lim_{i\rightarrow\infty} \overline{\mv}_{0,i, S}=\overline{\mv}_0.$$
See Figure \ref{Figure:the first bubble for Step 3} for a picture illustrating the scaling from $(\mR_i,\Phi_i,\nu_i)$ to $(\mR_{0,i},\Phi_{0,i},\nu_{0,i})$.
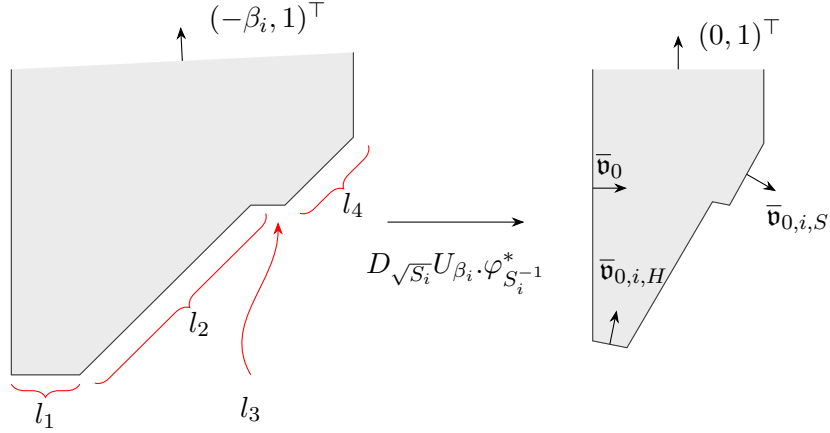
\begin{figure}[ht]
    \begin{tikzpicture}[scale=0.45]
        \draw (0,9)--(0,0)--(2,0)--(7,5)--(8,5)--(10,7)--(10,9.5);
        \filldraw[color=gray!50,fill=gray!50,opacity=0.3](0,9)--(0,0)--(2,0)--(7,5)--(8,5)--(10,7)--(10,9.5);

        \draw[-Stealth] (5,9.25)--(4.95,10.25);
        \node at (7.5,10.5) {$(-\beta_i,1)^\top$};

        \draw (17,9)--(17,1)--(18,0.8)--(20.5,5.1)--(21,5)--(22,6.82)--(22,9);
        \filldraw[color=gray!50,fill=gray!50,opacity=0.3] (17,9)--(17,1)--(18,0.8)--(20.5,5.1)--(21,5)--(22,6.82)--(22,9);

        \draw[-Stealth] (19.5,9)--(19.5,10);
        \node at (21.3,10) {$(0,1)^\top$};

        \draw[-Stealth] (17,5.5)-- node [above]{$\overline\mv_0$} (18,5.5);
        \draw[-Stealth] (17.5,0.9)--(17.7,1.9);
        \node at (18.2,3) {$\overline{\mv}_{0,i,H}$};
        \draw[-Stealth] (21.5,5.91)--(22.36,5.41);
        \node at (23,4.6) {$\overline{\mv}_{0,i,S}$};

        \draw[-Stealth] (11,4.5)--(15,4.5);

        \draw [draw = red,decorate,decoration={brace,amplitude=5pt,mirror,raise=4ex}]
  (0,1.2)--(2,1.2) node[midway,yshift=-3em]{};
        \node at (1,-1.2) {$l_1$};
        \draw [draw = red,decorate,decoration={brace,amplitude=5pt,mirror,raise=4ex}]
  (1.4,0.8)--(6.4,5.8) node[midway,yshift=-3em]{};
        \node at (5.5,1.5) {$l_2$};
        \draw [draw = red,decorate,decoration={brace,amplitude=5pt,mirror,raise=4ex}]
  (7.4,5.8)--(9.4,7.8) node[midway,yshift=-3em]{};
        \node at (10,5) {$l_4$};

        \draw [draw = red,-Stealth] (7,0) to [out=120,in=270](7.8,4.5);
        \node at (7,-1) {$l_3$};

        \node at (13,3) {$D_{\sqrt{S_i}}U_{\beta_i}.\varphi_{S_i^{-1}}^*$};

    \end{tikzpicture}
	\caption{The twisted scaling from $\mR_i$ to $\mR_{0,i}$. In this example $l_2\gg l_4\gg l_1\gg l_3$, and $S_i=l_2\to\infty,H_i=l_1$.}
	 \label{Figure:the first bubble for Step 3}   
\end{figure}

\

\noindent {\bf Step 4.} As in Section \ref{sec:degeneration of rod structures} after passing to a subsequence we obtain a bubble tree convergence. More precisely,  we first have the convergence of $(\mR_{0,i},\Phi_{0,i},\nu_{0,i})$ to a weak limit $(\mR,\Phi,\nu)$. For each cluster $\mC_{i,p}$ of $\mR_{0,i}$ with $m_p>1$ we consider the rescaled rod structures $\mR_{\mC_{i,p}}$ as in Section \ref{ss:uniform control of tame harmonic maps}, and the associated twisted rescaled augmented harmonic maps $(\Phi_{\mC_{i,p}},\nu_{\mC_{i,p}})$. Taking  the limits of $\mR_{\mC_{i,p}}$ and $(\Phi_{\mC_{i,p}},\nu_{\mC_{i,p}})$, we obtain the first level of bubble limits. Repeating this process for all of the further clusters we obtain all levels of bubble limits.

In this step we consider the bubble limit $(\mR', \Phi',\nu')$ where the rod $\mI_{j_0}$ survives and the goal is to show that $C^{-1}\leq H_i\leq C$.
We denote by $\mI'$ the rod of $\mR'$ that contains $\mI_{j_0}$. Notice that in general $\mI'$ may be bigger than $\mI_{j_0}$ due to possible merging of rods. By assumption, in all bigger scales no horizontal rod can survive.  We want to compute the normalized rod vector  along $\mI_{j_0}$ with respect to  $(\mR',\Phi',\nu')$. For this purpose we will compute the twist necessary in order to obtain the bubble limit $(\mR',\Phi',\nu')$. There are three cases. 

\

{\bf Case (A).} Both infinite rods of $\mR'$ arise from a slant rod. In particular, $\mR'$ is of Type $\AF_0$. In this case, let us label all the intermediate bubble limits between the bubble limit $(\mR', \Phi',\nu')$ and the bubble limit $(\mR, \Phi,\nu)$, in the order of decreasing scales as
$$(\mR_1', \Phi_1',\nu_1'),\ldots, (\mR_{K-1}', \Phi_{K-1}',\nu_{K-1}'),$$
with 
$$(\mR_0', \Phi_0',\nu_0'):=(\mR, \Phi,\nu),\quad (\mR_K', \Phi_K',\nu_K'):=(\mR', \Phi',\nu').$$ 
Note that the bubble $(\mR_0', \Phi_0',\nu_0')$ has its infinite rods arising from $\mI_0$ and $\mI_n$. From our specific choice of the rod vectors for $\mR_i$, we know that there is a unique $k_0\in\{1,\ldots,K\}$, such that 
\begin{itemize}
    \item all the bubbles $(\mR_k', \Phi_k',\nu_k')$ with $1\leq k<k_0$ have one of their infinite rods arising from $\mI_0$ or $\mI_n$ and the other infinite rod arising from a slant rod;
    \item all the bubbles $(\mR_k', \Phi_k',\nu_k')$ with $k\geq k_0$ have both infinite rods arising from a slant rod.
\end{itemize}
It is possible to have $k_0=1$, in which case all of the intermediate bubble limits have their infinite rods both arising from a slant rod. Moreover, for each bubble limit $(\mR_k',\Phi_k',\nu_k')$ with $k<k_0$, it must have Asymptotic Type $\AF_0$. For otherwise, the infinite rod vectors of $(\mR_k',\Phi_k',\nu_k')$  are different. In particular, if we denote  by the twisted rescaled augmented harmonic map that gives rise to this bubble limit, and denote the corresponding normalized rod vectors along the horizontal/slant/vertical rods as
$$\overline{\mv}_{k,i,H},\quad \overline{\mv}_{k,i,S},\quad \overline{\mv}_{k,i,0}$$
respectively, then we always have the additive relation
\begin{equation}\label{eq:additive relation}
    \overline{\mv}_{k,i,H}=-\overline{\mv}_{k,i,S}+\overline{\mv}_{k,i,0}.
\end{equation}
The assumption that the infinite rod vectors of $(\mR_k',\Phi'_k,\nu'_k)$ are different implies that the naturally associated lattice, which is generated by the limit of $\overline{\mv}_{k,i,S},\overline{\mv}_{k,i,0}$, is non-degenerate.
Since both infinite rods of $\mR'$ arise from a slant rod, by passing to the first non-trivial bubble limit of the convergence $(\mR_{k,i},\Phi_{k,i},\nu_{k,i})\to (\mR_k',\Phi_k',\nu_k')$ where one of the finite rod would be $\mI_{j_0}$, by Proposition \ref{p:deep scale angle goes to infinity} we arrive at a contradiction. Notice that this argument fails if $(\mR_k',\Phi_k',\nu_k')$ is of Asymptotic Type $\AF_0$, since in this case the associated lattice for $(\mR_k',\Phi_k',\nu_k')$ degenerates and one can not apply Proposition \ref{p:deep scale angle goes to infinity} directly.

If $k_0=1$, from the above and the discussion in Section \ref{ss:uniform control of tame harmonic maps} we know that the normalized rod vector along $\mI'$ is given by the limit of 
\begin{equation}\label{eq:twisted normalized horizontal rod vector}
    \overline{\mv}_{i,H}'=\sqrt{\mu_i^{-1}}D_{\sqrt{\mu_i}}.P_{\mv_{i, S}}.\overline{\mv}_{i, H},
\end{equation} 
where $\mu_i$ is the scale such that the bubble limit $(\mR',\Phi',\nu')$ arises, so $\mu_i\sim S_i^{-1}H_i$. Notice that we only need to rotate once using $P_{\mv_{i, S}}$, since all the intermediate scales share the same infinite rod vectors, and the rest contribution to the twist all comes from diagonal matrices of the form $D_\lambda$. Also from the explicit formula of $\overline{\mv}_{i, S}$ we know that $P_{\mv_{i, S}}$ is given by a counter-clockwise rotation by angle $\theta_i$, which is small and at the order $S_i^{-1}(1-\beta_i)$. By construction we know $C^{-1}\leq |\overline{\mv}_{i,H}'|\leq C$ for a uniform constant $C$. It is now a straightforward computation based on \eqref{eq:twisted normalized horizontal rod vector} and the explicit expression of $\overline{\mv}_{i,H}$ to see that $C^{-1}\leq H_i\leq C$.
In this case it is also easy to see that the total twist $$\sqrt{\mu_i^{-1}}D_{\sqrt{\mu_i}}.P_{\mv_{i, S}}.\sqrt{S_i^{-1}}D_{\sqrt{S_i}}U_{\beta_i}$$ is uniformly non-degenerate, such that $\Lambda$ induces a natural enhancement $\Lambda'$ for $\mR'$.  By assumption the cone angle along each rod in $\mR'$ is $2\pi$. It follows that $(\mR',\Phi',\nu', \Lambda')$ defines a 4-dimensional complete toric Ricci-flat manifold $(X',g')$. A priori from Section \ref{sec:degeneration of rod structures} we only know $\Phi'$ is globally tame by $\mR'$, but by Proposition \ref{p:ALF/AF rigidity}  we may conclude that $(X', g')$ is a toric gravitational instanton of Asymptotic Type $\AF_0$. An application of the uniqueness result Proposition \ref{p:MarsSimon} then gives that $\mR'$ has exactly one finite rod and $(X', g')$ must be a Schwarzschild metric.

If $k_0>1$, the situation is more complicated. Let us denote by $\mu_{1,i}$ the scale where the bubble limit $(\mR_{k_0-1},\Phi_{k_0-1},\nu_{k_0-1})$ arises, so $\mu_{1,i}$ is comparable to the length of the finite rods in $\mR_{0,i}$ that survive in the bubble limit $(\mR_{k_0-1},\Phi_{k_0-1},\nu_{k_0-1})$. Let us moreover denote by $\mu_{2,i}\mu_{1,i}$ the scale where the bubble $(\mR_{K},\Phi_{K},\nu_{K})$ arises, hence $\mu_{2,i}\mu_{1,i}$ is comparable to the length of the rod $\mI_{j_0}$ in $\mR_{0,i}$, which is then comparable to $S_i^{-1}H_i$.
For simplicity, also set 
\begin{equation}\label{eq:normalized rod vector u at k0-1}
    \overline{\mathfrak{u}}_{i,S}:=\sqrt{\mu_{1,i}^{-1}}D_{\sqrt{\mu_{1,i}}}.P_{\mv_0,\mv_{i,S}}.\overline{\mv}_{i,S}.
\end{equation}
By assumption the bubble limit $(\mR_{k_0-1},\Phi_{k_0-1},\nu_{k_0-1})$ has one of its infinite rods arising from $\mI_0$ or $\mI_n$, and the other infinite rod arising from a slant rod. 
The vector $\overline{\mathfrak{u}}_{i,S}$ is precisely the normalized rod vector along the slant rod, and the corresponding rod vector  is converging to $[(1,0)^\top]$ as $i\to\infty$, since the Asymptotic Type of $\mR_{k_0-1}$ is $\AF_0$. The rotation  $P_{\mv_0,\mv_{i,S}}$ is counter-clockwise by angle $\theta_{1,i}$, which is small at the order $S_i^{-1}(1-\beta_i)$. Simple computation gives that $S_i^{-1}\mu_{1,i}^{-1}\to0$.

From the discussion in Section \ref{ss:uniform control of tame harmonic maps} we know that in the bubble limit $(\mR',\Phi',\nu')$ the normalized rod vector along $\mI'$ is given by the limit of
\begin{equation}\label{eq:twisted normalized horizontal rod vector 2}
    \overline{\mv}_{i,H}'=\sqrt{\mu_{2,i}^{-1}}D_{\sqrt{\mu_{2,i}}}.P_{\overline{\mathfrak{u}}_{i,S}}.\sqrt{\mu_{{1,i}}^{-1}}D_{\sqrt{\mu_{1,i}}}.P_{\mv_{0},\mv_{i,S}}.\overline{\mv}_{i,H}.
\end{equation}
This is because in total we only need to rotate twice using $P_{\mv_0,\mv_{i,S}}$ and $P_{\overline{\mathfrak{u}}_{i,S}}$, since all the scales before $(\mR_{k_0},\Phi_{k_0},\nu_{k_0})$ share the same infinite rod vectors, and all the scales from $(\mR_{k_0},\Phi_{k_0},\nu_{k_0})$ to $(\mR_K,\Phi_K,\nu_K)$ also share the same infinite rod vectors.
See Figure \ref{Figure:step 4 case a} for a picture.
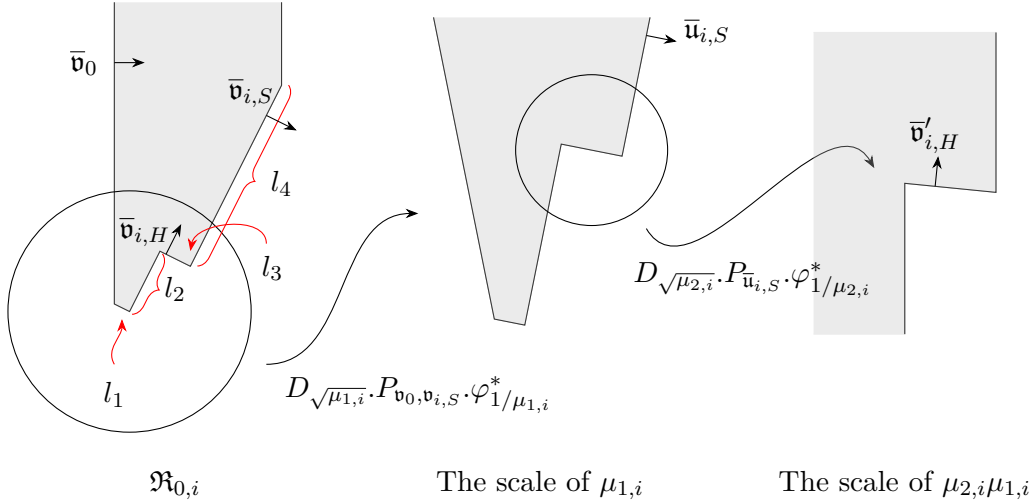
\begin{figure}[ht]
    \begin{tikzpicture}[scale=0.4]
        \draw (-2,10)--(-2,0)--(-1.5,-0.25)--(-0.5,1.75)--(0.5,1.25)--(3.5,7.25)--(3.5,10);
        \filldraw[color=gray!50,fill=gray!50,opacity=0.3](-2,10)--(-2,0)--(-1.5,-0.25)--(-0.5,1.75)--(0.5,1.25)--(3.5,7.25)--(3.5,10);
        \draw (-1.5,-0.25) circle (4);
        \draw [-Stealth] (3,-2) to [out=0,in=180] (8,3);
        \draw [draw = red,decorate,decoration={brace,amplitude=5pt,mirror,raise=4ex}] (-0.75,2)--(2.25,8) node[midway,yshift=-3em]{};
        \node at (3.5,4) {$l_4$};
        \draw [draw = red,decorate,decoration={brace,amplitude=5pt,mirror,raise=4ex}] (-2.9,0.4)--(-1.9,2.4) node[midway,yshift=-3em]{};
        \node at (0,0.5) {$l_2$};
        \draw [draw = red,-Stealth] (-2,-2) to [out=120,in=270](-1.75,-0.5);
        \node at (-2,-3) {$l_1$};
        \draw [draw = red,-Stealth] (3,2) to [out=90,in=80](0.5,1.75);
        \node at (3.2,1.2) {$l_3$};

        \draw[-Stealth] (-2,8)--(-1,8); 
        \node at (-3,8) {$\overline\mv_0$};

        \draw[-Stealth] (3,6.25)--(4,5.75);
        \node at (2.5,7) {$\overline\mv_{i,S}$};

        \draw[-Stealth] (-0.3,1.65)--(0.2,2.65);
        \node at (-1,2.4) {$\overline\mv_{i,H}$};

        \draw (8.5,9.5)--(10.5,-0.5)--(11.5,-0.7)--(12.7,5.3)--(14.7,4.9)--(15.62,9.5);
        \filldraw[color=gray!50,fill=gray!50,opacity=0.3](8.5,9.5)--(10.5,-0.5)--(11.5,-0.7)--(12.7,5.3)--(14.7,4.9)--(15.62,9.5);
        
        \draw (13.7,5.1) circle (2.5);
        \draw [-Stealth] (15.5,2.5) to [out=-60,in=120] (23,4.5);

        \draw[-Stealth] (15.5,8.9)--(16.5,8.7);
        \node at (17.5,9) {$\overline{\mathfrak{u}}_{i,S}$};

        \draw (24,-1)--(24,4)--(27,3.7)--(27,9);
        \filldraw[color=gray!50,fill=gray!50,opacity=0.3](21,-1)--(24,-1)--(24,4)--(27,3.7)--(27,9)--(21,9); 
        \draw[-Stealth] (25,3.9)--(25.1,4.9);
        \node at (25,5.6) {$\overline\mv_{i,H}'$};

        \node at (0,-6) {$\mR_{0,i}$};
        \node at (12,-6) {The scale of $\mu_{1,i}$};
        \node at (24,-6) {The scale of $\mu_{2,i}\mu_{1,i}$};

        \node at (8,-3) {$D_{\sqrt{\mu_{1,i}}}.P_{\mv_{0},\mv_{i,S}}.\varphi^*_{1/\mu_{1,i}}$};
        \node at (19,0.8) {$D_{\sqrt{\mu_{2,i}}}.P_{\overline{\mathfrak{u}}_{i,S}}.\varphi^*_{1/\mu_{2,i}}$};
        
    \end{tikzpicture}
	\caption{A twisted scaling-down $(\mR_{0,i},\Phi_{0,i},\nu_{0,i})$ with four finite rods $l_1\ll l_3\ll l_2\ll l_4\sim1$ and the intermediate scales between the bubble $(\mR,\Phi,\nu)$ and $(\mR',\Phi',\nu')$. In this example $k_0=2$, and $\mu_{1,i}\sim l_2,\mu_{2,i}\sim l_3/l_2$.}
	 \label{Figure:step 4 case a}   
\end{figure}
The rotation $P_{\overline{\mathfrak{u}}_{i,S}}$ is  counter-clockwise by angle $\theta_{2,i}$ that is small at the order $S_i^{-1}(1-\beta_i)\mu_{1,i}^{-1}$. A direct computation based on \eqref{eq:twisted normalized horizontal rod vector 2} and the above give that 
$$\overline\mv_{i,H}'\sim(\beta_i,\frac{S_i^{-1}}{\mu_{1,i}\mu_{2,i}})\sim(\beta_i,\frac{1}{H_i}).$$
Because of the convergence to the bubble $(\mR',\Phi',\nu')$, we have $C^{-1}\leq|\overline{\mv}_{i,H}'|\leq C$, which forces $C^{-1}\leq H_i\leq C$. Similar to the case $k_0=1$, we know the total twist is controlled and $\Lambda$ induces a natural enhancement $\Lambda'$ for $(\mR',\Phi',\nu')$. The toric gravitational instanton $(X',g')$ associated to $(\mR',\Phi',\nu',\Lambda')$ must be a Schwarzschild space by Proposition \ref{p:MarsSimon}.

\

{\bf Case (B).} Exactly one infinite rod of $\mR'$ arises from the rod $\mI_0$ of $\mR_i$. By assumption the other infinite rod of $\mR$ must come from a slant rod. It then follows that all the intermediate scales between the scale giving the bubble $(\mR', \Phi',\nu')$ and the scale giving the bubble $(\mR, \Phi,\nu)$ also have one infinite rod from $\mI_0$ and the other infinite rod from a slant rod.  Similar to {\bf Case (A)}, since passing to the bubble limit $(\mR',\Phi',\nu')$ the horizontal rod $\mI_{j_0}$ survives, if we denote the normalized rod vector along the horizontal/slant/vertical rods in $(\mR',\Phi',\nu')$ as 
$$\overline{\mv}_{\infty,H}',\quad \overline{\mv}_{\infty,S}',\quad \overline{\mv}_{\infty,0}',$$
then we also have the additive relation
\begin{equation}
    \overline{\mv}_{\infty,H}'=-\overline{\mv}_{\infty,S}'+\overline{\mv}_{\infty,0}'.
\end{equation}
This forces the infinite rod vectors of $(\mR',\Phi',\nu')$ to be different, as otherwise by the relation above $\mR'$ would have only one rod vector (since the rod vectors of the slant/vertical rods in $\mR'$ now are all the same which with the additive relation implies the rod vectors of the horizontal rods are also the same), and in particular, up to sign,  $\overline{\mv}_{\infty,H}'=\overline{\mv}_{\infty,S}'=\overline{\mv}_{\infty,0}'$. This violates the additive relation above.

To obtain the normalized rod vector of $(\mR',\Phi',\nu')$ along $\mI_{j_0}$, note that starting from $(\mR_{0,i},\Phi_{0,i},\nu_{0,i})$ we only need to rotate once using $P_{\mv_0,\mv_{i,S}}$, and the rest of the twist is solely contributed by diagonal matrices. The rotation by $P_{\mv_0,\mv_{i,S}}$ is again a counter-clockwise rotation with angle $\theta_i$ small at the order $S_i^{-1}(1-\beta_i)$.
It follows that the normalized rod vector along $\mI'$ is given by the limit of 
$$\overline\mv_{i,H}'=\sqrt{\mu_i^{-1}}D_{\sqrt{\mu_i}}.P_{\mv_{i, S}, \mv_0}.\overline\mv_{i, H},$$ 
where $\mu_i$ is the scale such that the bubble limit $(\mR',\Phi',\nu')$ arises hence $\mu_i\sim S_i^{-1}H_i$.  A computation shows the two infinite rod vectors are comparable to $[(1,\pm H_i^{-1})^\top]$ respectively. Since the infinite rod vectors of $(\mR',\Phi',\nu')$ are different, again this implies that $C^{-1}\leq H_i\leq C$.

In this case from the discussion in Section \ref{ss:uniform control of tame harmonic maps} we know $\mR'$ is of Asymptotic Type either $\Ae$ or $\ALFm$. It cannot be of Asymptotic Type $\ALFp$, since in $\mR_{0,i}$, between the rod vector $\mv_0=[(1,0)^\top]$ of $\mI_0$ and the slant rod vector $\mv_{0,i,S}=[(1-\beta_i,-S_i^{-1})^\top]$ of the slant rods we have
\begin{equation}\label{eq:alpha between I0 and slant}
    \alpha_{0,S}:=\mv_0\wedge\mv_{0,i,S}\sim S_{i}^{-1}\to0.
\end{equation}
Hence by our discussion of the five cases in Section \ref{ss:uniform control of tame harmonic maps}, if the infinite rod vectors of $\mR'$ are non-parallel, it is of Asymptotic Type $\ALFm$ or $\Ae$.
In either case as in {\bf Case (A)} one can argue that there is a natural enhancement $\Lambda'$ in the limit and we obtain a complete toric Ricci-flat 4-manifold $(X, g)$. If $\mR'$ is of Asymptotic Type $\Ae$, then by Proposition \ref{p:AE rigidity} we know $(X, g)$ is an asymptotically Euclidean gravitational instanton. But then $(X, g)$ has to be flat by the Bishop-Gromov volume comparison. So this case can not happen and $(X, g)$ is a toric gravitational instanton of Asymptotic Type $\ALFm$. As in {\bf Case (A)} one can then argue that $\mR'$ has exactly one finite rod and $(X', g')$ is an anti-Taub-Bolt space by Proposition \ref{p:MarsSimon}.

\

{\bf Case (C).} Exactly one infinite rod of $\mR$ arises from the rod $\mI_n$ of $\mR_i$. This is similar to the discussion in {\bf Case (B)} with only one difference that the bubble limit $(\mR',\Phi',\nu')$ now is of Asymptotic Type $\ALFp$. We have $C^{-1}\leq H_i\leq C$ and the bubble limit defines a Taub-Bolt metric by Proposition \ref{p:MarsSimon}. 

\

\noindent {\bf Step 5.}
In all the three cases in {\bf Step 4} we have shown that the maximum length of a horizontal rod is controlled by $C^{-1}\leq H_i\leq C$. Examining our arguments above we see that a horizontal rod will have a similar length control if it is the first one to see when we follow the bubble tree convergence, namely, in all of its previous scales no horizontal rods survive. 

Next we want to say that there is no further bubbling after this scale. As a consequence this shows that all horizontal rods are uniformly controlled. We need to divide into three cases according to {\bf Step 4}. 

\

{\bf Case (A).} If there is a further bubbling at some point, then we go to the first non-trivial bubble limit starting from this scale. The new bubble limit will have to be of Asymptotic Type $\AF_0$ or $\Ae$. In either case, by Proposition \ref{p:cone angle AF rod length go to zero} or \ref{p:angle comparision before limit} (together with the Bishop-Gromov volume comparison in the $\Ae$ case) we end up with a contradiction. The key point is that in the bubble limit $(\mR',\Phi',\nu')$, the naturally associated lattice $\Lambda'$ is non-degenerate so these two Propositions are applicable to further deeper scales of it.

\

{\bf Case (B).} and {\bf Case (C).} can be dealt with in a similar fashion. 

\

\noindent {\bf Step 6} It follows from {\bf Step 5} that all the slant rods have a uniform lower bound on their length. Combining with {\bf Step 4} we know their lengths must all go to infinity, for otherwise at the scale of horizontal rods (the length of them in $\mR_i$ all have uniform upper bound and lower bound away from zero by {\bf Step 4}) we would obtain a limit rod structure with more than one finite rod, which again contradicts the uniqueness result in  Proposition \ref{p:MarsSimon}. 

It then follows that if we consider pointed Gromov-Hausdorff limits based on a point on a horizontal rod we always obtain a  limit rod structure with exactly one finite rod. Again using Proposition \ref{p:ALF/AF rigidity} and Proposition \ref{p:MarsSimon} we know  that the limit gravitational instanton is either a (anti-)Taub-NUT space, a (anti-)Taub-Bolt space, or a Schwarzschild space. More precisely:
\begin{itemize}
    \item If the pointed Gromov-Hausdorff limit is based on $\mI_1$, then since as in \eqref{eq:alpha between I0 and slant} we have $\alpha_{0,S}\to0$, by our consideration in Section \ref{ss:uniform control of tame harmonic maps} and the structure of $\mR_{i,0}$, we know the limit is an anti-Taub-Bolt space.

    \item If $n=2k+1$, then the rod $\mI_{n-1}$ is a horizontal rod. Similar consideration shows that the pointed Gromov-Hausdorff limit based on $\mI_{n-1}$ is a  Taub-Bolt space, since in $\mR_{0,i}$ between the slant rod vector $\mv_{0,i,S}=[(1-\beta_i,-S_i^{-1})^\top]$ of $\mI_{n-2}$ and the rod vector $\mv_0=[(1,0)^\top]$ of $\mI_n$ we have
    $$\alpha_{S,0}:=\mv_{0,i,S}\wedge \mv_{0}\sim S_i\to\infty.$$

    \item If $n=2k$, then we can instead consider the pointed Gromov-Hausdorff limit based on the fixed point corresponding to the turning point $z_n$, and similar consideration shows that the limit is a  Taub-NUT space.

    \item For all the other horizontal rods, the pointed Gromov-Hausdorff limit is a Schwarzschild space.
\end{itemize}

\

{\bf Step 7}. Using Lemma \ref{lem:Gauss-Bonnet} we have 
$$\frac{1}{8\pi^2}\int_{X_n}|\Rm_{g_\beta}|^2\dvol_{g_\beta}=\chi(X_n)=n+1, $$
and for the Taub-NUT metric $g_{\text{TN}}$, Taub-Bolt metric $g_{\text{TB}}$, and Schwarzschild metric $g_{\text{Sc}}$ we have 
$$\frac{1}{8\pi^2}\int_X|\Rm_{g_{\text{TN}}}|^2\dvol_{g_{TN}}=1,$$
$$\frac{1}{8\pi^2}\int_{X}|\Rm_{g_{\text{TB}}}|^2\dvol_{g_{\text{TB}}}=\frac{1}{8\pi^2}\int_{X}|\Rm_{g_{\text{Sc}}}|^2\dvol_{g_{\text{Sc}}}=2.$$
It follows that the above pointed Gromov-Hausdorff limits have caught all the curvature energy of $g_{\beta_i}$ under the degeneration. Any other pointed Gromov-Hausdorff limit is the limit with locally uniformly bounded curvature.
\end{proof}

The above discussion suggests a possible new gluing construction of $\AFa$ gravitational instantons for special range of $\beta$, which may help understand the uniqueness question:

\begin{conjecture}[Uniqueness conjecture, II]
   In the setting of Theorem \ref{t:adjusting angle}, for all $\beta\in (0, 1)$, the length vector $\bl=\bl(\beta)$ is unique.
\end{conjecture}

A more ambitious question is to prove general black hole uniqueness theorem with fixed topology

\begin{question}[Uniqueness conjecture, III]
    For each $\beta\in (0,1)$, is the $\AFa$ gravitational instanton on $X_n$ discovered in Theorem \ref{t:main} unique up to isometry? 
\end{question}

A positive answer is known when $n=1,2$ by \cite{AADNS} under the assumption of an $S^1$ symmetry, where the uniqueness is proved by exploiting an identity characterising the property of being conformally K\"ahler.

\subsection{Explicit description}
Notice that our construction in Theorem \ref{t:adjusting angle} does not yield an explicit formula for the  axisymmetric harmonic map that corresponds to a toric $\AFa$ gravitational instanton. As discussed in Section \ref{sec:degree 0 and 1}, in the low degree case we 
can further reduce to the setting of axisymmetric harmonic functions, which can be explicitly written down. On the other hand,  there is a general method using integrable systems, namely the \emph{Inverse Scattering Transform} (see for example \cite{BZ}), to construct both Lorentzian and Riemannian solutions to the Einstein equation with symmetry. Indeed, the Chen-Teo metrics were first discovered using the Inverse Scattering Transform. One may wonder whether the gravitational instantons constructed in Theorem \ref{t:adjusting angle} can be recovered using the inverse scattering transform or another ansatz. It is also interesting to develop a systematic study relating the harmonic map viewpoint and the Inverse Scattering Transform viewpoint.

\begin{question}
  Is there an explicit description of axisymmetric harmonic maps which are strongly tamed by a rod structure,  when the degree is bigger than 1?
\end{question}

A related question is 

\begin{question}[Stability]
Are the metrics constructed in Theorem \ref{t:main} stable in terms of the second variation of the Einstein-Hilbert functional?
\end{question}
Notice that Biquard-Ozuch \cite{BO} showed that the Kerr metrics, Taub-Bolt metrics and  Chen-Teo metrics  are all unstable. It is a folklore open question that whether a gravitational instanton is stable if and only if it has special holonomy (i.e., locally K\"ahler up to orientation reversal).

\subsection{Classification}
It is  far-reaching to classify all gravitational instantons. The following is a more sensible question
\begin{question}
    Classify all (simply-connected) toric gravitational instantons. 
\end{question}

One expects that the similar strategy as developed in this paper may generate more examples of gravitational instantons of Asymptotic Type $\AFa$. Similarly, one may ask

\begin{question}
    Are there toric $\ALF^\pm$ or $\ALE$ gravitational instantons which are not locally Hermitian? Equivalently, are there augmented harmonic maps which are strongly tamed by a smooth enhanced rod structure of degree bigger than 1, such that the associated cone angles along all rods are all $2\pi$?
\end{question}

 This is related to the Riemannian black hole uniqueness question in the $\ALF^\pm$ case  and the Bando-Kasue-Nakajima conjecture \cite{BKN} in the $\ALE$ case. There are extra difficulties to construct toric gravitational instantons in these cases: in the $\ALF^\pm$ case the Mars-Simon uniqueness \cite{MarsSimon} gives constraints on the possible torus action;  in the $\ALE$ case due to the scaling invariance of the model map, there is one less rod length parameter than the angle parameters. 

\appendix

\section{Gauge fixing}

We first prove 
\begin{proposition}\label{prop:appendix 1}

    Given a rod structure $\mR$ of  Asymptotic Type  $\sharp\in\{\Ae,\ALFm,\ALFp, \AFa\}$ and a harmonic map $\Phi$ strongly tamed by $\mR$, with the normalized augmentation $\nu$ and the lattice $\Lambda_\mR$, the associated smooth Ricci-flat end $(E,g_{\Lambda_\mR})$ also has  Asymptotic Type  $\sharp$ with decay rate $\delta=1$.
\end{proposition}
\begin{proof}
    We prove the proposition for the $\AF_0$ case and the other cases are similar. Over the end $E$, which is the completion of $(\mathbb{H}^\circ\setminus B_N(0))\times\mathbb{R}^2/\Lambda_\mR$, there is the model metric $g^{\AF_0}$. It suffices to find a diffeomorphism $\varpi:E\to E$, such that $\varpi_*g_{\Lambda_\mR}$ is smooth, and 
    $$|\nabla_{g^{\AF_0}}^k(\varpi_*g_{\Lambda_{\mR}}-g^{\AF_0})|_{g^{\AF_0}}=O(\mathfrak{r}^{-k-1}).$$

    We construct the diffeomorphism on each annuli $A_{2^{j},2^{j+2}}$ with $j\in\mathbb{N}$ large, then patch them together. By the modified rescaling $\Phi_{2^j}=D_{\sqrt{2^j}}.\varphi_{2^{-j}}^*\Phi$, we scale the harmonic map $\Phi$ on $A_{2^{j-1},2^{j+3}}$ to $\Phi_{2^j}$ on $A_{1/2,8}$ which satisfies $d(\Phi_{2^j},\Phi^{\AF_0})\leq C2^{-j}$. Set
    $$\Phi_{2^j}=\frac{1}{\rho}\left(\begin{matrix}
        \rho^2e^{u_j}+e^{-u_j}v_j^2 & e^{-u_j}v_j \\ e^{-u_j}v_j & e^{-u_j} 
    \end{matrix}\right),\quad \Phi=\frac{1}{\rho}\left(\begin{matrix}
        \rho^2e^{u}+e^{-u}v^2 & e^{-u}v \\ e^{-u}v & e^{-u} 
    \end{matrix}\right)$$
    over $A_{1/2,8}$ and $A_{2^{j-1},2^{j+3}}$ as before, then $u_j=\varphi_{2^{-j}}^*u$ and $v_j=2^{-j}\varphi_{2^{-j}}^*v$. The uniform local regularity Proposition \ref{thm:C11 regularity} gives that on $A_{1/2,8}$
    \begin{equation}\label{eq:C1a scale to A12 appendix}
        \|u_j\|_{C^{1,\alpha}}\leq C2^{-j},\quad |\nabla^kv_j|\leq C_k2^{-j}\rho^{2-k}
    \end{equation}
    for $0<\alpha<1$ and $k=0,1,2$.

    \vspace{10pt}
    \noindent\textbf{Step (1). Construction of local gauge in each annulus.}
    \vspace{10pt}

    Now, under the augmentation $\varphi_{2^{-j}}^*\nu$ on $A_{1/2,8}$, with the lattice $\Lambda_{\mR}$, we get  the Ricci-flat metric $h$ associated to $\Phi_{2^j}$ defined on the completion of $A_{1/2,8}\times\mathbb{R}^2/\Lambda_{\mR}$ (denoted by $E_{1/2,8}$).  Define the sector region
    $$Sec_\tau:=\{\tau<\theta<\pi-\tau\}\subset\mathbb{H}^\circ.$$
    For a fixed small $\tau$, over $A_{1/2,8}\cap Sec_\tau$, the harmonic map equation is non-degenerate. Hence, by \eqref{eq:C1a scale to A12 appendix} on $A_{1/2,8}\cap Sec_\tau$ standard elliptic estimates give the improvement 
    \begin{equation}
        \|u_j\|_{C^{k,\alpha}}\leq C_k2^{-j},\quad \|v_j\|_{C^{k,\alpha}}\leq C_k2^{-j}
    \end{equation}
    for any $k\in\mathbb{N}$. Therefore, on the region $\pi^{-1}(A_{1/2,8}\cap Sec_\tau)\subset E_{1/2,8}$, under the coordinate $(x_1,x_2,x_3,x_4):=(r\cos\phi_1,r\sin\phi_1,z,\phi_2)$, we already have
    $$|\nabla_{g^{\AF_0}}^k(h-g^{\AF_0})|_{g^{\AF_0}}\leq \frac{C_k}{2^j}.$$

    As for the region in $A_{1,4}$ near the rod, we cover them by balls $B_t\subset A_{1/2,8}$ with fixed radius $1/4$. Then we cover the near rod region in $E_{1,4}:=\pi^{-1}(A_{1,4})$ by simply-connected open sets $U_{t,s}\subset E_{1/2,8}$ with smooth controlled boundary over $B_t$.  Introduce the coordinate
    $$x_1:=r\cos\phi_1,\quad x_2:=r\sin\phi_1,\quad x_3:=z,\quad x_4:=\phi_2$$
    on $U_{t,s}$. On each $U_{t,s}$, we solve for the harmonic coordinates
    \begin{equation}
        {h}^{kl}\frac{\partial^2}{\partial x_k\partial x_l}\overline{x}_j+{h}^{kl}{\Gamma}^r_{kl}\frac{\partial}{\partial x_r}\overline{x}_j=0,\quad \text{with }\overline{x}_j=x_j\text{ along }\partial U_{t,s}.\label{eq:harmonic coordinates for h appendix}
    \end{equation}
    Since we know the metric $h$ is $C^{0,1}$ under $(x_1,x_2,x_3,x_4)$ by Lemma \ref{lem:Lipschitz metric}, the harmonic coordinates are $W^{2,p}$ for any large $p$, hence is $C^{1,\alpha}$. Our uniform local regularity Proposition \ref{thm:C11 regularity} gives that under $(x_1,x_2,x_3,x_4)$
    \begin{equation}\label{eq:C11 bound for h appendix}
        |h_{kl}-g^{\AF_0}_{kl}|+|\Gamma^r_{kl}|\leq\frac{C}{2^j}.
    \end{equation}
    The elliptic estimate based on \eqref{eq:harmonic coordinates for h appendix} gives $\|\overline{x}_j-x_j\|_{W^{2,p}}\leq C2^{-j}$ for any large $p$, which implies $\|\overline{x}_j-x_j\|_{C^{1,\alpha}}\leq C2^{-j}$ under $(x_1,x_2,x_3,x_4)$. The Ricci-flat metric $h$ is smooth under $(\overline{x}_1,\overline{x}_2,\overline{x}_3,\overline{x}_4)$. Define the following map where we use the coordinates $(x_1,x_2,x_3,x_4)$ for $U_{t,s}$
    $$\zeta_{t,s}:U_{t,s}\to U_{t,s},$$
    $$x\mapsto \zeta_{t,s}(x):=(\overline{x}_1,\overline{x}_2,\overline{x}_3,\overline{x}_4).$$
    Shrinking the domain for $\zeta_{t,s}$ slightly, the map is a $C^{1,\alpha}$ diffeomorphism into its image when $j$ is large. Now the metric $(\zeta_{t,s})_*h$ under $(x_1,x_2,x_3,x_4)$ is smooth. The bound $\|\overline{x}_j-x_j\|_{C^{1,\alpha}}\leq C2^{-j}$ and \eqref{eq:C11 bound for h appendix} imply
    $$\|(\zeta_{t,s})_*h-g^{\AF_0}_{kl}\|_{C^{0,\alpha}}\leq \frac{C}{2^j}.$$
    Since the Ricci-flat equation is elliptic in harmonic coordinates, starting from the above, it is standard to show $\|(\zeta_{t,s})_*h-g^{\AF_0}_{kl}\|_{C^{k,\alpha}}\leq C_k2^{-j}$.

    \vspace{10pt}
    \noindent\textbf{Step (2). Equivariant patching to global gauge in each annulus.}
    \vspace{10pt}

    In the next we need to patch local diffeomorphisms on $\pi^{-1}(A_{1/2,8}\cap Sec_\tau)$ and $U_{t,s}$ together to construct a global equivariant diffeomorphism on $E_{1,4}$. Denote the identity diffeomorphism on $U_{0,0}:=\pi^{-1}(A_{1,4}\cap Sec_\tau)$ as $\zeta_{0,0}$. For the open cover of $E_{1,4}$ by $U_{t,s}$, we take standard cut-off functions $0\leq \chi_{t,s}\leq1$ that
    \begin{itemize}
        \item is smooth under the harmonic coordinate for $U_{t,s}$;
        \item has support in $U_{t,s}$; 
        \item equals one in a slightly smaller open subset of $U_{t,s}$ and $\sum\chi_{t,s}\equiv1$.
    \end{itemize}
    Consider the center-of-mass function
    $$\widetilde{\mathfrak{C}}_j:E_{1/2,8}\times E_{1/2,8}\to\mathbb{R}_+,$$
    $$(x,y)\mapsto \sum_{t,s}\chi_{t,s}(x)\cdot d_{h}(\zeta_{t,s}(x),y)^2.$$
    The function is almost invariant under the $\mathbb{T}$ action, in the sense that for $a\in \mathbb{T}$, $$\|\widetilde{\mathfrak{C}}_j(a\cdot x,a\cdot y)-\widetilde{\mathfrak{C}}_j(x,y)\|_{C^{k,\alpha}}\leq C_k2^{-j},$$ where the $C^{k,\alpha}$ norm is measured under the metric $h$. Denote by $\mathfrak{C}_j$ the average of  $\widetilde{\mathfrak{C}}_j$ over the $\mathbb T$ orbits.
    Now for any $x\in E_{1,4}$, we can minimize the function $y\mapsto \widetilde{\mathfrak{E}}_j(x,y)$. The minimizing point must take place at $y$  close to $x$, hence by convexity, the minimizing point is unique. Denote the minimizing point as $\mathfrak{z}_j(x)$. Then the map 
    $$\mathfrak{z}_j:E_{1,4}\to E_{1/2,8}$$
    is a $\mathbb{T}$-equivariant diffeomorphism into its image, and the push-forward metric satisfies
    $$\|(\mathfrak{z}_j)_*h-g^{\AF_0}_{kl}\|_{C^{k,\alpha}}\leq\frac{C_k}{2^j}.$$

    To recover the diffeomorphism on the annulus $A_{2^k,2^{j+2}}$, we can scale $A_{1,4}$ back to $A_{2^j,2^{j+2}}$, where the metric $h$ is rescaled to $2^{2j}h$. Because of the modified rescaling for $\Phi$, after further taking the quotient of $2^{2j}h$ by the $\mathbb{Z}_{2^j}$ action along $\phi_2$, we exactly recover the metric $g_{\Lambda_{\mR}}$ on $\pi^{-1}(A_{2^j,2^{j+2}})$. Since $\mathfrak{z}_j$ is $\mathbb{T}$-equivariant, it induces a $\mathbb{T}$-equivariant diffeomorphism $\overline{\mathfrak{z}}_j$ on $\pi^{-1}(A_{2^j,2^{j+2}})$ smoothing $g_{\Lambda_{\mR}}$ with the desired regularity.

    \vspace{10pt}
    \noindent\textbf{Step (3). Equivariant patching of adjacent annuli.}
    \vspace{10pt}

    Finally, we need to patch the diffeomorphisms $\overline{\mathfrak{z}}_j$ and $\overline{\mathfrak{z}}_{j+1}$ on the overlap of adjacent annuli $A_{2^j,2^{j+2}}$ and $A_{2^{j+1},2^{j+3}}$. Recall for the construction of $\mathfrak{z}_{j+1}$, we need to consider the modified rescaling $\Phi_{2^{j+1}}$ on $A_{1/2,8}$ and repeat the construction in \textbf{Step (1)}, which is equivalent to consider the modified rescaling $\Phi_{2^j}$ on $A_{1,16}$. Therefore, it suffices to patch the $\mathbb{T}$-equivariant diffeomorphisms $\mathfrak{z}_j$ on $A_{1,4}$ and $\mathfrak{z}_{j+1}$ on $A_{2,8}$. One just needs to perform the center-of-mass construction again. Hence we have obtained the desired global diffeomorphism $\varpi$.
\end{proof}

\

For any two harmonic maps $\Phi^1,\Phi^2$ that are strongly tamed by two rod structures $\mR_1,\mR_2$ with the same  Asymptotic Type  $\sharp\in\{\ALFm,\ALFp, \AFa\}$ and the same infinite rod vectors,  suppose moreover that 
\begin{equation}\label{eqn:uniform smallness} d(\Phi^1,\Phi^2)<\frac{\epsilon}{r}
\end{equation}
for small $\epsilon$ and $r>N$, where $N$ is a fixed large constant. Then under the normalized augmentations $\nu_1,\nu_2$, the lattices $\Lambda_{\mR_1}=\Lambda_{\mR_2}$, which we denote by $\Lambda$. These give rise to $C^{0,1}$ Ricci-flat ends $(E,g_{1,\Lambda})$ and $(E,g_{2,\Lambda})$. By Proposition \ref{prop:appendix 1} they are both of  Asymptotic Type  $\sharp$. We next prove

\begin{proposition}\label{prop:appendix 2}
    There is a constant $C(\epsilon)$ depending on $\epsilon$ and $k$ where $C(\epsilon)\to0$ as $\epsilon\to0$, such that there are diffeomorphisms $\varpi_1,\varpi_2:E\to E$, where on the end $(\varpi_1)_*g_{1,\Lambda}$ and $(\varpi_2)_*g_{2,\Lambda}$ are smooth with  Asymptotic Type  $\sharp$ and decay rate $\delta=1$, and with estimates 
    $$|\nabla^k_{g^\sharp}((\varpi_1)_*g_{1,\Lambda}-(\varpi_2)_*g_{2,\Lambda})|_{g^\sharp}<\frac{C(\epsilon)}{\mathfrak{r}^{k+1}}.$$
\end{proposition}

\begin{proof}
    We follow the proof of Proposition \ref{prop:appendix 1}. Again we prove the proposition for the $\AF_0$ case.  The diffeomorphisms $\varpi_1$ and $\varpi_2$ are exactly the ones constructed in the proof of Proposition \ref{prop:appendix 1}, but here we need to control the difference between them. Consider the modified rescaling 
    $$\Phi_{2^j}^1=\frac{1}{\rho}\left(\begin{matrix}
        \rho^2e^{u_j^1}+e^{-u_j^1}(v_j^1)^2 & e^{-u_j^1}v_j^1 \\ e^{-u_j^1}v_j^1 & e^{-u_j^1} 
    \end{matrix}\right),\quad \Phi_{2^j}^2=\frac{1}{\rho}\left(\begin{matrix}
        \rho^2e^{u_j^2}+e^{-u_j^2}(v_j^2)^2 & e^{-u_j^2}v_j^2 \\ e^{-u_j^2}v_j^2 & e^{-u_j^2} 
    \end{matrix}\right)$$
    defined on $A_{1/2,8}$. Then the condition \eqref{eqn:uniform smallness} gives  that on $A_{1/2, 8}$, $$d(\Phi^1_{2^j},\Phi^2_{2^j})\leq C\epsilon2^{-j},$$ which implies
    \begin{equation}\label{eq:distance between harmonic maps appendix}
        |u_j^{1}-u_{j}^2|+\rho^{-1}|v_j^1-v_j^2|\leq C\epsilon 2^{-j}.
    \end{equation}

    \vspace{10pt}
    \noindent\textbf{Step (1). Comparison of local gauge in each annulus.} 
    \vspace{10pt}

    Under the augmentations $\varphi_{2^{-j}}^*\nu_1,\varphi_{2^{-j}}^*\nu_2$ with the lattice $\Lambda$ there are the Ricci-flat metrics $h_1,h_2$ associated to $\Phi^1_{2^j},\Phi^2_{2^j}$ on $E_{1/2,8}$. Since the harmonic map equation is non-degenerate over the sector region $A_{1/2,8}\cap Sec_\tau$, by \eqref{eq:distance between harmonic maps appendix} it is straightforward to conclude that 
    \begin{equation}\label{eq:compare u v epislon appendix}
        \|u_j^1-u_j^2\|_{C^{k,\alpha}}\leq C_k\epsilon 2^{-j},\quad \|v_j^1-v_j^2\|_{C^{k,\alpha}}\leq C_k\epsilon 2^{-j}.
    \end{equation}
    Therefore on the region $U_{0,0}=\pi^{-1}(A_{1/2,8}\cap Sec_\tau)$, under the coordinates $(x_1,x_2,x_3,x_4)$, we already have
    $$|\nabla^k_{g^{\AF_0}}(h_1-h_2)|_{g^{\AF_0}}\leq\frac{C_k(\epsilon)}{2^j}$$

    For the near rod region in $E_{1,4}$ covered by the rest of $U_{t,s}$, we need to compare the harmonic coordinates 
    \begin{equation}
        {h_1}^{kl}\frac{\partial^2}{\partial x_k\partial x_l}\overline{x}_j^1+{h_1}^{kl}{\Gamma_1}^r_{kl}\frac{\partial}{\partial x_r}\overline{x}_j^1=0,\quad \text{with }\overline{x}_j^1=x_j\text{ along }\partial U_{t,s},\label{eq:harmonic coordinates for h1 appendix}
    \end{equation}
    \begin{equation}
        {h_2}^{kl}\frac{\partial^2}{\partial x_k\partial x_l}\overline{x}_j^2+{h_2}^{kl}{\Gamma_2}^r_{kl}\frac{\partial}{\partial x_r}\overline{x}_j^2=0,\quad \text{with }\overline{x}_j^2=x_j\text{ along }\partial U_{t,s},
    \end{equation}
   The solutions are in $W^{2,p}$ hence $C^{1,\alpha}$ on $U_{t,s}$.
    Proposition \ref{thm:C11 regularity} gives that on $U_{t,s}$
    \begin{equation}\label{eq:finer regularity appendix}
        |{h_1}_{kl}-{h_2}_{kl}|+|{\Gamma_1}^r_{kl}-{\Gamma_2}^r_{kl}|\leq \frac{C}{2^j}.
    \end{equation}
    Moreover, on the overlap with the sector region $U_{t,s}\cap U_{0,0}$, the above closeness improves to all the higher derivatives
    \begin{equation}\label{eq:finer regularity appendix with e}
        |\nabla^k({h_1}_{kl}-{h_2}_{kl})|\leq \frac{C_k(\epsilon)}{2^j}.
    \end{equation}
    Combining \eqref{eq:finer regularity appendix}-\eqref{eq:finer regularity appendix with e} it follows
    \begin{align}\label{eq:elliptic equation for difference of harmonic coordinates appendix}
        &{h_1}^{kl}\frac{\partial^2}{\partial x_k\partial x_l}(\overline{x}_j^1-\overline{x}_j^2)+{h_1}^{kl}{\Gamma_1}^r_{kl}\frac{\partial}{\partial x_r}(\overline{x}_j^1-\overline{x}_j^2)\\
        =\ &({h_1}^{kl}-{h_2}^{kl})\cdot A+({\Gamma_1}^{r}_{kl}-{\Gamma_2}^{r}_{kl})\cdot B,\quad \text{with }\overline{x}_j^1-\overline{x}_j^2=0\text{ along }\partial U_{t,s}. \notag
    \end{align}
    Here $A,B$ are functions in $L^p$ on $U_{t,s}$. Set the Green's function for the Dirichlet problem for ${h_1}^{kl}\frac{\partial^2}{\partial x_k\partial x_l}\cdot+{h_1}^{kl}{\Gamma_1}^r_{kl}\frac{\partial}{\partial x_r}\cdot$ on $U_{t,s}$ as $\mathcal{G}$, then 
    \begin{align}
        \overline{x}_j^1-\overline{x}_j^2=&\ \int_{U_{t,s}} (({h_1}^{kl}-{h_2}^{kl})\cdot A+({\Gamma_1}^{r}_{kl}-{\Gamma_2}^{r}_{kl})\cdot B)(y)\cdot\mathcal{G}(x,y)dy\\
        =&\ \int_{U_{t,s}\cap U_{0,0}} (({h_1}^{kl}-{h_2}^{kl})\cdot A+({\Gamma_1}^{r}_{kl}-{\Gamma_2}^{r}_{kl})\cdot B)(y)\cdot\mathcal{G}(x,y)dy \notag\\
        &+\int_{U_{t,s}\setminus U_{0,0}} (({h_1}^{kl}-{h_2}^{kl})\cdot A+({\Gamma_1}^{r}_{kl}-{\Gamma_2}^{r}_{kl})\cdot B)(y)\cdot\mathcal{G}(x,y)dy.\notag
    \end{align}
    So \eqref{eq:finer regularity appendix} and \eqref{eq:finer regularity appendix with e} give
    $$|\overline{x}_j^1-\overline{x}_j^2|\leq \frac{C(\epsilon)}{2^j}+\frac{C(\tau)}{2^j}.$$
    Here $C(\tau)\to0$ as $\tau\to0$. Hence, by choosing $\epsilon\ll\tau$ the elliptic estimate gives that on $U_{t,s}$
    \begin{equation}
        \|\overline{x}_j^1-\overline{x}_j^2\|_{C^{1,\alpha}}\leq \frac{C(\tau)}{2^j}.
    \end{equation}
    The bound improves to 
    \begin{equation}\label{eq:compare harmonic coordinates for h1 h2 on overlap}
        \|\overline{x}_j^1-\overline{x}_j^2\|_{C^{k,\alpha}}\leq {C_k(\tau)}{2^{-j}}
    \end{equation} 
    on the overlap $U_{t,s}\cap U_{0,0}$.
    Since on $U_{t,s}$ we have \eqref{eq:finer regularity appendix}, the ellipticity  implies that on $U_{t,s}$
    $$\|(\zeta_{t,s}^1)_*h_1-(\zeta_{t,s}^2)_*h_2\|_{C^{k,\alpha}}\leq \frac{C_k}{2^j}.$$
    However, since on $U_{t,s}\cap U_{0,0}$ we already have \eqref{eq:compare u v epislon appendix}, on $U_{t,s}\cap U_{0,0}$ we must have $$\|(\zeta_{t,s}^1)_*h_1-(\zeta_{t,s}^2)_*h_2\|_{C^{k,\alpha}}\leq {C(\tau)}{2^{-j}}.$$ Hence by making $\tau$ small and taking integration the above improves to
    $$\|(\zeta_{t,s}^1)_*h_1-(\zeta_{t,s}^2)_*h_2\|_{C^{k,\alpha}}\leq \frac{C_k(\tau)}{2^j}.$$

    \vspace{10pt}
    \noindent\textbf{Step (2). Comparison of equivariant patched gauge.} 
    \vspace{10pt}

    As in the construction performed in \textbf{Step (2)} of Proposition \ref{prop:appendix 1}, we have the center-of-mass functions $\widetilde{\mathfrak{C}}_j^1$ and $\widetilde{\mathfrak{C}}_j^2$. As $\widetilde{\mathfrak{C}}_j^1$ and $\widetilde{\mathfrak{C}}_j^2$ are almost invariant and the minimizing points for $y\mapsto{\mathfrak{C}}_j^1(x,y)$ and $y\mapsto{\mathfrak{C}}_j^2(x,y)$ only occur when $y$ is close to $x$, for simplicity we only need to consider $\widetilde{\mathfrak{C}}_j^1$ and $\widetilde{\mathfrak{C}}_j^2$ when $x$ and $y$ fall in a same $U_{t,s}$. Because of the $C(\tau)2^{-j}$-closeness between the harmonic coordinates and the smoothed metrics on $U_{t,s}$, we have 
    $$\|(\zeta_{t,s}^1)_*\widetilde{\mathfrak{C}}_j^1-(\zeta_{t,s}^2)_*\widetilde{\mathfrak{C}}_j^2\|_{C^{k,\alpha}}\leq C_k(\tau)2^{-j},$$
    where the $C^{k,\alpha}$ norm is measured by the standard coordinate $(x_1,x_2,x_3,x_4)$. Here we need the push-forward by $\mathfrak{z}_j^1,\mathfrak{z}_j^2$ since only after the push-forward $h_1,h_2$ are smooth under $(x_1,x_2,x_3,x_4)$. When $x\in U_{0,0}$, because of \eqref{eq:finer regularity appendix with e} and \eqref{eq:compare harmonic coordinates for h1 h2 on overlap},  there exists a $\sigma:=\sigma(\tau)\gg\tau$ that $\sigma\to0$ as $\tau\to0$, such that when $x\in\pi^{-1}(A_{1,4}\cap Sec_\sigma)$, whenever $\zeta_{t,s}(x)$ makes sense, $\zeta_{t,s}(x)\in U_{0,0}$. Clearly when $x\in\pi^{-1}(A_{1,4}\cap Sec_\sigma)$, for any $a\in\mathbb{T}$, $a\cdot x\in\pi^{-1}(A_{1,4}\cap Sec_\sigma)$ as well. Hence, when we minimize $y\mapsto\mathfrak{C}_j^1(x,y)$ and $y\mapsto\mathfrak{C}_j^2(x,y)$, we only need to consider $y\in U_{0,0}$.

    For $a\in\mathbb{T}$, $x\in\pi^{-1}(A_{1,4}\cap Sec_\sigma)$, and $y\in U_{0,0}$, we have 
    $$\|\widetilde{\mathfrak{C}}_j^1(a\cdot x,a\cdot y)-\widetilde{\mathfrak{C}}_j^1(x,y)\|_{C^{k,\alpha}}\leq C_k(\tau)2^{-j},\quad \|\widetilde{\mathfrak{C}}_j^2(a\cdot x,a\cdot y)-\widetilde{\mathfrak{C}}_j^2(x,y)\|_{C^{k,\alpha}}\leq C_k(\tau)2^{-j}.$$
    Hence $$\|\mathfrak{C}^1_j(x,y)-\mathfrak{C}^2_j(x,y)\|_{C^{k,\alpha}}\leq C_k(\tau)2^{-j}.$$ This implies that on $\pi^{-1}(A_{1,4}\cap Sec_\sigma)$
    $$\|(\mathfrak{z}_j^1)_*h_1-(\mathfrak{z}_j^2)_*h_2\|_{C^{k,\alpha}}\leq C_k(\tau)2^{-j}.$$
    Finally, since we obviously have $$\|(\mathfrak{z}_j^1)_*h_1-(\mathfrak{z}_j^2)_*h_2\|_{C^{k,\alpha}}\leq C2_k^{-j}$$ on the entire $E_{1,4}$, by taking integration we end with $$\|(\mathfrak{z}_j^1)_*h_1-(\mathfrak{z}_j^2)_*h_2\|_{C^{k,\alpha}}\leq C(\tau)2^{-j}$$ on the entire $E_{1,4}$. One can similarly control the difference between gauges on adjacent annuli. Hence after scaling back we finished the proof.
\end{proof}

\section{Mars-Simon uniqueness via positive mass theorem}\label{sec:appendix b}
Mars-Simon \cite{MarsSimon} proved a uniqueness theorem for certain gravitational instantons with an $\mathbb R$ action without isolated fixed points. The proof uses the positive mass theorem and is based on the method of Bunting and Masood-ul-Alam in general relativity. 
In this appendix we explain how to extend \cite{MarsSimon} to our setting in Section \ref{sec:discussion} of toric gravitational instantons, where in general isolated fixed points are allowed, by using a version of positive mass theorem with singularities. 

Let $(M,g)$ be a  toric gravitational instanton that arises from an augmented harmonic map $(\Phi,\nu)$ strongly tamed by an enhanced rod structure $(\Phi, \Lambda)$, where the Asymptotic Type of $\mR$ is assumed to be $\AF_0/\ALFp/\ALFm$ and the augmentation $\nu$ is normalized. Without loss of generality, we assume the infinite rod vectors of $\mR$ are $\bv_{\frac{\alpha}{2}}$ and $\bv_{-\frac{\alpha}{2}}$ with $\alpha\geq0$. 
\begin{proposition}\label{p:MarsSimon}
    Suppose the $S^1$ action induced by $(0,1)^\top\in \Lie(\dT^2)$ on $(M,g)$ is semi-free (i.e., the action is free away from its fixed points), then either $(M,g)$ is Type $\I$ up to reversing the orientation, or $(M,g)$ is isometric to a Taub-Bolt metric or a Schwarzschild metric.
\end{proposition}
\begin{proof}
    We follow closely the proof of \cite{MarsSimon}.
 Denote the quotient of $(M,g)$ by the $S^1$ action as $(N,h)$. The latter is a manifold with smooth boundaries and isolated singularities modeled on cones. More precisely, a local consideration shows that each fixed $2$-sphere of the $S^1$ action gives rise to a component of the boundary of $(N,h)$, and each isolated fixed point of the $S^1$ action gives rise to a  singularity modeled on $dr^2+\frac{1}{2}r^2g_{S^2}$, with $g_{S^2}$ being the round metric.   The metric $h$ is smooth away from the boundary and the isolated  singularities. 
 
 Denote by $K$ the Killing field associated to the $S^1$ action chosen suitably such that $V:=|K|_g\to1$ at infinity. By the discussion in Section \ref{subsec:reduction of Ricci-flat} there is an Ernst potential \footnote{In \cite{MarsSimon} the Ernst potential is denoted as $\omega$ and the Killing field $K$ is denoted as $\xi$.} $\mathcal{E}$  such that
    $$d\mathcal{E}=*(\eta\wedge d\eta)$$
    with $\eta=K^\flat$. From the assumption on the $S^1$ action and the asymptotic model, we know that $(N,h)$ is asymptotically flat at infinity, in the sense that $$h=g_{\mathbb{R}^3}+O(r^{-\delta}),$$ where $\mathfrak{r}$ is the distance to a base point and $\delta>0$.

    A computation shows (see (11)-(13) of \cite{MarsSimon}) that on $(N,h)$
    \begin{align}
        \Ric_{ij}&=V^{-1}\nabla_{i}\nabla_jV-\frac{1}{2}V^{-4}(\mathcal{E}_i\mathcal{E}_j-h_{ij}h^{kl}\mathcal{E}_k\mathcal{E}_l),\\
        \Delta V&=\frac{1}{2}V^{-3}h^{ij}\mathcal{E}_i\mathcal{E}_j,\\
        \Delta\mathcal{E}&=3V^{-1}h^{ij}\mathcal{E}_i V_j,
    \end{align}
    and the scalar curvature of $(N,h)$ is 
    \begin{equation}
        s=\frac{3}{2}V^{-4}g^{ij}\mathcal{E}_i\mathcal{E}_j\geq0.
    \end{equation}
    Since near infinity, $(N,h)$ is clearly asymptotically flat, we know that by adjusting $\mathcal{E}$ with a suitable constant we can assume that 
    $$\mathcal{E}=O(\mathfrak r^{-\epsilon}).$$
    Then we consider the conformal change $\gamma:=V^2h$ 
    and let 
    \begin{align}
        v_{\pm}&:=\frac{1-V^2\mp\mathcal{E}}{1+V^2\pm\mathcal{E}},\\
        \Theta&:=1-v_-v_+=\frac{4V^2}{(1+V^2+\mathcal{E})(1+V^2-\mathcal{E})}.
    \end{align}
    Another computation (see \cite{MarsSimon} Lemma 3-5) shows that $\gamma$ is again asymptotically flat at infinity. Let
    \begin{align}
        \Omega_{\pm}&:=\frac{1\pm\sqrt{\Theta}}{2\sqrt{\Theta}},\\
        g^{\pm}&:=\Omega_{\pm}^2\gamma=\frac{1}{16}(\sqrt{(1+V^2)^2-\mathcal{E}^2}\pm2V)^2h.
    \end{align}
   Define
    $$\mathcal{E}_{\pm}:=V^2\pm\mathcal{E}.$$
   Then a computation shows that
    \begin{align}
        \Delta_{g}\mathcal{E}_{\pm}&=V^{-2}|\nabla_g\mathcal{E}_{\pm}|_g^2\geq0.
    \end{align}
    Combining with $\mathcal{E}=O(\mathfrak{r}^{-\epsilon})$ and by applying the maximum principle on $M$, know that either $\mathcal{E}_{\pm}<1$ or $\mathcal{E}_{\pm}\equiv1$. Moreover, notice the identity $\mathcal{E}_{\pm}=2V^2-\mathcal{E}_{\mp}$, we conclude that $-1<\mathcal{E}_{\pm}\leq1$ and
    \begin{enumerate}
        \item either at least one of $\mathcal{E}_{\pm}$ satisfies that $\mathcal{E}_{\pm}\equiv1$ or
        \item both $\mathcal{E}_{\pm}$ satisfies $-1<\mathcal{E}_{\pm}<1$. 
    \end{enumerate}
    This also justifies that $$\Theta=\frac{4V^2}{(1+\mathcal{E}_{\pm})(1+\mathcal{E}_{\mp})}\leq 1,$$ and the equality holds only when case (1) occurs. By considering the asymptotic expansion at infinity, one can show that $(N,g^+)$ is still asymptotically flat but with vanishing mass at infinity. For $(N,g^-)$ one can show that the metric $g^-$ is incomplete at infinity, but it can be compactified by adding one point at infinity to a $C^2$ metric, with resulting manifold denoted as $(\widetilde{N},\widetilde{g^-})$. Both $g^+$ and $\widetilde{g^-}$ have non-negative scalar curvature. 

    Now, if we are in case (1), as in the first part of the proof of Theorem 1 in \cite{MarsSimon}, then $\gamma$ is locally flat and $(X,g)$ is precisely locally given by the Gibbons-Hawking ansatz. We then know it is Type $\I$ and  there is an explicit description from Section \ref{sec:degree 0 and 1}.

    If we are in case  (2), as in the second part of the proof of Theorem 1 in \cite{MarsSimon}, by local consideration near the boundary of $N$, we can patch $(N,g^{+})$ and $(\widetilde N,\widetilde{g^-})$ together, with the resulting manifold denoted as $(\mathcal{N},g_0)$. The key is that by computation along boundary of $(N,g^{\pm})$ the second fundamental form matches, which implies that $g_0$ is $C^{1,1}$ near the patched boundary. It may also have isolated singularities modeled on the cone $dr^2+\frac{1}{2}r^2g_{S^2}$, and is $C^{1,1}$ at the added point in $\widetilde N$. The space $(\mathcal{N},g_0)$ is smooth elsewhere, with only one asymptotically flat end and zero mass at infinity. A computation gives that $(\mathcal{N},g_0)$ also has non-negative scalar curvature (see (38) in \cite{MarsSimon}). Hence by the positive mass theorem with isolated singularities (see for example Theorem 1.1 in \cite{Dai}), we know $(N,g^{\pm})$ has to be flat. Starting from this, a more detailed computation gives that $(X,g)$ must be Taub-Bolt or Schwarzschild (this is the last paragraph of the proof to Theorem 1 in \cite{MarsSimon}).
\end{proof}

\bibliographystyle{alpha}
\bibliography{reference}

\newcommand{\etalchar}[1]{$^{#1}$}
\begin{thebibliography}{HKWX25}

\bibitem[AA24]{AA}
S.~Aksteiner and L.~Andersson.
\newblock {Gravitational Instantons and Special Geometry}.
\newblock {\em Journal of Differential Geometry}, 128(3):928--958, 2024.

\bibitem[AAD{\etalchar{+}}23]{AADNS}
S.~Aksteiner, L.~Andersson, M.~Dahl, G.~Nilsson, and W.~Simon.
\newblock {Gravitational Instantons with $S^1$ Symmetry}.
\newblock {\em arXiv preprint}, 2023.

\bibitem[BG23]{BG}
O.~Biquard and P.~Gauduchon.
\newblock {On Toric Hermitian ALF Gravitational Instantons}.
\newblock {\em Communications in Mathematical Physics}, 399(1):389--422, 2023.

\bibitem[BGL24]{BGL}
O.~Biquard, P.~Gauduchon, and C.~LeBrun.
\newblock {Gravitational Instantons, Weyl Curvature, and Conformally K\"ahler Geometry}.
\newblock {\em International Mathematics Research Notices}, 2024(20):13295--13311, 2024.

\bibitem[BKN89]{BKN}
S.~Bando, A.~Kasue, and H.~Nakajima.
\newblock {On a Construction of Coordinates at Infinity on Manifolds with Fast Curvature Decay and Maximal Volume Growth}.
\newblock {\em Inventiones Mathematicae}, 97(2):313--349, 1989.

\bibitem[BO23]{BO}
O.~Biquard and T.~Ozuch.
\newblock {Instability of conformally Kähler, Einstein metrics }.
\newblock {\em arXiv preprint}, 2023.

\bibitem[BZ78]{BZ}
V.~A. Belinskiĭ and V.~E. Zakharov.
\newblock {Integration of the Einstein Equations by Means of the Inverse Scattering Problem Technique and Construction of Exact Soliton Solutions}.
\newblock {\em Soviet Physics JETP}, 48(6):985--994, 1978.

\bibitem[Car73]{Carter}
Brandon Carter.
\newblock {Black hole equilibrium states}.
\newblock In C.~DeWitt and B.~S. DeWitt, editors, {\em Black Holes (Les Astres Occlus)}, pages 57--214, New York, 1973. Gordon and Breach.
\newblock Proceedings of the 1972 Les Houches Summer School.

\bibitem[CP02]{Calderbank2002}
D.~M.~J. Calderbank and H.~Pedersen.
\newblock {Self-Dual Einstein Metrics with Torus Symmetry}.
\newblock {\em Journal of Differential Geometry}, 60(3):485--521, 2002.

\bibitem[CT10]{ChenTeo2}
Y.~Chen and E.~Teo.
\newblock {Rod-Structure Classification of Gravitational Instantons with U(1)×U(1) Isometry}.
\newblock {\em Nuclear Physics B}, 838(1-2):207--237, 2010.

\bibitem[CT11]{Chen-Teo}
Y.~Chen and E.~Teo.
\newblock {A New AF Gravitational Instanton}.
\newblock {\em Physics Letters B}, 703(3):359--362, 2011.

\bibitem[Der83]{Der}
A.~Derdziński.
\newblock {Self-dual K\"ahler Manifolds and Einstein Manifolds of Dimension Four}.
\newblock {\em Compositio Mathematica}, 49(3):405--433, 1983.

\bibitem[DS14]{DS12}
S.~Donaldson and S.~Sun.
\newblock {Gromov–Hausdorff Limits of K\"ahler Manifolds and Algebraic Geometry}.
\newblock {\em Acta Mathematica}, 213(1):63--106, 2014.

\bibitem[DSW24]{Dai}
X.~Dai, Y.~Sun, and C.~Wang.
\newblock {The positive mass theorem for asymptotically flat manifolds with isolated conical singularities}.
\newblock {\em Science China Mathematics}, 67, 2024.

\bibitem[Ern68]{Ernst}
F.~J. Ernst.
\newblock {New Formulation of the Axially Symmetric Gravitational Field Problem}.
\newblock {\em Physical Review}, 167(5):1175--1178, 1968.

\bibitem[Gib79]{Gibbons}
G.~W. Gibbons.
\newblock {Gravitational Instantons: A Survey}.
\newblock In {\em Mathematical Problems in Theoretical Physics}, volume 116 of {\em Lecture Notes in Physics}, pages 282--287. Springer, 1979.

\bibitem[Har04]{Harmark}
T.~Harmark.
\newblock Stationary and axisymmetric solutions of higher-dimensional general relativity.
\newblock {\em Physical Review D}, 70(12):124002, 2004.

\bibitem[Haw77]{Hawking}
S.~W. Hawking.
\newblock {Gravitational Instantons}.
\newblock {\em Physics Letters A}, 60(2):81--83, 1977.

\bibitem[HIW11]{Hollands}
S.~Hollands, A.~Ishibashi, and R.~M. Wald.
\newblock A higher dimensional stationary rotating black hole must be axisymmetric.
\newblock {\em Communications in Mathematical Physics}, 302(3):631--674, 2011.

\bibitem[HKWX22]{HKWX}
Q.~Han, M.~Khuri, G.~Weinstein, and J.~Xiong.
\newblock {Asymptotic Analysis of Harmonic Maps with Prescribed Singularities}.
\newblock {\em arXiv preprint}, 2022.

\bibitem[HKWX25]{HKWX2}
Q.~Han, M.~Khuri, G.~Weinstein, and J.~Xiong.
\newblock {The Mass-Angular Momentum Inequality for Multiple Black Holes }.
\newblock {\em arXiv preprint}, 2025.

\bibitem[Joy95]{Joyce1995}
D.~D. Joyce.
\newblock {Explicit Construction of Self-Dual 4-Manifolds}.
\newblock {\em Duke Mathematical Journal}, 77(3):519--552, 1995.

\bibitem[KL21]{Kunduri}
H.~K. Kunduri and J.~Lucietti.
\newblock {Existence and Uniqueness of Asymptotically Flat Toric Gravitational Instantons}.
\newblock {\em Letters in Mathematical Physics}, 111(133), 2021.

\bibitem[Lap82]{Lapedes}
A.~S. Lapedes.
\newblock {Black Hole Uniqueness Theorems in Classical and Quantum Gravity}.
\newblock In S.-T. Yau, editor, {\em Seminar on Differential Geometry}, volume 102 of {\em Annals of Mathematics Studies}, pages 539--602. Princeton University Press, 1982.

\bibitem[LeB21]{LeBrun}
C.~LeBrun.
\newblock {Einstein Manifolds, Self-Dual Weyl Curvature, and Conformally K\"ahler Geometry}.
\newblock {\em Mathematical Research Letters}, 28:127--144, 2021.

\bibitem[Li23a]{Li2}
M.~Li.
\newblock {Classification Results for Conformally K\"ahler Gravitational Instantons}.
\newblock {\em arXiv preprint}, 2023.

\bibitem[Li23b]{Li1}
M.~Li.
\newblock {On 4-Dimensional Ricci-Flat ALE Manifolds}.
\newblock {\em arXiv preprint}, 2023.

\bibitem[Lot20]{Lott}
J.~Lott.
\newblock {The Collapsing Geometry of Almost Ricci-Flat 4-Manifolds}.
\newblock {\em Commentarii Mathematici Helvetici}, 95(1):79--98, 2020.

\bibitem[LS24]{LS2024}
M.~Li and S.~Sun.
\newblock {On Conical Asymptotically Flat Manifolds}.
\newblock {\em arXiv preprint}, 2024.

\bibitem[LT92]{LiTian}
Y.~Li and G.~Tian.
\newblock {Regularity of Harmonic Maps with Prescribed Singularities}.
\newblock {\em Communications in Mathematical Physics}, 149(1):1--30, 1992.

\bibitem[MM67]{MM}
R.~A. Matzner and C.~W. Misner.
\newblock {Gravitational Field Equations for Sources with Axial Symmetry and Angular Momentum}.
\newblock {\em Physical Review}, 154:1229--1232, 1967.

\bibitem[MS99]{MarsSimon}
M.~Mars and W.~Simon.
\newblock {On Uniqueness of the Taub-Bolt Instanton}.
\newblock {\em Classical and Quantum Gravity}, 16(11):3553--3560, 1999.

\bibitem[Ngu11]{Nguyen}
L.~Nguyen.
\newblock {Singular Harmonic Maps and Applications to General Relativity}.
\newblock {\em Communications in Mathematical Physics}, 301(2):411--441, 2011.

\bibitem[Nil24]{Nilsson}
Gustav Nilsson.
\newblock Topology of toric gravitational instantons.
\newblock {\em Differential Geometry and its Applications}, 96:Paper No.102171, 19 pp., 2024.

\bibitem[OR70]{OrlikRaymond}
P.~Orlik and F.~Raymond.
\newblock Actions of the torus on 4-manifolds. i.
\newblock {\em Transactions of the American Mathematical Society}, 152:531--559, 1970.

\bibitem[Ron96]{Rong}
X.~Rong.
\newblock {On the Fundamental Groups of Manifolds of Positive Sectional Curvature}.
\newblock {\em Annals of Mathematics}, 143(2):397--411, 1996.

\bibitem[War90]{Ward}
R.~S. Ward.
\newblock {Einstein-Weyl Spaces and $SU(\infty)$ Toda Fields}.
\newblock {\em Classical and Quantum Gravity}, 7(4):L95, 1990.

\bibitem[Wei20]{Weinstein}
G.~Weinstein.
\newblock {Harmonic Maps with Prescribed Singularities and Applications in General Relativity}.
\newblock {\em Advanced Studies in Pure Mathematics}, 85:479--489, 2020.

\bibitem[WS92]{WS}
A.~P. Whitman and W.~R. Stoeger.
\newblock {The Harmonic-Map Structure of the Axially Symmetric Stationary Einstein Equations}.
\newblock {\em General Relativity and Gravitation}, 24(6):641--658, 1992.

\bibitem[Wu19]{Wu}
P.~Wu.
\newblock {Einstein Four-Manifolds With Self-Dual Weyl Curvature of Nonnegative Determinant}.
\newblock {\em International Mathematics Research Notices}, 2021(2):1043--1054, 2019.

\bibitem[Yau82]{Yau1982problem}
S.-T. Yau.
\newblock {Problem Section}.
\newblock In S.-T. Yau, editor, {\em Seminar on Differential Geometry}, volume 102 of {\em Annals of Mathematics Studies}, pages 669--706. Princeton University Press, 1982.

\end{thebibliography}

\end{document}